\documentclass[leqno,12pt]{amsart}

\usepackage[eqnumber=subsubsection]{jdr-style}
\usepackage{jdr-macros,jdr-tikz}
\usepackage[all,tips]{xy}
\usepackage{longtable}
\usepackage{url}
\usepackage{verbatim}
\usepackage{tikz}
\usepackage{mathtools}
\usepackage{amscd}
\usepackage{xr}
\CompileMatrices

\mathtext{Int}
\mathtext{CH}
\mathtext{PL}
\mathtextb{sm}
\mathtextb{sp}
\mathtextb{can}
\mathtextb{rel}

\newtheorem{art}[subsection]{}

\newcommand{\T}{\bT}

\newcommand{\kcirc}{{ K^\circ}}
\newcommand{\ktilde}{{ \td{K}}}

\newcommand{\varphitrop}{{ \varphi_{\rm trop}}}
\newcommand{\torus}{\bT}

\let\id=\Id

\makeatletter
\@tfor\next:=ABCDEFGHIJKLMNOPQRSTUVWXYZ\do{%
  \expandafter\edef\csname\next S\endcsname{\noexpand\csname c\next\endcsname}%
  \expandafter\edef\csname\next cal\endcsname{\noexpand\csname c\next\endcsname}%
  \expandafter\edef\csname\next F\endcsname{\noexpand\csname f\next\endcsname}%
}
\makeatother

\def\Ocal{\mathscr O}
\def\OS{\mathscr O}
\def\LS{\mathfrak L}
\def\rG{\rm G}

\def\cAps{\cA_{\rm ps}}

\def\cAsm{\cA_{\rm sm}}
\def\cAsmc{\cA_{{\rm sm},c}}
\def\cDsm{\cD^{{\rm sm}}}

\def\bGm{\bG_{\rm m}}

\let\Linear=\bL
\let\divisor=\div
\newcommand{\metr}{\|\phantom{n}\|}
\newcommand{\absval}{|\phantom{n}|}

\newcommand{\z}{{(\Z,\Gamma)}}
\newcommand{\q}{{(\Q,\sqrt\Gamma)}}

\newcommand\secref[1]{\ref{#1}}

\bibalias{Lagerberg}{lagerberg12:super_currents}
\bibalias{CLD}{chambert_ducros12:forms_courants}
\bibalias{Jell}{jell16:poincare_lemma}
\bibalias{Liu}{liu20:tropical_cycle_classes}
\bibalias{JellDuality}{jell19:tropical_hodge_numbers}
\bibalias{Thuillier}{thuillier05:thesis}
\bibalias{Wanner}{wanner16:harmonic_functions}
\bibalias{JellOberwolfach}{jell19:oberwolfach}
\bibalias{gubler-martin}{gubler_martin19:zhangs_metrics}
\bibalias{Ducros}{ducros12:squelettes_modeles}
\bibalias{Berkovichetale}{berkovic93:etale_cohomology}
\bibalias{BerkovichSpectral}{berkovic90:analytic_geometry}
\bibalias{Bosch_lectures}{bosch14:lectures_formal_rigid_geometry}
\bibalias{GRW2}{gubler_rabinoff_werner:tropical_skeletons}
\bibalias{AllermannRau}{allermann_rau10:tropical_intersection_thy}
\bibalias{Gubler}{gubler16:forms_currents}
\bibalias{Gubler2}{gubler13:guide_tropical}
\bibalias{MaclaganSturmfels}{maclagan_sturmfel15:tropical_intro}
\bibalias{GK}{gubler_kunneman:tropical_arakelov}
\bibalias{Temkin_lecture_notes}{temkin15:berkovich_intro}
\bibalias{Fulton_intersection}{fulton98:intersec_theory}
\bibalias{gubler-crelle}{gubler98:local_heights_subvariet}
\bibalias{Vilsmeier}{vilsmeier21:monge_ampere}
\bibalias{Hartshorne}{hartshor77:algebraic_geometry}
\bibalias{Temkin00}{temkin00:local_properties}
\bibalias{StTe}{sturmfel_tevelev08:tropical_eliminat}
\bibalias{BPR}{baker_payne_rabinoff16:tropical_curves}
\bibalias{GKM}{gathmann_kerber_markwig09:tropical_fans}
\bibalias{BL}{bosch_lutkeboh85:stable_reduction_I}
\bibalias{KRZ1}{krzb16:uniform_bounds}
\bibalias{Wanner19}{wanner19:subharmonicity}
\bibalias{bosch-luetkebohmert1991}{bosch_lutkeboh91:degenerat_abelian_varieties}
\bibalias{foster-rabinoff-shokrieh-soto}{foster_rabinoff_etal18:tropical_theta_functions}
\bibalias{BRab}{baker_rabinoff14:skeleton_jacobian}
\bibalias{MikZharII}{mikhalkin_zharkov08:tropical_curves}
\bibalias{ABBR}{abbr14:lifting_harmonic_morphism_I}
\bibalias{BPR2}{baker_payne_rabinoff13:analytic_curves}
\bibalias{GJR2}{gubler_rabinoff:formsongraphs}

 \thanks{W.~Gubler and P.~Jell 
	were supported by the collaborative research 
	center SFB 1085 \emph{Higher Invariants---Interactions between Arithmetic Geometry and Global Analysis} funded by the Deutsche Forschungs\-gemeinschaft.}

\title{Forms on Berkovich spaces based on harmonic tropicalizations}

\author[W.~Gubler]{Walter Gubler}
\address{W. Gubler, Mathematik, Universit{\"a}t
	Regensburg, 93040 Regensburg, Germany}
\email{walter.gubler@mathematik.uni-regensburg.de}

\author[P.~Jell]{Philipp Jell}
\address{P. Jell, Mathematik, Universit{\"a}t
Regensburg, 93040 Regensburg, Germany}
\email{philipp.jell@mathematik.uni-regensburg.de}

\author[J.~Rabinoff]{Joseph Rabinoff}
\address{J. Rabinoff, Department of Mathematics, Trinity College of Arts and Sciences, Duke University, Durham, NC 27708, USA}
\email{jdr@math.duke.edu}

\begin{document}
\begin{abstract}
We introduce tropical skeletons for Berkovich spaces based on results of Duc\-ros. Then we study harmonic functions on good strictly analytic spaces over a non-trivially valued non-Archimedean field.
Chambert-Loir and Ducros introduced bigraded sheaves of smooth real-valued differential forms on Berkovich spaces by pulling back Lagerberg forms with respect to tropicalization maps.  We give a new approach in which we allow pullback by more general \defi{harmonic} tropicalizations to get a larger sheaf of differential forms with essentially the same properties, but with a better cohomological behavior.  A crucial ingredient is that tropical varieties arising from harmonic tropicalization maps are balanced.
\end{abstract}


\maketitle
\tableofcontents

\section{Introduction}  \label{section: introduction}

Arakelov theory is an arithmetic intersection theory on arithmetic varieties. The contributions of the Archimedean places are described analytically, while  intersection theory on models over valuation rings is used for the non-Archimedean places. It is an old dream to handle all places simultaneously in an analytic manner. This would have the advantage of replacing the models, which sometimes do not exist in sufficient generality, by the underlying Berkovich analytic spaces.

It is here that tropical geometry becomes very helpful. Lagerberg \cite{Lagerberg} has introduced
smooth $(p,q)$-forms  on $\R^n$  given in terms of coordinates
by
$$\alpha = \sum_{|I|=p, |J|=q}  \alpha_{IJ} \d'x_{i_1} \wedge \dots
\wedge \d'x_{i_p} \wedge \d''x_{j_1} \wedge \dots \wedge \d''x_{j_q}
$$
where $I$ (resp.\ $J$) consists of $i_1 < \dots < i_p$
(resp.\ $j_1 < \dots < j_q$) and $\alpha_{IJ}$ are $C^\infty$-functions on $\R^n$.
The smooth Lagerberg forms give rise to a bigraded differential sheaf $\cA^{\bullet,\bullet}$ of $\R$-algebras on $\R^n$ with alternating product $\wedge$ and with natural differentials $\d'$ and $\d''$ (see~\secref{Lagerberg forms}).
They are the analogues of the differential operators $\partial$ and
$\bar{\partial}$ in complex analysis.

One defines smooth Lagerberg forms on a weighted rational polyhedral complex $(\Pi,m)$ in $\R^n$ by restriction (see~\secref{Lagerberg forms}).  This is analogous to the definition of differential forms on singular complex analytic spaces. In applications, $(\Pi,m)$ will be a tropical variety. One has integrals and boundary integrals over $(\Pi,m)$ satisfying the formula of Stokes.

From now on, we consider a good, $d$-dimensional strictly analytic Berkovich space $X$ over a non-trivially valued non-Archimedean field $K$. For the basic notions used from non-archimedean geometry, we refer to \S \ref{Non-Archimedean geometry}. For a compact analytic subdomain $W$ of $X$, a  \emph{smooth tropicalization map} $F\colon W \to \R^n$ is given by $F(x)=(-\log|f_i(x)|)_{i=1,\dots,n}$ for invertible analytic functions $f_i$ on $W$. Chambert-Loir and Ducros \cite{CLD} used smooth tropicalization maps to lift smooth Lagerberg forms to $X$. In this way, they obtained a bigraded differential sheaf $\cA_{\rm sm}^{\bullet,\bullet}$ of $\R$-algebras on $X$ with alternating product $\wedge$ and with natural differentials $\d'$ and $\d''$ (see~\secref{Smooth forms and currents}).

This opens the door to the dream mentioned above. Many results hold as expected from complex geometry. Chambert-Loir and Ducros proved that there are analogues of smooth partitions of unity, of the formula of Stokes and of the Poincar\'e-Lelong formula. The $\d''$-operator induces an analogue of the Dolbeault complex and we can introduce the tropical Dolbeault cohomology by
\begin{equation} \label{intro: smooth Dolbeault cohomology}
H_{\rm sm}^{p,q}(X) \coloneqq \ker(\cA_{\rm sm}^{p,q}(X)\stackrel{\d''}{\longrightarrow} \cA_{\rm sm}^{p,q+1}(X))/\d''(\cA_{\rm sm}^{p,q-1}(X)).
\end{equation}
It was shown in \cite{Jell} that the sheaves $\cA_{\rm sm}^{p,q}$ satisfy a local Poincar\'e Lemma. If $X=Y^\an$ for a separated smooth scheme $Y$ of finite type over $K$, then Liu \cite{Liu} showed that there is an analogue of the cycle class map. However, it was realized in \cite{JellDuality} that the cohomology groups $H_{\rm sm}^{1,1}(X)$ can be infinite dimensional and that Poincar\'e duality can fail for smooth projective curves $X$ over $K$. The main problem can be already found in Thuillier's thesis \cite{Thuillier} and is a result of the fact that \emph{not all harmonic functions on curves are smooth}, as explained in Wanner's master thesis \cite{Wanner}.

It was suggested in \cite{JellOberwolfach} that we can overcome these problems by using more general tropicalization maps based on harmonic functions. The present paper gives the foundations of such a theory. With this approach, we will get a larger bigraded
differential sheaf $\cA^{\bullet,\bullet}$ of $\R$-algebras on $X$ consisting of \emph{weakly smooth forms}. The weakly smooth forms have all the nice properties of smooth forms, but they have a better cohomological behavior.

The notation and prerequisites for this paper are explained in Section~\ref{section: prerequisites}. We will use smooth forms and currents introduced  in \cite{CLD}. Some of their definitions are recalled later as a special case of our more general construction.

\subsection{Tropical Skeletons and Tropical Multiplicities} \label{intro:tropical skeletons and tropical multiplicities}
In Section~\ref{sec: tropical skeletons}, we assume that $X$ is compact and of pure dimension $d$, and we consider a morphism $\varphi \colon X \to \bT^\an$, where $\bT = \bGm^n$ is an $n$-dimensional split torus over $K$. In case $d=n$, Ducros \cite{ducros12:squelettes_modeles} has defined the skeleton $S_\varphi(X)$ as the preimage of the canonical skeleton of $\bT^\an$, and he has shown that the skeleton has the structure of a piecewise linear space  only depending on its underlying set and the analytic structure of $X$. Using Temkin's graded reductions, we will generalize this construction and define a \defi{tropical skeleton} $S_\varphi(X)$ for arbitrary $n$. By Proposition~\ref{prop:tropical.skeleton}, it is a closed subset of $X$ consisting only of Abhyankar points, and it depends only on the smooth tropicalization map $F \colon X \to \R^n$ induced by $\varphi$. We will see in Remark~\ref{rem: piecewise linear structure on tropical skeleton} that the tropical skeleton $S_\varphi(X)$ is a piecewise linear space of dimension at most $d$ closely related to the tropical variety $F(X)$. Indeed, the restriction of $F$ induces a surjective piecewise linear map $S_\varphi(X) \to F(X)$ with finite fibers. We will show in Remark~\ref{rem:trop.skel.GRW} that $S_\varphi(X)$ coincides with the tropical skeleton from \cite{gubler_rabinoff_werner16:skeleton_tropical} studied in the algebraic case.
In Section~\ref{sec: tropical multiplicities}, we will define and study positive \defi{tropical multiplicities $m_\varphi(X,x)$} at $x \in S_\varphi(X)$ and $m_\varphi(X,\omega)$ at $\omega \in F(X)$. The latter generalize the classical tropical multiplicities studied in tropical geometry.

\subsection{Formal and Semipositive Metrics}
Let $L$ be a line bundle on a paracompact good strictly analytic space $X$ over $K$. We consider a formal $\kcirc$-model $(\fX,\fL)$ of $(X,L)$, i.e.~an admissible formal scheme $\fX$ with generic fiber $X$ and a line bundle $\fL$ on $\fX$ with $\fL|_X=L$. Then we have an \defi{associated formal metric} $\metr_\fL$ on $L$ characterized by $\|s\|_\fL \equiv 1$ for any trivializing local section $s$ of $\fL$ (see~\cite[Section~2]{gubler_martin19:zhangs_metrics} for details). If we can choose $\fL$ such that the restriction of $\fL$ to the special fiber $\fX_s$ is nef, then we call $\metr_\fL$ \defi{semipositive}.

Let $L$ be a line bundle on a separated good strictly analytic space $X$ over $K$. We assume that the  boundary $\partial X$ is empty, which holds for example if $X$ is algebraic. The \defi{first Chern current} $c_1(L, \metr)$ of a continuous metric $\metr$ on $L$ is the current on $X$ defined locally by $\d'\d''[-\log\|s\|]$, where $s$ is a local trivializing section of $L$ and $[f]$ denotes the current associated to a continuous function $f$. As in complex analysis, this does not depend on the choice of $s$ and hence defines a current on $X$ (see~\cite[6.4.1]{chambert_ducros12:forms_courants}). We call such a current \defi{positive} if it is non-negative on compactly supported smooth positive forms (see~\secref{Smooth forms and currents}).
For $i \in \N$, let $\metr_i$ be a metric of $L$ such that for every $x \in X$ there is a strictly affinoid neighbourhood $W$ and $n_W  \in \N_{>0}$ such that the restriction of $\metr_i^{\otimes n_W}$ to $W$ is a semipositive formal metric. As formal metrics are continuous, we conclude that $\metr_i$ is a continuous metric on $L$.

\begin{thm} \label{semipositive and psh}
	Under the above assumptions, assume that the metrics $\metr_i$ converge uniformly to the metric $\metr$ of $L$. Then the first Chern current $c_1(L,\metr)$ is positive.
\end{thm}

Such  metrics play an important role for the contributions of non-Archimedean places in Arakelov theory and were originally introduced by Zhang \cite{zhang95} in an algebraic framework. We will prove Theorem~\ref{semipositive and psh} in Remark~\ref{formulation for metrics} based on the results from Sections~\ref{section: pl functions}--\ref{section: pl and harmonic functions}.

\subsection{Piecewise Linear and Harmonic Functions}
In Section~\ref{section: pl functions}, we will study piecewise $\Lambda$-linear functions on a good strictly analytic space $X$ over $K$ for a subgroup $\Lambda$ of $\R$.  Roughly, these are piecewise linear functions with slopes contained in $\Lambda$.

\begin{prop} \label{sheaf of piecewise linear functions}
Assume that $\Lambda = \Z$ or that $\Lambda$ is a divisible group containing $1$.
\begin{enumerate}
	\item \label{G-sheaf}
	The piecewise $\Lambda$-linear functions form a sheaf of $\Lambda$-modules on the $\rG$-topology of $X$, which is generated by the sheaf of piecewise $\Z$-linear functions (Remark~\ref{sec:PL.remarks}).
	\item \label{pull-back}
	Piecewise $\Lambda$-linearity is stable under pull-back (Lemma~\ref{lem:pl.basic.props}).
	\item \label{model}
	If $X$ is paracompact, then $h\colon X \to \R$ is piecewise $\Z$-linear if and only if there is  a formal $\kcirc$-model $(\fX,\fL)$ of $(X,\sO_X)$ with $h=-\log \|1\|_\fL$ (see~\secref{model theorem}).
	\item \label{analytic sheaf}
	Property \eqref{G-sheaf} holds also with respect to the analytic topology of $X$ (Proposition~\ref{prop:Lpl.analytic.nbhd}).
	\item \label{base extension}
	Piecewise  $\Lambda$-linearity is stable under extension $K'/K$ of the ground field (Lemma~\ref{lem:pl.basic.props}).
	\item \label{Galois}
	If $K'/K$  is normal in \eqref{base extension},  then base change gives precisely the ${\rm Aut}(K'/K)$-invariant piecewise $\Lambda$-linear functions on  $X\hat\otimes_K K'$ (Corollary~\ref{pl infinite Galois extensions}).
	\item \label{tropical}
	A real function $h \colon X \to \R$ is piecewise $\Z$-linear if and only if $h$ is locally the composition of a smooth tropicalization map with a piecewise $\z$-linear function on the corresponding tropical variety (Proposition~\ref{local description of piecewise linear}).

\end{enumerate}
Properties \eqref{G-sheaf}, \eqref{pull-back} and \eqref{model} characterize piecewise $\Lambda$-linear functions. We need the assumption $1 \in \Lambda$ to ensure that the piecewise $\Z$-linear functions are piecewise $\Lambda$-linear, which is used in~\eqref{G-sheaf}.
\end{prop}

In Section~\ref{sec:reduction-germs}, we  recall Temkin's reduction of germs, then we define the residue of a piecewise $\Lambda$-linear function at a point $x \in X$ as a canonical invertible sheaf on the reduction of the germ at $x$, following  \cite[Section~6]{chambert_ducros12:forms_courants}. This is an important tool used
in Section~\ref{section: pl and harmonic functions}, where we will explain the notion of harmonic functions. For non-Archimedean curves, they were investigated in Thuillier's thesis \cite{Thuillier}.
In any dimension, we can describe harmonic functions on good strictly analytic spaces as follows.

\begin{prop} \label{into:harmonic functions}
Assume that $\Lambda = \Z$ or that $\Lambda$ is a divisible group containing $1$.
\begin{enumerate}
	\item \label{harmonic are piecewise linear}
	Every $\Lambda$-harmonic function is piecewise $\Lambda$-linear (Definition~\ref{definition harmonic}).
	\item \label{sheaf of harmonic functions}
	The $\Lambda$-harmonic functions form a sheaf of $\Lambda$-modules generated by the sheaf of $\Z$-harmonic functions (Remark~\ref{sec:harmonic.remarks}, Proposition~\ref{prop: equivalence of harmonicity definitions}).
	\item \label{harmonic pullback}
	The  $\Lambda$-harmonic functions are stable under pull-back (Proposition~\ref{prop:harmonic.properties}).
	\item \label{integral harmonic}
	A function $h \colon X \to \R$ is $\Z$-harmonic if and only if every $x \in X$ has a strictly affinoid neighbourhood $W$ and a formal $\kcirc$-model $(\fW,\fL)$ of $(W,\sO_W)$ with $h|_W=-\log\|1\|_\fL$ such that $\fL$ restricts to a numerically trivial line bundle on $\fX_s$ (Proposition~\ref{prop:def.semipositive}).
	\item \label{rationality of harmonic}
	An $\R$-harmonic function is $\Lambda$-harmonic if and only if it is piecewise $\Lambda$-linear (see~\secref{sec:harmonic.remarks}).
	\item \label{harmonic and basechange}
	The $\Lambda$-harmonic functions are stable under base  extension $K'/K$ (Proposition~\ref{prop:harmonic.properties}).
	\item \label{Galois invariance harmonic}
	If $K'/K$  is normal in \eqref{harmonic and basechange},  then base change gives precisely the ${\rm Aut}(K'/K)$-invariant  $\Lambda$-harmonic functions on  $X\hat\otimes_K K'$ (Proposition~\ref{infinite Galois extensions}).
	\item  \label{curves}
	On a rig-smooth curve, a function is $\R$-harmonic if and only if it is harmonic in the sense of Thuillier (Proposition~\ref{Thuillier's harmonic}).
\end{enumerate}
Properties  \eqref{sheaf of harmonic functions}, \eqref{harmonic pullback} and  \eqref{integral harmonic} characterize $\Lambda$-harmonic functions.
\end{prop}

{In a subsequent paper, we will show that the maximum principle holds on $X \setminus \partial X$ for $\Lambda$-harmonic functions.}

\subsection{Piecewise Linear and Harmonic Tropicalizations} For a compact strictly analytic domain $W$ of $X$ of pure dimension $d$, we say that $h\colon W \to \R^n$ is a \defi{piecewise linear (resp.~harmonic) tropicalization} if the induced coordinate functions $h_i \colon W \to \R$ are piecewise $\Z$-linear (resp.~$\Z$-harmonic).
We will show in Section~\ref{section: weights for pl tropicalizations} that for every piecewise linear tropicalization $h\colon W \to \R^n$  we get \emph{canonical tropical weights} on the $d$-dimensional part $\Trop_h(W)_d \coloneqq h(W)_d$ of the \defi{tropical variety}  $\Trop_h(W) \coloneqq h(W)$. The latter is a union of rational polytopes in $\R^n$ of dimension at most $d$. The main idea is that piecewise linear tropicalizations are locally covered by smooth tropicalization maps for which we can use the tropical weights from Section~\ref{sec: tropical multiplicities}.
Note that smooth tropicalization maps are in particular harmonic tropicalizations; in this case the tropical multiplicities agree with the  ones from~\secref{intro:tropical skeletons and tropical multiplicities}. For a piecewise linear tropicalization $h$, we will define an associated \defi{$d$-skeleton $S_h(W)_d$} as a canonical compact piecewise linear space in $W$ associated to $h$, and we will endow its $d$-dimensional faces with canonical positive tropical weights such that
\begin{equation} \label{d-skeleton vs tropical variety}
h _*(S_h(W)_d) =\Trop_h(W)_d
\end{equation}
as an identity of weighted polytopal complexes in $\R^n$.
We will show that the underlying weighted polytopal complexes satisfy a functoriality property generalizing the well-known Sturmfels--Tevelev formula from tropical geometry.

\subsection{The Balancing Condition} It turns out that the key property of the smooth tropicalization maps used in~\cite{CLD} is that tropical varietes are balanced outside the image of the  boundary; this is more or less equivalent to the formula of Stokes. It is now crucial to extend this to harmonic tropicalizations.

\begin{thm} \label{intro: canonical tropical weights}
	Let $h\colon W \to \R^n$ be a harmonic tropicalization on a compact strictly analytic subdomain $W$ of $X$ of pure dimension $d$. Then the canonical tropical weights  satisfy the balancing condition in every $(d-1)$-dimensional face of $\Trop_h(W)$  not contained in $h({S_h(W)_d\cap}\partial W)$.
\end{thm}

The balancing condition will be explained in Section~\ref{section: balancing condition}, and Theorem~\ref{intro: canonical tropical weights} will be proven in Theorem~\ref{thm balancing} by first showing a local version for a germ $(X,x)$ at $x \in X$. The balancing theorem generalizes the well-known fact from tropical geometry that tropical varieties of closed subschemes of a split torus over $K$ are balanced. Note that if $W=X$ is algebraic, no boundary $\partial W$ occurs. Theorem~\ref{intro: canonical tropical weights} is also a generalization of \cite[Theorem~3.6.1]{CLD}, where the statement is shown for smooth tropicalization maps. We use this to show  first a local version for smooth germs of tropicalization maps, then we will use that every harmonic tropicalization is locally covered by a smooth tropicalization map, and finally we will employ a trick with the graph to descend the balancing condition. Here, Theorem~\ref{semipositive and psh}
 and the tropical Poincar\'e--Lelong formula are applied.

\subsection{Weakly Smooth Forms}
In Section~\ref{section: differential forms}, we will introduce weakly smooth forms on good strictly analytic spaces over $K$.
The idea is to lift Lagerberg forms by using harmonic tropicalizations, in analogy with the definition of smooth forms~\cite{CLD}. The properties of weakly smooth forms are summarized as follows:

\begin{thm} \label{intro: characterization weakly smooth}
There is a bigraded differential sheaf $\AS^{\bullet,\bullet}$ of $\R$-algebras on $X$ with an alternating product $\wedge$ and differentials $\d',\d''$ whose elements are called weakly smooth forms. These sheaves satisfy the following properties:
\begin{enumerate}
	\item \label{functoriality of pull-back}
	For a morphism $f\colon X' \to X$, there is a homomorphism $f^*\colon \cA_X \to  f_*\cA_{X'}$ of sheaves of bigraded differential $\R$-algebras. This pull-back is functorial (see~\secref{functoriality}).
	\item \label{harmonic lift}
	If $X$ is compact with a harmonic tropicalization map $h\colon X \to \R^n$, there is an injective homomorphism $h^*\colon \cA^{\bullet,\bullet}(h(X))\to \cA^{\bullet,\bullet}(X)$ of bigraded differential $\R$-algebras lifting smooth Lagerberg forms on the tropical variety $h(X)$ to $X$ (Proposition~\ref{harmonic lift to weakly smooth forms}).
    \item \label{functoriality of harmonic lift}
    If $X,h$ are as in \eqref{harmonic lift} and $f\colon X' \to X$ is a morphism of compact good strictly analytic spaces over $K$, then $(h\circ f)^*=f^* \circ h^*$ (Proposition~\ref{harmonic lift to weakly smooth forms}).
    \item \label{locally given by lifts}
    For $\omega \in \cA(X)$, any $x \in X$ has a strictly affinoid neighbourhood $W$ with a harmonic tropicalization $h \colon W \to \R^n$ such that $\omega|_W=h^*(\alpha)$ for some $\alpha \in \cA(h(W))$ (see~\secref{piecewise smooth forms}).
    \item \label{relation to smooth forms}
    The sheaf $\cA_{\rm sm}^{\bullet,\bullet}$ of smooth forms on $X$ introduced by Chambert-Loir and Ducros is a subsheaf of bigraded differential $\R$-algebras of $\cA^{\bullet,\bullet}$ (Proposition~\ref{subsheaves of forms}).
    \item \label{characterization of harmonic functions}
    The elements of $\cA^{0,0}(X)$ are continuous functions on $X$. A real function $f$ on $X$ is $\R$-harmonic if and only if $f\in \cA^{0,0}(X)$ and $\d'\d''f=0$ (Proposition~\ref{weakly smooth harmonic functions}).
\end{enumerate}
The sheaves  $\cA^{\bullet,\bullet}$ of bigraded differential $\R$-algebras are characterized up to unique isomorphism by properties \eqref{functoriality of pull-back}, \eqref{harmonic lift}, \eqref{functoriality of harmonic lift} and \eqref{locally given by lifts}.
\end{thm}

\subsection{Integration and Strong Currents}
In Section~\ref{section: integration and currents}, we give a theory of integration for weakly smooth forms. Since weakly smooth forms are piecewise smooth, this can be easily deduced from the integration theory for smooth forms in \cite{CLD} and the inclusion-exclusion formula. We will prove  the formulae of \defi{Stokes} (Theorem~\ref{theorem of Stokes})
and \defi{Green} (Theorem~\ref{theorem of Green}) for weakly smooth forms.

In Section~\ref{section: strong currents}, we will introduce strong currents on a pure dimensional separated strictly analytic space $X$ with no boundary as a topologically dual notion of compactly supported weakly smooth forms. As the  boundary of $X$ is empty, the balancing condition in Theorem~\ref{intro: canonical tropical weights} yields that the boundary integral in the formula of Stokes vanishes, and hence
we can embed the space of weakly smooth forms into the space of strong currents. We will show that the \defi{Poincar\'e--Lelong formula} holds in the sense of strong currents (Theorem~\ref{Poincare-Lelong}).

\subsection{Dolbeault Cohomology and Curves}
The Dolbeault complex $0 \to \AS^{p,0}(X) \to \dots \to \AS^{p,d}(X) \to 0$ for $\d''$ induces the \defi{Dolbeault cohomology groups} $H^{p,q}(X)$, now using the sheaf of weakly smooth forms $\AS$ and not the sheaf $\AS_{\rm sm}$ of smooth forms as in \eqref{intro: smooth Dolbeault cohomology}. We will show in Section~\ref{section: Dolbeault cohomology} that this new Dolbeault cohomology satisfies also a \defi{local Poincar\'e Lemma} (Theorem~\ref{Poincare Lemma in the stalks}) and we will give an analogue of \emph{Liu's cycle class map} for this new Dolbeault cohomology with the same properties as before (Theorem~\ref{Liu's cycle class theorm}).

To see that $H^{p,q}(X)$ is better behaved than $H_{\rm sm}^{p,q}(X)$ from \eqref{intro: smooth Dolbeault cohomology}, we consider a compact rig-smooth curve $X$. Then
 $H^{p,q}(X)$ has the (finite) dimension  expected from the case of compact Riemann surfaces, and Poincar\'e duality holds if $X$ is proper. This will be shown in our companion paper \cite{GJR2}. As a special case, we mention here the following result:

\begin{thm} \label{intro: Hodge diamond}
Let $Y$ be a smooth projective curve over $K$ and let $b$ the first Betti number of $X=Y^\an$. Then  $\dim(H^{p,q}(X))$ is equal to $1$ (resp.~$b$) if $p+q \in \{0,2\}$ (resp.~if $p+q=1$).
\end{thm}

The idea is to introduce $(p,q)$-forms on a weighted metric graph with boundary and to compute the Dolbeault cohomology first in this combinatorial setting. We will then show that if $K'/K$ is a finite Galois extension such that $X' \coloneq X\tensor_K K'$ admits a strictly semistable model, then $H^{p,q}(X')$ is canonically isomorphic to the Dolbeault cohomology of its associated skeleton.  By a Galois descent argument, the theorem follows.
Although Theorem~\ref{intro: Hodge diamond} will be proven in \cite{GJR2}, it relies on the foundations developed here. In particular, we will use the results about the Galois action on weakly smooth forms and on Dolbeault cohomology discussed in Section~\ref{section: Galois action}.

In Section~\ref{section: results for curves}, we will restrict ourselves to the case of a rig-smooth curve. We will show that our notion of harmonic functions agrees with Thuillier's notion from \cite{Thuillier}. If the curve has a strictly semistable model, then we will compare  harmonic functions on the curve to harmonic functions on the skeleton. At the end, we will give a regularization result (Theorem~\ref{regularization theorem}) for subharmonic functions on a rig-smooth curve generalizing a theorem of Wanner.

\subsection{Canonical Tropicalizations for Abelian Varieties}
In Section~\ref{section: canonical tropicalization of abelian varieties}, we will show that for an abelian variety $A$ over $K$, the canonical tropicalization map $\trop \colon A^\an \to N_\R$ is a harmonic tropicalization. Here, we denote by $N$ the cocharacter lattice of the torus $\bT$ involved in Raynaud's uniformization theory for $A$
. In Section~\ref{abelian threefold}, we will give an example of an abelian threefold with mixed reduction such that the canonical tropicalization is not a smooth tropicalization map. This shows that weakly smooth forms arise very naturally for abelian varieties.

Let $M_\R=\Hom(N,\R)$ and $\Lambda^{p,q} M_\R \coloneqq \Lambda^p M_\R \otimes_\R \Lambda^q M_R$ be the bigraded exterior algebra of $M_\R$. Then there is a canonical injective homomorphism $\Lambda^{p,q} M_\R \to H^{p,q}(A)$ for all $p,q \in \N$. This was shown by Miriam Prechtel in her thesis \cite{prechtel23} based on results of this paper. We do not know if this is an isomorphism as in the complex case.

\subsection{Acknowledgments}
We thank the participants of the Oberseminar Arakelov Theory in the summer term 2021 for  useful comments about an earlier draft of the paper and Andreas Mihatsch for his remarks on the first arXiv version and the referee for the suggestions to improve the presentation of the paper.

\section{Prerequisites} \label{section: prerequisites}

\subsection{General Notation and Conventions} \label{general notation and conventions}
The set of natural numbers $\N$ includes $0$. A set-theoretic inclusion $S \subset T$ allows equality.  For an additive subgroup $\Lambda\subset\R$ we let $\sqrt\Lambda$ denote its divisible hull. If $P$ is any abelian group, we let $P_\Lambda \coloneqq P \otimes_\Z \Lambda$ be the tensor product with $\Lambda$.

Rings are always assumed to be commutative and with $1$.
For a ring $A$, the group of invertible elements is denoted by $A^\times$.
A \defi{variety} over a field $F$ is an integral scheme which is of finite type and separated over $\Spec\, F$.

A space is \defi{compact} if it is quasi-compact and Hausdorff. A \defi{paracompact space} is a Hausdorff space such that every open covering has a locally finite open refinement.

\subsection{Non-Archimedean Fields} \label{nonArch fields}
Throughout the whole paper, we fix a field $K$ that is complete with respect to a non-trivial non-Archimedean complete absolute value denoted $\absval\colon K\to \R$.  We will call this a \defi{non-trivially valued non-Archimedean field}.  A \defi{non-Archimedean extension field} is a non-Archimedean field $K'$ equipped with an isometric embedding $K\inject K'$.

We will make use of Temkin's theory of graded residue fields introduced in~\cite{temkin04:local_properties_II}; see also  \cite[2.2]{ducros14:structur_des_courbes_analytiq} and \cite[Section~0]{ducros12:squelettes_modeles}.
We use $\td K$ to denote the ordinary residue field and $\td K^\bullet$ for the graded residue field.  Beware that Temkin and Ducros use $\td K^1$ for the ordinary residue field and $\td K$ for the graded residue field.

\medskip\noindent
We will use the following notation concerning non-Archimedean fields. \\[1mm]
\null\begin{tabular}{rl}
  \makebox[1.5cm][r]{\hfil$K$} & A non-Archimedean field. \\
  $\absval$ & $\colon K\to\R$, the absolute value on $K$. \\
  $v(\scdot)$ & $=-\log|\scdot|$, the associated valuation. \\
  $K^\circ$ & The valuation ring of $K$. \\
  $K^{\circ\circ}$ & The maximal ideal in $K^\circ$. \\
  $\td K$ & $= K^\circ / K^{\circ\circ}$, the residue field of $K$. \\
  $\Gamma$ & $= v(K^\times)$, the value group of $K$. \\
  $\td K^{\leq r}$ & $= \{x\in K\mid |x|\leq r\}$ for $r\in\R$. \\
  $\td K^{<r}$ & $= \{x\in K\mid |x|< r\}$ for $r\in\R$. \\
  $\td K^r$ & $= \td K^{\leq r}/\td K^{<r}$. \\
  $\td K^\bullet$ & $= \Dsum_{r\in\R}\td K^r$, the graded residue field. \\
  $e(L/K)$ & $= [|L^\times|: |K^\times|]$, the ramification index of a field extension.  \\
  $f(L/K)$ & $= [\td L : \td K]$, the residue degree.  \\
\end{tabular}
\medskip

\noindent
We have
\begin{equation}\label{eq:n.ef}
  [\td L^\bullet:\td K^\bullet] = e(L/K)f(L/K)
\end{equation}
by \cite[2.2.27.1]{ducros14:structur_des_courbes_analytiq}. 
The left hand side is defined as the dimension of $\td L^\bullet$ as a graded $\td K^\bullet$-vector space, which is the cardinality of any basis as in ordinary linear algebra  \cite[2.2.16]{ducros14:structur_des_courbes_analytiq}.

\subsection{Non-Archimedean geometry}  \label{Non-Archimedean geometry}
We will use Berkovich analytic spaces  as defined in~\cite{Berkovichetale}. We are mainly concerned with good analytic spaces $X$ over a non-trivially valued non-Archimedean field $K$. These are the analytic spaces considered in~\cite{BerkovichSpectral}, so this book is our main reference for non-Archimedean geometry. Here, \defi{good} means that every $x \in V$ has a neighbourhood isomorphic to the Berkovich spectrum $\sM(\sA)$ of an affinoid algebra $\sA$.

When considering strictly analytic spaces, we will always use the $\rG$-topology  generated by the strictly affinoid domains \cite[1.3]{berkovic93:etale_cohomology}. This is most natural for comparison to rigid analytic geometry: see~\cite[1.6]{Berkovichetale}.

\medskip\noindent
We will use the following notation concerning Berkovich analytic spaces.  Let $X$ be an analytic space and let $\sA$ be a $K$-affinoid algebra.
\null\begin{longtable}{rl}
    \makebox[1.5cm][r]{\hfil$
      \sM(\sA)$} & The Berkovich spectrum of the affinoid algebra $\sA$. \\
  $\sA^\circ$ & The subring of power-bounded elements of $\sA$. \\
  $\sA^{\circ\circ}$ & The ideal of topologically nilpotent elements of $\sA^\circ$. \\
  $\td\sA$ & $= \sA^\circ/\sA^{\circ\circ}$, the canonical reduction of $\sA$. \\
  $\Ocal_X$ & The sheaf of analytic functions on a good analytic space $X$. \\
  $\kappa(x)$ & The residue field $\sO_{X,x}/\fm_{X,x}$ at a point $x\in X$. \\
  $\sH(x)$ & $=\kappa(x)^\wedge$, the completed residue field at a point $x\in X$. \\
  $\chi_x$ & $\colon\sA\to\sH(x)$, the canonical homomorphism for $x\in\sM(\sA)$. \\
  $\td\chi_x$ & $\colon\td\sA\to\td\sH(x)$, the reduction of $\chi_x$. \\
  $\td\chi_x^\bullet$ & $\colon\sA\to\td\sH(x)^\bullet$, the composition of $\chi_x$ with $\sH(x)\to\td\sH(x)^\bullet$. \\
        $\del X$ & The \defi{boundary} of $X$~\cite[2.5]{BerkovichSpectral}. \\
  $\Int(X)$ & $= X \setminus \del X$, the \defi{interior} of $X$. \\
  $\del(X/S)$ & The \defi{relative boundary} of a morphism $X\to S$. \\
  $\Int(X/S)$ & $= X\setminus\del(X/S)$, the \defi{relative interior} of $X\to S$. \\
        $\sX^\an$ & The analytification of a $K$-scheme $\sX$ of finite type~\cite[3.4]{BerkovichSpectral}. \\
     \end{longtable}

Note that $\sX^\an$ is good strictly analytic space over $K$ with empty boundary.

\subsection{Formal geometry and models}  \label{Formal geometry and models}
We will make extensive use of admissible formal models of analytic spaces.  We refer the reader to~\cite{Bosch_lectures,bosch_lutkeboh93:formal_rigid_geometry_I} for the definitions. 

\medskip\noindent
We will use the following notation concerning admissible formal schemes.  Let $\fX$ be an admissible formal $K^\circ$-scheme, i.e., a flat formal $K^\circ$-scheme locally of topological finite type such that $\fX$ has a locally finite atlas by affine formal schemes.

\null\begin{tabular}{rl}
  \makebox[1.5cm][r]{\hfil$\fX_\eta$} & The \defi{generic fiber} of $\fX$, a  strictly $K$-analytic space. \\
  $\fX_s$ & The \defi{special fiber} of $\fX$, a $\td K$-scheme locally of finite type. \\
  $\red_\fX$ & $\colon\fX_\eta\to\fX_s$, the functorial \defi{reduction map}. \\
\end{tabular}

\medskip
An affine admissible formal $K^\circ$-scheme is the formal spectrum $\Spf(A)$ of a  $K^\circ$-algebra $A$ which is topologically of finite type and flat over $K^\circ$. 
The generic fiber of $\Spf(A)$ is defined as the spectrum of the $K$-affinoid algebra $A\tensor_{K^\circ}K$; this construction globalizes under the above assumptions to give the generic fiber of $\fX$.
If $\fX$ is separated, then it follows from our above finiteness assumption on the atlas and from \cite[Theorem~1.6.1]{berkovic93:etale_cohomology} that the generic fiber $\fX_\eta$  is a paracompact strictly analytic space. Below we will see a converse.

Let $X$ be a good  paracompact strictly analytic space over $K$.
A \defi{formal $\kcirc$-model $\fX$ of $X$} is an admissible formal $\kcirc$-scheme $\fX$
with an identification $\fX_\eta = X$.   We say that a formal $\kcirc$-model $\fX'$ \defi{dominates} $\fX$ if there exists a morphism $\fX'\to\fX$ inducing the identity on generic fibers.  A theorem of Raynaud states that the set of isomorphism classes of $\kcirc$-models of $X$ is non-empty and filtered with respect to the dominance relation: see~\cite[8.4]{Bosch_lectures} and~\cite[2.4]{gubler-martin}. A paracompact space is Hausdorff, and as $X$ is also good, it follows that $X$ is a separated Berkovich space \cite[Proposition~3.1.5]{berkovic90:analytic_geometry}.  Hence we see that every formal $\kcirc$-model of $X$ is separated by using \cite[Proposition~4.7]{bosch_lutkeboh93:formal_rigid_geometry_I}.

If $K$ is algebraically closed and $X$ is reduced, then for any formal $\kcirc$-model $\fX$ of $X$ there is a dominating $\kcirc$-model $\fX'$ of $X$ with reduced special fiber and with $\fX' \to \fX$ finite. There is a minimal $\kcirc$-model with these properties. It is given over a formal affine open subset $\Spf(A)$ by $\Spf((A\otimes_\kcirc K)^\circ)$.  See~\cite[1.10]{gubler98:local_heights_subvariet} and~\cite[Proof of Proposition~3.5]{gubler_martin19:zhangs_metrics}.

A formal $\kcirc$-model $\fX$ of $X$ can be modified by an \defi{admissible formal blowing up}. This is a dominating morphism $\pi\colon \fX' \to \fX$ obtained by a blowing up a center supported in $\fX_s$. See~\cite[8.2]{Bosch_lectures} for details.

Let $L$ be a line bundle on $X$, i.e., a locally free $\Ocal_X$-module of rank $1$.  A \defi{formal $\kcirc$-model $(\fX,\LS)$ of $(X,L)$} is a line bundle $\LS$ on a formal $\kcirc$-model $\fX$ of $X$ together with an identification $\LS|_X = L$.

\subsection{Polyhedral geometry} \label{Polyhedral geometry}
In this subsection, we collect the notions from convex geometry that we will use throughout the paper.  We suggest that the reader browse through this subsection in a first read-through to become familiar with our conventions regarding polyhedra, polytopes, polyhedral complexes, subcomplexes, and fans, as there are several competing definitions used in the literature. At a later stage, one might come back to this subsection if needed.

Usually, we have the following setup for convex geometry.\\[1mm]
\null\begin{tabular}{rl}
  \makebox[1.5cm][r]{\hfil$N$} & A \defi{lattice}, i.e.\ a free $\Z$-module of finite rank. \\
   $N_\R$ & $= N\tensor\R$, the ambient real vector space. \\
   $N_\Gamma$ & $= N\tensor\Gamma$ for an additive subgroup $\Gamma\subset\R$. \\
   $M$ & $=\Hom(N,\Z)$, the dual lattice. \\
   $M_\R$ & $=\Hom(N,\R)$, the dual vector space. \\
   $M_\Gamma$ & $=\Hom(N,\Gamma)$ for an additive subgroup $\Gamma\subset\R$. \\
   $\angles{\scdot,\scdot}$ & $\colon M_\R\times N_\R\to\R$, the evaluation pairing. \\
\end{tabular}

\medskip
Now we recall some notions from convex geometry in~$N_\R$.  These notions are often introduced in the special case of the lattice $\Z^n$ inside $\R^n$.  We fix a subring $R$ of $\R$ and an $R$-submodule $\Gamma$ of $\R$. In applications, we use this mainly for $R=\Z$ and $\Gamma = v(K^\times)$.

An \defi{affine function} on $N_\R$ is a function of the form the form $u+c$ with $u \in M_\R$ and $c \in \R$. We say that the affine function is \defi{$(R,\Gamma)$-linear} if $u \in M_\R$ and $c \in \Gamma$.  For another lattice $N'$, we say that a function $L\colon N_\R' \to N_\R$ is \defi{$(R,\Gamma)$-linear} if $f \circ L$ is  $(R,\Gamma)$-linear on $N_\R'$ for each $(R,\Gamma)$-linear function $f\colon N_\R \to \R$.

In the above and what follows, we abbreviate the prefix ``$(R,\R)$-'' (in which no condition is required of the translation) simply by ``$R$-''

A \defi{polyhedron} $\sigma$ in $N_\R$ is a finite intersection of half-spaces $H_i \coloneqq \{f_i \geq 0\}$ for affine functions $f_i$ on $N_\R$. If we can choose all $f_i$ to be $(R,\Gamma)$-linear, then we say that $\sigma$ is an $(R,\Gamma)$-polyhedron.  A \defi{polytope} is a bounded polyhedron.
A  \defi{face} of a polyhedron $\sigma$ is the intersection of $\sigma$ with the boundary of a half-space $H \supset \sigma$. We also allow $\sigma$ and $\emptyset$ as faces of $\sigma$.

A \defi{polyhedral complex} in $N_\R$ is a finite collection $\Pi$ of polyhedra such that for any $\sigma \in \Pi$, all faces of $\sigma$ are contained in $\Pi$, and such that for $\sigma, \rho \in \Pi$, the intersection $\sigma \cap \rho$ is a face of both $\sigma$ and $\rho$.  A \defi{$(R,\Gamma)$-polyhedral complex} is a polyhedral complex consisting of $(R,\Gamma)$-polyhedra.
If all $\sigma \in \Pi$ are polytopes, then $\Pi$ is called a \defi{polytopal complex}. The \defi{support} $|\Pi|\coloneqq \bigcup_{\sigma \in \Pi} \sigma$ of a polytopal complex $\Pi$ is called a \defi{polytopal set}.

A \defi{subcomplex} of $\Pi$ is a polyhedral complex $\Pi' \subset \Pi$ 
A \defi{subdivision} of $\Pi$ is a polyhedral complex $\Pi''$ with $|\Pi|=|\Pi''|$, such that each $\sigma \in \Pi$ is a union of polyhedra in $\Pi''$.

A \defi{(rational) fan} in $N_\R$ is a $(\Z,\{0\})$-polyhedral complex. It consists of $\Z$-polyhedral cones centered at $0$. For a $\Z$-polyhedral complex $\Pi$ in $N_\R$ and $\omega \in |\Pi|$, we define the \defi{star around $\omega$} as $\{\sigma \in \Pi \mid \omega \in \sigma\}$. By translation to the origin, the star generates a fan in $N_\R$.

For a $\Z$-polyhedron $\sigma$ in $N_\R$, let $\Linear_\sigma$ be the $\R$-linear subspace of $N_\R$ generated by $\{\omega_1-\omega_2 \mid \omega_1,\omega_2 \in \sigma\}$. Note that $\Linear_\sigma$ is generated by the saturated lattice $N_\sigma \coloneqq N \cap \Linear_\sigma$.

A \defi{weighted polyhedral complex of dimension $d$} is a polyhedral complex $\Pi$ in $N_\R$ such that any maximal $\sigma \in \Pi$ is $d$-dimensional and comes with a \defi{weight} $m_\sigma \in \Z$. We consider  weighted $(R,\Gamma)$-polyhedral complexes $(\Pi,m)$ and $(\Pi',m')$ to be equivalent if we can replace them by $(R,\Gamma)$-subdivisions such that they agree outside the weight zero part. More precisely, this means that there is a subdivision $\Pi''$ of $\Pi$ such that for every $d$-dimensional face $\sigma$ of $\Pi''$ contained in some $\Delta \in \Pi$ with $m_\Delta \neq 0$, there exists $\Delta' \in \Pi'$ with $\sigma \subset \Delta'$ and $m_{\Delta'}=m_\Delta$, and the same holds with the role of $\Pi$ and $\Pi'$ reversed.
Equivalence classes of weighted $(R,\Gamma)$-polyhedral complexes form an abelian group by adding the weights of a joint $(R,\Gamma)$-subdivision. The zero element is given by $0 \coloneqq (N_\R,0)$.

An \defi{abstract polyhedral complex} is a finite collection $\Pi$ of polyhedra satisfying the same conditions as above, where each $\sigma \in \Pi$ is a full-dimensional polyhedron in a real vector space $\R^{n_\sigma}$, but where we do not require that all polyhedra be contained in the same ambient space $\R^n$. We call $\Pi$ an \defi{abstract $(R,\Gamma)$-polyhedral complex} if for a face $\tau$ of $\sigma \in \Pi$, the induced affine map $\R^{n_\tau}\to \R^{n_\sigma}$ is $(R,\Gamma)$-linear.  An \defi{$(R,\Gamma)$-piecewise linear space} is the support of an abstract polyhedral complex, and the \defi{$(R,\Gamma)$-piecewise linear structure} is given by the polyhedra of $\Pi$, well-defined up to $(R,\Gamma)$-subdivison.

A function $L\colon S' \to S$ between piecewise $(R,\Gamma)$-linear spaces $S,S'$ is a \defi{piecewise $(R,\Gamma)$-linear map} if there exist polyhedral complexes $\Pi, \Pi'$ with $|\Pi|=S,|\Pi'|=S'$ such that for every $\rho \in \Pi'$, there is $\sigma \in \Pi$ with $L(\rho)=\sigma$, and $L|_\rho$ induces an $(R,\Gamma)$-linear map $\R^{n_{\rho}} \to \R^{n_\sigma}$.

We are mainly concerned with piecewise $(\Z,\Gamma)$-linear maps $f\colon|\Pi| \to \R$ on a weighted $\z$-polyhedral complex $(\Pi,m)$ of dimension $d$. Such functions play the role of rational functions in tropical intersection theory. We have an associated \defi{tropical Weil divisor} $\divisor(f)$ which is a $(d-1)$-dimensional weighted $\z$-polyhedral  subcomplex of a subdivision of $\Pi$.  See~\cite[Section~6]{AllermannRau}.

Let $S$ be a piecewise $(\Z,\Gamma)$-linear space. A weighted abstract $(\Z,\Gamma)$-polyhedral complex of dimension $d$ is given by endowing the $d$-dimensional faces $\Delta$ of an abstract polyhedral complex $\Pi$ defining $S$ with weights $m_\Delta \in \Z$. The resulting weighted polyhedral complex is denoted by $(\Pi_d,m)$. As above,  weighted abstract polyhedral complexes are \defi{equivalent} if they agree after subdivision. For a piecewise $(\Z,\Gamma)$-linear map $L \colon S \to S'$ of piecewise $(\Z,\Gamma)$-linear spaces,  the \defi{push-forward} $L_*(\Pi_d,m)$ is a weighted $(\Z,\Gamma)$-polyhedral complex  of dimension $d$ on $S'$, well-defined up to equivalence. See \cite[Section~7]{allermann_rau10:tropical_intersection_thy}.

\subsection{Tori and tropicalizations} \label{Tropical geometry}
Let $\torus$ be an $n$-dimensional split torus over the non-Archimedean field $K$, and let $N$ be the cocharacter lattice of $\torus$.  Then the \defi{tropicalization map} is the function $\trop\colon \torus^\an \to N_\R$ defined as follows.  For $u$ in the character lattice $M = \Hom(N,\Z)$, let $\chi^u$ be the corresponding character of $\torus = \Spec(K[M])$. Then the image $\trop(x)\in N_\R$ of $x \in\torus^{\an}$ is characterized by $$\angles{u,\trop(x)}  = -\log|\chi^u(x)|$$
for all $u \in M$. If we choose toric coordinates $x_1,\dots,x_n$ on $\torus$ identifying $N$ with $\Z^n$, then $\trop(x) = (-\log|x_1(x)|,\ldots,-\log|x_n(x)|)$ under the identification $N_\R\cong\R^n$.  Working analytically has the advantage that $\trop$ is a continuous proper map.

There is a canonical section $\iota\colon N_\R \to\bT^\an = \Spec(K[M])^\an$ of the tropicalization map $\trop\colon \bT^\an \to N_\R$ defined by
\begin{equation}\label{eq:skeleton.norm}
 -\log\bigg\|\sum_{u\in M}a_u\xi^u\bigg\|_{\iota(\omega)} = \min\bigl\{ v(a_u) +
 \angles{u,\omega}\mid a_u\neq 0 \bigr\}.
\end{equation}
Its image is the \defi{canonical skeleton} $S(\bT)$ of $\bT$, defined in~\cite[6.3]{BerkovichSpectral}.

Given an analytic space $X$, a morphism $\phi\colon X\to\bT^\an$ is called a \defi{moment map}.  We let $\phi_{\trop} = \trop\circ\phi\colon X\to N_\R$ be the corresponding \defi{smooth tropicalization map}.

\medskip\noindent
We will use the following notations concerning tropical geometry and tori. \\[1mm]
\null\begin{tabular}{rl}
  \makebox[1.5cm][r]{$\bT$} & A split torus over $K$. \\
  $M$ & The character lattice of $\bT$. \\
  $N$ & The cocharacter lattice of $\bT$. \\
  $\xi^u$ & The character of $\bT$ corresponding to $u\in M$. \\
  $\trop$ & $\colon \bT^\an\to N_\R$, the tropicalization map. \\
  $\iota$ & $\colon N_\R\to\bT^\an$, the section of $\trop$. \\
  $S(\bT)$ & $= \iota(N_\R)$, the \defi{canonical skeleton} of $\bT$. \\
  $\phi_{\trop}$ & $= \trop\circ\phi$, the smooth tropicalization map associated to $\phi\colon X\to\bT^\an$.
\end{tabular}

\subsection{Lagerberg forms} \label{Lagerberg forms} Let $N$ be a lattice of rank $n$. Lagerberg \cite{Lagerberg} has introduced the sheaf
$\cA^{p,q}\coloneqq\cA^p \otimes_{C^\infty} \cA^q$ of $(p,q)$-forms on $N_\R$, where on the right we use the sheaf $\cA^\bullet$ of usual smooth differential forms on $N_\R$. Such $(p,q)$-forms are called \defi{smooth Lagerberg forms}. As in complex geometry, the smooth Lagerberg forms give rise to a bigraded differential sheaf $\cA^{\bullet,\bullet}$ of $\R$-algebras with an alternating product $\wedge$ and natural differentials $\d'\colon\cA^{p,q} \to \cA^{p+1,q}$ and $\d''\colon\cA^{p,q} \to\cA^{p,q+1}$.

The \emph{Lagerberg involution} is the unique involution $J$ of the sheaf  $\cA$ of $\R$-algebras which leaves the smooth functions fixed and which satisfies $J\d'=\d''J$. It maps $\cA^{p,q}$ isomorphically onto $\cA^{q,p}$.
For an open subset $\Omega$ of $N_\R$, we call $\omega \in \cA^{p,p}(\Omega)$ a \defi{positive Lagerberg form} if there are $\alpha_1, \dots, \alpha_m \in \cA^{p,0}(\Omega)$ and smooth functions $f_j \colon \Omega \to \R_{\geq 0}$ with
$$\omega = \sum_{j=1}^m (-1)^{\frac{p(p-1)}{2}} f_j \cdot \alpha_j \wedge J\alpha_j.$$
Smooth positive Lagerberg forms are closed under products. We refer to \cite[Section~2]{burgos-gubler-jell-kuennemann1} for a detailed discussion.

For the following generalization, we refer to \cite[Section~3]{Gubler}.
We fix a weighted $\Z$-polyhedral complex $\CS$ of dimension $d$ in $N_\R$ with underlying piecewise $\Z$-linear space $S$.
The sheaf $\cA_S^{\bullet,\bullet}=\cA^{\bullet,\bullet}$ of smooth Lagerberg forms on $S$ is given on $\Omega \cap S$ for an open subset $\Omega$ of $N_\R$ by restricting the smooth Lagerberg forms on $\Omega$ to $\Omega \cap S$.
This means precisely that we identify $\omega, \omega' \in \cA(\Omega)$ if and and only if $\omega$ and $\omega'$ restrict to the same Lagerberg form on $\relint(\Delta) \cap \Omega$ for each face $\Delta$ of $\CS$. By restriction, the Lagerberg involution $J$ is well-defined on smooth Lagerberg forms on $S$.
There are functorial pull-backs $L^*\colon L_*\cA_{S'}^{\bullet,\bullet}\to \cA_S^{\bullet,\bullet}$ for piecewise $\Z$-linear maps $L\colon S' \to S$ of piecewise $\Z$-linear spaces.

An element $\omega$ of $\cA_S^{p,p}$ is  \defi{symmetric} if $J\omega = (-1)^p \omega$. We call a smooth Lagerberg form on $S$ of type $(p,p)$ \defi{positive} if it is the restriction of a positive smooth Lagerberg form on $N_\R$. A positive Lagerberg form is symmetric.
The positive Lagerberg forms are obviously closed under products.
It follows from \cite[Corollary~2.2.5]{burgos-gubler-jell-kuennemann1} that every symmetric smooth Lagerberg form on $S$ is the difference of two positive smooth Lagerberg forms.

We denote by $\cA_c(\CS)$ the space of compactly supported forms.
The underlying integral structure of $S$ gives well-defined integrals
$\int_\CS \omega$ and boundary integrals%
\footnote{Note that there is a sign mistake in the definition of the boundary integral in~\cite{gubler16:forms_currents}; see the erratum in \cite{gubler13:forms_currents_analytif_algebraic_variety}.}
 $\int_{\partial\CS}\eta$
 for  $\omega \in \cA_c^{d,d}(S)$ and  $\eta \in \cA_c^{d,d-1}(S)$.  Stokes' Theorem holds with respect to $\d''$.   We refer to \cite[3.2]{gubler_kunneman:tropical_arakelov} for the definition of \defi{Lagerberg currents} on $\Omega \cap S$ as certain linear functionals on $\cA_c(\Omega \cap S)$. The space of Lagerberg currents on $\Omega \cap S$ of type $(p,q)$ is denoted by $\cD_{p,q}(\Omega \cap S)$, which means they act on compactly supported Lagerberg forms of type $(p,q)$.  We call  $T \in \cD_{p,p}(\Omega \cap S)$ \defi{symmetric} if $TJ=(-1)^pT$.
 A symmetric Lagerberg current $T \in \cD_{p,p}(\Omega \cap S)$ is called \defi{positive} if $T(\omega) \geq 0$ for all positive $\omega \in \cA_c^{p,p}(\Omega \cap S)$.

Note that any piecewise $\Z$-linear function $f\colon S \to \R$ gives rise to a unique Lagerberg current  $[f]$ acting on $\omega \in\cA_c^{d,d}(S)$ by $\langle \omega, [f] \rangle = \int_{\CS} f \, \omega$. The \emph{tropical Poincar\'e--Lelong formula}~\cite[Theorem~0.1]{GK} states that $\d'\d''[f]$ is the current of integration over the tropical Weil divisor $\divisor(f)$.
We refer to the summary below for more details.

\medskip\noindent
Here is a summary of the notations we will use concerning Lagerberg forms.  Let $\cC$ be a $d$-dimensional weighted $\Z$-polyhedral complex with support $S$. \\[1mm]
\null\begin{tabular}{rl}
  \makebox[1.5cm][r]{\hfil$\cA^{\bullet,\bullet}$} &
    \begin{minipage}[t]{0.8\linewidth}
      The sheaf of differential bigraded algebras of Lagerberg forms on $S$.
    \end{minipage} \\
  $\cA^{\bullet,\bullet}_c$ & The cosheaf of compactly supported forms. \\
  $\wedge$ & The alternating product on $\cA^{\bullet,\bullet}$. \\
  $\d'$ & $\colon\cA^{p,q}\to\cA^{p+1,q}$, the first differential on $\cA^{\bullet,\bullet}$. \\
  $\d''$ & $\colon\cA^{p,q}\to\cA^{p,q+1}$, the second differential on $\cA^{\bullet,\bullet}$. \\
  $\int_\cC$ & $\colon\cA_c^{d,d}(S)\to\R$, the integral. \\
  $\int_{\del\cC}$ & $\colon\cA_c^{d,d-1}(S)\to\R$, the boundary integral. \\
  $[f]$ & The Lagerberg current of a piecewise $\Z$-linear function $f\colon S\to\R$. \\
   \makebox[1.5cm][r]{\hfil$\cD_{\bullet,\bullet}$} &
  \begin{minipage}[t]{0.8\linewidth}
  	The bigraded sheaf of Lagerberg currents on $S$.
  \end{minipage} \\
       $\d'$ & $\colon\cD_{p,q}\to\cD_{p-1,q}$, given by  $\angles{\d'T,\omega} = (-1)^{p+q+1}\angles{T,\d'\omega}$  for $\omega \in \cA_c^{p-1,q}$. \\
       $\d''$ & $\colon\cD_{p,q}\to\cD_{p,q-1}$, given by $\angles{\d''T,\omega} = (-1)^{p+q+1}\angles{T,\d''\omega}$ for $\omega \in \cA_c^{p,q-1}$. \\
\end{tabular}

\subsection{Smooth forms and currents} \label{Smooth forms and currents}
Chambert-Loir and Ducros~\cite{CLD} have introduced a sheaf of bigraded differential $\R$-algebras $\cA_\sm^{\bullet,\bullet}$ of $(p,q)$-forms on a Berkovich space $X$ over $K$.  The differentials are denoted by $\d',\d''$; they behave like the differential operators $\partial$ and $\bar\partial$ in complex geometry.  In this paper, the elements of $\cA_\sm^{p,q}$ will be called \defi{smooth $(p,q)$-forms}. There is a functorial pull-back with respect to morphisms of analytic spaces.

Let $X$ be a compact analytic space with a moment map  $\phi\colon X\to\bGm^{n,\an}$, let $S = \phi_{\trop}(X)$, so that $S$ is a piecewise $\Z$-linear space.  There is a natural homomorphism $\phi_{\trop}^* \colon \cA_S^{\bullet,\bullet} \to \cA_{{\rm sm}}^{\bullet,\bullet}$ of differential $\R$-algebras. The Lagerberg involution $J$ is also defined on the sheaf $\cA_{{\rm sm}}$.
We call $\omega \in \cAsm^{p,p}(X)$ \defi{symmetric} if $J\omega=(-1)^p\omega$.
A form $\omega \in \cAsm^{p,p}(X)$ is called \defi{positive} if every $x \in X$ has a neighborhood $W$ which is a compact analytic domain admitting a moment map  $\phi\colon W\to\bGm^{n,\an}$ such that
$\omega|_W=\phi_{\trop}^*(\alpha)$ for a positive Lagerberg form $\alpha \in \cA^{p,p}(h(W))$.
Elements of $\cAsm^{0,0}$ are functions on open subsets of $X$ which we call \defi{smooth functions}. For details, we refer to~\cite[Section~3]{CLD}.

We will frequently use the fact that smooth partitions of unity exist on a good paracompact analytic space~\cite[Proposition~3.3.6]{chambert_ducros12:forms_courants}.

If $X$ is a topologically separated analytic space over $K$ with $\partial X = \emptyset$, then it is shown in~\cite[Section~4]{CLD} that for every open subset $U$ of $X$, the compactly supported smooth $(p,q)$-forms on $U$ form a locally convex space $\cA_{{\rm sm},c}^{p,q}(U)$; its dual is called the space of \defi{currents} and is denoted by $\cDsm_{p,q}(U)$.  There are differential operatiors $\d'\colon\cDsm_{p,q}\to\cDsm_{p-1,q}$ and $\d''\colon\cDsm_{p,q}\to\cDsm_{p,q-1}$ which are defined by $\angles{\d'T,\alpha} = (-1)^{p+q+1}\angles{T,\d'\alpha}$ and $\angles{\d''T,\alpha} = (-1)^{p+q+1}\angles{T,\d''\alpha}$, respectively.  The signs are chosen so that $\d'[\alpha]=[\d'\alpha]$ and $\d''[\alpha]=[\d''\alpha]$, where $\alpha\in\cA_{\sm,c}^{\bullet,\bullet}$ and $[\alpha]$ is the associated current of integration.
We call $T \in \cDsm_{p,p}(X)$  \defi{symmetric} if $TJ=(-1)^pT$. A symmetric current $T \in \cDsm_{p,p}(X)$ is called \defi{positive} if $T(\omega)\geq 0$ for all positive forms $\omega \in \cA_{{\rm sm},c}^{p,q}(X)$.  See~\cite[Sections~5.3, 5.4]{chambert_ducros12:forms_courants}.

\medskip\noindent
Here is a summary of the notations we will use concerning smooth forms.   \\[1mm]
\null\begin{tabular}{rl}
  \makebox[1.5cm][r]{\hfil$\cA_\sm^{\bullet,\bullet}$} &
    \begin{minipage}[t]{0.8\linewidth}
      The sheaf of differential bigraded algebras of smooth forms on an analytic space as introduced in~\cite{chambert_ducros12:forms_courants}.
    \end{minipage} \\
  $\wedge$ & The alternating product on $\cA_\sm^{\bullet,\bullet}$. \\
  $\cDsm_{\bullet,\bullet}$ & The sheaf of currents. \\
       $\d',\d''$ & The differentials on $\cA_\sm^{\bullet,\bullet}$ and on $\cDsm_{\bullet,\bullet}$. \\
\end{tabular}

\section{Tropical skeletons} \label{sec: tropical skeletons}

Let $X$ be a compact strictly  analytic space over $K$ of pure dimension $d$.  We consider a moment map $\varphi\colon X \to \bT^{\an}$ to a split torus $\bT$ over $K$, and we let $N$ and $M$ be the cocharacter and character lattices of $\bT$, respectively.  Note that the moment map $\phi$ is equivalent to a homomorphism $M\to\Gamma(X,\sO_X^\times)$. We will use results of Ducros \cite{ducros12:squelettes_modeles} to attach to $\varphi$ a tropical skeleton in $X$ which generalizes the algebraic case considered in \cite{gubler_rabinoff_werner16:skeleton_tropical}.

In fact, the assumption that $X$ is a strictly analytic space was done only for simplicity of notation and can be easily generalized to $\Gamma$-strictly analytic space for any subgroup $\Gamma$ of $\R$ containing the value group of $K$ as defined in \cite[0.20]{ducros12:squelettes_modeles}.

\begin{defn}\label{def:tropical.variety}
	The \defi{tropical variety} of $\phi\colon X\to\bT^\an$ is the image $\phi_{\trop}(X)\subset N_\R$.
\end{defn}

\begin{art}[Polyhedral Structure on Tropicalizations]\label{sec:trop.is.polyhedral}
	A result of Berkovich, which might be seen as an  analytic version of the Bieri--Groves theorem, shows that that the tropical variety $\varphitrop(X)$ is the support of a $\z$-polytopal complex $\Pi$ in $N_\R$ of dimension at most $d$.  A theorem of Ducros~\cite[Theorem~3.2]{Ducros} shows that we may choose $\Pi$  such  that $\varphitrop(\partial X)$ is contained in a subcomplex of dimension at most $d-1$.   If $X$ is strictly affinoid, then $\varphitrop(\partial X)$ is equal to the support of such a subcomplex.
\end{art}

\begin{art}[Tropicalizations of Germs]\label{sec:ducros.germs}
	Ducros~\cite[Theorem~3.4]{Ducros} has also shown that for every $x \in X$, there is a compact strictly analytic neighbourhood $V(x)$ of $x$ such that for all compact strictly analytic neighbourhoods $W \subset V(x)$ of $x$, the germs of the tropical varieties $\varphitrop(W)$ and $\varphitrop(V(x))$ agree at $\varphitrop(x)$.
\end{art}

\begin{defn}\label{def:germ.of.smooth.trop}
	The \defi{tropical variety $\phi_{\trop}(X,x)$ of a germ $(X,x)$} is the germ of $\phi_{\trop}(V(x))$ at $\phi_{\trop}(x)$.  The \defi{tropical dimension of $(X,x)$} is $\dim\phi_{\trop}(X,x)$.
\end{defn}

\begin{art}\label{sec:germ.smooth.trop.properties}
	The tropical variety of a germ $(X,x)$ is the germ of a fan at $\varphitrop(x)$.  This germ may be strictly smaller than the germ of $\varphitrop(X)$ at $\varphitrop(x)$. For example, the tropical variety of the germ of $X$ at a rigid point is always a point.  More generally, the tropical dimension of $(X,x)$ is bounded in Theorem~\ref{thm:trop.dim.reduction} below.
\end{art}

In the following, we will use graded residue fields introduced in~\secref{nonArch fields}.  Recall from~\secref{Non-Archimedean geometry} that $\td\chi_x^\bullet\colon\sA\to\td\sH(x)^\bullet$ is the canonical homomorphism to the graded reduction at a point $x\in\sM(\sA)$.  This extends in an obvious way to $\sO_{X,x}$.

\begin{notn}\label{notn:tildaMx}
	Let $\td K^\bullet(\td M(x))$ denote the graded subfield of the graded residue field $\td\sH(x)^\bullet$ generated over $\td K^\bullet$ by $\td M(x) = \{\td\chi_x^\bullet(\phi^*u)\mid u\in M\}$.  We set
	\[ d_\phi(X,x) \coloneq \trdeg\bigl(\td K^\bullet(\td M(x))/\td K^\bullet\bigr). \]
\end{notn}

The transcendence degree of an extension of graded fields is defined in the same way as ordinary fields; see~\cite[Section~0]{ducros12:squelettes_modeles}.  Alternatively, it can be defined as follows.

\begin{rem} \label{Abhyankar inequality}
	We note that we have
	$$\trdeg\bigl( \td \sH(x)^\bullet/\td K^\bullet \bigr)= \dim_\Q \left( |\sH(x)^\times|_\Q / |K^\times|_\Q \right) + \trdeg\bigl(\td \sH(x)/\td K  \bigr) \leq \dim_x(X) = d$$
	as shown in~\cite[0.12]{Ducros} and by \cite[1.4.6]{ducros18:families}.	If equality is satisfied, then $x$ is called an \defi{Abhyankar point}. In any case, we get $d_\varphi(X,x) \leq d$.
\end{rem}

\begin{rem} \label{connection to usual residue field}
  If $\omega = \phi_{\trop}(x)\in N_\Gamma$ then we choose a lift $\lambda\colon M\to K^\times$ of $\angles{\scdot,\omega}\colon M\to\Gamma$, and we let $\td K(\td M(x))$ denote the subfield of $\td\sH(x)$ generated over $\td K$ by $\{\td\chi_x(\phi^*(u/\lambda(u)))\mid u\in M\}$; this does not depend on the choice of $\lambda$.  In this case we can compute $d_\phi(X,x)$ using just classical residue fields by~\cite[0.12]{Ducros}:
	\[ d_\phi(X,x) = \trdeg\bigl( \td K(\td M(x))/\td K \bigr). \]
\end{rem}

\begin{eg}\label{eg:res.field.skeleton.point}
  Choose a basis $u_1,\ldots,u_n$ for $M$.  Let $\xi$ be a point of the canonical skeleton $S(\bT)$, and let $r_i = |u_i(\xi)|$ for $i=1,\ldots,n$.  Then the homomorphism  $\td K^\bullet(t_1/r_1,\ldots,t_n/r_n)\to\td\sH(\xi)^\bullet$ sending $t_i\mapsto\td\chi_\xi^\bullet(u_i)$ is an isomorphism by~\cite[2.2.46.6]{ducros14:structur_des_courbes_analytiq} or~\cite[0.14]{Ducros}.  In other words, $\td\sH(\xi)^\bullet$ is a purely transcendental extension of $\td K^\bullet$ of transcendence degre $n$.  Moreover, if $\phi\colon X\to\bT^\an$ is a moment map and $\phi(x) = \xi$, then $\td K^\bullet(\td M(x))$ is the image of $\td\sH(\xi)^\bullet\to\td\sH(x)^\bullet$.
\end{eg}

The next theorem of  Ducros will be crucial.

\begin{thm}[{\cite[Theorem~3.4]{ducros12:squelettes_modeles}}]%
	\label{thm:trop.dim.reduction}
	We have $\dim\phi_{\trop}(X,x)\leq d_\phi(X,x)$, with equality if $x\in\Int(X)$, in which case $\phi_{\trop}(X,x)$ is pure dimensional.  Moreover, if $x\in\Int(X)$ and $\dim(\bT) = d_\phi(X,x)$, then $\phi_{\trop}(X,x) = (N_\R,\phi_{\trop}(x))$.
\end{thm}

The following definition makes sense for any $K$-analytic space, although we will confine ourselves to compact spaces in the results that follow.

\begin{defn}\label{def:trop.skel}
	The \defi{tropical skeleton of $\phi\colon X\to\bT^\an$} is the subset
	\[ S_\phi(X) = \bigl\{ x\in X\bigm| d_\phi(X,x) = \dim_x(X) = d \bigr\}. \]
\end{defn}

By Theorem~\ref{thm:trop.dim.reduction}, if $\dim\phi_{\trop}(X,x)=d$ then $x\in S_\phi(X)$, and if $x\in S_\phi(X)\setminus\del X$ then $\dim\phi_{\trop}(X,x)=d$.

The following properties are immediate consequences of various results of Ducros.

\begin{prop}\label{prop:tropical.skeleton}
  The tropical skeleton satisfies the following properties.
  \begin{enumerate}
  \item\label{item:trop.skeleton.abhyankar}
    If $x\in S_\phi(X)$, then $x$ is an Abhyankar point of $X$.
  \item\label{item:trop.skeleton.subdomain}
    If $U\subset X$ is a compact strictly analytic domain in $X$, then $S_\phi(U) = S_\phi(X)\cap U$.
  \item\label{item:trop.skeleton.closed}
    The tropical skeleton $S_\phi(X)$ is closed in $X$.
  \item\label{item:trop.skeleton.samedim}
    If $\dim(\bT)=d = \dim(X)$ then $S_\phi(X) = \phi\inv(S(\bT))$.
  \item\label{item:trop.skeleton.basechange}
    Let $K'/K$ be a non-Archimedean extension field, let $X' = X\hat\tensor_K K'$, let $\pi\colon X'\to X$ be the structure morphism, and let $\phi' = \phi\hat\tensor_K K'\colon X'\to\bT^\an\hat\tensor_K{K'}$.  Then $\pi(S_{\phi'}(X')) = S_\phi(X)$, and for any $x \in S_\varphi(X)$, the  Shilov boundary of $\pi^{-1}(x)$ is finite and is equal to $\pi^{-1}(x) \cap S_{\varphi'}(X')$.
  \item \label{item:algebraic extension}
    If $K'/K$ is the completion of an algebraic extension, then $S_{\phi'}(X')=\pi^{-1}(S_\varphi(X))$.
  \item\label{item:trop.skeleton.1units}
    If $\phi'\colon X\to\bT$ satisfies $\phi_{\trop}=\phi'_{\trop}$, then $S_\phi(X)\cap\Int(X) = S_{\phi'}(X)\cap\Int(X)$.
  \item\label{item:trop.skeleton.morphism}
    Let $f\colon X'\to X$ be a morphism of compact strictly $K$-analytic spaces of pure dimension $d$, and let $\phi' = \phi\circ f\colon X'\to\bT^\an$.  Then $S_{\phi'}(X') = f\inv(S_\phi(X))$.
  \end{enumerate}
\end{prop}

\begin{proof}
	(\ref{item:trop.skeleton.abhyankar})~\ This follows from Remark~\ref{Abhyankar inequality}.

	(\ref{item:trop.skeleton.subdomain})~\ This is clear from the definition because $U$ and $X$ have the same completed residue fields at points of~$U$, see \cite[1.2.9]{ducros18:families}.

	(\ref{item:trop.skeleton.closed})~\ Choose a basis for $M$, and let $f_1,\ldots,f_n\in\Gamma(X,\sO_X)^\times$ be the image of this basis under $\phi^*$.  For $I = \{i_1,\ldots,i_d\}\subset\{1,2,\ldots,n\}$ of size $d$, we let
	\[ U_I = \bigl\{ x\in X\mid \td\chi_x(f_{i_1}),\ldots,\td\chi_x(f_{i_d})\in\td\sH(x)^\bullet \text{ are algebraically dependent over } \td K^\bullet \bigr\}. \]
	Then $U_I$ is open by Lemma~\ref{lem:alg.dep.open}, and $X\setminus S_\phi(X) = \bigcap_I U_I$.

	(\ref{item:trop.skeleton.samedim})~\ For $x\in S_\phi(X)$ we have $d = d_\phi(X,x) = \trdeg(\td K^\bullet(\td M(x))/\td K^\bullet)$, so $\phi(x)\in S(\bT)$ by~\cite[0.13]{Ducros}.  Conversely, let $x\in X$ and suppose that $\xi = \phi(x)$ is contained in $S(\bT)$.  Then $\trdeg(\td K^\bullet(\td M(x))/\td K^\bullet)=d$ by Example~\ref{eg:res.field.skeleton.point}.

	(\ref{item:trop.skeleton.basechange})~\ 
        Let $x\in S_\phi(X)$. Recall that $X$ has pure dimension $d$.  Choosing $d$ elements $u_1,\ldots,u_d$ of $M$ gives rise to a morphism $\psi\colon X\to\bT'$, where $\bT'$ is a $d$-dimensional torus.  By~(\ref{item:trop.skeleton.samedim}) above, the images of $u_1,\ldots,u_d$ in $\td\sH(x)^\bullet$ are algebraically independent over $\td K^\bullet$ if and only if  $x\in\psi\inv(S(\bT'))$.  We fix a choice of such $u_1,\ldots,u_d$; this is essentially what Ducros calls an ``Abhyankar presentation'' in~\cite[3.2.12.2]{ducros14:structur_des_courbes_analytiq}.  Let $\xi'\in S(\bT'\hat\tensor_K K')$ be the unique point of the skeleton above $\psi(x)\in S(\bT')$ and let $\psi' \coloneqq \psi\hat\tensor_K K'\colon X'\to\bT^\an\hat\tensor_K{K'}$.    By~\cite[Proposition~3.2.13]{ducros14:structur_des_courbes_analytiq}, the Shilov boundary of $\pi\inv(x)$ is equal to 
        $\pi\inv(x)\cap (\psi')\inv(\xi')$. In particular, the set 
        $\pi\inv(x)\cap (\psi')\inv(\xi')$ is independent of the choice of $u_1,\ldots,u_d$.  If 
        $x'\in\pi\inv(x)\cap (\psi')\inv(\xi')$, then the images of $u_1,\ldots,u_d$ in $\td\sH(x')^\bullet$ are algebraically independent over $(\td K')^\bullet$ because $\psi'(x')\in S(\bT'\hat\tensor_KK')$, so we have $x'\in S_{\phi'}(X')$ a fortiori.  
        Conversely, let $x$ be any point of $X$ such that there is a point $x'\in \pi\inv(x)\cap S_{\phi'}(X')$, then there exist $u_1,\ldots,u_d\in M$ whose images in $\td\sH(x')^\bullet$ are algebraically independent over $(\td K')^\bullet$; their images in $\td\sH(x)^\bullet$ are then algebraically independent over $\td K^\bullet$. It follows again from~(\ref{item:trop.skeleton.samedim}) that $x \in S_\varphi(X)$.
        Using this choice of $u_1,\ldots,u_d$, the above argument shows that $x'$ is in the Shilov boundary of $\pi\inv(x)$.  Since the Shilov boundary of $\pi\inv(x)=\sM(\sA\hat\otimes_K \sH(x))$ is nonempty, this also proves $\pi(S_{\phi'}(X')) = S_\phi(X)$.

    (\ref{item:algebraic extension})~\ This follows from
    (\ref{item:trop.skeleton.basechange}) and \cite[Corollary~1.3.6]{berkovic90:analytic_geometry}.

    (\ref{item:trop.skeleton.1units})~\ Let $u_1,\ldots,u_n$ be a basis for $M$, and let $f_i = \phi^*u_i$ and $f_i' = \phi'^*u_i$.  If $\psi = \phi\inv\phi'\colon X\to\bT$ then $\psi_{\trop}(X) = \{0\}$.  If $x\in\Int(X)$ then $d_\psi(X,x) = \dim\psi_{\trop}(X,x) = 0$ by Theorem~\ref{thm:trop.dim.reduction}, so $\td\chi_x^\bullet(f_i)\inv\td\chi_x^\bullet(f_i')$ is algebraic over $\td K^\bullet$ for all $i$.  Hence
	\[\trdeg(\td K^\bullet(\td\chi_x^\bullet(f_1),\ldots,\td\chi_x^\bullet(f_n))/\td K^\bullet)
	= \trdeg(\td K^\bullet(\td\chi_x^\bullet(f_1'),\ldots,\td\chi_x^\bullet(f_n'))/\td K^\bullet) \]
	for all $x\in\Int(X)$, so $x\in S_\phi(X)$ if and only if $x\in S_{\phi'}(X)$.

        (\ref{item:trop.skeleton.morphism})~\ If $f(x') = x$ then $d_{\phi'}(X',x') = d_\phi(X,x)$ because $\td K^\bullet(\td M(x')) = \td K^\bullet(\td M(x))$ as subfields of $\td\sH(x')^\bullet$.
\end{proof}

We used the following lemma in the proof of Proposition~\ref{prop:tropical.skeleton}(\ref{item:trop.skeleton.closed}).

\begin{lem}\label{lem:alg.dep.open}
	Let $X$ be a $K$-analytic space, let $f_1,\ldots,f_n\in\Gamma(X,\sO_X)^\times$, and let
	\[ U = \bigl\{ x\in X\mid \td\chi^\bullet_x(f_1),\ldots,\td\chi^\bullet_x(f_n)\in\td\sH(x)^\bullet \text{ are algebraically dependent over } \td K^\bullet \bigr\}. \]
	Then $U$ is open.
\end{lem}

\begin{proof}
  Let $f\colon X\to\bGm^{n,\an}$ be the morphism induced by $(f_1,\ldots,f_n)$.  Then $X\setminus U$ is equal to the inverse image of the skeleton $S(\bGm^{n,\an})$ by~\cite[0.13]{Ducros}.  The latter set is closed, so $U$ is open.
\end{proof}

\begin{rem}\label{eg:tropical.skeleton.general}
  Suppose that $K$ is an algebraically closed field, that $\sX$ is a smooth algebraic $K$-curve, and that $\phi_{\rm alg}\colon\sX\inject\bT$ is a closed immersion.  Let $X = \sX^\an$ and $\phi = \phi_{\rm alg}^\an\colon X^\an\inject\bT^\an$, and note that $\del X=\emptyset$ by~\cite[Theorem 3.4.1]{berkovic90:analytic_geometry}.  In this case, there is a well-understood theory of skeletons of $X$ associated to semistable models of the smooth completion of $\sX$, regarded as a curve with marked points: see for instance~\cite{BPR2}.  Let $S\subset X$ be such a skeleton and let $\tau_S\colon X\to S$ be the canonical retraction map.  By~\cite[Theorem~5.15]{BPR2}, the map $\phi_{\trop}\colon X\to N_\R$ factors through $\tau_S$.  It follows that $\phi_{\trop}$ is locally constant on $X\setminus S$---see for example~\cite[Lemma~2.13]{BPR2}---so $\dim\phi_{\trop}(X,x) = 0$ for $x\notin S$.  Thus Theorem~\ref{thm:trop.dim.reduction} yields $S_\phi(X)\subset S$, and in fact, that $S_\phi(X)$ is equal to the locus of points $x\in S$ such that $\phi_{\trop}$ is not constant in a neighborhood of $x$ in $S$.  This locus can be computed explicitly using the slope formula~\cite[Theorem~5.15]{BPR2}.
\end{rem}

\begin{eg}\label{eg:tropical.skeleton.specific}
  As a specific case of Remark~\ref{eg:tropical.skeleton.general} taken from~\cite[Example~2.7]{BPR} (to which the reader can refer for details), let $p\geq 5$ be a prime, let $K = \bC_p$, and let $X=\sX^\an$ for the plane curve $\sX\subset\bGm^2$  defined by the equation
  \[ x^3y - x^2y^2 - 2xy^3 - 3x^2y + 2xy - p = 0. \]
  The closure of $\sX$  in the projective space $\bP_K^2$ is given by the homogeneous equation
  \[x^3y - x^2y^2 - 2xy^3 - 3x^2yz + 2xyz^2 - pz^4 = 0 \]
  which defines	a smooth quartic of genus $3$;
  the zeros and poles of $x$ and $y$ occur at the points $P_1 = (0:1:0),\,P_2 = (1:0:0),\,Q_1 = (2:1:0),$ and $Q_2 = (-1:1:0)$.  
  The same homogeneous equation describes a closed subscheme of $\bP_{K^\circ}^2$ which is the minimal regular proper semistable model $\sX_{\min}$ of this smooth quartic. The special fiber of  $\sX_{\min}$ is given by the equation $xy(x+y-z)(x-2y-2z)=0$. By \cite[Proposition 2.4.4]{berkovic90:analytic_geometry}, there are unique points $A,B,C,D$ with reduction to the special fiber equal to the generic points of the four lines  $x=0$, $x+y=z$, $y=0$ and $x=2y+2z$, respectively. 
  The minimal skeleton of $\sX$ is given by the edges of a tetrahedron with the vertices $A,B,C,D,$ with an additional ray emanating from each vertex in the directions of $P_1,Q_2,P_2,$ and $Q_1$, respectively. 
  Note that the skeleton is the dual graph of the special fiber of $\sX_{\min}$ with vertices corresponding to irreducible components, edges corresponding to  double points, and where the rays correspond to the marked points $P_1,P_2,Q_1,Q_2$. 
   The tropicalization $\phi_{\trop}(X)$ is a triangle with vertices at $(0,0),\,(0,1),$ and $(1,0)$, together with three rays emanating from these vertices in the directions of $(-1,-1),\,(3,-1),$ and $(-1,3)$, respectively.  The map $\phi_{\trop}$ takes $A$ to $(1,0)$, $C$ to $(0,1)$, and both $B$ and $D$ to $(0,0)$, and it takes each segment in $S$ to the line segment connecting the tropicalizations of its endpoints.  Importantly, $\phi_{\trop}$ takes the segment $BD$ to the \emph{point} $(0,0)$.  See Figure~\ref{fig:genus3.trop}.  It follows that $S_\phi(X)$ consists of $S$ minus the interior of the segment $BD$.
\end{eg}

\begin{figure}[ht]
  \centering
  \begin{tikzpicture}[>=Latex, line join=bevel, thick,
    every node/.style={inner sep=1pt}]
    \coordinate (D) at (0,0);
    \coordinate (A) at (2,0);
    \coordinate (C) at (0,2);
    \coordinate (B) at (-1,-.5);
    \path (A) node[below] {$A$}
          (B) node[above left] {$B$}
          (C) node[above right] {$C$}
          (D) node[above right] {$D$};
    \draw (C) -- (A) -- (D) -- cycle;
    \draw (A) -- (B);
    \draw[thin, dashed] (B) -- (D);
    \draw (C) -- (B);
    \draw[->] (A) -- ++(1.5,-.5) node[right] {$P_1$};
    \draw[->] (C) -- ++(-.5,1.5) node[above] {$P_2$}; 
    \draw[->] (B) -- ++(-1,-1) node[below] {$Q_1$};
    \draw[->] (D) -- ++(-1,-1) node[below] {$Q_2$};

    \draw[->] (3.2,1) to["$\phi_{\trop}$" above] ++(3,0);

    \begin{scope}[xshift=7cm]
      \coordinate (TB) at (0,0);
      \coordinate (TA) at (2,0);
      \coordinate (TC) at (0,2);
      \path (TB) node[below right] {$(0,0)$}
            (TA) node[above right] {$(1,0)$}
            (TC) node[left] {$(0,1)$}
      ;
      \draw (TA) to[{"$\scriptstyle 2$" above}] (TB);
      \draw (TB) to[{"$\scriptstyle 2$" left}] (TC);
      \draw (TC) -- (TA);
      \draw[->] (TA) -- ++(1.5,-.5) node[right] {$P_1$};
      \draw[->] (TB) to[{"$\scriptstyle 2$" above left}] ++(-1,-1) node[below] {$Q_1,Q_2$};
      \draw[->] (TC) -- ++(-.5,1.5) node[above] {$P_2$};
    \end{scope}
  \end{tikzpicture}

  \caption{\label{fig:genus3.trop}The skeleton $S$ of the curve $X$ from Example~\ref{eg:tropical.skeleton.specific}, and its tropicalization.  The dashed line segment $BD$ is included in $S$ but not in $S_\phi(X)$.  The small numbers indicate tropical multiplicities; see Example~\ref{eg:tropical.skeleton.multiplicities}.}
\end{figure}

\begin{defn}\label{def:generic.quotient}
  Let $S$ be a \emph{polyhedral set} in $N_\R$ of pure dimension $d$: in other words, $S$ is a finite union of $d$-dimensional polyhedra. 
  Let $N'$ be a free abelian group of rank $d$.  We say that a homomorphism $Q\colon N\to N'$ is \defi{generic with respect to $S$} provided that for every $\omega\in S$, the image of the germ $(S,\omega)$ under $Q_\R\colon N_\R\to N'_\R$ also has dimension $d$.  We say that $Q\colon N\to N'$ is \defi{generic} if it is generic with respect to some unspecified $S$.  We apply the above terminology also to the homomorphism of tori $q\colon\bT\to\bT'$ inducing the homomorphism $Q\colon N\to N'$ on cocharacter lattices.
\end{defn}

The polyhedral set $S$ is always the support of a finite polyhedral complex $\Pi$ of pure dimension $d$.
Then $Q$ is generic with respect to $S$ if and only if $Q_\R$ is injective on every $d$-dimensional face $\sigma\in\Pi$.  Note that a generic $Q$ must have full rank,%
\footnote{One could require $Q$ to be surjective as well, but this is not important for our purposes, since an isogeny of tori induces a bijection on skeletons.}
in which case it embeds $M' = \Hom(M,\Z)$ as a sublattice of $M$.

\begin{thm}\label{thm:tropical.skeleton.projection}
  Let $\bT'$ be a split torus of dimension $d=\dim(X)$, let $q\colon\bT\to\bT'$ be generic with respect to $\phi_{\trop}(X)$, and let $\psi = q\circ\phi$.  Then $S_\phi(X)\cap\Int(X) \subset S_\psi(X)\subset S_\phi(X)$.  Moreover, we have $S_\psi(X) = S_\phi(X)$ for generic $q$.
\end{thm}

\begin{proof}
  Let $M'$ be the character lattice of $\bT'$.  For $x\in X$ we have $d_\psi(X,x)\leq d_\phi(X,x)$ because $M'\subset M$, so $S_\psi(X)\subset S_\phi(X)$.  If $x\in S_\phi(X)$ and $x\notin\del X$ then Theorem~\ref{thm:trop.dim.reduction}  yields $d = \dim\phi_{\trop}(X,x)$, so $d = \dim\psi_{\trop}(X,x)$ by genericity of $q$; hence $d_\psi(X,x)=d$ and $x\in S_\psi(X)$, again by Theorem~\ref{thm:trop.dim.reduction}.  This proves the first assertion.

	It remains to show that $\del X\cap S_\phi(X) \subset S_\psi(X)$ for generic~$q$.  We induct on $\dim(X)$; when $\dim(X)=0$ there is nothing to show.  By Proposition~\ref{prop:tropical.skeleton}(\ref{item:trop.skeleton.subdomain}) we may assume that $X = \sM(\sA)$ is affinoid.  Let $K'/K$ be a non-Archimedean extension field, define $X',\pi$, and $\phi'$ as in Proposition~\ref{prop:tropical.skeleton}(\ref{item:trop.skeleton.basechange}), and let $\psi'\colon X'\to\bT'^\an\hat\tensor_K K'$ be the base change of $\psi$ (with respect to some choice of $q$).  By~Proposition~\ref{prop:tropical.skeleton}(\ref{item:trop.skeleton.basechange}), we have $\pi(S_{\psi'}(X')) = S_\psi(X)$ and $\pi(S_{\phi'}(X')) = S_\phi(X)$, so if $S_{\phi'}(X')\subset S_{\psi'}(X')$ then $S_\phi(X)\subset S_\psi(X)$.  Hence we may extend scalars to assume $K$ is algebraically closed and $v(K^\times) = \R$.  In this case, $\trdeg(\td K^\bullet(\td M(x))/\td K^\bullet) = \trdeg(\td K(\td M(x))/\td K)$ for all $x\in X$ as seen in Remark~\ref{connection to usual residue field}, and similarly for $\psi$; this means that we may work with ordinary reductions, which is important in the next paragraph.

	First we prove the pointwise version: that is, we claim that for $x\in S_\phi(X)$, we have $x\in S_\psi(X)$ for generic~$q$.  (When $x\in\Int(X)$ one can choose $q$ to be generic with respect to $\phi_{\trop}(X,x)$, but when $x\in\del X$ it is harder to see with respect to which polyhedra one must choose $q$ to be generic.)  By our assumptions on $K$, we may translate to assume $\phi_{\trop}(x) = 0$.  Then the homomorphism $K[M]\to\sH(x)$ restricts to a homomorphism $K^\circ[M]\to\sH(x)^\circ$, hence induces a homomorphism $\td K[M]\to\td\sH(x)$.  Let $\td\fp\subset\td K[M]$ be the kernel of this homomorphism, and let $B = \td K[M]/\td\fp$ and $Y = \Spec(B)$.  Then $Y$ is an integral scheme, and its function field $\td K(Y) = Q(B) = \td K(\td M(x))\subset\td \sH(x)$ has transcendence degree $d$ over $\td K$ by hypothesis, so $\dim(Y)=d$.  Let $\Trop(Y)\subset N_\R$ denote the tropicalization of $Y$, considered as a closed subscheme of the torus $\bT_{\td K}$ over the trivially valued field $\td K$.  This is a fan of pure dimension $d$.  If $q\colon\bT\to\bT'$ is generic with respect to $\Trop(Y)$ then $Y\to\bT'_{\td K}$ is a dominant morphism, so the function field $\td K(M')$ of $\bT'_{\td K}$ injects into the function field $\td K(Y)$ of $Y$ such that $\td K(Y)/\td K(M')$ is a finite extension, and hence $\trdeg(\td K(\td M'(x))/\td K) = d$; it follows that $x\in S_\psi(X)$.  This proves the pointwise claim.

	To globalize the above construction, we use a description of $\del X$ due to Ducros.  Choose $f\in\sA$ nonzero, let $F\colon X\to\bA^{1,\an}$ be the morphism determined by $f$, and let $\eta = \iota(s)\in\bGm^{\an}\subset\bA^{1,\an}$ for some $s\in\sqrt{|K^\times|}=\R$ (recall from~\secref{Tropical geometry} that $\iota\colon\R\to\bGm^{\an}$ is the section of $\trop\colon\bGm^{\an}\to\R$).   The fiber $Y = F\inv(\eta)$ is a strictly $\sH(\eta)$-affinoid space.
	Ducros~\cite[Lemme~3.1]{Ducros} observes that $\del X$ is the union of finitely many such $Y$, so it suffices to show $Y\cap S_\phi(X)\subset S_\psi(X)$ for generic $q$. Since $\eta$ is an Abhyankar point, it follows from \cite[Lemma~1.5.11]{ducros18:families} that $Y$ has pure dimension $d-1$.

	For $x \in Y$, we note that $\sH(x)$ is the same whether $x$ is regarded as a point of $X$ or of $Y$. It follows that $Y\cap S_\phi(X)\subset S_\phi(Y)$, where we abuse notation and write $\phi\colon Y\to(\bT\tensor_K\sH(\eta))^\an$ for the restriction of the base change.  By the inductive hypothesis as applied to  $Y$, if $M''\subset M$ is a generic sublattice of rank $d-1$, then for all $x\in Y\cap S_\phi(X)$ we have $\trdeg(\td\sH(\eta)(\td M''(x))/\td\sH(\eta))=d-1$.  Noting that $\td\sH(\eta) = \td K(\td f(x))$ as a subfield of $\td\sH(x)$ by Example~\ref{eg:res.field.skeleton.point}, we thus have $\trdeg(\td K(\td M''(x), \td f(x))/\td K)=d$.

	Let $M''\subset M$ be a generic sublattice of rank $d-1$ as above, let $M'\subset M$ be a sublattice of rank $d$ containing $M''$, and let $q\colon\bT\to\bT' = \Spec(K[M'])$ be the corresponding homomorphism of tori.  Let $x\in Y\cap S_\phi(X)$, and suppose that $\trdeg(\td K(\td M'(x))/\td K)=d$.  We have $\trdeg(\td\sH(x)/\td K)=d$ in any case, so $\td f(x)$ is algebraic over $\td K(\td M'(x))$.  By Lemma~\ref{lem:alg.dep.open}, there is a neighbourhood $U_x$ of $x$ such that $\td f(y)$ is algebraic over $\td K(\td M'(y))$ for all $y\in U_x$.  Since $\trdeg(\td K(\td M''(y), \td f(y))/\td K)=d$ for all $y\in Y\cap S_\phi(X)$ by the previous paragraph, this implies $\trdeg(\td K(\td M'(y))/\td K)=d$ for $y\in U_x\cap Y\cap S_\phi(X)$. For $\psi = q \circ \varphi$, this shows that $Y\cap S_\psi(X)$ is open in $Y\cap S_\phi(X)$; it is closed in $Y\cap S_\phi(X)$ by~Proposition~\ref{prop:tropical.skeleton}(\ref{item:trop.skeleton.closed}), so it is a union of connected components.  In particular, if $x\in S_\psi(X)$ for a single point $x$ of $Y\cap S_\phi(X)$, then the entire connected component of $Y\cap S_\phi(X)$ containing $x$ is also contained in $S_\psi(X)$

	Choose representatives $x_1,\ldots,x_r$ of the connected components of $Y\cap S_\phi(X)$.  By the pointwise construction, if $q$ is generic then $\{x_1,\ldots,x_r\}\subset S_\psi(X)$.  Since a generic rank-$(d-1)$ sublattice $M''$ of a generic rank-$d$ sublattice $M'$ is also generic, this shows $Y\cap S_\phi(X)\subset S_\psi(X)$ for generic $q$.
\end{proof}

\begin{cor}\label{cor:tropical.skeleton}
  The map $\phi_{\trop}|_{S_\phi(X)}\colon S_\phi(X)\to\phi_{\trop}(X)$ has finite fibers, and if $\phi_{\trop}(X)$ has dimension $d$ at $\omega$, then $\omega\in\phi_{\trop}(S_\phi(X))$.  
\end{cor}

\begin{proof}
  Let $\omega\in\phi_{\trop}(X)$.  Choose $q\colon\bT\to\bT'$ generic as in Theorem~\ref{thm:tropical.skeleton.projection}, so that $S_\phi(X) = S_\psi(X) = \psi\inv(S(\bT'))$ for $\psi = q\circ\phi$ by Proposition~\ref{prop:tropical.skeleton}(\ref{item:trop.skeleton.samedim}).  Let $\omega' = Q_\R(\omega)$ and let $\xi' = \iota(\omega')\in S(\bT')$.  Since $S_\psi(X) = \psi\inv(S(\bT'))$ we have $\psi\inv(\xi') = S_\psi(X)\cap\psi_{\trop}\inv(\omega')$.  A dimension argument based on the Abhyankhar inequality~\cite[Exemple~2.4.3]{CLD} shows that $\psi\inv(\xi')$ is finite, so $S_\phi(X)\cap\phi_{\trop}\inv(\omega) \subset S_\psi(X)\cap\psi_{\trop}\inv(\omega')$ is finite.

  The set $\phi_{\trop}\inv(\omega)$ is compact, so it is contained in a finite union of the interiors of compact strictly analytic domains $V(x_1),\ldots,V(x_r)$ as in~\secref{sec:ducros.germs} for $x_1,\ldots,x_r\in\phi_{\trop}\inv(\omega)$.  If $U$ is the union of the interiors of the $V(x_i)$ then $X\setminus U$ is closed and disjoint from $\phi_{\trop}\inv(\omega)$, so $\phi_{\trop}(X\setminus U)$ is closed and does not contain $\omega$; thus $\bigcup_{i=1}^r\phi_{\trop}(V(x_i))$ is a neighbourhood of $\omega$ in $\phi_{\trop}(X)$.  It follows that the germ of $\phi_{\trop}(X)$ at $\omega$ is equal to $\bigcup_{i=1}^r\phi_{\trop}(V(x_i))$, so if $\phi_{\trop}(X)$ has dimension $d$ at $\omega$ then $\dim\phi_{\trop}(V(x_i)) = d$ for some~$i$.  Thus $d_\phi(X,x_i)=d$, so $x_i\in S_\phi(X)$ by Theorem~\ref{thm:trop.dim.reduction}.
\end{proof}

\begin{rem}\label{rem:trop.skel.GRW}
	Suppose that $\phi\colon\sX\inject\bT$ is the inclusion of an equidimensional closed subscheme.   A tropical skeleton $S_{\Trop}(\sX)$ was introduced in this context in~\cite{gubler_rabinoff_werner:tropical_skeletons}, using Shilov points of fibers of tropicalization.  We claim that $S_{\phi}(\sX^\an) = S_{\Trop}(\sX)$. Indeed, Theorem~\ref{thm:tropical.skeleton.projection} and Proposition~\ref{prop:tropical.skeleton}(\ref{item:trop.skeleton.samedim}) show that $S_{\phi}(U)=S_{q \circ \varphi}(U)=(q\circ\varphi|_U)^{-1}(S(\bT'))$ for a compact strictly analytic subdomain $U$ of $\bT^\an$ and a generic homomorphism  $q:\bT \to \bT'$ to a split torus $\T'$ of dimension $d$.
	The proof of~\cite[Proposition~4.6]{gubler_rabinoff_werner:tropical_skeletons} and
	\cite[Lemma~4.4]{gubler_rabinoff_werner:tropical_skeletons} yield  $S_{\Trop}(\sX)=(q \circ\varphi)^{-1}S_{\Trop}(\bT')=(q \circ\varphi)^{-1}S(\bT')$. The claim follows from Proposition~\ref{prop:tropical.skeleton}(\ref{item:trop.skeleton.subdomain}).
\end{rem}

\begin{rem}  \label{rem: piecewise linear structure on tropical skeleton}
	By results of Ducros~\cite{Ducros}, the tropical skeleton $\phi\inv(S(\bT))$ has a canonical structure of a piecewise $\q$-linear space (called a $\q$-skeleton in \textit{ibid.}), completely determined by the underlying set and the analytic structure of $X$.

	To see that, we proceed as in the proof of Proposition~\ref{prop:tropical.skeleton}(\ref{item:trop.skeleton.closed}). Let $f_1,\ldots,f_n\in\Gamma(X,\sO_X)^\times$ be the image of a basis of $M$ under $\phi^*$.  For $I \subset\{1,2,\ldots,n\}$ of size $d$, let $\varphi_I\colon X \to \bG_{\rm m}^{d,\an}$ be the moment map given by $(f_i)_{i \in I}$. Then we have seen that
	\begin{gather*}
          S_\varphi(X)=\bigcup_I\left\{x \in X \biggm\vert
            \begin{split}
            &\td\chi_x(f_{i_1}),\ldots,\td\chi_x(f_{i_d})\in\td\sH(x)^\bullet \\
            \text{ are}&\text{ algebraically independent over } \td K^\bullet 
            \end{split}\right\}
        \end{gather*}
	and  Proposition~\ref{prop:tropical.skeleton}(\ref{item:trop.skeleton.samedim}) gives
	$$S_\varphi(X)= \bigcup_I S_{\varphi_I}(X)    =\bigcup_I \varphi_I^{-1}S(\bG_{\rm m}^d).$$
	It follows from  Ducros~\cite[Theorem~5.1]{Ducros} that $S_\varphi(X)$ is a $\q$-skeleton. The piecewise $\q$-linear structure is canonical by~\cite[4.2]{Ducros}.

    By~\cite[Theorem~5.1]{Ducros}, the map $\varphitrop\colon S_\phi(X) \to \varphitrop(X)$ is a piecewise $\q$-linear map with finite fibers.
    We conclude that there is a  unique way to endow  $S_\varphi(X)$ with the structure of a piecewise $\z$-linear space such that $\varphi_{\trop}$ restricts to a $\z$-linear isomorphism on every face.  Note that the $\z$-linear structure depends on~$\varphi_{\trop}$. 
    This is illustrated in Example~\ref{example for dependence on phi} below.
\end{rem}

\begin{rem}[The $d$-skeleton]\label{rem:d-skeleton}
  It follows from Remark~\ref{rem: piecewise linear structure on tropical skeleton} that the locus
  \[ S_\phi(X)_d \coloneq \bigl\{ x\in S_\phi(X)\mid \dim\phi_{\trop}(X,x) = d \bigr\} \]
  is the union of the $d$-dimensional faces of $S_\phi(X)$.  This is again a $\q$-skeleton, and it inherits the structure of a piecwise $\z$-linear space.  The map $S_\phi(X)_d\to\phi_{\trop}(X)$ has finite fibers and surjects onto the $d$-dimensional part of $\phi_{\trop}(X)$.  We call $S_\phi(X)_d$ the \defi{$d$-skeleton of $(X,\phi)$}.

  {Since $\dim\phi_{\trop}(X,x) = d$ for any $x\in S_\phi(X)\cap\Int(X)$ as we have seen in Theorem \ref{thm:trop.dim.reduction}, it is clear that $S_\phi(X)_d\cap\Int(X) = S_\phi(X)\cap\Int(X).$}
\end{rem}

\begin{eg}  \label{example for dependence on phi}
  We choose $X=\T^\an$ and $\varphi\colon X \to \T^\an$ the analytification of the map $\T \to \T$ given by $x \mapsto x^n$ for some non-zero $n \in \N$. Then $S_\varphi(X)=S(\T)$, which we identify with $N_\R$ using the homeomorphism $\trop\colon S(\T)\to N_\R$. Then the $\z$-linear structure of $S_\varphi(X)$ is given by the lattice $\frac{1}{n}N$ as $\varphitrop\colon S(\T) \to S(\T)$ is multiplication by $n$. We conclude that $\varphi$ induces a homeomorphism $S_\varphi(X) \to S(\T)$ which is an isomorphism of $\q$-linear spaces, but not of $\z$-linear spaces.
\end{eg}

\section{Tropical multiplicities} \label{sec: tropical multiplicities}

Let $X$ be a compact good strictly $K$-analytic space of pure dimension $d$.  We fix a moment map $\varphi\colon X \to \bT^{\an}$ to a split torus $\bT$ over $K$, and we let $N$ and $M$ be the cocharacter and character lattices of $\bT$, respectively.  In this section, we will define tropical multiplicities $m_\varphi(X,x)$ and $m_\varphi(X,\omega)$ for all $x \in X$ and all $\omega \in N_\R$. For $\omega$ in the relative interior of a $d$-dimensional face $\sigma$ of the tropical variety $\phi_{\trop}(X)$, the tropical multiplicity $m_\varphi(X,\omega)$ agrees with the tropical weight of $\sigma$ induced by the construction of Chambert-Loir and Ducros  (see Remark~\ref{rem: multiplicity of the Antoines}),
so if $\phi$ is the analytification of the embedding of a closed subvariety then we get the classical tropical weight of $\sigma$.  In this algebraic setting, we will show in Lemma~\ref{lem:mult.from.BPR} that $m_\varphi(X,x)$ agrees with the relative tropical multiplicity $m_\rel(x)$ introduced in~\cite{baker_payne_rabinoff16:tropical_curves} for all $x\in X$, and that $m_\varphi(X,\omega)$ agrees with the tropical multiplicity $m_{\Trop}(\omega)$ defined in \textit{ibid.}\ for all $\omega\in N_\R$.  As our tropical multiplicities are defined for all points of $X$ and $N_\R$, including boundary points, they provide a common generalization of these constructions.  Our definition also has the advantage that it does not depend \emph{a priori} on any choices.

Similarly as in the previous section, all results extend to $\Gamma$-strictly analytic spaces over $K$ for any subgroup $\Gamma$ of $\R$ containing the value group of $K$.

\medskip
If $A$ is an Artin local ring, its \defi{length} is defined to be the length of $A$ as a module over $A$.
We will use the tropical skeleton $S_\varphi(X)$ and the notation from Section~\ref{sec: tropical skeletons}.

\begin{defn}\label{def:trop.mult}
	Let $x\in X$.  If $K$ is algebraically closed, then the \defi{tropical multiplicity of $x$ with respect to $\phi$} is defined to be
	\[ m_\phi(X,x) =
	\begin{cases}
	\length(\sO_{X,x})\,[\td\sH(x)^\bullet:\td K^\bullet(\td M(x))] & \text{if } x \in S_\phi(X) \\
	0 & \text{otherwise.}
	\end{cases}\]
	In general, let $K'/K$ be the completion of an algebraic closure of $K$, let $X' = X\hat\tensor_K K'$ with structure morphism $\pi\colon X'\to X$, let $\bT_{K'} = \bT\tensor_K K'$, and let $\phi' = \phi\hat\tensor_K K'\colon X'\to\bT_{K'}^\an$.  Then we define
	\begin{equation}\label{eq:trop.mult.non.closed}
	m_\phi(X,x) = \sum_{\pi(x')=x}m_{\phi'}(X', x').
	\end{equation}
	For $\omega\in N_\R$ we define
	\begin{equation}\label{eq:trop.mult}
	m_\phi(X,\omega) = \sum_{\phi_{\trop}(x)=\omega} m_\phi(X,x).
	\end{equation}
\end{defn}

\begin{art}  \label{rem: tropical multiplicities}
	Some remarks:
	\begin{enumerate}
		\item\label{tropmult 1}
		Clearly $m_\phi(X,x)$ only depends on the restriction of $\phi$ to the germ $(X,x)$.
		\item\label{tropmult 2}
		If $x\in S_\phi(X)$ then the ring $\sO_{X,x}$ is Artinian by~\cite[Example~3.2.10]{ducros18:families}.  In particular, if $X$ is reduced then $\sO_{X,x}$ is a field and $\length(\sO_{X,x}) = 1$.
		\item\label{tropmult 3}
                  Suppose that $K$ is algebraically closed.  Let $x\in S_\phi(X)$.  Then $\td\sH(x)^\bullet$ is finitely generated over $\td K^\bullet$ by~\cite[3.2.12.3]{ducros14:structur_des_courbes_analytiq}.  Since
		$\trdeg\bigl(\td K^\bullet(\td M(x))/\td K^\bullet\bigr)=d=\trdeg\bigl( \td\sH(x)^\bullet/\td K^\bullet\bigr)$, we conclude that
		$\td\sH(x)^\bullet$ is algebraic over $\td K^\bullet(\td M(x))$ and we have $m_\phi(X,x) < \infty$.
		\item\label{tropmult 4}
		For general $K$, the sum~\eqref{eq:trop.mult.non.closed} is finite by
		(\ref{item:trop.skeleton.basechange}) and (\ref{item:algebraic extension}) in Proposition~\ref{prop:tropical.skeleton}.
		\item\label{tropmult 5} For general $K$, we have $m_\phi(X,x) > 0$ if and only if $x\in S_\phi(X)$ by Proposition~\ref{prop:tropical.skeleton}(\ref{item:trop.skeleton.basechange}).
		\item\label{tropmult 6}
		In Definition~\ref{def:trop.mult}, we can take $K'/K$ to be any algebraically closed non-Archimed\-ean extension field by Proposition~\ref{item:trop.mult.basechange} below.
              \item\label{tropmult 7}
                The fibers of $S_\phi(X)\to\phi_{\trop}(X)$ are finite by Corollary~\ref{cor:tropical.skeleton}, so the sum~\eqref{eq:trop.mult} is finite by~(\ref{tropmult 5}).  If $\phi_{\trop}(X)$ has dimension $d$ at a point $\omega\in\phi_{\trop}(X)$, then $m_\phi(X,\omega) > 0$ by Corollary~\ref{cor:tropical.skeleton} and~(\ref{tropmult 5}).
                \item\label{tropmult 8}
                  Let $x\in S_\phi(X)$ with $\omega = \phi_{\trop}(x)\in N_\Gamma$.  Then $\trdeg(\td K(\td M(x))/\td K) = d$ by  Remark~\ref{connection to usual residue field}, so $[|\sH(x)^\times|:|K^\times|]$ is finite.
                  If $K$ is algebraically closed, then $|K^\times|$ is divisible, so $|\sH(x)^\times|=|K^\times|$; then~\eqref{eq:n.ef} yields
                  \[ [\td\sH(x)^\bullet:\td K^\bullet(\td M(x))] = [\td\sH(x):\td K(\td M(x))], \]
                  so in this case it is not necessary to use graded residue fields.  By~(\ref{tropmult 6}) the tropical multiplicity can always be defined entirely in terms of ordinary residue fields.
                 \item\label{tropmult 9} The use of base change to the algebraic closure in the definition of $m_\varphi(X,x)$ is necessary.
                   It implies that the tropical multiplicity $m_\varphi(X,\omega)$ is invariant under base extensions of non-archimedean fields (Proposition~\ref{item:trop.mult.basechange}), and that it coincides with the multiplicities in~\cite{CLD} (Proposition~\ref{item:trop.mult.samedim}).                 	
    	\end{enumerate}
\end{art}

\begin{eg}
  We illustrate Remark~\ref{rem: tropical multiplicities}(\ref{tropmult 9}) here in a zero-dimensional example. Let $L/K$ be a finite, separable field extension. Then $X=\sM(L)$ consists of a single point $x$ given by the unique absolute value on $L$ extending the given one on $K$, and $\sO_{X,x} = \sH(x) = L$.  Define $\varphi\colon X \to \bG_{{\rm m},K}^\an = \Spec(K[t^{\pm1}])^\an$ by the $K$-algebra homomorphism $K[t^{\pm1}]\to L$ sending $t\mapsto 1$.  Then $\phi_{\trop}(X) = \{0\}$, and the tropical skeleton $S_\varphi(X)$ is equal to $X$.

  Let $K'$ be the completion of an algebraic closure of $K$,  let $X' = X\hat\tensor_K K' = \sM(L\hat\tensor_K K')$, and let $\phi' = \phi_{K'}\colon X'\to\bG_{{\rm m},K'}^\an$.  Note that $L\hat\tensor_K K' = L\tensor_K K'$ is a product of $[L:K]$ copies of $K'$, so that $X' = \{x_1',\ldots,x_{[L:K]}'\}$ is a collection of $[L:K]$ rig-points, and $\sO_{X',x_i'} = \sH(x_i') = K'$ for each $i$.  This implies $S_{\phi'}(X') = X'$.

  Let $\omega = 0\in\R = N_\R$. Using Definition \ref{def:trop.mult}, we get
                 $$m_\varphi(X,\omega)=  m_\phi(X,x)= \sum_{i=1}^{[L:K]} m_{\varphi'}(X',x_i') = \sum_{i=1}^{[L:K]} \length(\sO_{X',x_i'})\,[\widetilde{K'}{}^\bullet:\widetilde{K'}{}^\bullet]=[L:K]$$    
             and 
              $$m_\varphi(X',\omega)=  \sum_{i=1}^{[L:K]} m_{\varphi'}(X',x_i')= \sum_{i=1}^{[L:K]} \length(\sO_{X',x_i'})\,[\widetilde{K'}{}^\bullet:\widetilde{K'}{}^\bullet]=[L:K],$$
              which illustrates invariance of the tropical multiplicities under base change to $K'$.

              Let $m_\phi'(X,x)=\length(\sO_{X,x})[\td\sH(x)^\bullet:\td K^\bullet]$: this would be the definition of the multiplicity without first passing to the algebraic closure.  The identity \eqref{eq:n.ef} yields that
          \[ [\td\sH(x)^\bullet:\td K^\bullet]=[\td L^\bullet:\td K^\bullet]=e(L/K)f(L/K), \]
      where $e(L/K)$ is the ramification index and $f(L/K)$ is the classical residue degree. Recall that a non-archimedean field $K$ is called \emph{stable} if $e(F/K)f(F/K)=[F:K]$ for any finite field extension $F/K$ (the inequality $\leq$ holds in general by~\cite[Proposition~3.1.3/2]{bosch_guntzer_remmert84:non_archimed_analysis}). 
      If $K$ is a stable field, then $e(L/K)f(L/K)=[L:K]$, so $m_\phi'(X,x) = m_\phi(X,x)$, and it is not necessary to pass to the completion of an algebraic closure in this case.
  
      We take the following example of a non-archimedean field $K$ which is not stable from \cite{bosch_guntzer_remmert84:non_archimed_analysis}. Let $K'$ be the completion of an algebraic closure of $\Q_2$ and let $K$ be the topological closure of the subfield of $K'$ generated by all $2^n$-roots of $2$. Then $K$ is a non-archimedean field which has additive value group $\bigcup_n \frac{1}{2^n} \Z$. It follows that every quadratic field extension of $K$ is unramified over $K$. It is shown in the example given in \cite[\S 3.6.1]{bosch_guntzer_remmert84:non_archimed_analysis} that $L=K(\sqrt{3})$ satisfies $f(L/K)=1$, so that $K$ is not stable. For this choice of $L$, we have $m'_\phi(X,x) = e(L/K)f(L/K)=1$, but $m_\varphi(X,\omega)= [L:K] = 2$.
\end{eg}

\begin{art}[Fundamental Cycles]\label{sec:fund.cycle}
  In~\cite[2.7]{gubler-crelle}, the fundamental cycle $[X]$ of the analytic space $X$ was defined as follows.  Let $U = \sM(\sA)$ be a strictly affinoid domain in $X$.  The cycle $[\Spec(\sA)]$ of the noetherian scheme $\Spec(\sA)$ has the form $\sum_{Y'}m_{Y'}[Y']$, where the sum runs over the irreducible components $Y'$ of $\Spec(\sA)$.  This construction is compatible with pull-back with respect to flat morphisms of noetherian schemes, from which it follows that these cycles glue to give a cycle $[X]$.  This cycle has the form $[X] = \sum_Y m_Y[Y]$, the sum being taken over all irreducible components $Y$ of $X$ (see~\cite[1.5]{ducros18:families} for a definition).  In particular, there is a canonical multiplicity $m_Y$ associated to an irreducible component $Y$.
\end{art}

The next lemma relates multiplicities of irreducible components to tropical multiplicities.  As a special case, if $K$ is algebraically closed then the factor of $\length(\sO_{X,x})$ in Definition~\ref{def:trop.mult} is simply the multiplicity of the irreducible component of $X$ containing $x$.

\begin{lem}\label{lem:mult.fund.cycle}
  Let $\iota\colon X_{\red}\inject X$ be the reduction of $X$ and let $\psi = \phi\circ\iota$.  If $x\in S_\phi(X)$, then $x$ is contained in a unique irreducible component $Y$ of $X$, and $m_\phi(X,x) = m_Y\,m_\psi(X_{\red},x)$.
\end{lem}

\begin{proof}
  Since multiplicities are local, we may pass to a neighbourhood of $x$ to assume that $X = \sM(\sA)$ is strictly affinoid.  If $x\in S_\phi(X)$ then $x$ is an Abhyankar point by Proposition~\ref{prop:tropical.skeleton}(\ref{item:trop.skeleton.abhyankar}), so it is contained in a unique irreducible component $Y$ by~\cite[Remark~1.5.9]{ducros18:families}.  Let $\fp\subset\sA$ be the kernel of the semi-norm corresponding to $x$.  Then $m_Y = \length(A_\fp)$.  The homomorphism $A_\fp\to\sO_{X,x}$ is flat by~\cite[Theorem~2.1.4]{Berkovichetale}, and $\fp\sO_{X,x}$ is the maximal ideal of $\sO_{X,x}$ by~\cite[Example~3.2.10]{ducros18:families}, so $m_Y = \length(\sO_{X,x})$ by~\cite[Lemma~A.4.1]{fulton98:intersec_theory}.

  Let $K'$ be the completion of an algebraic closure of $K$, let $X' = X\hat\tensor_K K'$, let $\pi\colon X'\to X$ be the structure morphism, let $\phi' = \phi\hat\tensor_K K'\colon X'\to\bT^\an\hat\tensor_K K'$ and let $x\in S_\phi(X)$.  We have $\pi\inv(x) \subset S_{\phi'}(X')$ by Proposition~\ref{prop:tropical.skeleton}(\ref{item:algebraic extension}).  For $x'\in\pi\inv(x)$ the map $\sO_{X,x}\to\sO_{X',x'}$ is flat by~\cite[Corollary~2.1.3]{Berkovichetale}.  Again it follows from that~\cite[Lemma~A.4.1]{fulton98:intersec_theory} that
  \[ \length(\sO_{X',x'}) = \length(\sO_{X,x})\,\length(\sO_{X',x'}/\fm_{X,x}\sO_{X',x'})
    = m_Y\,\length(\sO_{X',x'}/\fm_{X,x}\sO_{X',x'}).
  \]
  Since $\sO_{X,x}$ is Artinian, the ideal $\fm_{X,x}$ is its nilradical, so $\sO_{X_{\red},x} = \sO_{X,x}/\fm_{X,x}$.  Letting $x''$ denote the point of $X''\coloneq X_{\red}\hat\tensor_K K'$ mapping to $x'$, we thus have $\sO_{X'',x''} = \sO_{X',x'}/\fm_{X,x}\sO_{X',x'}$, so that $\length(\sO_{X',x'}) = m_Y\,\length(\sO_{X'',x''})$.  The lemma follows from this equality.
\end{proof}

We spend the rest of this section proving several important properties of the tropical multiplicity.

\begin{prop}  \label{item:trop.mult.samedim}
	Let $\varphi\colon X \to \bT^\an$ be a moment map as above and suppose that $\dim(\bT) = d$. Let $x\in S_\phi(X)$ and let $\xi = \phi(x)$.
	Then $m_\phi(X,x) = \dim_{\sH(\xi)}(\sO_{\phi\inv(\xi),x})$.  The local ring $\sO_{\bT^\an,\xi}$ is a field $\kappa(\xi)$, and if $x \in \Int(X)$, then $m_\phi(X,x) = \dim_{\kappa(\xi)}(\sO_{X,x})$.
\end{prop}

\begin{proof}
By Proposition~\ref{prop:tropical.skeleton}(\ref{item:trop.skeleton.samedim}), we have $\xi \in S(\bT)$ and hence $\xi$ is an Abhyankar point.
Suppose first that $K$ is algebraically closed.    The fiber $\phi\inv(\xi)$ is a finite $\sH(\xi)$-analytic space by~\cite[Lemma~1.5.11]{ducros18:families}.  In particular, the residue field $\kappa(x)$ of $x$ in $\phi\inv(\xi)$ is a finite normed $\sH(\xi)$-vector space, so $\kappa(x) = \sH(x)$ is complete.  We have the equality
	\[ \dim_{\sH(\xi)}(\sO_{\phi\inv(\xi),x}) = \length(\sO_{\phi\inv(\xi),x})\,[\sH(x):\sH(\xi)] \]
	by~\cite[Lemma~A.1.3]{fulton98:intersec_theory}.  Since $x$ is an Abhyankar point of $X$, the ring $\sO_{X,x}$ is Artinian, as is $\sO_{\phi\inv(\xi),x}$ since $\dim\phi\inv(\xi)=0$.  The homomorphism of local rings $\sO_{X,x}\to\sO_{\phi\inv(\xi),x}$ is regular by~\cite[Theorem~6.3.7]{ducros18:families}, so in particular, it is flat and $\sO_{\phi\inv(\xi),x}/\fm_{X,x}\sO_{\phi\inv(\xi),x}$ is a field.  It follows from~\cite[Lemma~A.4.1]{fulton98:intersec_theory} that $\length(\sO_{X,x}) = \length(\sO_{\phi\inv(\xi),x})$.  The field $\sH(\xi)$ is \emph{stable} by  Ducros~\cite[Theorem~4.3.14]{ducros14:structur_des_courbes_analytiq}, which means  that $[\sH(x):\sH(\xi)] = [\td\sH(x)^\bullet:\td\sH(\xi)^\bullet]$ as seen in  \eqref{eq:n.ef}.  Furthermore, we have $[\td\sH(x)^\bullet:\td\sH(\xi)^\bullet] = [\td\sH(x)^\bullet:\td K^\bullet(\td M(x))]$ by Example~\ref{eg:res.field.skeleton.point}.  This proves that $m_\phi(X,x)=\dim_{\sH(\xi)}(\sO_{\phi\inv(\xi),x})$ when $K$ is algebraically closed.

	When $K$ is not algebraically closed, we use the notations of Definition~\ref{def:trop.mult} and take $K'$ to be the completion of an algebraic closure of $K$.  Let $\omega = \phi_{\trop}(x)$.  By Proposition~\ref{prop:tropical.skeleton}(\ref{item:trop.skeleton.samedim}) and Corollary~\ref{cor:tropical.skeleton}, there are finitely many points in $\phi_{\trop}\inv(\omega)\cap S_\phi(X) = \phi\inv(\xi)$; hence Proposition~\ref{prop:tropical.skeleton}(\ref{item:trop.skeleton.subdomain}) shows that
	we may shrink $X$ to assume $\phi\inv(\xi) = \{x\}$.  Let $\xi'$ be the unique point of $S(\bT\tensor_K {K'})$ mapping to $\xi$.  Any point in $\phi'^{-1}(\xi')$ maps to $\phi\inv(\xi) = \{x\}$ under $\pi$, so $\pi\inv(x)\cap S_{\phi'}(X') = \phi'^{-1}(\xi')$ using Proposition~\ref{prop:tropical.skeleton}(\ref{item:trop.skeleton.samedim}) as applied to $\varphi'$.  Thus we calculate
	\begin{equation}\label{eq:m.phi.basechange}\begin{split}
	m_\phi(X,x) &= \sum_{x'\mapsto x}m_{\phi'}(X',x')
	= \sum_{\phi'(x') = \xi'} \dim_{\sH(\xi')}(\sO_{\phi'^{-1}(\xi'),x'}) \\
	&= \dim_{\sH(\xi')}\Gamma(\phi'^{-1}(\xi'), \sO_{\phi'^{-1 }(\xi')}) \\
	&= \dim_{\sH(\xi)}\Gamma(\phi\inv(\xi), \sO_{\phi\inv(\xi)})
	= \dim_{\sH(\xi)}(\sO_{\phi\inv(\xi),x}),
	\end{split}\end{equation}
	where the fourth equality holds true because $\phi'^{-1}(\xi') = \phi\inv(\xi)\tensor_{\sH(\xi)}\sH(\xi')$.

	Since $\xi$ is contained in the canonical skeleton $S(\bT)$ by  Proposition~\ref{prop:tropical.skeleton}(\ref{item:trop.skeleton.samedim}), it follows that $\Ocal_{\bT^\an,\xi} = \kappa(\xi)$ is a field by~\cite[0.19]{Ducros}.
	As above, we know that $x$ is isolated in $\phi\inv(\xi)$. If $x \in \Int(X)$, then $\phi$ is finite at $x$ by~\cite[Proposition~1.5.5(ii) and Corollary~3.1.10]{Berkovichetale}.  By~\cite[Proposition~3.1.4]{Berkovichetale}, we can shrink $X$ so that $X = \sM(\sA)$ and $Y = \phi(X) = \sM(\sB)$ are affinoid and $X\to Y$ is finite.
	Using $\phi\inv(\xi) = \{x\}$, we get
	$\sA\tensor_\sB\kappa(\xi) = \sO_{X,x}$
	by~\cite[Lemma~2.1.6]{Berkovichetale}, and the fiber $\phi\inv(\xi)$ is equal to $\sM(\sA\hat\tensor_\sB\sH(\xi)) = \sM(\sO_{\phi\inv(\xi),x})$, so
	\[ \sO_{\phi\inv(\xi),x} = \sA\hat\tensor_\sB\sH(\xi) = \sA\tensor_\sB\sH(\xi)
	= (\sA\tensor_\sB\kappa(\xi))\tensor_{\kappa(\xi)}\sH(\xi)
	= \sO_{X,x}\tensor_{\kappa(\xi)}\sH(\xi), \]
      where the second equality
	is~\cite[Proposition~3.7.3/6]{bosch_guntzer_remmert84:non_archimed_analysis}.
	It follows that
        \[ \dim_{\kappa(\xi)}(\sO_{X,x}) = \dim_{\sH(\xi)}(\sO_{\phi\inv(\xi),x}) = m_\phi(X,x). \]
\end{proof}

Recall from~\secref{Polyhedral geometry} that if $\sigma\subset N_\R$ is a $\Z$-polyhedron of dimesion $d$, then  $\bL_\sigma$ is the subspace underlying its affine span, and $N_\sigma = N\cap\bL_\sigma$ is the induced lattice.

\begin{prop} \label{item:trop.mult.projection}
  Let $x\in\Int(X)$, and suppose that $\phi_{\trop}(X,x)$ is the germ of a  plane $\sigma$ in $N_\R$ of dimension $d$.  Let $\bT' = \bGm^d$, let $q\colon\bT\to\bT'$ be a generic homomorphism with respect to $\sigma$ in the sense of Definition~\ref{def:generic.quotient}, and let $\psi = q\circ\phi$.  Let $Q\colon N\to N'$ be the homomorphism of cocharacter lattices corresponding to $q$.  Then $x \in S_\varphi(X)$, the index $[N':Q(N_\sigma)]$ is finite, and for $\xi'=\psi(x)$ we have
		\begin{equation}\label{eq:antoines.multiplicity}
		m_\phi(X,x) = \frac 1{[N':Q(N_\sigma)]}\dim_{\kappa(\xi')}(\sO_{X,x})
		= \frac 1{[N':Q(N_\sigma)]} m_\psi(X,x).
		\end{equation}
\end{prop}

\begin{proof}
By Theorem~\ref{thm:trop.dim.reduction}, we have $x \in S_\varphi(X)$.
Since $Q_\R$ is by hypothesis injective on $\sigma$ and since $N_\sigma$ and $N'$ both have rank $d$, it follows that $[N':Q(N_\sigma)] < \infty$.  The second equality in~\eqref{eq:antoines.multiplicity} holds by Proposition~\ref{item:trop.mult.samedim}.  Let $\bT''$ be the split torus with cocharacter lattice $Q(N_\sigma)$.  Then $q\colon\bT\to\bT'$ factors as a homomorphism $q'\colon\bT\to\bT''$ followed by a homomorphism $p\colon\bT''\to\bT'$ of degree $\deg(p) = [N':Q(N_\sigma)]$.  Let $\psi'=q'\circ\phi$ and let $\xi'' = \psi'(x) = q'(\xi)$.  Then $[\kappa(\xi''):\kappa(\xi')] = \deg(p) = [N':Q(N_\sigma)]$ by~\cite[2.4.1]{CLD}, so
	\[ m_{\psi'}(X,x) = \dim_{\kappa(\xi'')}(\sO_{X,x})
	= \frac{\dim_{\kappa(\xi')}(\sO_{X,x})}{[\kappa(\xi''):\kappa(\xi')]}
	= \frac{m_\psi(X,x)}{[N':Q(N_\sigma)]}
	\]
	by Proposition~\ref{item:trop.mult.samedim}.  Hence we may assume $Q(N_\sigma) = N'$.

	First assume that $K$ is algebraically closed.  Consider the annihilator $N_\sigma^\perp\subset M$ of $N_\sigma$.  Let $\bT_\sigma = \Spec(K[N_\sigma^\perp])$, let $q_\sigma\colon\bT\to\bT_\sigma$ be the homomorphism induced by the inclusion $N_\sigma^\perp\inject M$, and let $\psi_\sigma = q_\sigma\circ\phi\colon X\to\bT_\sigma$.  Unwinding the definitions, we find that $\psi_{\sigma,\trop}(X,x)$ is a single point.  Since $x\in\Int(X)$, Theorem~\ref{thm:trop.dim.reduction} implies that $\td K^\bullet(\td N_\sigma^\perp(x))$ is algebraic over $\td K^\bullet$.  Then $\td N_\sigma^\perp(x)\subset\td K^\bullet$ by~\cite[Lemme~2.3.5]{ducros14:structur_des_courbes_analytiq}.  Let $M'\subset M$ be the character lattice of $\bT'$.  Noting that
	$  (M'+N_\sigma^\perp)/N_\sigma^\perp \cong M'/(M' \cap N_\sigma^\perp)=M'$ and that
	the dual of $M'  \to M/N_\sigma^\perp$ is $N_\sigma \to N'$, we deduce
	$[M:(M'+N_\sigma^\perp)] = [N':Q(N_\sigma)]=[N':N']=1$ and hence $M = M'+N_\sigma$.  By the above,
	\[\begin{split} \frac{m_\phi(X,x)}{\length(\sO_{X,x})}
	&= [\td\sH(x)^\bullet:\td K^\bullet(\td M(x))]
	= [\td\sH(x)^\bullet:\td K^\bullet((M' + N_\sigma^\perp)^\sim(x))] \\
	&= [\td\sH(x)^\bullet:\td K^\bullet(\td M'(x))]
	= \frac{m_\psi(X,x)}{\length(\sO_{X,x})},
	\end{split}\]
	which proves $m_\phi(X,x) = m_\psi(X,x)$.

	When $K$ is not algebraically closed, we use the notations of Definition~\ref{def:trop.mult}, and we let $\psi' = \psi\hat\tensor_K K'\colon X'\to\bT'$.
	For any point $x'\in\pi\inv(x)$ we have $x'\in\Int(X') \cap S_{\varphi'}(X')$ by~\cite[Proposition~1.5.5]{berkovic93:etale_cohomology} and by~Proposition~\ref{prop:tropical.skeleton}(\ref{item:algebraic extension}).
	 Since $\phi'_{\trop} = \phi_{\trop}\circ\pi$ we have $\phi'_{\trop}(X',x')\subset\phi_{\trop}(X,x)$.  On the other hand, since $x'\in S_{\phi'}(X')$ we have $d_{\phi'}(X',x') = d$, so $\phi'_{\trop}(X',x') = \phi_{\trop}(X,x)$ is also a $d$-plane by Theorem~\ref{thm:trop.dim.reduction}.  Then
	\[ m_\phi(X,x) = \sum_{\pi(x') = x} m_{\phi'}(X',x')
	= \sum_{\pi(x') = x} m_{\psi'}(X',x')
	= m_\psi(X,x)
	\]
	by the above.
\end{proof}

\begin{prop}  \label{item:trop.mult.basechange}
Let $X' = X\hat\tensor_K K'$ for a non-Archimedean extension field $K'/K$, let $\pi\colon X'\to X$ be the structure morphism, and let $\phi' = \phi\hat\tensor_K K'\colon X'\to\bT^{\an}\hat\tensor_K K'$.  For $x\in X$ we have
\[ m_\phi(X,x) = \sum_{\pi(x')=x} m_{\phi'}(X',x'). \]
\end{prop}

\begin{proof}
	By Remark~\ref{rem: tropical multiplicities}(\ref{tropmult 5}),
	we may assume that $x \in S_\varphi(X)$ as otherwise all tropical multiplicities are zero.
	We reduce immediately to the case where $K$ and $K'$ are algebraically closed.  Choose a generic projection $q\colon\bT\to\bT'$ as in Theorem~\ref{thm:tropical.skeleton.projection}, so that $S_\psi(X) = S_\phi(X)$ for $\psi = q\circ\phi$.  Let $M'\subset M$ be the character lattice of $\bT'$.  Let $\bT'_{K'} = \bT'\tensor_K K'$ and let $\psi' = \psi\hat\tensor_K K'\colon X'\to\bT'^{\an}_{K'}$.  Let $\xi' = \psi(x) = q(\xi)\in S(\bT')$ and let $\xi''\in S(\bT'_{K'})$ be the unique point mapping to $\xi'$.  We shrink $X$ so that $\psi\inv(\xi') = \{x\}$ as in~Proposition~\ref{item:trop.mult.samedim}, so that $\psi'^{-1}(\xi'') = \pi\inv(x)\cap S_{\phi'}(X')$ using Proposition~\ref{prop:tropical.skeleton}(\ref{item:trop.skeleton.samedim}).

	By~\cite[Corollaire~4.3.15]{ducros14:structur_des_courbes_analytiq}, there is a unique point $x'\in\psi'^{-1}(\xi'') = \pi\inv(x)\cap S_{\phi'}(X')$.
	By Proposition~\ref{item:trop.mult.samedim} and using $\psi'^{-1}(\xi'')=\psi\inv(\xi')\tensor_{\sH(\xi')}\sH(\xi'')$ from its proof, we have
	\[ m_\psi(X,x) = \dim_{\sH(\xi')}(\sO_{\psi\inv(\xi'),x})
	= \dim_{\sH(\xi'')}(\sO_{\psi'^{-1}(\xi''),x'}) = m_{\psi'}(X',x'), \]
	so we need to show that $m_\psi(X,x)/m_\phi(X,x) = m_{\psi'}(X',x')/m_{\phi'}(X',x')$.
	This follows immediately from~ Proposition~\ref{item:trop.mult.projection} when $x\in\Int(X)$ and $\phi_{\trop}(X,x)$ is a $d$-plane (in which case both quotients are equal to $[N':Q(N_\sigma)]$); here we argue in general.  We have
	\[ \frac{m_\psi(X,x)}{m_\phi(X,x)}
	= \frac{\length(\sO_{X,x})[\td\sH(x)^\bullet:\td K^\bullet(\td M'(x))]}{
		\length(\sO_{X,x})[\td\sH(x)^\bullet:\td K^\bullet(\td M(x))]}
	= [\td K^\bullet(\td M(x)):\td K^\bullet(\td M'(x))]
	\]
	and likewise for $m_{\psi'}(X',x')/m_{\phi'}(X',x')$, so we need to show
	\begin{equation}\label{eq:finite.transcendental.basechange}
	[\td K^\bullet(\td M(x)):\td K^\bullet(\td M'(x))] =
	[\td K'^\bullet(\td M(x')):\td K'^\bullet(\td M'(x'))].
	\end{equation}
	This is an elementary exercise in graded algebra (modulo a result on geometric reducedness due to Ducros, cited below); the proof proceeds exactly as for ordinary fields.  Note that Ducros uses in \cite{ducros14:structur_des_courbes_analytiq} corpoids and anneloids instead of graded fields and graded rings, which means that he replaces the direct sum of the graded pieces by their disjoint union. This does not give any change in the used claims or arguments.

	Let $E = \td K^\bullet(\td M'(x))$, $F = \td K^\bullet(\td M(x))$, and $E' = \td K'^\bullet(\td M'(x'))$.  We claim that $F\tensor_E E'$ is a graded field.  By Example~\ref{eg:res.field.skeleton.point}, the graded fields $E = \td\sH(\xi')^\bullet$ and $E' = \td\sH(\xi'')^\bullet$ are purely transcendental extensions of $\td K^\bullet$ and $\td K'^\bullet$, respectively, with the same generators, so we have $E' = Q(E\tensor_{\td K^\bullet}\td K'^\bullet)$, where $Q(\scdot)$ denotes the quotient graded field.   By~\cite[2.2.30]{ducros14:structur_des_courbes_analytiq}, the graded ring $F\tensor_E (E\tensor_{\td K^\bullet}\td K'^\bullet) = F\tensor_{\td K^\bullet}\td K'^\bullet$ is an integral domain, so its localization $F\tensor_E Q(E\tensor_{\td K^\bullet}\td K'^\bullet) = F\tensor_E E'$ is also an integral domain.  Since $F\tensor_E E'$ is finite over the graded field $E'$ by Remark~\ref{rem: tropical multiplicities}(\ref{tropmult 3}), it is also a graded field~\cite[2.2.20.1]{ducros14:structur_des_courbes_analytiq}.

	It follows that the homomorphism $F\tensor_E E' = \td K^\bullet(\td M(x))\tensor_{\td K^\bullet(\td M'(x))}\td K'^\bullet(\td M'(x'))\to\td\sH(x')^\bullet$ is injective.  Its image is $K'^\bullet(\td M(x))$, which proves~\eqref{eq:finite.transcendental.basechange}.
	\end{proof}

\begin{prop}\label{item:trop.mult.1units}
		Let $\phi\colon X\to\bT^\an$ and $\phi'\colon X\to\bT^\an$ be  moment maps with  $\phi_{\trop}=\phi'_{\trop}$.  Then $m_\phi(X,x) = m_{\phi'}(X,x)$ for all $x\in\Int(X)$.
\end{prop}

\begin{proof}
	We easily reduce to the case when $K$ is algebraically closed.    Note that $S_{\phi}(X)\cap\Int(X) = S_{\phi'}(X)\cap\Int(X)$ by Proposition~\ref{prop:tropical.skeleton}(\ref{item:trop.skeleton.1units}).
	We conclude from  Remark~\ref{rem: tropical multiplicities}(\ref{tropmult 5})
	that we may assume that $x \in S_{\phi}(X)\cap\Int(X) = S_{\phi'}(X)\cap\Int(X)$ as otherwise the tropical multiplicities are zero.
	Let $u_1,\ldots,u_n$ be a basis for $M$, and let $f_i = \phi^*u_i$ and $f_i' = \phi'^*u_i$.  Using~\cite[Lemme~2.3.5]{ducros14:structur_des_courbes_analytiq}, we see as in the proof of Proposition~\ref{prop:tropical.skeleton}(\ref{item:trop.skeleton.1units}) that $\td\chi_x^\bullet(f_i)/\td\chi_x^\bullet(f_i')\in \td K^\bullet$ for all $i$.  Since $\td K^\bullet(\td M(x))$ is the field generated over $\td K^\bullet$ by $\td\chi_x^\bullet(f_1),\ldots,\td\chi_x^\bullet(f_1)$ using the moment map $\phi$, and by $\td\chi_x^\bullet(f_1'),\ldots,\td\chi_x^\bullet(f_1')$ using~$\phi'$, it is clear that $m_\phi(X,x) = m_{\phi'}(X,x)$.
\end{proof}

\begin{rem} \label{rem: multiplicity of the Antoines}
  Recall from~\secref{sec:trop.is.polyhedral} that we can choose a $(\Z,\Gamma)$-polytopal complex $\Pi$ with support  $\phi_{\trop}(X)$ such that $\phi_{\trop}(\del X)$ is contained in a subcomplex of dimension at most $d-1$.  Let $\omega \in \relint(\sigma)$ for a $d$-dimensional face $\sigma$ of $\Pi$,
and pick a homomorphism $q\colon \bT \to \bT'$ to a $d$-dimensional split torus $\bT'$ such that the corresponding homomorphism $Q\colon N \to N'$ of cocharacter lattices is injective on $\sigma$.
We consider $x \in X$ with $\varphitrop(x)=\omega$, so $x \in \Int(X)$. It follows from Theorem~\ref{thm:trop.dim.reduction} and Proposition~\ref{item:trop.mult.projection} that $x \in S_\varphi(X)$ if and only if $\varphitrop(X,x)$ is $d$-dimensional, which means here that it is the germ of $\sigma$ at $\omega$. If this is the case, then for $\psi = q \circ \varphi$, we see that  $\psi_{\trop}(X,x)$ is the germ of $N_\R'$ at $Q_\R(\omega)$, and hence
$\xi'=(q \circ \varphi)(x) \in S(\bT')$.

By \eqref{eq:antoines.multiplicity} and Remark~\ref{rem: tropical multiplicities}(\ref{tropmult 5}), we have
	\begin{equation} \label{tropical multiplicity omega}
          m_\phi(X,\omega) = \frac 1{[N':Q(N_\sigma)]}  \sum_{\substack{x \in S_\varphi(X)\\
		\varphitrop(x)=\omega}} \dim_{\kappa(\xi')}(\sO_{X,x}).
	\end{equation}
        Chambert-Loir and Ducros \cite[3.5]{chambert_ducros12:forms_courants} use that the sum on the right is constant on $\relint(\sigma)$ to define a canonical calibration on $\sigma$. This is equivalent to defining the tropical multiplicity $m_\sigma$ of $\sigma$ in $\varphitrop(X)$ by the right hand side of \eqref{tropical multiplicity omega}.  See \cite[Remark~7.6]{gubler16:forms_currents}.
\end{rem}

For convenience of the reader, we repeat here this crucial argument of Chambert-Loir and Ducros in the context of our tropical multiplicities.

\begin{prop} \label{prop: trop.mult.constant}
Let $\sigma$ be a $d$-dimensional face of $\varphitrop(X)$ with $\relint(\sigma)\cap \varphitrop(\partial X)= \emptyset$. Then the tropical multiplicity $m_\varphi(X,\omega)$ is constant in  $\omega \in \relint(\sigma)$.
\end{prop}

\begin{proof}
  Replacing $X$ by $\varphi^{-1}_{\trop}(\sigma)$, we may assume that $\sigma =\varphitrop(X)$. Pick $q\colon \bT \to \bT'$ as in Remark~\ref{rem: multiplicity of the Antoines}, which means here that $q$ is generic with respect to $\varphitrop(X)=\sigma$. Let $\psi = q \circ \varphi$ and let $\xi'$ be the unique point in $S(\bT')$ with $\trop(\xi')=Q_\R(\omega)$. 
  Since the points of $X$ over $\omega$ are interior  and using $\sigma =\varphitrop(X)$, it follows from Theorem~\ref{thm:tropical.skeleton.projection} that $\{x \in S_\varphi(X)\mid \varphitrop(x)=\omega\}=\{x \in S_\psi(X)\mid \psi_{\trop}(x)=Q_\R(\omega)\}$. We conclude that this set agrees with
  $\{x \in X \mid \psi(x)=\xi'\}$ as we have $S_\psi(X)=\psi^{-1}(S(\bT'))$ by Proposition~\ref{prop:tropical.skeleton}(\ref{item:trop.skeleton.samedim}) and so we get
\begin{equation} \label{degree sums}
  \sum_{\substack{x \in S_\varphi(X)\\
		\varphitrop(x)=\omega}} \dim_{\kappa(\xi')}(\sO_{X,x})
            = \sum_{\substack{x \in X\\
		\psi(x)=\xi'}} \dim_{\kappa(\xi')}(\sO_{X,x}).
\end{equation}
By~\cite[2.4.3]{CLD}, the map $\psi$ is finite and flat over a neighbourhood of $\xi'$. It follows from \cite[3.2.5]{ducros14:structur_des_courbes_analytiq}  that the right hand side of  \eqref{degree sums}     is the degree of $\psi$ over $\xi'$.   By  \cite[3.2.6]{ducros14:structur_des_courbes_analytiq}, the degree of a finite flat morphism is locally constant. Combining these facts with  \eqref{tropical multiplicity omega}  and  \eqref{degree sums}, we deduce that $m_\varphi(X,\omega)$ is constant in a neighbourhood of $\omega$. Since $\relint(\sigma)$ is connected, we get the claim.
\end{proof}

\begin{defn}\label{def:trop.mult.face}
	Choose a $(\Z,\Gamma)$-polytopal complex $\Pi$ with support $\phi_{\trop}(X)$ such that $\phi_{\trop}(\del X)$ is contained in a subcomplex of dimension at most $d-1$, as in~Remark \secref{rem: multiplicity of the Antoines}.  Let $\sigma$ be a $d$-dimensional face of $\Pi$.  Then we define the \defi{tropical weight}  of $\sigma$ by  $m_\sigma \coloneqq m_\phi(X,\sigma) \coloneq m_\phi(X,\omega)$ for any $\omega\in\relint(\sigma)$. This is well-defined by Proposition~\ref{prop: trop.mult.constant}.
\end{defn}

\begin{rem} \label{rem: tropical multiplicities constant on faces of skeleton}
We have seen in Remark~\ref{rem: piecewise linear structure on tropical skeleton} that the tropical skeleton $S_\varphi(X)$ has the structure of a piecewise $(\Z,\Gamma)$-linear space such that $\varphitrop$ induces a piecewise $(\Z,\Gamma)$-linear map to $\varphitrop(X)$ which is finite to one. Replacing $\Pi$ from Definition~\ref{def:trop.mult.face} by a suitable subdivision, this piecewise linear structure can be given by an abstract $(\Z,\Gamma)$-polytopal complex with support $S_\varphi(X)$ such that $\varphitrop$ maps every face $\Delta$ of this complex by a $(\Z,\Gamma)$-linear isomorphism onto a face $\sigma$ of $\Pi$. If $\dim(\Delta)=d$, then we claim that $m_\varphi(X,x)$ is constant in $x \in \relint(\Delta)$ and we define the \defi{tropical weight  $m_\Delta \coloneqq m_\varphi(X,\Delta) \coloneqq m_\varphi(X,x)$ in $\Delta$}.  This makes $S_\phi(X)$ into a weighted abstract $\z$-polytopal complex supported on the $d$-skeleton $S_\phi(X)_d$  (Remark~\ref{rem:d-skeleton}).

To show this, we note first that we may shrink $X$ to assume that $S_\varphi(X)= \Delta$ by using Remark~\ref{rem: tropical multiplicities}(\ref{tropmult 1}) and Proposition~\ref{prop:tropical.skeleton}(\ref{item:trop.skeleton.subdomain}). Then for $x \in \relint(\Delta)$, we have $\omega \coloneqq \varphitrop(x) \in \relint(\sigma)$ and $m_\varphi(X,x)=m_\varphi(X,\omega)$.  Hence the claim follows from Proposition~\ref{prop: trop.mult.constant}.
\end{rem}

\begin{art}[Relative Multiplicities from~\cite{BPR}]\label{sec:relative.mult}
  Suppose that $K$ is algebraically closed.  Let $\sX\inject\bT$ be a reduced closed subscheme of pure dimension $d$, let $X = \sX^\an$, and let $\phi\colon X\inject\bT^\an$ denote the inclusion.  In this situation, the following  \defi{relative multiplicity} was introduced in~\cite[Definition~4.21]{BPR}.

  Let $\omega \in \phi_{\trop}(X)$.  If $\omega\in N_\Gamma$ then $X^\omega\coloneq\phi_{\trop}\inv(\omega)$ is a strictly $K$-affinoid subdomain $\sM(\sA_\omega)\subset X$.  There are two natural admissible formal $K^\circ$-models of $X^\omega$: one is the canonical model $\fX^\omega_\can = \Spf(\sA_\omega^\circ)$, and the other is the polyhedral model $\fX^\omega = \Spf(A_\omega)$, where $A_\omega$ is the subalgebra of $\sA^\circ_\omega$ topologically generated over $K^\circ$ by $\{\chi^u/\lambda(u)\mid u\in M\}$ as in Remark~\ref{connection to usual residue field}.  There is a natural morphism $\iota\colon\fX^\omega_\can\to\fX^\omega$ which induces a finite morphism on special fibers $\iota_s\colon\fX^\omega_{\can,s}\to\fX^\omega_s$ that sends generic points to generic points.  For $x\in X$ with $\phi_{\trop}(x)=\omega$, if $x$ is a Shilov boundary point of $X^\omega$ then the reduction $\red_{\fX^\omega_\can}(x)\in\fX^\omega_{\can,s}$ is the generic point of an irreducible component $\fC_x$ of $\fX_{\can,s}$; the relative multiplicity $m_\rel(x)$ is defined to be $[\fC_x:\iota_s(\fC_x)]$, the degree of $\fC_x$ over its image in $\fX^\omega_s$.  If $x\in\phi_{\trop}\inv(\omega)$ is not a Shilov boundary point then $m_\rel(x)\coloneq0$.

If $\omega\notin N_\Gamma$ then we define $m_\rel$ on $\phi_{\trop}\inv(\omega)$ after extending scalars: choose an algebraically closed non-Archimedean extension field $K'/K$ with value group $\Gamma'$ large enough that $\omega\in N_{\Gamma'}$.  We set $X' = X\hat\tensor_K K'$, we let $\pi\colon X'\to X$ be the structure morphism, and for $x\in\phi_{\trop}\inv(\omega)$ we define $m_\rel(x) = \sum_{\pi(x')=x} m_\rel(x')$.

For $\omega\in\phi_{\trop}(X)$,   the quantity $\sum_{\phi_{\trop}(x)=\omega}m_\rel(\omega)$ is the ``classical'' tropical multiplicity $m_{\Trop}(\omega)$ as defined in terms of initial degenerations, see~\cite[Proposition~4.24]{BPR}.
\end{art}

\begin{lem}\label{lem:mult.from.BPR}
  In the situation of~\secref{sec:relative.mult}, we have $m_\rel(x) = m_\phi(X,x)$ for all $x\in X$, and $m_{\Trop}(\omega) = m_\phi(X,\omega)$ for all $\omega\in\phi_{\trop}(X)$.
\end{lem}

\begin{proof}
  The second statement follows from the first.  Let $x\in X$ and let $\omega = \phi_{\trop}(x)$.  Replacing $X$ by an appropriate base change and using  Proposition~\ref{item:trop.mult.basechange}, we may assume $x\in N_\Gamma$.  Then $S_\phi(X)\cap X^\omega$ is the Shilov boundary of $X^\omega$ by Remark~\ref{rem:trop.skel.GRW}.  Hence we may assume $x\in S_\phi(X)\cap X^\omega$ as otherwise all tropical multiplicities are zero.  Using the notations of~\secref{sec:relative.mult}, we let $\td K(\fC_x)$ denote the function field of $\fC_x$.  The homomorphism $\sA_\omega\to\sH(x)$ induces an isomorphism $\td K(\fC_x)\isom\td\sH(x)$ by~\cite[Proposition~2.4.4(ii)]{BerkovichSpectral}.  By the definition of $A_\omega\subset\sA_\omega^\circ$, the function field of $\iota_s(\fC_x)$ is precisely the subfield $\td K(\td M(x))$ of Remark~\ref{connection to usual residue field}.  We have $\length(\sO_{X,x})=1$ because $X$ is reduced, so
  \[ m_\phi(X,x) = [\td\sH(x):\td K(\td M(x))] = [\fC_x:\iota(\fC_x)] = m_\rel(x) \]
by Remark~\ref{rem: tropical multiplicities}(\ref{tropmult 8}).
\end{proof}

\begin{eg}[Continuation of Example~\ref{eg:tropical.skeleton.specific}]
  \label{eg:tropical.skeleton.multiplicities}
  We continue with the notation in Example~\ref{eg:tropical.skeleton.specific}.  By~\cite[Theorem~5.8]{BPR}, for any point $P$ in the interior of an edge $e$ in the skeleton $S$, the relative multiplicity $m_\rel(P)$ is equal to the ``expansion factor'' of the edge~$e$ with respect to the tropicalization map $\varphitrop$.  In this situation, the expansion factor of every edge is equal to $1$, except for the edge $BD$, which has expansion factor $0$ as it is crushed to a point.  
  
  We claim that  $m_\phi(X,P)=1$ for $P\in\{A,B,C,D\}$. This can be checked directly from the definition of $m_\phi(X,P)$. We give the proof in case $P=A$ which reduces to the generic point of the line $x=0$ in the special fiber of $\sX_{\min}$.  
  By \cite[Proposition 2.4.4(ii)]{berkovic90:analytic_geometry}, we have $\td\sH(A)=\td K(y)$. Recall that $\trop(A)=(1,0)$. If $M=\Z^2$ be the character lattice of $\T=\bGm^2$, then  $\td K(\td M(A))$ is the subfield of $\td\sH(A)$ generated by the reductions of $x/p$ and of $y/1=y$. This shows $\td\sH(A)=\td K(\td M(A))$, and hence  \ref{rem: tropical multiplicities}\eqref{tropmult 8} yields  that $m_\phi(X,A)=[\td\sH(A):\td K(\td M(A))]=1$. The proof for the points $P=B,C,D$ is similar. 
  
  We conclude that $m_\rel(P) = m_\phi(X,P) =1$ for every point $P$ of $S_\phi(X)$.  
  It follows that $m_\phi(X,\omega)$ just counts the preimages of $\omega$ in $S_\phi(X)$ under $\phi_{\trop}$; these tropical multiplicities are indicated in Figure~\ref{fig:genus3.trop}.
\end{eg}

\section{Piecewise linear functions} \label{section: pl functions}

In this section we discuss piecewise linear functions on analytic spaces.  Most of these results are standard in one form or another: for instance, see~\cite{CLD,gubler-crelle,gubler-martin}.  However, the terminology in the literature is somewhat nonstandard, so it is important to fix our ideas.  We will find it useful to work slightly more generally with piecewise linear functions with slopes contained in an additive subgroup $\Lambda\subset\R$, and in some cases new arguments are needed.  In practice we will usually take $\Lambda=\Z,\Q,$ or $\R$, in which case either $\Lambda=\Z$ or $\Lambda$ is divisible.

In this section, $K$ is a non-Archimedean field with a non-trivial valuation $v$ and value group $\Gamma \coloneqq v(K^\times) \subset \R$.  We consider a good strictly $K$-analytic space $X$ and an additive subgroup $\Lambda\subset\R$. Recall from~\secref{Non-Archimedean geometry} that a strictly analytic space is \defi{good} provided that that every point has a strictly affinoid neighbourhood.  Most strictly analytic spaces occurring in practice are good: for example, any strictly affinoid space, a space with no boundary, or an analytic subdomain of the analytification of an algebraic variety. We restrict our attention to good spaces as then the coherent modules on the Berkovich space $X$ are the same as the coherent modules on the space $X_G$ endowed with the $\rG$-topology \cite[Proposition 1.3.4]{berkovic93:etale_cohomology}, which makes it reasonable to work with line bundles on $X$. 
 We always use the $\rG$-topology on $X$ generated by the strictly affinoid domains. Let $C(X)$ be the space of continuous real functions on $X$.

\begin{defn} \label{piecewise linear function}
  A function $h \colon X \to \R$ is called \defi{piecewise $\Lambda$-linear} or \defi{$\Lambda$-PL} if there exists a $\rG$-covering $\{X_i\}$ of $X$ such that for each $i$, there are finitely many $\lambda_{ij}\in\Lambda$ and  $f_{ij}\in \sO(X_i)^\times$ with  $h|_{X_i} = \sum_j\lambda_{ij}\log|f_{ij}|$.  A $\Z$-PL function is called \defi{piecewise linear} or \defi{PL}.

  We denote the group of PL functions on $X$ by $\PL(X)$, and the group of $\Lambda$-PL functions by $\PL_\Lambda(X)$.  These are subgroups of $C(X)$.
\end{defn}

\begin{art}\label{sec:PL.remarks}
  Some remarks:
  \begin{enumerate}
  \item A function $h\colon X\to\R$ is PL if and only if there exists a $\rG$-covering $\{X_i\}$ of $X$ and  $f_i\in \cO(X_i)^\times$ such that $h|_{X_i} = \log|f_i|$.
  \item The set of $\Lambda$-PL functions on $X$ forms a group since any two $\rG$-coverings $\{X_i\}$ and $\{X_j'\}$ admit the common refinement $\{X_i\cap X_j'\}$.
  \item The restriction of a $\Lambda$-PL function to a strictly analytic domain is $\Lambda$-PL.
  \item The presheaf $U\mapsto\PL_\Lambda(U)$ is a sheaf in the $\rG$-topology, and therefore in the analytic topology as well.
  \item A function $h\colon X\to\R$ is $\Lambda$-PL if and only if there exists a $\rG$-covering $\{X_i\}$ of $X$ such that for each $i$, there are finitely many $\lambda_{ij}\in\Lambda$ and $h_{ij}\in\PL(X_i)$ with $h|_{X_i} = \sum_j\lambda_{ij}h_{ij}$.
  \item A function $h\colon X\to\R$ is $\Q$-PL if and only if, locally in the analytic topology, there exists a positive integer $N$ such that $Nh$ is PL.
  \item If $\Lambda=\Z$ or $\Lambda=\Q$ then the minimum and maximum of two $\Lambda$-PL functions is again $\Lambda$-PL by~\cite[Proposition~2.12(d)]{gubler-martin}.  
  \end{enumerate}
\end{art}

\begin{eg}
  The assertion of~\secref{sec:PL.remarks}(7) is not true for $\Lambda$ strictly containing~$\Q$.   For example, if the value group of $K$ is $\Q$ and $X=\bGm^\an = \Spec(K[t^{\pm1}])^\an$ is the split torus of rank $1$, then the tropicalization of any strictly affinoid domain in $X$ is a union of compact intervals with rational vertices. We conclude that for $\Lambda=\R$, the function $h \coloneqq \pi \log|t|+1=-\pi\trop +1$ is $\Lambda$-PL, but $\max(h,0)$ is not $\Lambda$-PL as $\max(h,0)=\varphi \circ \trop$ for  $\varphi(u)=\max(-\pi u+1,0)$, and as the piecewise linear function $\varphi$ has a bend at the irrational number $1/\pi$.
\end{eg}

Here we list some basic functoriality properties of $\Lambda$-PL functions.

\begin{lem}\label{lem:pl.basic.props}
  Let $h\colon X\to\R$ be a function.
  \begin{enumerate}
  \item\label{item:pl.prop.pullback}
    Let $f\colon X'\to X$ be a morphism of good strictly $K$-analytic spaces.  If $h\in\PL_\Lambda(X)$ then $h\circ f\in\PL_\Lambda(X')$.
  \item\label{item:pl.prop.basechange}
    Let $K'/K$ be a non-Archimedean field extension, let $X' = X\hat\tensor_K K'$, and let $\pi\colon X'\to X$ be the structure morphism.  If $h\in\PL_\Lambda(X)$ then $h\circ\pi\in\PL_\Lambda(X')$.
  \item\label{item:pl.prop.reduction}
    Let $\iota\colon X_{\red}\to X$ be the inclusion of the reduction of $X$, and let $h' = h\circ\iota$.  Then $h\in\PL_\Lambda(X)$ if and only if $h'\in\PL_\Lambda(X_{\red})$.
  \end{enumerate}
\end{lem}

\begin{proof}
  The first two assertions are obvious.  For the third, by~(\ref{item:pl.prop.pullback}), the only thing to show is that if $h'$ is $\Lambda$-PL then $h$ is  $\Lambda$-PL.  Let $U = \sM(\sA)\subset X$ be a strictly affinoid domain, and let $U_{\red} = \sM(\sA_{\red})\subset X_{\red}$ be its reduction.  Suppose that $h'|_{U_{\red}} = \sum\lambda_i\log|g_i|$ for $g_i\in\sA_{\red}^\times$.  If $f_i\in\sA$ lifts $g_i$ then $f_i\in\sA^\times$ is a unit, and $h|_U = \sum\lambda_i\log|f_i|$.  Since $X_{\red}$ has a $\rG$-cover by such affinoid domains, the same is true for $X$, so $h$ is $\Lambda$-PL.
\end{proof}

At this point we assume that the reader is familiar with formal geometry as recalled in~\secref{Formal geometry and models}.  We give an interpretation of  $\Lambda$-PL functions in terms of formal geometry which we will use to prove some more subtle properties of piecewise $\Lambda$-linearity.  Recall that a \defi{formal $\kcirc$-model} of a paracompact strictly $K$-analytic space $\fX$ is an admissible formal scheme $\fX$ over $\kcirc$ equipped with an identification $\fX_\eta = X$.

\medskip
A \defi{frame} of a line bundle on an open set is a nowhere vanishing section.

\begin{art}[Functions from Formal Models of $\sO_X$] \label{model function}
  We assume now that $X$ is paracompact.  This allows us to use Raynaud's theory of formal $K^\circ$-models, as in~\secref{Formal geometry and models}. We will later apply this only locally on a good, strictly analytic space, where we can always find a \emph{compact} strictly analytic neighbourhood (even a strictly affinoid neighbourhood as $X$ is good).  Hence the reader loses nothing by assuming here and in the following that $X$ is compact.  (Paracompactness has the advantage that one can apply these considerations also to the analytification of algebraic varieties over $K$.)
  
  Let $\fX$ be a formal $K^\circ$-model of $X$. The isomorphism classes of line bundles  $\fL$ on $\fX$ equipped with an identification $\fL_\eta=\sO_X$ form an abelian group denoted by $M(\fX)$.  Then $\fL \in M(\fX)$ determines a canonical metric $\metr_\fL$ on $\sO_X$, defined as follows: if $s$ is a frame of $\fL$ over a formal open subset $\fU\subset\fX$, then we set $\|s\| = 1$ on $\fU_\eta$.  This metric is well-defined because any two frames differ by a unit with constant absolute value $1$.

  Equivalently, we get a canonical function $h_\fL = -\log\|1\|_\fL\colon X\to\R$ with $-\log\|f\|_\fL = -\log|f| + h_\fL$ for an analytic function $f$ on any open subset $U$ of $X$. If $U=\fU_\eta$ for a formal open subset $\fU$ of $\fX$ and if $f\in\Gamma(\fU,\fL)\subset\Gamma(\fU_\eta,\sO_X)$ is a frame of $\fL$,
then by definition $0 = -\log|f| + h_\fL$, so $h_\fL$ is PL because $X$ has a $\rG$-cover by such~$U=\fU_\eta$.
\end{art}

\begin{art}[PL Functions are Model Functions] \label{model theorem}
  If $X$ is paracompact, then \cite[Proposition~2.10]{gubler-martin} yields  that $h\colon X\to\R$ is PL if and only if there exists a formal $K^\circ$-model $(\fX,\fL)$ of $(X,\sO_X)$ such that $h = -\log\|1\|_\fL$.
\end{art}

\begin{art}[$\Lambda$-PL Functions from Formal Models]\label{sec:equiv.pl.setup}
  For paracompact $X$,   we extend~\secref{model theorem} to the case of $\Lambda$-PL functions as follows.  Choose a formal $\kcirc$-model $\fX$ of $X$, and set $M(\fX)_\Lambda \coloneqq M(\fX) \otimes_\Z \Lambda$.  
  We have a unique homomorphism
  $h\colon M(\fX)_\Lambda \longrightarrow C(X)$, written $\fL\mapsto h_\fL$, such that $h_{\lambda\fL} = -\lambda\log \|1\|_\fL$ for all $\fL \in M(\fX)$ and $\lambda\in\Lambda$.  Note that if $\fL = \sum_{i=1}^m\lambda_i\fL_i\in M(\fX)_\Lambda$ then $h_\fL = \sum_{i=1}^m\lambda_i h_{\fL_i}\in\PL_\Lambda(X)$ by~\secref{sec:PL.remarks}(5).

  If $f\colon\fX'\to\fX$ is a morphism of formal $K^\circ$-models of $X$, then $h_{f^*\fL} = h_\fL$ for all $\fL\in M(\fX)_\Lambda$.  More generally, if $f\colon\fX'\to\fX$ is any morphism of separated admissible formal schemes
  then $h_{f^*\fL} = h_\fL\circ f_\eta$ for all $\fL\in M(\fX)_\Lambda$.
  (See~\secref{Formal geometry and models} for separatedness of formal models.)

  For $\fL \in M(\fX)_\Lambda$, we get a unique continuous metric $\metr_\fL$ on $\sO_X$  defined by $-\log \|f\|_\fL = -\log|f| + h_\fL$ for an analytic function $f$ on any open subset $U$ of $X$.
\end{art}

We want to extend the result of~\secref{model theorem} locally in the analytic topology to $\Lambda$-coefficients as in~\secref{sec:equiv.pl.setup}.  To do this we need to show that a $\Lambda$-PL function is  locally \emph{in the analytic topology} a linear combination of PL functions with $\Lambda$-coefficients.  (For our purposes, we could have taken the statement of Proposition~\ref{prop:Lpl.analytic.nbhd} as the definition of $\Lambda$-PL; we chose Definition~\ref{piecewise linear function} because it is more compatible with the terminology in the literature.)

\begin{prop}\label{prop:Lpl.analytic.nbhd}
  Suppose that $\Lambda=\Z$ or that $\Lambda$ is divisible.  A function $h\colon X\to\R$ is $\Lambda$-PL if and only if for every $x\in X$, there exists an open neighbourhood $U$ of $x$ such that $h|_U = \sum_{i=1}^m\lambda_ih_i$ for finitely many $\lambda_1,\ldots,\lambda_m\in\Lambda$ and $h_1,\ldots,h_m\in\PL(U)$.
\end{prop}

Before proving the proposition, we state the desired corollary.

\begin{cor}\label{cor:hL.is.pl}
  Suppose that $\Lambda=\Z$ or that $\Lambda$ is divisible.  A real function $h\colon X\to\R$ is  $\Lambda$-PL if and only if for every point $x\in X$, there exists a compact strictly analytic neighbourhood $U$ of $x$, a formal $K^\circ$-model $\fU$ of $U$, and $\fL\in M(\fU)_\Lambda$, such that $h|_U = h_\fL$.  In this case, such a neighbourhood $U$ can be chosen arbitrarily small.
\end{cor}

\begin{proof}
  Suppose that $h$ is  $\Lambda$-PL.  By Proposition~\ref{prop:Lpl.analytic.nbhd}, after shrinking $X$, we may assume that $h = \sum_{i=1}^m\lambda_ih_i$ for $\lambda_1,\ldots,\lambda_m\in\Lambda$ and PL functions $h_1,\ldots,h_m\colon X\to\R$. 
  Since every point of $X$ has a compact strictly analytic neighbourhood, we may also assume that $X$ is compact. By~\secref{model theorem}, for each $i$ there is a model $(\fX_i,\fL_i)$ of $(X,\sO_X)$ such that $h_i = h_{\fL_i}$.  Finding a model $\fX$ dominating all $\fX_i$, we may assume all $\fX_i = \fX$.  Then $h = h_\fL$ for $\fL = \sum_{i=1}^m\lambda_i\fL_i$.  Conversely, if $\fL = \sum_{i=1}^m\lambda_i\fL_i$ on $U$ for $\lambda_1,\ldots,\lambda_m\in\Lambda$ and $\fL_1,\ldots,\fL_m\in M(\fU)$ then $h_\fL = \sum_{i=1}^m \lambda_i h_{\fL_i}$ is  $\Lambda$-PL by~\secref{sec:PL.remarks}(5).  
  We can choose $U$ sufficiently small because $U$ can be taken arbitrarily small in Proposition~\ref{prop:Lpl.analytic.nbhd}.
\end{proof}

We will use the following lemma in the proof of Proposition~\ref{prop:Lpl.analytic.nbhd}; we will use it again several times.

\begin{lem}\label{lem:lambda.assume.li}
  Suppose that $\Lambda$ is a divisible additive subgroup of $\R$.  Let $P$ be an abelian group, and let $x_1,\ldots,x_r\in\Lambda\tensor_\Z P$ be a finite collection of elements.  Then there is a $\Q$-linearly independent subset $\{\lambda_1,\ldots,\lambda_n\}\subset\Lambda$ such that each $x_i$ has the form $\sum_{j=1}^n \lambda_i\tensor x_{ij}$ for $x_{ij}\in P$.
\end{lem}

\begin{proof}
  Write $x_i = \sum_{j=1}^{m_i}\lambda_{ij}\tensor y_{ij}$.  Let $\lambda_1,\ldots,\lambda_n$ be a basis of the $\Q$-subspace of $\Lambda$ generated
  by the $\lambda_{ij}$.  Since $\Lambda$ is divisible, we may assume that every $\lambda_{ij}$ is a $\Z$-linear combination of $\lambda_1,\ldots,\lambda_n$.  Expanding $\lambda_{ij}$ in the sum $\sum_{j=1}^{m_i}\lambda_{ij}\tensor y_{ij}$ and rearranging proves the lemma.
\end{proof}

\begin{proof}[Proof of Proposition~\ref{prop:Lpl.analytic.nbhd}]
 The ``if'' direction is trivial because a covering by open sets is a $\rG$-covering.  Suppose then that $h$ is $\Lambda$-PL.  If $\Lambda=\Z$ then $h$ is itself PL, so we may assume $\Lambda$ is divisible.  Fix $x\in X$, and replace $X$ by a compact strictly analytic neighbourhood of $x$ to assume $X$ is compact.  Choose a $\rG$-covering $\{X_i\}_{i\in I}$ of $X$ by compact strictly analytic domains such that $h|_{X_i} = \sum_{j=1}^{m_i}\lambda_{ij}\log|f_{ij}|$ for finitely many $\lambda_{ij}\in\Lambda$ and  $f_{ij} \in \sO(X_i)^\times$. We may assume that every $X_i$ contains $x$ and that $I$ is a finite set.   By Lemma~\ref{lem:lambda.assume.li}, there are $\Q$-linearly independent elements $\lambda_1, \dots, \lambda_m\in\Lambda$ such that for every $i \in I$, we have $h|_{X_i}=\sum_{k=1}^m \lambda_k h_{ik}$ with $h_{ik}=\log|g_{ik}|$ for   $g_{ik}\in \sO(X_i)^\times$.

 We claim that there is a neighbourhood $U$ of $x$ such that for all $i,j \in I$ and all $y \in U \cap X_i \cap X_j$ we have
 \begin{equation} \label{functions extend up to constants}
 h_{ik}(y)-h_{ik}(x)=h_{jk}(y)-h_{jk}(x).
 \end{equation}
 To prove this, we pick $i,j \in I$, set $X_{ij} \coloneqq X_i \cap X_j$ and consider the smooth tropicalization map $H_{ij}\colon X_{ij} \to \R^{2m}$ given by $y \mapsto (h_{ik}(y),h_{jk}(y))_{k=1,\dots,m}$.
 By~\secref{sec:trop.is.polyhedral} and~\secref{sec:ducros.germs}, the tropical variety of the germ $(X_{ij},x)$ is the germ at $H_{ij}(x)$ of a finite union of $(\Z,\Gamma)$-polytopes $\Delta$. Choosing them sufficiently small, we may assume every $\Delta$ is  contained in $H_{ij}(X_{ij})$. It is enough to show~\eqref{functions extend up to constants}
 for $y \in U_{ij} \coloneqq \bigcup_{\Delta} H_{ij}^{-1}(\Delta)$ as $U_{ij}$ is a neighbourhood of $x$ in $X_{ij}$, hence has the form $U\cap X_{ij}$ for a neighbourhood $U$ of $x$ in $X$.

 We fix $i,j \in I$ and  pick a polytope $\Delta$ as above. Since $\Delta$ is a $\Z$-polytope, the set $(H_{ij}(x) + \Q^{2m})\cap\Delta$ is dense in $\Delta$.
 By construction, there are coordinate functions $\ell_{ik}, \ell_{jk}$ of $\R^{2m}$ such that $h_{ik}=\ell_{ik}\circ H_{ij}$ and $h_{jk}=\ell_{jk}\circ H_{ij}$.  Hence it is enough to show~\eqref{functions extend up to constants} for all $y \in X_{ij}$ with $H_{ij}(y)\in \Delta$ and with $\omega=H_{ij}(y)-H_{ij}(x)\in \Q^{2m}$. Using
 \begin{equation*}
 \begin{split}
 0&= \sum_k \lambda_k h_{ik}(y) - \sum_k \lambda_k h_{jk}(y) - \sum_k \lambda_k h_{ik}(x) + \sum_k \lambda_k h_{jk}(x)\\
 &= \sum_k \lambda_k \left(h_{ik}(y)-  h_{ik}(x) - h_{jk}(y) +h_{jk}(x) \right)
 \end{split}
 \end{equation*}
and $h_{ik}(y)-  h_{ik}(x) - h_{jk}(y) +h_{jk}(x)=\ell_{ik}(\omega)-\ell_{jk}(\omega)\in \Q$, the $\Q$-linear independence of $\lambda_1,\dots,\lambda_m$ yields~\eqref{functions extend up to constants} for all such $y$. This finishes the proof of the claim.

Let us fix $i \in I$. We deduce from \eqref{functions extend up to constants} that there are constants $c_{ijk}$ such that $h_{ik}=h_{jk}+c_{ijk}$ on $U \cap X_{ij}$ for all $j
\in I$ and $k=1,\dots,m$. Evaluating at a rig-point of $U \cap X_{ij}$, we see that $c_{ijk}\in\sqrt\Gamma$.  Replacing $h$ by an integer multiple, we may assume $c_{ijk}=-v(\gamma_{ijk}) \in \Gamma$ for some $\gamma_{ijk} \in K^\times$. We conclude that  $\gamma_{ijk} g_{jk}$ is an invertible function on $X_j$. We claim that
there is a unique piecewise linear function $h_k\colon U \to \R$ that agrees with $\log|\gamma_{ijk} g_{jk}|=c_{ijk}+h_{jk}$ on $U \cap X_j$ for every $j \in I$. To see that the functions $c_{ijk}+h_{jk}$ indeed glue, we know from \eqref{functions extend up to constants} that they agree on intersections up to constants and the claim follows as $c_{ijk}+h_{jk}(x)=h_{ik}(x)$ does not depend on $j \in I$.
We note that
$$\sum_k \lambda_k h_{jk}(x)= h(x)= \sum_k \lambda_k h_{ik}(x)= \sum_k \lambda_k \left(h_{jk}(x)+c_{ijk} \right)$$
and hence $\sum_k \lambda_k c_{ijk}=0$. For $y \in U \cap X_j$, we deduce that
$$h(y)= \sum_k \lambda_k h_{jk}(y)= \sum_k \lambda_k h_{jk}(y) + \sum_k \lambda_k c_{ijk}= \sum \lambda_k h_k(y)$$
and hence $h= \sum \lambda_k h_k$ on $U$. This proves the proposition.
\end{proof}

With these foundations in place, we can use formal models to show that piecewise $\Lambda$-linearity can sometimes be checked after pull-back, at least when $\Lambda$ is divisible.

\begin{prop}\label{prop:pl.finite.flat}
  Let $f\colon X'\to X$ be a finite, flat, surjective morphism of good strictly $K$-analytic spaces.  Suppose that $\Lambda$ is divisible. If $h\colon X\to\R$ is a function and $h' = h\circ f$, then $h\in\PL_\Lambda(X)$ if and only if $h'\in\PL_\Lambda(X')$.
\end{prop}

\begin{proof}
  By Lemma~\ref{lem:pl.basic.props}(\ref{item:pl.prop.pullback}), we only need to show that if $h'$ is $\Lambda$-PL then so is $h$.  By~\cite[Proposition~3.2.7]{Berkovichetale}, the map $f$ is open.  Let $x\in X$, choose $x'\in X'$ mapping to $x$, and let $U'$ be an open neighbourhood of $x'$ such that $h'|_{U'} = \sum_{i=1}^m\lambda_i h_i'$ for $\lambda_1,\ldots,\lambda_m\in\Lambda$ and  $h_1',\ldots,h_m'\in\PL(U')$.  Such a neighbourhood exists by Proposition~\ref{prop:Lpl.analytic.nbhd}.  By~\cite[Lemma~3.1.2, Definition~3.2.5]{Berkovichetale}, after shrinking $U'$, the morphism $f|_{U'}\colon U'\to f(U')$ is finite and flat.  Thus we may replace $X'$ by $U'$ and $X$ by $f(U')$ to assume $h' = \sum_{i=1}^m\lambda_i h_i'$ on $X'$.  Shrinking further, we may assume $X$ and $X'$ are compact.  By~\secref{model theorem}, each PL function $h_i'$ is induced by a line bundle on a formal model of $X'$.  Passing to a dominating formal model, we may assume that there is a formal $K^\circ$-model $\fX'$ of $X'$ and $\fL_i'\in M(\fX')$ such that $h_i' = h_{\fL_i'}$ for each $i$.

  Using Raynaud's theorem \cite[Theorem~8.4.3, Lemma~8.4.4]{bosch14:lectures_formal_rigid_geometry}, after passing to an admissible formal blowup, we can extend $f$ to a morphism $f\colon\fX'\to\fX$ for some formal $K^\circ$-model $\fX$ of $X$.  By~\cite[Theorem~5.2, Corollary~5.3(b)]{bosch_lutkeboh93:formal_rigid_geometry_II}, after passing to admissible formal blowups of $\fX$ and $\fX'$, we may assume $f\colon\fX'\to\fX$ is flat and quasi-finite.  By~\cite[Corollary~4.4]{Temkin00} the morphism $f\colon\fX'_s\to\fX_s$ is proper, so it is finite; by Nakayama's lemma, the morphism $f\colon\fX'\to\fX$ is finite.  Working locally, we may thus assume $\fX = \Spf(A)$ and $\fX' = \Spf(A')$ are formal affines, with $\fX$ connected.  Then $A'$ is a finite, flat $A$-module, which is locally of finite presentation by~\cite[Theorem~7.3.4]{Bosch_lectures}, so it is a locally free $A$-module of some constant rank $n$ by~\cite[\href{https://stacks.math.columbia.edu/tag/02KB}{Lemma 02KB}]{stacks-project}.  Thus we can define the \emph{norm} $N(\fL')$ of a line bundle $\fL'$ on $\fX'$ as in~\cite[6.5]{egaII}.  This is a line bundle on $\fX$ and a model of $\sO_X^{\tensor n} = \sO_X$, so we get a homomorphism $N\colon M(\fX')\to M(\fX)$.  Tensoring with $\Lambda$, we have an induced homomorphism $N\colon M(\fX')_\Lambda\to M(\fX)_\Lambda$.

  Now we can transfer the considerations in \cite[8.5]{boucksom_eriksson21} to the analytic setting.  For $\fL'\in M(\fX')$ and $x\in X$, we claim that
\begin{equation} \label{multiplicity formula for norm bundle}
 \|1(x)\|_{N(\fL')} = \prod_{y\mapsto x} \|1(y)\|^{m_{y}}_{\fL'},
\end{equation}
where $m_{y} = \dim_{\sH(x)}(\sO_{f\inv(x),y})$ is the multiplicity of $y$ in $f\inv(x)$.  With this in place, letting $\fL = N(\fL')$, it follows that
\[ h_\fL(x) = \sum_{y\mapsto x} m_{y}\,h_{\fL'}(y). \]
Extending by linearity, the above equation holds for $\fL'\in M(\fX')_\Lambda$ and $\fL = N(\fL')\in M(\fX)_\Lambda$.  Letting $\fL' = \sum_{i=1}^m \lambda_i \fL_i'$ and $\fL = N(\fL')$, we have $h' = h_{\fL'}$, and for $x\in X$ we have
\[ h_\fL(x) = \sum_{y\mapsto x} m_{y}\,h_{\fL'}(y)
  = \sum_{y\mapsto x} m_{y}\, h\circ f(y)
  = h(x) \sum_{y\mapsto x} m_{y} = n\, h(x). \]
It follows that $h = h_{n\inv\fL}$ is $\Lambda$-PL.

It remains to prove~\eqref{multiplicity formula for norm bundle}.  Let $x \in X$, and let $y_1, \dots, y_r$ be the points of $f^{-1}(x)$.  Let $R_j\coloneq\sO_{f\inv(x),y_j}$, let $F \coloneqq \sH(x)$, and let $m_j\coloneq\dim_F(R_j) = m_{y_j}$.  We have the decomposition  $R \coloneqq \sO(f^{-1}(x))= \prod_j R_j$ into local Artinian rings.   Working locally on $\fX=\Spf(A)$, by~\cite[\href{https://stacks.math.columbia.edu/tag/0BUT}{Lemma 0BUT}]{stacks-project} we may assume that $\fL'$ is trivial over $\fX' = \Spf(A')$.  Choose a frame $s'$ for $\fL'$ and let $\gamma = 1/s'\in (A'\tensor_{K^\circ}K)^\times$, so $\|1\|_{\fL'} = |\gamma|$ on $X'$.  The norm $s \coloneq N(s')$ is a frame of $N(\fL)$, and $1/s = N_{A'/A}(\gamma)\in (A\tensor_{K^\circ} K)^\times$, so
$$\|1(x)\|_{N(\fL')}=|N_{A'/A}(\gamma)(x)|=|N_{R/F}(\gamma)|.$$
Let $K_j$ be the residue field of $R_j$, let $\gamma_j$ be the image of $\gamma$ in $R_j$, and let $\overline \gamma_j$ be the residue class in $K_j$. A Jordan--H\"older argument (see~\cite[2.1c, p.9]{serre97:mordell_weil}) shows that
$N_{R_j/F}(\gamma_j)=N_{K_j/F}(\overline \gamma_j)^{[R_j:F]/[K_j:F]}$.  Since the canonical valuation on $\sH(y_j)$ extends the canonical valuation on the complete non-archimedean field $F=\sH(x)$, we have
$$|\gamma(y_j)|^{[K_j : F]}=|\overline\gamma_j|^{[K_j:F]}=|N_{K_j/F}(\overline \gamma_j))|=|N_{R_j/F}(\gamma_j)|^{[K_j:F]/[R_j:F]}.$$
Putting all these together, we deduce
$$\|1(x)\|_{N(\fL')} = |N_{R/F}(\gamma)| = \prod_j |N_{R_j/F}(\gamma_j)| = \prod_j |\gamma(y_j)|^{m_j} = \prod_j \|1(y_j)\|^{m_{j}}_{\fL'}.
$$
This proves the claim in \eqref{multiplicity formula for norm bundle} and hence the proposition.
\end{proof}

Proposition~\ref{prop:pl.finite.flat} gives the following variant of Lemma~\ref{lem:pl.basic.props}(\ref{item:pl.prop.basechange}).

\begin{cor}\label{cor:pl.basechange}
  Let $X$ be a good strictly $K$-analytic space, let $K'$ be the completion of an algebraic field extension of $K$, let $X' = X\hat\tensor_K K'$, and let $\pi\colon X'\to X$ be the structure morphism.  Suppose that $\Lambda$ is divisible.  If $h\colon X\to\R$ is a function and $h' = h\circ \pi$, then $h\in\PL_\Lambda(X)$ if and only if $h'\in\PL_\Lambda(X')$.
\end{cor}

\begin{proof}
  By Lemma~\ref{lem:pl.basic.props}(\ref{item:pl.prop.basechange}), we only need to show that if $h'$ is $\Lambda$-PL then so is $h$.  When $K'$ is a \emph{finite} extension of $K$ then this is a special case of Proposition~\ref{prop:pl.finite.flat}.  We reduce to this case as follows.

  Since piecewise $\Lambda$-linearity is a local property in the analytic topology, we may assume that $X$ is compact.
  Then $X'$ is also compact, so by Proposition~\ref{prop:Lpl.analytic.nbhd}
  there are finitely many compact strictly analytic domains $U_i'$ such that the  interiors $V_i'$ of the $U_i'$ cover $X'$ and such that $h_i' \coloneqq h'|_{U_i'}$ has the form  $h_i' = \sum_j \lambda_{ij} h_{ij}'$ for finitely many $h_{ij}'\in\PL(U_i')$ and $\lambda_{ij} \in \Lambda$.   By~\cite[Propositions~2.7 and~2.10]{gubler_martin19:zhangs_metrics}, the  functions $h_{ij}'$ extend to PL functions on $X'$. We denote the extensions also by $h_{ij}'$ leading to an extension $\sum_j \lambda_{ij} h_{ij}'$ of $h_i'$ which we denote also by $h_i'$.
  By~\cite[Proposition~2.18(b)]{gubler_martin19:zhangs_metrics}, there is a finite subextension $F/K$ of $K'/K$ such that for all $i,j$, there are PL functions $h_{ij}$ on $X_F \coloneqq X \otimes_K F$ whose pull-backs to $X'$ agree with $h_{ij}$. Then the pull-back of the $\Lambda$-PL function $h_i \coloneqq \sum_j \lambda_{ij} h_{ij}$ agrees with $h_i'$ on $X'$.
  As seen at the beginning, it is enough to show that the pull-back of $h$ to $X_F$ is $\Lambda$-PL, so we may replace $K$ by $F$ and $X$ by $X_F$.  Piecewise $\Lambda$-linearity is a local property, so we show it in a neighbourhood of any $y \in X$.
  It follows from \cite[Proposition~1.3.5, Corollary~1.3.6]{berkovic90:analytic_geometry} that the structure map $p:X' \to X$ is open and surjective, and hence the sets $p(V_i')$ form an open covering of $X$. For $y \in X$, we choose $i$ such that $y \in p(V_i')$. For all $x'\in V_i'$ and $x = p(x')$, we have
  $h(x)=h'(x')= h_i'(x')= h_i(x)$,
  proving that $h$ is $\Lambda$-PL on the neighbourhood $p(V_i')$ of $y$.
\end{proof}

\begin{cor} \label{pl infinite Galois extensions}
  Suppose that $\Lambda$ is divisible.  Let $L/K$ be a (potentially infinite) normal extension with automorphism group $G \coloneqq \Aut(L/K)$. Let $K'$ be the completion of $L$ and let $X' \coloneqq  X\hat\tensor_K K'$, with structure morphism $\pi\colon X'\to X$. Then $h\mapsto h\circ\pi$ induces an isomorphism $\PL_\Lambda(X)\isom\PL_\Lambda(X')^G$, where $\PL_\Lambda(X')^G$ is the space of $G$-invariant $\Lambda$-PL functions on $X'$.
\end{cor}

\begin{proof}
By~\cite[Proposition~1.3.5, Corollary~1.3.6]{berkovic90:analytic_geometry}, we have $X=X'/G$, so the claim follows from Corollary~\ref{cor:pl.basechange}.
\end{proof}

\section{Reduction of Germs and residues of PL Functions}\label{sec:reduction-germs}

We will characterize semipositive and harmonic functions in~\secref{section: pl and harmonic functions} in the language of Temkin's theory of reductions of germs~\cite{Temkin00}, using residues of line bundles with piecewise linear metrics from~\cite[Section~6]{CLD}.  In this section, we recall the theory following quite closely~\cite[6.1]{CLD}, and prove then the preliminary results we need.

By a \defi{variety} over a field $k$ we mean an integral, separated $k$-scheme $V$ of finite type.  The field of rational functions on $V$ is denoted $\cR(V)$.

\begin{art}[Riemann--Zariski spaces]\label{sec:rz.spaces}
  Let $k$ be a field, let $F/k$ be a field extension, and let $\bP_{F/k}$ be the set of  valuations on $F$ that are trivial on $k$, up to equivalence.  For $\nu\in\bP_{F/k}$ we let $\sO_\nu\subset F$ denote the corresponding valuation ring.  For a subset $A\subset F$ we set $\bP_{F/k}\{A\} = \{\nu\in\bP_{F/k}\mid \nu(f)\geq0\text{ for all }f\in A\}$.  Subsets of the form $\bP_{F/k}\{f_1,\ldots,f_n\}$ are said to be \defi{affine}.  Note that $\bP_{F/k} = \bP_{F/k}\{\emptyset\}$ is itself affine.
  We endow $\bP_{F/k}$ with the topology generated by affine subsets.  We equip $\bP_{F/k}$ with the sheaf of rings $\sO\colon U\mapsto \{f\mid \nu(f)\geq0 \text{ for all }\nu\in U\}$.  For $\nu\in\bP_{F/k}$ the localization of $\sO$ at $\nu$ is canonically identified with $\sO_\nu$, so $\bP_{F/k}$ is a locally ringed space.
  It is shown in~\cite[Corollary~1.3]{Temkin00} (see~\cite[Section~2]{temkin04:local_properties_II} for an erratum) that any affine open subset of $\bP_{F/k}$ is quasi-compact.
 \end{art}

\begin{art}[Models]\label{sec:rz.models}
  We keep the notation in~\secref{sec:rz.spaces}.  Let $V$ be a variety equipped with a $k$-algebra homomorphism from the function field $\cR(V)$ to $F$.  We let $\bP_{F/k}\{V\}$ denote the set of all valuations $\nu\in\bP_{F/k}$ such that $\Spec(F)\to V$ extends to a morphism $\Spec(\sO_\nu)\to V$.  Such an extension is unique if it exists, so there is a map of sets $\pi_V\colon \bP_{F/k}\{V\}\to V$ sending $\nu$ to the image of the closed point of $\Spec(\sO_\nu)\to V$.  It is easy to see that $\bP_{F/k}\{V\}$ is a finite union of affine open subsets, so it is quasi-compact.  The morphism $\pi_V\colon\bP_{F/k}\{V\}\to V$ is a surjective morphism of ringed spaces by~\cite[6.1.1]{CLD}.   Conversely, if $U$ is any quasi-compact open subset of $\bP_{F/k}$ then $U = \bP_{F/k}\{V\}$ for some such $V$ by~\cite[Proposition~1.4]{Temkin00}; we call $V$ a \defi{model} of $U$.%
  \footnote{In~\cite{CLD} these are called \emph{premodels}.}

  Let $V$ and $V'$ be varieties equipped with $k$-algebra homomorphisms $\cR(V)\to F$ and $\cR(V')\to F$.  A dominant morphism $V'\to V$, compatible with the morphisms $\cR(V)\to F$ and $\cR(V')\to F$, is unique if it exists, in which case $\bP_{F/k}\{V'\}\subset\bP_{F/k}\{V\}$.  In this case we say that $V'$ \defi{dominates} $V$.  We have equality $\bP_{F/k}\{V'\}=\bP_{F/k}\{V\}$ if and only if $V'\to V$ is proper by the valuative criterion of properness for integral schemes (see~\cite[Corollaire~7.3.10(ii)]{egaII} and~\cite[Lemma~3.2.1]{temkin11:riemann_zariski}).
\end{art}

\begin{art}[Functoriality]\label{sec:functoriality.RZ}
  Let $F'/k'$ be an extension of $F/k$: that is, $F'$ is an extension field of $F$ and $k'$ is a subfield of $F'$ containing $k$.  Restriction of valuations gives a canonical map $\phi\colon \bP_{F'/k'}\to\bP_{F/k}$.  The inverse image of $\bP_{F/k}\{f_1,\ldots,f_n\}$ is $\bP_{F'/k'}\{f_1,\ldots,f_n\}$, so $\phi$ is continuous.  One checks that $\phi$ induces a morphism of locally ringed spaces.  If $k'=k$ then $\phi$ is surjective, and if $F'=F$ then $\phi$ is injective.

  Let $V$ (resp.\ $V'$) be a $k$-variety (resp.\ $k'$-variety) equipped with a $k$-algebra homomorphism $\cR(V)\to F$ (resp.\ a $k'$-algebra homomorphism $\cR(V')\to F'$).  As in~\secref{sec:rz.models}, a dominant morphism $V'\to V$ of $k$-schemes, compatible with the morphisms $\cR(V)\to F$ and $\cR(V')\to F'$, is unique if it exists, in which case $\phi(\bP_{F'/k'}\{V'\})\subset\bP_{F/k}\{V\}$.  In this case we also say that $V'$ \defi{dominates} $V$.
\end{art}

\begin{lem}\label{lem:models.functoriality}
  Let $F'/k'$ be an extension of $F/k$, and let $\phi\colon\bP_{F'/k'}\to\bP_{F/k}$ be the canonical morphism.  Let $V$ be a $k$-variety equipped with a $k$-algebra homomorphism $\cR(V)\to F$, and let $V'$ be the closure of the image of the canonical $k'$-morphism $\Spec(F')\to V\tensor_k k'$.  Then $\bP_{F'/k'}\{V'\} = \phi\inv(\bP_{F/k}\{V\})$.
\end{lem}

\begin{proof}
  This is similar to~\cite[Lemma~3.2.1]{temkin11:riemann_zariski}.  Let $\nu'\in\bP_{F'/k'}$ and let $\nu = \phi(\nu') = \nu'|_F$.  If $\nu'\in\bP_{F'/k'}\{V'\}$ then we have a $k'$-morphism $\Spec(\sO_{\nu'})\to V'$ extending $\Spec(F')\to V'$.  Let $x\in V$ be the image of the closed point of $\Spec(\sO_{\nu'})$ under the composition $\Spec(\sO_{\nu'})\to V'\to V$.  Then $\Spec(\sO_{\nu'})\to V$ factors through $\Spec(\sO_{V,x})$, so we have a $k$-algebra homomorphism $\sO_{V,x}\to\sO_{\nu'}$.  It follows that the image of $\sO_{V,x}$ is contained in the intersection $\sO_\nu = \sO_{\nu'}\cap F$, so $\Spec(F)\to V$ extends to a morphism $\Spec(\sO_\nu)\to V$, and hence $\nu\in\bP_{F/k}\{V\}$.  Conversely, suppose that $\nu\in\bP_{F/k}\{V\}$.  Then $\Spec(\sO_{\nu'})$ maps to $V$ via $\Spec(\sO_{\nu'})\to\Spec(\sO_\nu)\to V$, so $\Spec(F')\to V\tensor_k k'$ extends to $\Spec(\sO_{\nu'})\to V\tensor_k k'$.  The image of this morphism is contained in $V'$, so $\nu'\in\bP_{F'/k'}\{V'\}$.
\end{proof}

As a special case of the next lemma, we see that the poset of models of a quasi-compact open subset $U\subset\bP_{F/k}$ is filtered.

\begin{lem}\label{lem:models.dominate}
  Let $F'/k'$ be an extension of $F/k$, and let $\phi\colon\bP_{F'/k'}\to\bP_{F/k}$ be the canonical morphism.  Let $U'\subset\bP_{F'/k'}$ and $U\subset\bP_{F/k}$ be quasi-compact open subsets such that $\phi(U')\subset U$.  Let $V'$ and $V$ be models of $U'$ and $U$, respectively.  Then there is a model $W'$ of $U'$ that dominates both $V$ and $V'$.
\end{lem}

\begin{proof}
  Let $W$ be the closure of the image of the canonical $k'$-morphism $\Spec(F')\to V\tensor_k k'$, as in Lemma~\ref{lem:models.functoriality}.  Then this lemma shows that $W\to V$ is dominant and $\phi\inv(U) = \bP_{F'/k'}\{W\}$, so we can replace $V$ by $W$ to assume $k=k'$ and $F=F'$.  The given morphisms $\Spec(F)\to V$ and $\Spec(F)\to V'$ determine a morphism $\Spec(F)\to V\times_k V'$; let $W'$ be the closure of the image of this morphism.  This is a variety.  The generic point of $W'$ maps to the generic points of $V$ and $V'$, so $W'\to V$ and $W'\to V'$ are dominant.  We must show that $\bP_{F/k}\{W'\} = U$.  For $\nu\in U$ we have unique morphisms $\Spec(\sO_v)\to V$ and $\Spec(\sO_v)\to V'$, which induces a morphism $p\colon \Spec(\sO_v)\to V\times_k V'$ extending $\Spec(F)\to V\times_k V'$.  The image of $p$ is contained in $W'$, so $\nu\in\bP_{F/k}\{W'\}$.
\end{proof}

\begin{art}[Invertible sheaves on Riemann--Zariski spaces]
  With the notation in~\secref{sec:rz.spaces}, let $U\subset\bP_{F/k}$ be a quasi-compact open subset.  Then $U$ is a ringed space, so it makes sense to discuss invertible sheaves on $U$.  We let $\Pic(U)$ denote the group of isomorphism classes of invertible sheaves, and for an additive subgroup $\Lambda\subset\R$ we set $\Pic(U)_\Lambda = \Pic(U)\tensor_\Z\Lambda$.

  Let $V$ be a model of $U$ as in~\secref{sec:rz.models}, with canonical morphism $\pi_V\colon U\to V$.  If $L_V$ is an invertible sheaf on $V$, then $L\coloneq\pi_V^*L_V$ is an invertible sheaf on $U$.  It is an important fact that every invertible sheaf on $U$ arises in this way. This is explained in~\cite[6.1.3]{CLD}. We call $(V,L_V)$ a \defi{model} of $(U,L)$.

  More generally, for an additive subgroup $\Lambda$ of $\R$  we have a group homomorphism $\Pic(V)_\Lambda\to\Pic(U)_\Lambda$.  Any element of $\Pic(U)_\Lambda$ can be expressed as a finite sum $L = \sum_{i=1}^n\lambda_i L_i$ for $\lambda_i\in\Lambda$ and $L_i\in\Pic(U)$.  Choosing a model for each $L_i$ and passing to a dominating model, we see that any $L\in\Pic(U)_\Lambda$ has the form $\pi_V^*L_V$ for a model $V$ of $U$ and $L_V\in\Pic(V)_\Lambda$.  We call such $(V,L_V)$ a \defi{model} of $(U,L)$ as above.
\end{art}

\begin{art}[Nef line bundles]\label{sec:nefness}
  Let $X$ be a separated $k$-scheme of finite type and let $L$ be a line bundle on $X$.  For an integral, proper closed curve $C\subset X$ we have a well-defined intersection number $c_1(L).C\coloneq\deg(L|_C)\in\Z$.  We say that $L$ is \defi{numerically effective} or \defi{nef} if $c_1(L).C\geq 0$ for all such curves $C$, and that $L$ is \defi{numerically trivial} if $c_1(L).C=0$ for all $C$ (equivalently, if $L$ and $L\inv$ are both nef).

  More generally, let $\Lambda\subset\R$ be an additive subgroup and let $C\subset X$ be an integral, proper closed curve.  We have an induced homomorphism $L\mapsto c_1(L).C\colon\Pic(X)_\Lambda\to\Lambda$, and we say that $L$ is \defi{nef} (resp.\ \defi{numerically trivial}) if $c_1(L).C\geq 0$ (resp.\ $c_1(L).C=0$) for all such~$C$.

  The set of nef line bundles in $\Pic(X)_\Lambda$ forms a monoid (resp.\ a cone if $\Lambda=\R$), and the set of numerically trivial line bundles form a subgroup (resp.\ an $\R$-subspace).
\end{art}

Here we list some classical properties of nefness.  We sketch some proofs because in the literature, the scheme $X$ is generally assumed to be proper.

\begin{lem}\label{lem:nef.properties}
  Let $X$ be a separated $k$-scheme of finite type and let $L\in\Pic(X)_\Lambda$.
  \begin{enumerate}
  \item The line bundle $L$ is nef if and only if its image in $\Pic(X)_\R$ is nef.
  \item Let $k'/k$ be a field extension and let $X' = X\tensor_k k'$, with structure morphism $\pi\colon X'\to X$.  Then $L$ is nef if and only if $\pi^*L$ is nef.
  \item Let $f\colon X'\to X$ be a morphism of separated, finite-type $k$-schemes.  If $L$ is nef then $f^*L$ is nef, and the converse holds if $f$ is proper and surjective.
  \item Let $\iota\colon X_{\red}\to X$ be the reduction map.  Then $L$ is nef if and only if $\iota^*L$ is nef.
  \end{enumerate}
  The same statements hold with ``nef'' replaced by ``numerically trivial.''
\end{lem}

\begin{proof}
  The last assertion holds because $L$ is numerically trivial if and only if $L$ and $L\inv$ are nef.
  Assertion~(1) is obvious because nefness is a numerical property, and
  (4) is clear because $X$ and $X_{\red}$ have the same integral subcurves.

  (2)~\ Since intersection numbers are compatible with base extension, it is clear that if $\pi^*L$ is nef, then also $L$ is nef.
   To prove the converse, we assume that $L$ is nef. First suppose that $k'/k$ is algebraic. Since every integral proper subcurve $C'$ of $X'$ is defined over a finite subextension, we may assume that $k'/k$ is a finite extension. Then we may view $X'$ as a scheme of finite type over $k$ and $\pi$ as a finite morphism over $k$. Then $\pi(C')$ is a proper subcurve of $X$ and we get $c_1(\pi^*L).C'=c_1(L).\pi_*(C') \geq0$ by the projection formula. This proves the algebraic case of (2), so we may assume $k$ is algebraically closed to prove the claim in general. Let $C'$ again be an integral proper subcurve of $X'$. Since $C'$ is defined over a finitely generated subextension, we may assume that $k'$ is a finitely generated field extension of $k$. Then $k'$ is the function field of a variety $V$ over $k$. Let $\pi_V\colon  X \times_k V \to V$ and $\pi_X\colon X\times_k V\to X$ be the structure morphisms. The morphism $\pi \colon X' \to X$ is the generic fiber of $\pi_V$.  Since $C'$ is projective and flat over $k'$, the spreading out property shows that there is a non-empty open subset $U$ of $V$ such that the   closure $\cC'$ of $C'$ in $\pi_V^{-1}(U)$ is flat and projective over $U$. By~\cite[Proposition~VI.2.9]{kollar96},
   the degree of the fiber $\cC_u'$ with respect to  $\pi_X^*(L)$ is independent of the choice of $u\in U$. Let $\eta$ be the generic point of $U$ and let $u$ be a closed point of $U$. Since $\cC_u' \cong C$ for a proper closed subcurve $C$ of $X$, the claim follows from
   $$c_1(\pi^*L).C'=\deg_{\pi_V^*(L)}(\cC_\eta')=\deg_{\pi_V^*(L)}(\cC_u')=\deg_L(C)=c_1(L).C \geq 0.$$

   (3)~\ If $L$ is nef  then $f^*L$ is nef because if $C'\subset X'$ is an integral proper $k$-curve then $f(C')\subset X$ is integral and proper of dimension  $\leq 1$, so the claim follows from the projection formula.  Conversely, if $f^*L$ is nef and if $f$ is a proper surjective map, then $L$ is nef because every integral proper $k$-curve $C\subset X$ is the image of an integral proper $k$-curve $C'\subset X'$ (take the closure of a closed point in the fiber over the generic point of $C$).
\end{proof}

\begin{defn}
  Let $U$ be a quasi-compact open subset of $\bP_{F/k}$, let $\Lambda\subset\R$ be an additive subgroup, and let $L\in\Pic(U)_\Lambda$.  We say that $L$ is \defi{nef} if there exists a model $(V,L_V)$ of $(U,L)$ with $L_V$ nef.  We say that $L$ is \defi{numerically trivial} if both $L$ and $L\inv$ are nef.
\end{defn}

\begin{art}
  The set of nef elements of $\Pic(U)_\Lambda$ forms a monoid (pass to a dominating model and use Lemma~\ref{lem:nef.properties}(3)), resp.\ a cone if $\Lambda=\R$; the set of numerically trivial elements forms a group (resp.\ an $\R$-subspace).  An element of $\Pic(U)_\Lambda$ is nef if and only if its image in $\Pic(U)_\R$ is nef.
\end{art}

The following lemma shows that nefness can be checked on any model.

\begin{lem}\label{lem:nef.all.models}
  Let $U$ be a quasi-compact open subset of $\bP_{F/k}$ and let $L\in\Pic(U)_\Lambda$ be nef.  If $(V,L_V)$ is a model of $L$, then $L_V$ is nef.
\end{lem}

\begin{proof}
  By definition, there is a model $(V',L_{V'})$ of $L$ such that $L_{V'}$ is nef.  Let $W$ be a model dominating both $V$ and $V'$ as in Lemma~\ref{lem:models.dominate}, and let $L_W$ and $L_W'$ be the pullbacks of $L_V$ and $L_{V'}$ to $W$, respectively.
  It is shown in~\cite[6.1.3]{CLD} that $\Pic(U)$ is the inductive limit of the groups $\Pic(V)$ over all models $V$ of $U$.  Since tensor products commute with inductive limits, we have $\Pic(U)_\Lambda = \varinjlim\Pic(V)_\Lambda$.  Hence we can replace $W$ by a larger model such that $L_W = L_W'$ in $\Pic(W)$.  Then $L_W$ is nef, so $L_{V'}$ is nef because $W\to V'$ is proper, using Lemma~\ref{lem:nef.properties}(3) twice.
\end{proof}

The next lemma is an analogue of Lemma~\ref{lem:nef.properties}.

\begin{lem}\label{lem:nef.properties.RZ}
  Let $F'/k'$ be an extension of $F/k$, and let $\phi\colon\bP_{F'/k'}\to\bP_{F/k}$ be the canonical morphism.  Let $U'\subset\bP_{F'/k'}$ and $U\subset\bP_{F/k}$ be quasi-compact open subsets such that $\phi(U')\subset U$.  If $L\in\Pic(U)_\Lambda$ is nef, then $\phi^*L\in\Pic(U')_\Lambda$ is nef.  The converse is true if $U'=\phi\inv(U)$, and if one of the following conditions holds:
  \begin{enumerate}
  \item $k'/k$ is algebraic, or
  \item $\trdeg(F/k)<\infty$ and $\trdeg(k'F/k')=\trdeg(F/k)$.
  \end{enumerate}
  The same statements hold with ``nef'' replaced by ``numerically trivial.''
\end{lem}

Here $k'F$ denotes the smallest subfield of $F'$ containing both $k'$ and $F$.
It is always true that $\trdeg(k'F/k')\leq\trdeg(F/k)$.

\begin{proof}
  Suppose that $L$ is nef.  Let $(V,L_V)$ be a model of $(U,L)$, and let $W$ be the closure of the image of the canonical $k'$-morphism $\Spec(F')\to V\tensor_k k'$, as in Lemma~\ref{lem:models.functoriality}.  The pullback of $L$ to $W$ is nef by Lemma~\ref{lem:nef.properties}(2,3), so we may replace $V$ by $W$ to assume $k=k'$.  By Lemma~\ref{lem:models.dominate}, there is a model $V'$ of $U'$ dominating $V$; let $f\colon V'\to V$ be the dominating morphism.  Then $f^*L_V$ is nef by Lemma~\ref{lem:nef.properties}(3), so $\phi^*L$ is nef.

  Now suppose that $\phi^*L$ is nef and that $U' = \phi\inv(U)$.  First we treat the case when $k'$ is algebraic over $k$.  If $k'=k$ then any model $(V,L_V)$ of $(U,L)$ is also model of $(U',\phi^*L)$, so $L_V$ is nef by Lemma~\ref{lem:nef.all.models}.  In general, the map $\bP_{F'/k'}\to\bP_{F/k}$ factors as $\bP_{F'/k'}\to\bP_{F'/k}\to\bP_{F/k}$, so by the previous sentence we may assume $F=F'$.  Note that $\bP_{F/k}=\bP_{F/k'}$ and $U=U'$.  Let $(V,L_V)$ be a model of $L$ and let $V'$ be a model of $U'$.  We may assume that $V'$ dominates $V$ by Lemma~\ref{lem:models.dominate}; let $f\colon V'\to V$ be the dominating morphism.  Then $(V',f^*L_V)$ is a model of $(U',\phi^*L)$, so $f^*L_V$ is nef.  The model $V'$ is defined over a finite extension of $k$: that is, there exists a finite extension $k''/k$ contained in $k'$ and a variety $V''$ over $k''$ such that $V' = V''\tensor_{k''}k'$.  After enlarging $k''$, we may assume that $V''$ also dominates $V$; let $f'\colon V''\to V$ be the dominating morphism.  By Lemma~\ref{lem:nef.properties}(2), the line bundle $f'^*L$ is nef.  Replacing $k'$ with $k''$ and $V'$ with $V''$, we may assume that $[k':k]<\infty$.  In this case, the variety $V'$ is also a model of $U$.  Any proper $k$-curve in $V'$ is also a proper $k'$-curve, so $f^*L_V$ is nef when we consider $V'$ as a $k$-variety.  It follows from Lemma~\ref{lem:nef.all.models} that $L$ is nef.

  Now assume that $d = \trdeg(F/k)<\infty$, and that $\trdeg(k'F/k')=d$.  First we reduce to the case when $k$ is algebraically closed.  Let $F''$ be a field containing $F'$ and an algebraic closure $\bar k$ of $k$.  Pulling back via $\bP_{F''/k'}\to\bP_{F'/k'}$, we may replace $F'$ with $F''$ to assume $\bar k\subset F'$.  We have $\trdeg(\bar kF/\bar k) = \trdeg(F/k) = d$, and
  \[ \trdeg( k'\bar k F/k' \bar k) = \trdeg(k'\bar kF/k') = \trdeg(Fk'/ k') = d. \]
  Applying case~(1) to the extension $\bar kF/\bar k$ of $F/k$ and using the above displayed equation, we may replace $k$ by $\bar k$, $F$ by $\bar kF$, and $k'$ by $k'\bar k$ to assume $k=\bar k$.

  Choose a model $(V,L_V)$ of $(U,L)$ with $\dim(V) = d$.  Since $k$ is algebraically closed, the scheme $V'\coloneq V\tensor_k k'$ is a variety.  As $F/\cR(V)$ is algebraic, we have $\trdeg(k'\cR(V)/k') = \trdeg(k'F/k') = d$.  Let $x$ be the image of $\Spec(F')\to V'$, so $x$ maps to the generic point of $V$.  The residue field $\kappa(x)$ contains both $k'$ and $\cR(V)$, so $\trdeg(\kappa(x)/k') = d$.  It follows that $x$ is the generic point of $V'$.  By Lemma~\ref{lem:models.functoriality}, the variety $V'$ is a model of $U'$, so $\phi^*L_V$ is nef by Lemma~\ref{lem:nef.all.models}, hence $L_V$ is nef by Lemma~\ref{lem:nef.properties}(3).
\end{proof}

\begin{art}[Temkin's reductions of germs]\label{sec:temkin.reduction.germs}
  Let $X$ be a good strictly $K$-analytic space and let $x\in X$.  Temkin's reduction $\red(X,x)$ of the germ $(X,x)$ is an affine open subset of $\bP_{\td\sH(x)/\td k}$, where $\td\sH(x)$ is the reduction of the completed residue field $\sH(x)$ at $x$ as in~\secref{Non-Archimedean geometry}.  The reduction $\red(X,x)$ can be constructed as follows.%
  \footnote{Reductions of germs are much more subtle when $X$ is not good or strictly analytic, and a major motivation of Temkin's work is to understand goodness and strictness by studying such reductions.  In our case, the construction is relatively simple.}
  Let $U$ be a compact strictly analytic neighbourhood of $x$ in $X$ and let $\fU$ be a formal $K^\circ$-model of $U$.  The closure $V_{\fU,x}$ of the reduction $\red_\fU(x)$ is variety over $\td k$ by~\cite[Proposition~4.7]{bosch_lutkeboh93:formal_rigid_geometry_I}, and there is a canonical homomorphism  $\cR(V_{\fU,x})\to\td\sH(x)$ by~\cite[2.4]{berkovic90:analytic_geometry}.  We define $\red(X,x)$ to be $\bP_{\td\sH(x)/\td k}\{V_{\fU,x}\}$.  Then $\red(X,x)$ is an \emph{affine} (open) subset of $\bP_{\td\sH(x)/\td k}$ by~\cite[Lemma~2.2]{Temkin00}.  Let $\pi_{\fU,x}\colon\red(X,x)\to V_{\fU,x}$ be the canonical morphism of~\secref{sec:rz.models}

  Reduction of germs $(X,x)\mapsto\red(X,x)$ is functorial in $(X,x)$.   It also respects extension of the ground field, in the following sense.  Let $K'/K$ be a non-Archimedean extension field, let $X' = X\hat\tensor_K K'$ with structure morphism $\pi\colon X'\to X$, and let $x'\in X'$ and $x = \pi(x')$.  Let $\phi\colon\bP_{\td\sH(x')/\td K'}\to\bP_{\td\sH(x)/\td K}$ be the natural morphism.  Then $\phi\inv(\red(X,x)) = \red(X',x')$.  See~\cite[Lemma~4.2]{Temkin00} (and~\cite[Section~5]{temkin04:local_properties_II} for a correction).
\end{art}

We will need the fact that models of the form $V_{\fU,x}$ as in~\secref{sec:temkin.reduction.germs} are cofinal in the system of all models of $\red(X,x)$.

\begin{lem}\label{lem:formal.model.cofinal}
  Let $X$ be a good strictly $K$-analytic space, let $x\in X$, and let $V$ be a model of $\red(X,x)$.  Then there exists a compact strictly analytic neighbourhood $U$ of $x$ and a formal $K^\circ$-model $\fU$ of $U$ such that $V_{\fU,x}$ dominates $V$.
\end{lem}

\begin{proof}
  Choose a finite affine open covering $V = \bigcup V_i$.  Then $V_i = \Spec(\ktilde[\td f_{ij}:j\in J_i])$ for finitely many $\td f_{ij}\in\td\sH(x)$.  We pick lifts $f_{ij}\in\sO_{X,x}$.  Let $U$ be a strictly affinoid neighbourhood of $x$ with $f_{ij}\in \sO(U)$, and let $U_i \coloneqq \{x \in U \mid |f_{ij}(x)| \leq 1, \, j \in J_i\}$.  By~\cite[Proposition~2.3(ii,iii)]{temkin00:local_properties}, the reductions $\red(U_i,x)$ cover $\red(X,x)$, so $(\bigcup U_i,x) = (X,x)$.  We may shrink $U$ to assume that $U = \bigcup U_i$ is a (strictly affinoid) neighbourhood of~$x$.

  We claim that we can find a formal $K^\circ$-model $\fU$ of $U$ that is covered by formal opens $\fU_i$ with $(\fU_i)_\eta = U_i$ and $f_{ij}\in\sO(\fU_i)$.  Write $U_i = \sM(\sA_i)$ for a strictly affinoid algebra $\sA_i$.  Choose a surjection from a Tate algebra $T$ onto $\sA_i$ such that each $f_{ij}$ is the image of a generator.  The image $A_i$ of $T^\circ$ in $\sA_i$ is a flat $K^\circ$-algebra topologically of finite presentation containing $f_{ij}$; hence $\fU_i\coloneq\Spf(A_i)$ is an admissible formal model of $U_i$ with $f_{ij}\in\sO(\fU_i)$.  By~\cite[Lemma~8.4.5]{Bosch_lectures}, there exists an admissible formal $K^\circ$-model $\fU$ of $U$ and a covering by formal open subschemes $\fU_i'\subset\fU$ such that $(\fU_i')_\eta = U_i$.  By Lemma~8.4.4(d) of \textit{loc.\ cit.}, there is a formal $K^\circ$-model $\fU_i''$ of $U_i$ that is an admissible formal blowup of both $\fU_i$ and $\fU_i'$.  By Proposition~8.2.13 of \textit{loc.\ cit.}, there is an admissible formal blowup $\phi\colon\fU'\to\fU$ such that $\phi\inv(\fU_i')$ dominates $\fU_i''$.  Replacing $\fU$ by $\fU'$ and $\fU_i$ by $\phi\inv(\fU_i')$ proves the claim.

  Consider the closure $V' \coloneq V_{\fU,x}$ of $\red_\fU(x)$ in $\fU_s$.  By construction the function field $\cR(V')$ contains all $\td f_{ij}$, so $\cR(V)\subset\cR(V')$ and hence there is a dominant rational map $V'\to V$.  Moreover, for each $i$ we have $\sO(V_i)\subset\sO((\fU_i)_s\cap V')$ because $f_{ij}\in\sO(\fU_i)$, so $(\fU_i)_s\cap V' \to V_i$ is a morphism, and hence $V'\to V$ is a morphism.
\end{proof}

We will also use the following theorem of Temkin.

\begin{thm}[{\cite[Theorem~4.1]{Temkin00}}]\label{thm:reduction.functor}
  Let $f\colon X'\to X$ be a morphism of good strictly $K$-analytic spaces, let $x'\in X'$, and let $x = f(x')$.  Let $\phi\colon\bP_{\td\sH(x')/\td K}\to\bP_{\td\sH(x)/\td K}$ be the natural morphism.  Then $x\in\Int(X'/X)$ if and only if $\phi\inv(\red(X,x))=\red(X',x')$.
\end{thm}

\begin{art}[Residues of PL functions]\label{sec:reduction.pl.metric}
  Let $h\colon X\to\R$ be a PL function on a good strictly $K$-analytic space $X$.  For $x\in X$, Chambert--Loir and Ducros define a canonical invertible sheaf $L_h(x)$ on the ringed space $\red(X,x)$, which is called the \defi{residue} of $h$ at $x$.  It can be described as follows, using~\cite[Exemple~6.6.8]{CLD}.  Let $U$ be a compact strictly analytic neighbourhood of $x$.  By~\secref{model theorem}, there exists a formal $K^\circ$-model $(\fU,\fL)$ of $(U,\sO_X|_U)$ such that $h = -\log\|1\|_\fL$ on $U$.  This $\fL$ restricts to an invertible sheaf $\fL_{\fU,x}$ on $V_{\fU,x}$, and we get an invertible sheaf $L_h(x) = \pi_{\fU,x}^*\fL_{\fU,x}$ on $\red(X,x)$.

  Let $\PL(X,x)$ denote the group of germs of PL functions on $X$ at $x$; this is intrinsic to the germ $(X,x)$.  Applying the above to germs of PL functions at $x$, we obtain a homomorphism $L\colon\PL(X,x)\to\Pic(\red(X,x))$, denoted $h\mapsto L_h(x)$.
\end{art}

\begin{lem}\label{lem:model.lb.trivial}
  Let $(X,x)$ be a good strictly $K$-analytic germ.  If $h\in\PL(X,x)$ has the form $\log|f|$ for an invertible function $f$, then $L_h(x) = 0$ in $\Pic(\red(X,x))$.
\end{lem}

\begin{proof}
  Choose a compact strictly analytic space $X$ and a point $x\in X$ representing $(X,x)$ such that there exists an invertible function $f$ on $X$ with $h=\log|f|$.  Since $X$ is compact, there exists a formal $\kcirc$-model $\fX$ of $X$.  Let $\fL = \sO_\fX\cdot f$, regarded as a subsheaf of the sheaf $\fU\mapsto\Gamma(\fU_\eta,\sO_X)$ on $\fX$.  Then $\|1\|_\fL=|f|\inv$, and hence $h=-\log\|1\|_\fL$. Since $\fL$ is trivial, so is $\fL_{\fX,x}$, so we conclude that $L_h(x) = 0$.
\end{proof}

  Now we want to extend~\secref{sec:reduction.pl.metric} to $\Lambda$-PL functions.  Let $\PL_\Lambda(X,x)$ denote the group of germs of $\Lambda$-PL functions on $(X,x)$.
  We call $h \in \PL_\Lambda(X,x)$ \emph{$\Lambda$-linear} if there are $f_1, \dots, f_r \in \sO_{X,x}^\times$ and $\lambda_1, \dots, \lambda_r \in \Lambda$ such that $h=\sum_{j=1}^r \lambda_j \log |f_j|$ in a neighbourhood of $x$.

\begin{lem}\label{lem:reduction.depends.function}
  Suppose that $\Lambda=\Z$ or that $\Lambda$ is divisible. Then there is a unique homomorphism
  \[ h\mapsto L_h(x)\colon\PL_\Lambda(X,x)\To\Pic(\red(X,x))_\Lambda \]
  such that, if $h = \sum_{i=1}^m\lambda_i h_i$ with $\lambda_i\in\Lambda$ and $h_i\in\PL(X,x)$, then $L_h(x) = \sum_{i=1}^m\lambda_i L_{h_i}(x)$. 
  The kernel of this homomorphism consists precisely of the $\Lambda$-linear functions in $\PL_\Lambda(X,x)$.
\end{lem}

\begin{proof}
We have to show that  $L_h(x) \coloneqq \sum_{i=1}^m\lambda_i L_{h_i}(x)$ is well-defined.  If $\sum_{i=1}^m\lambda_i h_i=0$ for $\lambda_i\in\Lambda$ and $h_i\in\PL(X,x)$, then we need $\sum_{i=1}^m\lambda_i L_{h_i}(x)=0$. By Lemma~\ref{lem:lambda.assume.li}, we may assume $\{\lambda_1,\ldots,\lambda_m\}$ is linearly independent over $\Q$.  Choose a compact strictly analytic space $X$ and a point $x\in X$ representing $(X,x)$ such that each $h_i$ extends to a PL function $h_i\colon X\to\R$.  Shrinking $X$, we may assume $X$ is connected and  $\sum_{i=1}^m\lambda_ih_i = 0$ on~$X$.

  We claim that each $h_i$ is constant on $X$.   Using \cite[Lemma~1.6.2]{berkovic93:etale_cohomology}, we see that $X$ is $\rG$-connected and hence the claim is $\rG$-local on $X$. It follows from~\cite[Proposition~9.1.4/8]{bosch_guntzer_remmert84:non_archimed_analysis} that $X$ is $\rG$-locally connected, so we may assume that $h_i = -\log|f_i|$ for an invertible analytic function $f_i$ on a connected strictly analytic domain $U\subset X$.  Let $\bT = \bGm^m$ and let $\phi = (f_1,\ldots,f_m)\colon U\to\bT^\an$, so $\phi_{\trop} = (h_1,\ldots,h_m)$.  Since $\sum_{i=1}^m\lambda_i h_i=0$, the tropical variety $\phi_{\trop}(U)$ is contained in the hyperplane $H = \{(x_1,\ldots,x_m)\in\R^m\mid\lambda_1x_1+\cdots+\lambda_mx_m = 0\}$.  If $\phi_{\trop}$ is not constant, then $\phi_{\trop}(U)$ has positive dimension because $U$ is connected.  It follows from~\secref{sec:trop.is.polyhedral} that $\phi_{\trop}(U)$ contains a $(\Z,\Gamma)$-polytope of dimension one, which is a line segment with rational slopes.  The tangent space of  this segment contains  a non-zero vector in $H$ of the form $(a_1,\ldots,a_n)\in\Q^n$, which contradicts the linear independence of $\{\lambda_1,\ldots,\lambda_n\}$ over $\Q$.  This proves the claim; moreover, since $\phi_{\trop}(U)$ is a $(\Z,\Gamma)$-point in $\R^n$, this shows that $h_i$ is a constant function $\gamma_i$ for $\gamma_i\in\sqrt\Gamma$.

  Choose a positive integer $N$ such that $N\gamma_i\in\Gamma$ for all $i$, so that $N\gamma_i = -\log|\beta_i|$ for $\beta_i\in K^\times$.  Then $Nh_i = -N\log|\beta_i|$, so $L_{Nh_i}(x) = 0$ by Lemma~\ref{lem:model.lb.trivial}.  Thus $NL_h(x) = \sum_{i=1}^m \lambda_i L_{Nh_i}(x) = 0$, so $L_h(x) = 0$ because $\Lambda$ is divisible.
  
  To prove the last claim, it is obvious from the definition of the residue $L_h(x)$ that $\Lambda$-linear functions $h \in \PL_\Lambda(X,x)$ are in the kernel of the homomorphism. Conversely, let $h \in \PL_\Lambda(X,x)$ with $L_h(x)=0$. We note first that for $\Lambda =\Z$, the claim follows from \cite[6.6.2]{chambert_ducros12:forms_courants}. Indeed, Chambert-Loir and Ducros show that $h$ induces a metric on $\sO_{X}$ which is isometric to the trivial metric on a neighbourhood $U$ of $x$ and hence there is $f \in \sO(U)^\times$ such that $h=\log |f|$ on $U$. It is clear that the case $\Lambda=\Z$ readily yields the case $\Lambda=\Q$ by clearing denominators.
  
  For $\Lambda$ divisible, we reduce to $\Lambda=\Q$ as follows.
  There are $h_1, \dots, h_r \in \PL_\Q(X,x)$  and $\lambda_1,\dots, \lambda_r \in \Lambda$ with $h=\sum_{i=1}^r \lambda_i h_i$. By Lemma~\ref{lem:lambda.assume.li} again, we may assume that $\lambda_1,\dots, \lambda_r$ are $\Q$-linearly independent. It is enough to prove that $L_{h_i}(x)=0$ for all $i \in \{1,\dots,r\}$ as then the special case $\Lambda=\Q$ above shows that $h_i$ is $\Q$-linear, and hence that $h$ is $\Lambda$-linear. We argue by contradiction.  Then there is a $\Q$-linear form $\ell\colon\Pic(\red(X,x))_\Q\to\Q$ with $\ell(L_{h_i}(x))\neq 0$ for at least one $i\in \{1,\dots,r\}$.   This extends by linearity to a homomorphism $\ell\colon\Pic(\red(X,x))_\Lambda\to\R$, and 
  $$0=\ell(L_h(x))=\sum_{i=1}^r \lambda_i \ell(L_{h_i}(x))$$
  is a non-trivial $\Q$-linear combination of $\lambda_1,\dots,\lambda_r$ which is a contradiction.
\end{proof}

\begin{art}[Functoriality of residues]\label{sec:reduction.function.functoriality}
  Let $f\colon(X',x')\to(X,x)$ be a morphism of good strictly $K$-analytic germs. For  $h\in\PL_\Lambda(X,x)$, we get $h' = h\circ f\in\PL_\Lambda(X',x')$.  We have an induced morphism $\td f\colon\red(X',x')\to\red(X,x)$, and $L_{h'}(x') = \td f^*L_h(x)$ by~\cite[6.6.7]{CLD}.  A similar statement holds for extension of scalars.
\end{art}

\section{Semipositivity and harmonic functions} \label{section: pl and harmonic functions}

In this section, we study positivity conditions on piecewise linear functions.  We will introduce the crucial notion of a \defi{harmonic} function, on which our theory of differential forms rests.  We consider a good strictly $K$-analytic space $X$ and an additive subgroup $\Lambda\subset\R$.  \emph{We will always assume that $\Lambda=\Z$ or that $\Lambda$ is divisible.}

We can characterize semipositivity either in terms of formal geometry or residues of $\Lambda$-PL functions, as elaborated in~\secref{sec:reduction-germs}.  A line bundle $\fL$ on an admissible formal scheme $\fX$ is called \defi{nef} (resp.\ \defi{numerically trivial}) if its restriction to $\fX_s$ has the same property.  In other words, the line bundle $\fL$ is nef if the degree of $\fL|_C$ is nonnegative for every closed integral curve $C\subset\fX_s$ which is proper over $\td K$.%
\footnote{If $\sX$ is a proper $K^\circ$-model of a proper variety $X$ over $K$ and $\sL$ is a line bundle on $\sX$ with $L=\sL|_X$ such that the restriction of $\sL$ to the special fiber of $\sX$ is nef, then $L$ is a nef line bundle on $X$ by~\cite[4.8]{gubler_martin19:zhangs_metrics}.}

Let $\fX$ be a formal $K^\circ$-model of a paracompact good strictly $K$-analytic space $X$.  Recall from~\secref{model function} that $M(\fX)$ denotes the group of models of $\sO_X$, and that we have defined a homomorphism $h\mapsto h_\fL\colon M(\fX)_\Lambda\to\PL(X)_\Lambda$. 
We will also use the residue $L_h(x)$ of $h$ at $x \in X$ from~\secref{sec:reduction.pl.metric}.

\begin{prop}\label{prop:def.semipositive}
  Let $h\colon X\to\R$ be a $\Lambda$-PL function and  $x\in X$.  The following are equivalent:
  \begin{enumerate}
  \item There exists a paracompact neighbourhood $U$ of $x$ in $X$, a formal $K^\circ$-model $\fU$ of $U$, and a nef (resp.\ numerically trivial) element $\fL\in M(\fU)_\Lambda$, such that $h|_U = h_\fL$.
  \item The residue $L_h(x)\in\Pic(\red(X,x))_\Lambda$ is nef (resp.\ numerically trivial).
  \end{enumerate}
\end{prop}

\begin{proof}
  The first assertion implies the second by the description of $L_h(x)$ in~\secref{sec:reduction.pl.metric}.  We will prove the nef case, as the numerically trivial case is identical.  The following argument is similar to~\cite[Proposition~6.5]{gubler_kuenneman19:positivity}.  Suppose that $L_h(x)$ is nef.  By Corollary~\ref{cor:hL.is.pl}, there exists a compact strictly analytic neighbourhood $U$ of $x$, a formal $K^\circ$-model $\fU$ of $U$, and an element $\fL\in M(\fU)_\Lambda$, such that $h=h_\fL$.  As in~\secref{sec:reduction.pl.metric} we let $V_{\fU,x}$ be the closure of $\red_\fU(x)$ and we let $\fL_{\fU,x}$ be the restriction of $\fL$ to $V_{\fU,x}$.  Then $(V_{\fU,x},\fL_{\fU,x})$ is a model of $(\red(X,x),L_h(x))$, so $\fL_{\fU,x}$ is nef by Lemma~\ref{lem:nef.all.models}.  Let $W$ be a compact, strictly analytic neighbourhood of $x$ contained in the open set $\red_\fU\inv(V_{\fU,x})$.  By~\cite[Lemma~4.4]{bosch_lutkeboh93:formal_rigid_geometry_I}, there exists an admissible formal blowup $f\colon\fU'\to\fU$ such that $W = \fW_\eta$ for a formal open subset $\fW\subset\fU'$.  Note that $h|_W = h_{f^*\fL|_\fW}$. We claim that $f^*\fL|_\fW$ is nef.

  By functoriality of reductions, the set-theoretic image of $\fW$ under $f$ is contained in $V_{\fU,x}$.  In particular, if $C$ is any integral, proper subcurve of $\fW_s$, then $f(C)\subset V_{\fU,x}$.  Since $C$ is reduced, the schematic image of $C$ is contained in $V_{\fU,x}$, so we have a commutative square of morphisms of $\td k$-schemes:
  \[\xymatrix @=.25in{
      C \ar[r] \ar[d]_f & {\fW_s} \ar[d]^f \\
      {V_{\fU,x}} \ar[r] & {\fU_s}
    }\]
  Since $\fL_{\fU,x}$ is nef, we have $f^*\fL_{\fU,x}.C\geq 0$ by the projection formula.
\end{proof}

\begin{defn} \label{definition harmonic}
  A $\Lambda$-PL function $h\colon X\to\R$ is called \defi{semipositive} at a point $x\in X$ if it satisfies the equivalent conditions of Proposition~\ref{prop:def.semipositive}, and it is \defi{harmonic} at $x$ if both $h$ and $-h$ are semipositive at $x$.  We say that $h$ is \defi{semipositive} if it is semipositive at all points of $X$.

  A \defi{harmonic function on $X$} is an $\R$-PL function $h\colon X\to\R$ which is harmonic at all points of $X$.   A $\Lambda$-PL harmonic function is called \defi{$\Lambda$-harmonic.}
\end{defn}

\begin{art}\label{sec:harmonic.remarks}
  Some remarks:
  \begin{enumerate}
  \item\label{item:harm.rmk.numtrivial}
    A $\Lambda$-PL function $h$ is harmonic at $x\in X$ if and only if $L_h(x)$ is numerically trivial.  Equivalently, there exists a paracompact neighbourhood $U$ of $x$, a formal $K^\circ$-model $\fU$ of $U$, and a numerically trivial element $\fL\in M(\fU)_\Lambda$, such that $h=h_\fL$.
  \item\label{item:harm.rmk.R}
    Semipositivity and harmonicity are insensitive to the group $\Lambda$, in the following sense: if $h\in\PL_\Lambda(X)$, then $h$ is semipositive at a point $x$ if and only if its image in $\PL_\R(X)$ is semipositive at $x$.
  \item\label{item:harm.rmk.germ}
    Semipositivity and harmonicity of $h$ at $x$ is intrinsic to the germ of $h$ in $\PL_\Lambda(X,x)$.  In particular, the semipositive (resp.\ harmonic) functions on $X$ form a sheaf in the analytic topology. (They do \emph{not} form a sheaf in the $\rG$-topology).
  \item\label{item:harm.rmk.locus}
    It follows from Proposition~\ref{prop:def.semipositive}(1) that the locus of all points $x\in X$ at which $h$ is semipositive (resp.\ harmonic) is open.
  \item\label{item:harm.rmk.monoid}
    The set of semipositive functions at a point forms a monoid (resp.\ a cone if $\Lambda=\R$), and the set of harmonic functions at a point forms a group (resp.\ an $\R$-subspace).
  \item\label{item:harm.rmk.smooth}
    If $h = \log|f|$ for an invertible analytic function $f$ in a neighbourhood of $x$, then $h$ is harmonic at $x$ by Lemma~\ref{lem:model.lb.trivial} (the converse does not hold: see Remark~\ref{harmonic non-smooth functions}).
  \item\label{item:harm.rmk.analogy} Harmonic functions generalize smooth functions in the same way that numerical equivalence generalizes rational equivalence in algebraic geometry.
  \end{enumerate}
\end{art}

\begin{lem}\label{lem:harmonic.global}
  Suppose that $X$ is paracompact.  Let $\fX$ be a formal $K^\circ$-model of $X$ and let $\fL\in M(\fX)_\Lambda$.  Then $h_\fL$ is semipositive (resp.\ harmonic) if and only if $\fL$ is nef (resp.\ numerically trivial).
\end{lem}

\begin{proof}
  The harmonic case follows from the semipositive case.  If $\fL$ is nef then $h_\fL$ is semipositive by definition.  Suppose then that $h\coloneq h_\fL$ is semipositive.  Let $C\subset\fX_s$ be an integral, proper closed curve.  By surjectivity of $\red_\fX$, there exists a point $x\in X$ reducing to the generic point of $C$.  Then $(C,\fL|_C)$ is a model of $(\red(X,x), L_h(x))$, so $\fL|_C$ is nef by Lemma~\ref{lem:nef.all.models}.  In particular, $c_1(\fL).C\geq 0$, so $\fL$ is nef.
\end{proof}

Combined with~\secref{model theorem}, we see that a PL function $h\colon X\to\R$ on a paracompact space $X$ is semipositive (resp.\ harmonic) if and only if there exists a formal $K^\circ$-model $\fX$ of $X$ and a nef (resp.\ numerically trivial) line bundle $\fL\in M(\fX)$ such that $h = h_\fL = -\log\|1\|_\fL$.

We have the following alternative characterization of harmonic functions.  
We do not know if the same characterization holds for semipositive functions.

\begin{prop}\label{prop: equivalence of harmonicity definitions}
  Let $h\colon X\to\R$ be a $\Lambda$-PL function and let $x\in X$.  Then $h$ is harmonic at $x$ if and only if there exists a strictly analytic neighbourhood $U$ of $x$ such that $h|_U = \sum_{i=1}^m\lambda_ih_i$ for 
  $\Z$-harmonic functions $h_1,\ldots,h_m$ on $U$
  and $\lambda_1,\ldots,\lambda_m\in\Lambda$.
\end{prop}

\begin{proof}
  The proposition is vacuous when $\Lambda=\Z$,  so assume $\Lambda$ is divisible.  The ``if'' direction is trivial.  Suppose then that $h$ is harmonic at $x$.  By Proposition~\ref{prop:Lpl.analytic.nbhd}, after shrinking $X$, we may assume that $h = \sum_{i=1}^m \lambda_i h_i$ for $\lambda_1,\ldots,\lambda_m\in\Lambda$ and $h_1,\ldots,h_m\in\PL(X)$.  By Lemma~\ref{lem:lambda.assume.li}, we may assume that $\lambda_1,\ldots,\lambda_m$ are $\Q$-linearly independent.  By Lemma~\ref{lem:reduction.depends.function}, the residue $L_h(x)$ is equal to $\sum_{i=1}^m \lambda_i L_{h_i}(x)$.  Let $V$ be a model for $\red(X,x)$ in the sense of~\secref{sec:rz.models} such that each $L_{h_i}(x)$ is the pullback of a line bundle $L_{V,i}$ on $V$.  Let $C\subset V$ be an integral, proper closed curve.  Since $L_h(x)$ is numerically trivial, we have $\sum_{i=1}^m \lambda_i\,c_1(L_{V,i}(x)).C = 0$.  But $c_1(L_{V,i}(x)).C\in\Z$ for each $i$, so by linear independence, we have $c_1(L_{V,i}(x)).C = 0$ for each $i$.  Thus $L_{V,i}$ is numerically trivial, so $h_i$ is harmonic.
\end{proof}

Now we gather some functoriality properties of semipositive and harmonic functions.

\begin{prop}\label{prop:harmonic.properties}
	Let $X$ be a good strictly $K$-analytic space and let $h \in \PL_\Lambda(X)$.

    Let $f\colon X'\to X$ be a morphism of good strictly $K$-analytic spaces and let $h' = h\circ f$.
    \begin{enumerate}
    	\item[(1)]\label{item:harm.prop.pullback}
    	For $x' \in X'$, let $x = f(x')$.   If $h$ is semipositive at $x$, then $h'$ is semipositive at $x'$.  The converse holds if $x'\in\Int(X'/X)$.
    	\item[($1^\prime$)]\label{item:harm.prop.pullback2}
    	If $h$ is semipositive, then $h'$ is semipositive.  The converse holds if $f$ is proper and surjective.
    \end{enumerate}
Let $K'/K$ be a non-Archimedean field extension, let $X' = X\hat\tensor_K K'$, let $\pi\colon X'\to X$ be the structure morphism, let $x'\in X'$, let $x = \pi(x')$, and let $h' = h\circ\pi$.
	\begin{enumerate}
		\item[(2)]\label{item:harm.prop.basechange}
	  If $h$ is semipositive at $x$, then $h'$ is semipositive at $x'$.  The converse holds if we have $\trdeg(\td K'\td\sH(x)/\td K') = \trdeg(\td\sH(x)/\td K)$, where $\td K'\td\sH(x)$ is the smallest subfield of $\td\sH(x')$ containing $\td\sH(x)$ and $\td K'$.
		\item[($2^\prime$)]\label{item:harm.prop.basechange2}
		The function $h$ is semipositive if and only if $h'$ is semipositive.
    \end{enumerate}
	Let $\iota\colon X_{\red}\to X$ be the inclusion of the reduction of $X$ and let $h' = h\circ\iota$.
	\begin{enumerate}
	\item[(3)]\label{item:harm.prop.reduction}
	  For $x' \in X_{\red}$, the function  $h$ is semipositive at $x =\iota(x')$ if and only if $h'$ is semipositive at $x'$.
		\item[($3^\prime$)]\label{item:harm.prop.reduction2}
		The function $h$ is semipositive if and only if $h'$ is semipositive.
	\end{enumerate}
	The  statements (1)--(3) and ($1^\prime$)--($3^\prime$) hold with ``semipositive'' replaced by ``harmonic.''
\end{prop}

\begin{proof}
  In each case, harmonicity follows easily from semipositivity.  In all cases, the function  $h'$ is $\Lambda$-PL by Lemma~\ref{lem:pl.basic.props}.

  We prove first (1). If $h$ is semipositive at $x$  then $h'$ is semipositive at $x'$ by Lemma~\ref{lem:nef.properties.RZ} and functoriality of reduction.  The converse holds if $x\in\Int(X'/X)$ by Theorem~\ref{thm:reduction.functor} and Lemma~\ref{lem:nef.properties.RZ}(1), using functoriality of residues~\secref{sec:reduction.function.functoriality}.  Assertion~($1^\prime$) follows because a proper morphism has no boundary.

 To prove (2), let $\phi\colon\bP_{\td\sH(x')/\td K'}\to\bP_{\td\sH(x)/\td K}$ be the natural morphism.  We have seen in~\secref{sec:temkin.reduction.germs} that  $\phi\inv(\red(X,x)) = \red(X',x')$.  The assertions now follow from Lemma~\ref{lem:nef.properties.RZ} and functoriality of residues~\secref{sec:reduction.function.functoriality}.
 To prove ($2^\prime$), by Corollary~\ref{cor:hL.is.pl} we may shrink $X$ such that $X$ is compact with  formal model $\fX$  and such that there is $\fL \in M(\fX)_\Lambda$ with $h=h_\fL$. Then the claim follows from Lemma~\ref{lem:nef.properties}(2) and Lemma~\ref{lem:harmonic.global}.

The claims (3) and ($3^\prime$) follow from~(\ref{item:harm.prop.pullback},$1^\prime$) since $X_{\red}\to X$ is proper and surjective.
\end{proof}
The following proposition allows us to pass to Galois extensions when studying harmonic and semipositive functions.

\begin{prop} \label{infinite Galois extensions}
  Suppose that $\Lambda$ is divisible.  Let $L/K$ be a (potentially infinite) normal extension with automorphism group $G \coloneqq \Aut(L/K)$. Let $K'$ be the completion of $L$ and let $X' \coloneqq  X\hat\tensor_K K'$, with structure morphism $\pi\colon X'\to X$.  Then $h\mapsto h\circ\pi$ induces an isomorphism from the monoid of semipositive $\Lambda$-PL functions on $X$ (resp.\ the group of $\Lambda$-harmonic
  functions on $X$) to the monoid of $G$-invariant  semipositive $\Lambda$-PL functions on $X'$ (resp.\ the group of $G$-invariant $\Lambda$-harmonic
functions on $X'$).
\end{prop}

\begin{proof}
This is an immediate consequence of Corollary~\ref{pl infinite Galois extensions} and Proposition~\ref{prop:harmonic.properties}($2^\prime$).
\end{proof}

Our next goal is to approximate locally semipositive metrics by local Fubini--Study metrics (Proposition~\ref{prop:approx.semipositive}).  This will be a crucial ingredient in proving the balancing condition in {Section}~\secref{section: balancing condition}.

Recall that a line bundle on a scheme $Y$ is called \defi{semiample} if some positive tensor power is globally generated.  {We say that $L \in \Pic(Y)_\R$ is \defi{ample} (resp.~\defi{semiample}) if there are ample (resp.~semiample) line bundles $H_1,\dots,H_r$ and $\lambda_1,\dots, \lambda_r \in \R_{\geq 0}$ such that $L=\sum_{j=1}^r \lambda_j H_j \in \Pic(Y)_\R$. 
An ample $L$ is semiample and a semiample $L$ is nef.}

\begin{defn}\label{def:lfs.func}
Let $h\colon X\to\R$ be an {$\R$}-PL function.  We say that $h$ is \defi{Fubini--Study} at a point $x\in X$ if $L_h(x)$ is semiample.  We say that $h$ is \defi{locally Fubini--Study} or \defi{LFS} if it is Fubini--Study at every point of $X$.
\end{defn}

Clearly if $h$ is Fubini--Study at $x$ then it is semipositive at $x$.

\begin{art}\label{lfs.is.max}
It follows from~\cite[Proposition~6.8.2]{CLD} that $h$ is Fubini--Study at $x$ if and only if
{there is an open neighbourhood $U$ of $x$ such that $h|_U$ is an $\R_{\geq 0}$-linear combination of functions of the form $\max\{-\log|f_1|,\ldots,-\log|f_n|\}$ for invertible analytic functions $f_1,\ldots,f_n$ on $U$.} 
\end{art}

The following proposition can be seen as a local analogue of the fact that one can approximate a nef $\Q$-line bundle by ample $\Q$-line bundles on a projective variety.

\begin{prop}\label{prop:approx.semipositive}
Let $h\colon X\to\R$ be an {$\R$}-PL function that is semipositive at a point $x\in\Int(X)$.  Then there exists a strictly analytic neighbourhood $U$ of $x$ and a sequence $h_1,h_2,\ldots$ of {$\R$}-PL LFS functions on $U$ such that $h_i\to h|_U$ uniformly.
\end{prop}

\begin{proof}
We may shrink $X$ to assume it is compact.  
{By Corollary \ref{cor:hL.is.pl}, we may assume that there exists a formal $K^\circ$-model $\fX$ of $X$ and a line bundle $\fL\in M(\fX)_\R$ with $h = h_\fL$,} 
using the notation of~\secref{sec:equiv.pl.setup}.  Since $x$ is inner, 
the closure $V_{\fX,x}$ of $\red_\fX(x)$ in $\fX_s$ is proper (with the notation of~\secref{sec:temkin.reduction.germs}): see \cite[Lemma~6.5.1]{chambert_ducros12:forms_courants} and~\cite[Appendix~A]{Vilsmeier}.  By Chow's lemma, there is a proper birational morphism $V\to V_{\fX,x}$ with $V$ projective over $\td K$. Let $L'$ be an ample line bundle on $V$.  By~\cite[6.6.5]{CLD}, there is a PL function $h'$ on the germ $(X,x)$ such that $(V,L')$ is a model of the residue $L_{h'}(x)$.  Let $U$ be a compact strictly analytic neighbourhood of $x$ on which $h'$ is defined.  Then there is a formal $K^\circ$-model $\fU$ of $U$ and a line bundle $\fL'$ on $\fU$
with $h' = h_{\fL'}$.  By Lemma~\ref{lem:formal.model.cofinal}, there is a compact strictly analytic neighbourhood $U'$ of $x$ and a formal $K^\circ$-model $\fU'$ of $U'$ such that the closure $V' \coloneq V_{\fU',x}$ of $\red_{\fU'}(x)$ dominates $V$.
We may assume that $U' \subset U$ and that this inclusion extends to a morphism $f\colon\fU' \to \fU$ by using Raynaud's theorem. Let $L'|_{V'}$ be the pull-back of $L'$ with respect to $V' \to V$. Then $L'|_{V'}$ and $f^*(\fL')|_{V'}$ are both models for $L_{h'}(x)$. We conclude that there is a dominating model $V''$ of $\red(X,x)$ such that the pull-backs of $L'|_{V'}$ and $f^*(\fL')|_{V'}$  to $V''$ agree. Using Lemma~\ref{lem:formal.model.cofinal} again and the arguments above, we may assume that $V''=V_{\fU'',x}$ for a formal $K^\circ$-model $\fU''$ of a compact strictly analytic neighbourhood $U''$ of $x$ contained in $U'$ such that the inclusion $U'' \subset U'$ extends to a morphism $\fU'' \to \fU'$. Moreover, we may assume that $U'' \subset X$ extends to a morphism $\fU'' \to \fX$. Replacing $\fX$ by $\fU''$ and  $\fL, \fL'$ by their pull-backs to $\fU''$,
we put ourselves in the following situation:
\begin{enumerate}
	\item The closure $V'$ of $\red_\fX(x)$ in $\fX_s$ dominates $V$.
	\item There is {$\fL \in M(\fX)_\R$ and $L \in \Pic(V)_\R$}
	such that $h = h_\fL$ and $\fL|_{V'}$ is the pullback of $L$ under $V'\to V$, and such that $(V,L)$ is a model of $L_h(x)$.
	\item There is a line bundle $\fL'$ on $\fX$ and an ample line bundle $L'$ on $V$ such that $h' = h_{\fL'}$ and $\fL'|_{V'}$ is the pullback of $L'$ under $V'\to V$.
\end{enumerate}
We recycle notation now and let $U \coloneq \red_\fX\inv(V')$.  {By anticontinuity of the reduction map, the set $U$} is an open neighbourhood of $x$.  For  $\epsilon\in\Q_{>0}$ let $L_\epsilon\coloneq L + \epsilon L'$, $\fL_\epsilon \coloneq \fL + \epsilon\fL'$, and $h_\epsilon \coloneq h_{\fL_\epsilon} = h + \epsilon h'$.  Then $h_\epsilon$ is an  {$\R$}-PL function on $X$ and $\fL_\epsilon|_{V'}$ is the pullback of $L_\epsilon$ under $V'\to V$.  Note that $h_\epsilon\to h$ uniformly as $\epsilon\to0$.  If $h$ is semipositive at $x$ then $L$ is nef by Lemma~\ref{lem:nef.all.models}, so $L_\epsilon$ is ample as a consequence of Kleiman's theorem, see \cite[Corollary 1.4.10]{Laz1}.   It follows that $\fL_\epsilon|_{V'}$ is semiample.  If $y\in U$ then $\red_\fX(y)\in V'$, so the closure $Y$ of $\red_\fX(y)$ in $\fX_s$ is contained in $V'$.  Since $h = h_\fL$ and $h' = h_{\fL'}$ on all of $X$, the pair $(Y,\fL_\epsilon|_Y)$ is a model of $L_{h_\epsilon}(y)$ by the construction in~\secref{sec:reduction.pl.metric}.  Since $\fL_\epsilon|_Y$ is again semiample, we see that $h_\epsilon$ is Fubini--Study at $y$.  Thus $h_\epsilon$ is LFS on $U$.
\end{proof}

\begin{rem}
Let $\sX$ be a proper scheme over $K$.  It is shown  in~\cite[Section~5]{gubler-martin} that
a pointwise limit of semipositive $\Q$-PL metrics is semipositive.
We expect the analogous result to hold in our situation.  In particular, we expect the converse of Proposition~\ref{prop:approx.semipositive} to hold true: namely, if an {$\R$}-PL function $h$ is a uniform limit of {$\R$}-PL LFS functions on a strictly analytic neighbourhood $U$ of an interior point $x$, then $h$ is semipositive at $x$.  This would imply that if $h$ and $h'$ are {$\R$}-PL functions that are semipositive at $x$, then so is $\max\{h,h'\}$.  Indeed, we can write $h = \lim h_i$ and $h' = \lim h_i'$ for {$\R$}-PL LFS functions $h_i,h_i'$ in a neighbourhood $U$ of $x$.  By~\secref{lfs.is.max}, the functions $\max\{h_i,h_i'\}$ are {$\R$}-PL LFS functions on $U$, and we have $\max\{h,h'\} = \lim\max\{h_i,h_i'\}$.
\end{rem}

Now we relate semipositivity to positive forms and currents. See~\secref{Smooth forms and currents} for the definitions of positivity based on the corresponding notions for Lagerberg forms and Lagerberg currents introduced in \secref{Lagerberg forms}.
The next result serves as a replacement for \cite[Lemmas~5.2.3 and~5.3.3]{chambert_ducros12:forms_courants} where the result is claimed for strongly positive forms; this is not correct, as explained in \cite[Example~2.3.6]{burgos-gubler-jell-kuennemann1}.

\begin{lem} \label{positive and symmetric forms}
Let $X$ be a Hausdorff analytic space over $K$ and let $\omega \in \cA_{{\rm sm},c}^{p,p}(X)$ be symmetric. Then there are positive forms $\omega_1,\omega_2 \in \cA_{{\rm sm},c}^{p,p}(X)$ such that $\omega = \omega_1- \omega_2$.
\end{lem}

\begin{proof}
The support of $\omega$ is compact and hence has a paracompact open neighbourhood in $X$ by \cite[Lemme~2.1.6]{chambert_ducros12:forms_courants}. Replacing $X$ by this open neighbourhood, we may assume that $X$ is paracompact. Then we have smooth partitions of unity available \cite[Proposition~3.3.6]{chambert_ducros12:forms_courants}, so we may argue locally. Hence we may assume that there is an open subset $U$ of $X$ such that the support of $\omega$ is contained in $U$ and such that $W \coloneqq \bar U$ is a compact analytic domain in $X$ with  a moment map $\varphi \colon W \to \bGm^{n,\an}$ and with $\omega = \varphi_{\rm trop}^*(\alpha)$ for a symmetric form $\alpha \in \cA^{p,p}(\varphitrop(W))$.  It follows from~\secref{Lagerberg forms} that $\alpha=\alpha_1-\alpha_2$ for two smooth positive forms $\alpha_1,\alpha_2$ on $\varphitrop(W)$. As usual, the existence of smooth partitions of unity yield that there is a smooth non-negative function $\psi$ with support in $U$ such that $\psi \equiv 1$ on $\supp(\omega)$. Setting $\omega_i \coloneqq \psi \cdot \varphi_{\trop}^*(\alpha_i)$ for $i=1,2$, we get the claim.
\end{proof}

\begin{art}[Continuous Plurisubharmonic Functions] \label{continuous psh functions}
For simplicity, we suppose that $X$ is a boundaryless%
\footnote{Currents in analysis work well only in the boundaryless case, as otherwise there are disturbing residues along the boundary. Note that Chambert-Loir and Ducros also omit the boundary in their definition of  currents \cite[4.2.2]{chambert_ducros12:forms_courants}, as their currents act on smooth forms with compact support disjoint from the boundary.} good separated strictly analytic space to use currents on $X$ as in~\secref{Smooth forms and currents}. Using the theory of integration from \cite[3.7--3.10]{chambert_ducros12:forms_courants}, every continuous real function $h$ on $X$ has an \defi{associated current} $[h]$ given by
\[ [h] \colon \cA_{{\rm sm},c}(X) \longrightarrow \R \, , \quad \omega \mapsto \int_X h \cdot \omega \]
where the integral is defined by a limit process using that $h$ is a uniform limit of smooth functions on a neighbourhood of the support of $\omega$ (see~\cite[Corollaire~5.4.7, 5.4.8]{chambert_ducros12:forms_courants}). We say that $h$ is \defi{plurisubharmonic} (short: psh) if the current $\d'\d''[h]$ is positive. A partition of unity argument shows that plurisubharmonicity is a local notion on $X$ (see~\cite[Lemme~5.5.2]{chambert_ducros12:forms_courants}).

By \cite[Lemme~5.5.3]{chambert_ducros12:forms_courants}, a smooth function $h$ is psh if and only if every $x \in X$ has a compact strictly analytic neighbourhood $V$ with a moment map $\varphi\colon V \to \bGm^{r,\an}$ such that $h|_V= u \circ \varphitrop$ for a continuous function $u$ on $\varphitrop(V)$ which restricts to a convex function on each face. This characterization shows that smooth psh functions admit pull-backs with respect to morphisms.
\end{art}

\begin{thm} \label{semipositive implies psh}
Let $X$ be a good separated strictly analytic space over $K$ and let $h\colon X \to \R$ be a semipositive {$\R$}-PL function. Then $h$ restricts to a psh function on $X \setminus \partial X$.
\end{thm}

\begin{proof}
  By restriction to $X \setminus \partial X$, we may assume that $X$ is boundaryless. The claim is local in $X$ and so we may assume that $h$ is a uniform limit of {$\R$}-PL LFS functions $h_i$ on $X$ by using Proposition~\ref{prop:approx.semipositive}. 	It follows from~\secref{lfs.is.max} and from \cite[Corollaire~6.3.2]{chambert_ducros12:forms_courants} that an  $\R$-PL locally Fubini--Study function is locally a uniform limit of smooth psh functions and hence $h_i$ is psh.
  For any positive form $\omega \in \cA_{{\rm sm},c}(X)$, we deduce from uniform convergence that
$$ \langle \omega, \d'\d''[h] \rangle = \int_X h \cdot \d'\d'' \omega = \lim_i \int_X h_i \cdot \d'\d'' \omega \geq 0.$$
We conclude that $h$ is psh.
\end{proof}

 \begin{cor} \label{Chern current of numerically trivial}
	Let $h$ be a harmonic function on a good separated strictly analytic space $X$ over $K$.  Then the restriction of the current $\d'\d''[h]$ to $X \setminus \partial X$ is zero.
\end{cor}

\begin{proof}
	We may assume that $X$ is boundaryless. It follows from Theorem~\ref{semipositive implies psh} that $h$ and $-h$ are psh on $X$.
	By linearity in the irreducible components and by arguing locally, we may assume that $X$ is of pure dimension $d$. Let $\eta \in \cA_{{\rm sm},c}^{d-1,d-1}(X)$. Since $\d'\d''[h]$ is a symmetric current on $X$, we have
	$$\langle \eta, \d'\d''[h] \rangle = \frac{1}{2}\left(\langle \eta, \d'\d''[h] \rangle + (-1)^{d-1}\langle J\eta, \d'\d''[h] \rangle \right)=\frac 12\langle \omega, \d'\d''[h] \rangle $$
	for the symmetric form $\omega \coloneqq \eta + (-1)^{d-1}J\eta \in  \cA_{{\rm sm},c}^{d-1,d-1}(X)$.
	By Lemma~\ref{positive and symmetric forms}, there are positive forms $\omega_1,\omega_2 \in  \cA_{{\rm sm},c}^{d-1,d-1}(X)$ such that $\omega=\omega_1-\omega_2$. As $h$ and $-h$ are psh, we get
	$$2\langle \eta, \d'\d''[h] \rangle = \langle \omega_1, \d'\d''[h] \rangle - \langle \omega_2, \d'\d''[h] \rangle = 0-0=0,$$
	proving the claim.
	\end{proof}

\begin{rem} \label{formulation for metrics}
In this section, we have focused on the concept of semipositive and harmonic functions. In fact, all the results immediately transfer to $\Lambda$-PL metrics $\metr$ of a line bundle $L$ over a good strictly analytic space $X$. Indeed, the choice of a local frame $s$	gives rise to a $\Lambda$-PL function $h \coloneqq -\log \|s\|$, to which we can apply the above results. Then Theorem~\ref{semipositive implies psh} yields that on a good separated strictly analytic space $X$ every semipositive {$\R$}-PL metric of $L$ has a positive first Chern current $c_1(L,\metr)$ as defined in \cite[6.4.1]{chambert_ducros12:forms_courants}. The same holds for uniform limits of such metrics on $L$ which includes all the metrics considered by Zhang \cite{zhang95} over projective varieties.
\end{rem}

\section{Piecewise Linear and Harmonic Tropicalizations} \label{section: weights for pl tropicalizations}

In this section, we assume that $X$ is a compact good strictly $K$-analytic space of pure dimension $d$.  We will define a piecewise linear tropicalization map to be a tuple $h = (h_1,\ldots,h_n)$ of PL functions on $X$.  We will endow $h(X)\subset\R^n$ with the structure of a weighted $\z$-polytopal complex, which is canonical up to subdivision.  The difficult part is defining weights on $h(X)$, which we will do $\rG$-locally using the results of~\secref{sec: tropical multiplicities}.

\begin{defn}
	A \defi{piecewise linear (PL) tropicalization} is a map
	\begin{align*}
	h = (h_1,\dots,h_n) \colon X \To \R^n
	\end{align*}
	where $h_1,\dots,h_n$ are piecewise linear functions.
	It is called a \defi{harmonic tropicalization} if in addition all the $h_i$ are harmonic (which means that all of the $h_i$ are $\Z$-harmonic).  It is called a \defi{smooth tropicalization map} if $h_i = \log|f_i|$ for $f_i\in\Gamma(X,\sO_X^\times)$.
\end{defn}

\begin{art} \label{remark on tropicalization maps}
	Remarks:
	\begin{enumerate}
		\item A smooth tropicalization map has the form $\phi_{\trop}$ for a moment map $\phi\colon X\to\bGm^{n,\an}$.  These are the tropicalizations considered in~\cite{CLD}.
		\item A smooth tropicalization map is harmonic by Remark~\ref{sec:harmonic.remarks}(6).
		\item A piecewise linear tropicalization is $\rG$-locally a smooth tropicalization map.
		\item Piecewise linear, harmonic, and smooth tropicalization maps all pull back under morphisms of analytic spaces and arbitrary change of base: see Lemma~\ref{lem:pl.basic.props} and Proposition~\ref{prop:harmonic.properties}.
	\end{enumerate}
\end{art}

\begin{prop} \label{local description of piecewise linear}
	A function $h\colon X \to \R$ is piecewise linear if and only if each $x\in X$ has a strictly affinoid neighbourhood $U$ such that $h|_U = F \circ h'$ for a smooth tropicalization map $h'\colon U \to \R^n$ and a piecewise $\z$-linear function $F\colon h'(U)\to\R$.
\end{prop}

\begin{proof}
	Suppose that $h$ is piecewise linear.  Fix $x\in X$.  As $X$ is a good analytic space, we may assume that it is strictly affinoid.  There is a finite $\rG$-covering $X = \bigcup_{i \in I} X_i$ by strictly affinoid domains $X_i$ on which $h = -\log|f_i|$ for an invertible function $f_i$ on $X_i$.  Shrinking $X$, we may assume that $x\in\bigcap X_i$.  By the Gerritzen--Grauert theorem, we may assume that each $X_i$ is a rational domain in $X$, i.e.\ $X_i = X\bigl(\frac{g_{i}}{u_i}\bigr)$ for finitely many analytic functions $(g_{ij})_{j \in J_i},u_i$ without common zeros on $X$, where $g_i \coloneqq (g_{ij})_{j \in J_i}$.  Since $x\in X_i$ we have $u_i(x)\neq 0$; shrinking $X$ again, we can assume all $u_i$ are units on $X$.  Then we may replace $g_{i}$ by $g_{i}/u_i$ to assume $X_i = X(g_{i})$ is a Weierstrass domain in $X$.  If some $g_{ij}(x) = 0$ then we replace $g_{ij}$ by $g_{ij}+\lambda$ for any scalar $\lambda\in K^\times$ with $|\lambda|<1$, noting that $|g_{ij}(y)|\leq 1$ if and only if $|g_{ij}(y)+\lambda|\leq 1$.  Shrinking $X$, we assume that all $g_{ij}$ are invertible on $X$.

	Since $f_i$ is invertible on $X_i$, there is $\varepsilon>0$ such that $|f_i| \geq \varepsilon$ on $X_i$.  Since $X_i$ is a Weierstrass domain, the ring $\sO(X)$ is dense in $\sO(X_i)$, so there exists $a_i \in \Ocal(X)$ such that $|f_i-a_i| < \varepsilon$ on $X_i$. Then the ultrametric triangle inequality shows that we may replace $f_i$ by $a_i$.  Using that $|f_i(x)|\geq\epsilon>0$, we may shrink $X$ to assume that all $f_i \in \sO(X)^\times$.

	Define a moment map $\varphi\colon X \to \bGm^{N,\an}$ by $\phi = (f_i,g_{ij})_{i \in I, j \in J_i}$, and let $h' \coloneq \phi_{\trop}$ be the induced smooth tropicalization map.  By construction, if $p_i\colon\R^N\to\R$ is the projection onto the factor of $\R$ corresponding to $f_i$, then $h|_{X_i} = p_i \circ h'$.  Define $F\colon h'(X)\to\R$ by $F|_{h'(X_i)} = p_i$.  This is well-defined because if $y\in X_i$ and $y'\in X$ satisfy $h'(y)=h'(y')$, then $y'\in X_i$ as well because $|g_{ij}(y')|=|g_{ij}(y)|\leq 1$ for all $j\in J_i$.  Thus we get a piecewise $\z$-linear function $F\colon h'(X)\to\R$ such that $h = F\circ h'$.   This proves one direction.

	For the converse, we may assume $X = U$ is strictly affinoid and that $h = F\circ h'$, where $h' = \phi_{\trop}$ for a moment map $\phi\colon X\to\bGm^{n,\an}$ and $F\colon \phi_{\trop}(X)\to\R$ is piecewise $\z$-linear.  Suppose that $F$ is $\z$-linear on a closed $\z$-polytope $\sigma\subset \phi_{\trop}(U)$.  Then there exists an affine homomorphism (homomorphism followed by translation) $q\colon\bGm^n\to\bGm$ such that $h = (q\circ\phi)_{\trop}$ on $\phi_{\trop}\inv(\sigma)$.  Thus $h$ has the form $\log|f|$ on $\phi_{\trop}\inv(\sigma)$ for some $f \in \sO(\phi_{\trop}\inv(\sigma))^\times$.  Since $X$ has a $\rG$-covering by
	such strictly analytic domains $\phi_{\trop}\inv(\sigma)$, this concludes the proof.
\end{proof}

\begin{art}[Covering by a Smooth Tropicalization Map] \label{smooth tropicalization covering}
	Let $h\colon X \to \R^n$ be a piecewise linear tropicalization map. Applying Proposition~\ref{local description of piecewise linear} to the coordinate functions $h_1, \dots, h_n$, we get locally on $X$ a smooth tropicalization map $h'\colon U \to \R^m$ and a piecewise $\z$-linear map $F\colon h'(U) \to \R^n$ with $h=F \circ h'$ on $U$.
\end{art}

\begin{defn} \label{smooth covering tropicalization}
	A piecewise linear tropicalization $h\colon X \to \R^n$ is \emph{covered by a smooth tropicalization map $h'\colon U \to \R^m$} on a compact strictly analytic subdomain $U\subset X$ if there is a piecewise $\z$-linear function $F\colon h'(U) \to \R^n$ such that $h=F \circ h'$.
\end{defn}

It follows from~\secref{smooth tropicalization covering} that locally on $X$, a piecewise linear tropicalization is always covered by a smooth tropicalization map.

\begin{art}[Polytopal structure on PL tropicalizations] \label{choice of polytopal structure}
	Let $h\colon X\to\R^n$ be a PL tropicalization map.  There is a $\z$-polytopal complex $\Pi$ of dimension at most $d$ with support $h(X)$ such that $h(\partial X)$ is contained in a subcomplex of dimension at most $d-1$.  This follows from~\secref{sec:trop.is.polyhedral} by restricting $h$ to any  $\rG$-open piece where $h$ is a smooth tropicalization map.  Explicitly, there is a finite $\rG$-covering $(X_j)_{j \in J}$ of $X$ by compact strictly analytic subdomains such that $h|_{X_j}$ is the smooth tropicalization map associated to a moment map $\varphi_j\colon X_j \to \bGm^{n,\an}$.  Then $(\phi_j)_{\trop}(X_j)$ is a finite union of $\z$-polytopes, and we have $\del X\subset\bigcup_{j \in J}\del X_j$.
\end{art}

\begin{defn}\label{def:pl.multiplicities}
	Let $h\colon X\to\R^n$ be a PL tropicalization.  We say that $h$ is \defi{smooth at a point $x\in\Int(X)$} if there exists a compact strictly analytic neighbourhood $U$ of $x$ such that $h|_U$ is a smooth tropicalization map, i.e., if $h|_U = \phi_{\trop}$ for a moment map $\phi\colon U\to\bGm^{n,\an}$.  Note that $x\in\Int(U) = \Int(U/X)\cap\Int(X)$.  If $h$ is smooth at $x\in\Int(X)$, the \defi{tropical multiplicity of $x$ with respect to $h$} is the number $m_h(X,x) \coloneq m_\phi(X,x)$ defined in Definition~\ref{def:trop.mult}.

	We say that $h$ is \defi{smooth over a point $\omega\in h(X)$} if $h\inv(\omega)$ can be covered by compact strictly analytic domains $U$ such that $h|_U$ is a smooth tropicalization map and $\omega\notin h(\del U)$.

	If $h$ is smooth over $\omega$, the \defi{tropical multiplicity of $\omega$ with respect to $h$} is the number $m_h(X,\omega)\coloneq\sum_{h(x)=\omega}m_h(X,x)$.

\end{defn}

\begin{art}\label{rem:pl.multiplicities}
	Some remarks:
	\begin{enumerate}
		\item By Proposition~\ref{item:trop.mult.1units}, the quantity $m_h(X,x)$ does not depend on the choice of $U$ or $\phi$.
		\item Suppose that $h$ is smooth over $\omega\in h(X)$.
		Let $U\subset X$ be a compact strictly analytic domain such that $\omega\notin h(\del U)$.  Then $h\inv(\omega)\cap U = h\inv(\omega)\cap\Int(U)$ is open and closed in $h\inv(\omega)$.  Since $h\inv(\omega)$ can be covered by finitely many such $U$,
		it follows from Remark~\ref{rem: tropical multiplicities}(\ref{tropmult 7}) that $m_h(X,\omega)$ is a nonnegative integer.
		\item By~\secref{choice of polytopal structure}, there is a $\z$-polytopal complex $\Pi$ of dimension at most $d$ with support $h(X)$ such that $h$ is smooth over the complement of a  subcomplex of dimension at most $d-1$.  (Choose $\Pi$ such that each $h(\del X_i)$ is contained in such a subcomplex.)
		\item If $h$ is smooth over $\omega\in h(X)$ and $h(X)$ has dimension $d$ at $\omega$, then $m_h(X,\omega) > 0$ by Remark~\ref{rem: tropical multiplicities}(\ref{tropmult 7}).
		\item If $h$ is itself a smooth tropicalization map, i.e., if $h = \phi_{\trop}$ for a moment map $\phi\colon X\to\bGm^{n,\an}$, then $m_h(X,x) = m_\phi(X,x)$ and $m_h(X,\omega) = m_\phi(X,\omega)$ for all $x\in\Int(X)$ and $\omega\in h(X)\setminus h(\del X)$.
	\end{enumerate}
\end{art}

In order to pass from points to faces, we need the following consequence of Proposition~\ref{prop: trop.mult.constant}.

\begin{prop}\label{prop:pl.trop.mult.constant}
	Let $h\colon X\to\R^n$ be a PL tropicalization, and choose a $\z$-polytopal complex structure $\Pi$ on $h(X)$ such that $h$ is smooth over the complement of a subcomplex of dimension at most $d-1$, as in Remark~\ref{rem:pl.multiplicities}(3).  Let $\sigma$ be a $d$-dimensional face of $\Pi$.  Then the tropical multiplicity $m_h(X,\omega)$ is constant in $\omega\in\relint(\sigma)$.
\end{prop}

\begin{proof}
	Fix $\omega\in\relint(\sigma)$, and cover $h\inv(\omega)$ by finitely many compact strictly analytic domains $U_i$ such that $\omega\notin h(\del U_i)$ and such that  $h|_{U_i} = (\phi_i)_{\trop}$ for moment maps $\phi_i\colon U_i\to\bGm^{n,\an}$.   Let $Z$ be the complement in $X$ of $\bigcup\Int(U_i)$.  Then $Z$ is a closed subset, so it is compact, and hence $h(Z)$ is a closed subset of $\R^n$ not containing $\omega$.  Thus there exists an open neighbourhood $\Omega$ of $\omega$ in $\relint(\sigma)$ such that $h\inv(\Omega)$ is contained in $\bigcup\Int(U_i)$.  Shrinking $\Omega$, we may assume that it is disjoint from $\bigcup h(\del U_i)$.  For $\omega'\in\Omega$ we have $m_h(U_i,\omega') = m_{\phi_i}(U_i,\omega')$, and the latter quantity is locally constant in $\omega'\in\Omega$ by Proposition~\ref{prop: trop.mult.constant}.  We have  $\del(U_i\cap U_j)\subset\del(U_i)\cup\del(U_j)$, so $h(\del(U_i\cap U_j))$ is disjoint from $\Omega$, and thus $m_h(U_i\cap U_j,\omega')$ is locally constant in $\omega'\in\Omega$ as well.  The same holds for triple intersections and so on, so the proposition is true by the inclusion-exclusion principle.
\end{proof}

We can now make a definition analogous to Definition~\ref{def:trop.mult.face}.

\begin{defn}[Weights on PL tropicalizations]\label{def:pl.trop.mult.face}
	Let $h\colon X\to\R^n$ be a PL tropicalization.  Choose a $(\Z,\Gamma)$-polytopal complex $\Pi$ with support $h(X)$ such that $h$ is smooth over the complement of a subcomplex of dimension at most $d-1$, as in Remark~\ref{rem:pl.multiplicities}(3).  Let $\sigma$ be a $d$-dimensional face of $\Pi$.  Then we define the \defi{tropical weight}  of $\sigma$ by  $m_\sigma \coloneqq m_h(X,\sigma) \coloneq m_h(X,\omega)$ for any $\omega\in\relint(\sigma)$. This is well-defined by Proposition~\ref{prop:pl.trop.mult.constant}.

Let $\Pi_d \subset \Pi$ be the set of $d$-dimensional faces of $\Pi$. We consider the $d$-dimensional part $h(X)_d$ of $h(X)$ as the weighted $(\Z,\Gamma)$-polytopal complex $(\Pi_d,m)$, well-defined up to subdivision.
\end{defn}

Tropical weights are a ``geometric'' notion, in the following sense.

\begin{prop} \label{base change and tropical multiplicities}
	Let $K'/K$ be an extension of non-Archimedean fields, let $X' = X\hat\tensor_K K'$ with structure morphism $\pi\colon X'\to X$, let $h\colon X\to\R^n$ be a PL tropicalization, and let $h' = h\circ\pi$.  Then $h'(X')_d = h(X)_d$ as weighted $\z$-polytopal complexes.
\end{prop}

\begin{proof}
	This is an immediate consequence of Proposition~\ref{item:trop.mult.basechange}.
\end{proof}

\begin{art}[Change of tropicalization] \label{linear maps and d-skeleton}
	Let $h\colon X\to\R^n$ be a PL tropicalization and let $L\colon \R^n \to \R^m$ be a piecewise $\z$-linear map.  Then $h'\coloneq L \circ h\colon X \to \R^m$ is a PL tropicalization.  Indeed, choosing a $\z$-polytopal complex structure $\Pi$ on $h(X)$ such that $L$ is $\z$-linear on every face of $\Pi$, we see that $\rG$-locally on $X$, each coordinate of $h'$ is a $\Z$-linear combination of the coordinates of $h$ plus a scalar in $v(K^\times)$.
\end{art}

In the following, we will use the results of Ducros and of~\secref{sec: tropical multiplicities} to introduce an intrinsic closed subset $S(X,h)_d$ of $X$ with a piecewise $\z$-linear structure which covers the tropical variety $h(X)_d$, and we will endow $S(X,h)_d$ with canonical weights.

\begin{art}[The $d$-skeleton of a PL tropicalization]\label{rem:pl.d-skeleton}
	Let $h\colon X\to\R^n$ be a PL tropicalization map.  The \defi{$d$-skeleton of $h$} is the subset
	\[ S_h(X)_d \coloneq \bigl\{ x\in X\mid \dim h(X,x)=d \bigr\}. \]
	We claim that $S_h(X)_d$ is a $\q$-skeleton of $X$ in the sense of Remark~\ref{rem: piecewise linear structure on tropical skeleton}.  Let $X = \bigcup_{j\in J}X_j$ be a $\rG$-covering of $X$ by compact strictly analytic subdomains such that $h|_{X_j} = (\phi_j)_{\trop}$ for a moment map $\phi_j\colon X_j\to\bGm^{n,\an}$, as in~\secref{choice of polytopal structure}.  Then $S_h(X)_d = \bigcup_j S_{\phi_j}(X_j)_d$, and the latter is a $\q$-skeleton of $X$ by Remark~\ref{rem:d-skeleton} and~\cite[Theorem~5.1 and Proposition~4.9]{ducros12:squelettes_modeles}.  As in Remark~\ref{rem: piecewise linear structure on tropical skeleton}, there is a unique way to endow $S_h(X)_d$ with a $\z$-linear structure (depending on $h$) such that $h$ restricts to a $\z$-linear isomorphism on every face.  The map $S_h(X)_d\to h(X)$ has finite fibers and surjects onto the $d$-dimensional part $h(X)_d$ of $h(X)$. The constructions are independent of the choice of $X_j$ and $\varphi_j$. By Proposition~\ref{prop:tropical.skeleton}(\ref{item:trop.skeleton.abhyankar}), all elements of $S_h(X)_d$ are Abhyankar points.
\end{art}

\begin{lem} \label{harmonic d-skeleton and generic projection}
	Let $Q \colon \Z^n \to \Z^m$ be a homomorphism such that $Q_\R$ is injective on every $d$-dimensional face of $h(X)$. Then we have $S_h(X)_d=S_{Q_\R\circ h}(X)_d$.
\end{lem}

\begin{proof}
	Setting $L \coloneqq Q_\R$, we have seen in~\secref{linear maps and d-skeleton} that $h' \coloneqq L \circ h$ is a PL tropicalization. By definition, the $d$-skeleton $S_h(X)_d$ is the set of $x \in X$ such that $\dim h(X,x)=d$ and similarly for $h'$. Since $Q_\R$ is injective on every $d$-dimensional face of $h(X)$, we have $\dim h(X,x)=\dim h'(X,x)$ for every $x \in X$. This proves $S_h(X)_d = S_{h'}(X)_d$.
\end{proof}

\begin{art}[Weights on the $d$-skeleton]\label{rem:pl.mult.d-skeleton}
	We continue with the notation in Remark~\ref{rem:pl.d-skeleton}.
	After subdivision of the source and the target, we can choose
	an abstract $(\Z,\Gamma)$-polytopal complex with support $S_h(X)_d$
	and a $\z$-polytopal complex $\Pi$ with support $h(X)$ such that
	\begin{enumerate}
		\item Every $S_{\varphi_j}(X_j)_d=X_j \cap S_h(X)_d$ is a subcomplex of $S_h(X)_d$.
		\item The restriction of $h$ to any face  of $S_h(X)_d$ is a $(\Z,\Gamma)$-linear isomorphism onto a face of $\Pi$.
		\item Every $h(\partial X_j)$ is contained in subcomplex of $\Pi$ of dimension at most $d-1$.
	\end{enumerate}
	For a $d$-dimensional face $\Delta$ of $S_h(X)$, there is $j$ such that $\Delta$ is a face of $S_{\varphi_j}(X_j)$. For $x \in \relint(\Delta)$, we conclude that $m_h(X,x)= m_{\varphi_j}(X,x)$ is the weight $m_{\varphi_j}(X_j,\Delta)$ of the $d$-dimensional face $\Delta$ of $S_{\varphi_j}(X_j)$ considered in  Remark~\ref{rem: tropical multiplicities constant on faces of skeleton} and hence does not depend on the choice of $x \in \relint(\Delta)$. Defining $m_\Delta\coloneq m_h(X,\Delta) \coloneqq m_h(X,x)$,  this makes $S_h(X)_d$ into a weighted abstract $\z$-polytopal complex. The weighted polytopal complex $h(X)_d$ is the push-forward of the weighted polytopal complex $S_h(X)_d$, i.e. for any $d$-dimensional face $\sigma$ of $\Pi$, we have
	\begin{equation} \label{tropical multiplicities as push-forward}
	m_\sigma = \sum_{h(\Delta)=\sigma} m_\Delta.
	\end{equation}
\end{art}

\begin{lem}[Functoriality of the $d$-skeleton] \label{morphisms and d-skeleton}
	Let $f\colon X' \to X$ be a morphism of compact good strictly analytic spaces of pure dimension $d$ and let $h\colon X \to \R^n$ be a PL tropicalization. Then  $h' \coloneqq h \circ f$ is a PL tropicalization with $f(S(X',h')_d)  \subset S(X,h)_d$ inducing a piecewise $\z$-linear map $S_{h'}(X')_d \to S_h(X)$ which is finite to one.
\end{lem}

\begin{proof}
	Using the characterization of the $d$-skeleton by germs of tropicalizations, it is clear that $f(S_{h'}(X')_d)\subset S_h(X)_d$. Since $h'\colon S_{h'}(X')_d \to \R^n$ factorizes through $h\colon S_{h}(X)_d \to \R^n$ and both are finite-to-one PL maps which determine the PL-structure on the respective $d$-skeletons, we deduce the claim.
\end{proof}

\begin{rem} \label{preimage of the d-skeleton}
We have seen in Proposition~\ref{prop:tropical.skeleton}(\ref{item:trop.skeleton.morphism}) that the preimage of the tropical skeleton is the tropical skeleton. However, in the setting of Lemma~\ref{morphisms and d-skeleton}, such an identity
\begin{equation} \label{preimage d-skeleton}
S_{h'}(X')_d = f\inv(S(X,h)_d)
\end{equation}
is not always true. A counterexample is when $X$ is the closed annulus $\{\frac12\leq|x|\leq1\}$ in $\bGm^\an$ and  $f\colon X'\subset X$ is the inclusion of the annulus $X'\coloneqq \{|x|=1\}$.  For the smooth tropicalizatio map $h\colon X\to\R$, given by $x \mapsto -\log|x|$, we have $S(X',h')_d=\emptyset$ but $f\inv(S(X,h)_d)\neq\emptyset$.

However, if $f$ is open (in particular, this applies if $f$ is finite and flat by \cite[Proposition~3.2.7]{berkovic93:etale_cohomology}, then \eqref{preimage d-skeleton} is true. This follows from the definition of the $d$-skeleton in terms of germs of tropicalizations.
\end{rem}

\begin{art}[Compatible face structures] \label{compatible face structures}
	We continue with the setup of~\secref{morphisms and d-skeleton} and we use the above $\rG$-covering $X = \bigcup_{j\in J}X_j$ again. Then the compact strictly analytic subdomains $X_j'$ form a $\rG$-covering of $X'$. We choose a face structure on $S_h(X)_d$ (resp.~on $S_{h'}(X')$) as in~\secref{rem:pl.mult.d-skeleton}. After subdivisions, we may assume that the same $\z$-polytopal complex $\Pi$ in~\secref{rem:pl.mult.d-skeleton} works for $S_h(X)_d$ and $S_{h'}(X')_d$, and that the restriction of $f$ to each face of $S_{h'}(X')_d$ is a $\z$-linear isomorphism onto a face of $S_h(X)_d$.
\end{art}

\begin{prop} \label{functoriality of weights}
	Let $f\colon X' \to X$ be a morphism of compact good strictly analytic spaces of pure dimension $d$, let $h\colon X \to \R^n$ be a PL tropicalization and let $h' \coloneqq h \circ f$. We choose compatible face structures on $S_{h'}(X')_d$ and $S_h(X)_d$ as in~\secref{compatible face structures}. Let $\Delta'$ be a $d$-dimensional face of $S_{h'}(X')_d$ and let $x' \in \relint{\Delta'}$. Then $\Delta \coloneqq f(\Delta)$ is a $d$-dimensional face of $S_{h'}(X')_d$. For $x \coloneqq f(x')$, we have $x \in \relint(\Delta)$, the local ring $\sO_{X',x'}$ is a finite module over the artinian ring $\sO_{X,x}$ and
	\begin{equation} \label{local functorial for multiplicities}
	m_{\Delta'} \cdot \length_{\sO_{X,x}}(\sO_{X,x})= m_\Delta \cdot \length_{\sO_{X,x}}(\sO_{X',x'}).
	\end{equation}
\end{prop}

\begin{proof}
	By the choice of the face structures,  $\Delta \coloneqq f(\Delta)$ is a $d$-dimensional face of $S_{h'}(X')_d$ and $\sigma \coloneqq h(\Delta)=h'(\Delta')$  is a $d$-dimensional face of $\Pi$.
	Then $\Delta$ is a face of some subcomplex  $S_{\varphi_j}(X_j)_d$ associated to the moment map $\phi_j\colon X_j\to\bGm^{n,\an}$. Let us choose a  homomorphism $Q\colon \Z^n \to \Z^d$ which is generic with respect to $h(X)$ and hence generic with respect to $h'(X')\subset h(X)$. Let $\bT = \bGm^d$ and let $q\colon \bGm^n \to \bT$ be the homomorphism of tori associated to $Q$. We set $\psi_j \coloneqq q \circ \varphi_j$ and $\psi_j' \coloneqq \psi_j \circ f$.

	Let us choose $x' \in \relint(\Delta')$. Then $x \coloneqq f(x')\in \relint(\Delta)$ and $\omega \coloneq h(x) \in \relint(\sigma)$. By Lemma~\ref{harmonic d-skeleton and generic projection}, we have  $x \in S_{\psi_j}(X_j)_d$ and hence  $\xi \coloneqq \psi_j(x) \in S(\bT)$ by Proposition~\ref{prop:tropical.skeleton}(4).
	In the following, we denote the residue fields of the local rings $\sO_{\bT^\an,\xi},$ $\sO_{X,x},$ and $\sO_{X',x'}$ by $\kappa(\xi),\kappa(x),$ and $\kappa(x')$, as in~\secref{Non-Archimedean geometry}.
	By the choice of the face structures, we have $\relint(\Delta') \subset \Int(X')$ and $\relint(\Delta)\subset \Int(X)$. By Proposition~\ref{item:trop.mult.samedim}, we conclude that $\sO_{\bT^\an,\xi}$ is the residue field $\kappa(\xi)$ and that the local rings $\sO_{X,x}$ and $\sO_{X',x'}$ are finite-dimensional $\kappa(\xi)$-algebras. We deduce from Proposition~\ref{item:trop.mult.projection} that
	\begin{equation} \label{local functorial for multiplicities over xi}
	m_{\Delta'} \cdot \dim_{\kappa(\xi)}(\sO_{X,x})= m_\Delta \cdot \dim_{\kappa(\xi)}(\sO_{X',x'}).
	\end{equation}
	Finite-dimensionality over $\kappa(\xi)$ yields that $\sO_{X,x}$ is an artinian ring and that $\sO_{X',x'}$ is a finite module over $\sO_{X,x}$. By \cite[Lemma~A.1.3]{fulton98:intersec_theory}, we have
	$$\length_{\sO_{X,x}}(\sO_{X',x'})= \length_{\sO_{X',x'}}(\sO_{X',x'}) \cdot [\kappa(x'): \kappa(x)].$$
	Similarly, we can express $\dim_{\kappa(\xi)}(\sO_{X,x})$ and $\dim_{\kappa(\xi)}(\sO_{X',x'})$ in terms of the multiplicity of the local rings and the residue degrees over $\kappa(\xi)$. This and \eqref{local functorial for multiplicities over xi} yield \eqref{local functorial for multiplicities}.
\end{proof}

\begin{rem} \label{functoriality in reduced case}
	In the above setting, suppose that $\sO_{X,x}$ is reduced. Since $\sO_{X,x}$ is an artinian local ring, we conclude that it is equal to its residue field $\kappa(x)$, and  \eqref{local functorial for multiplicities} yields
	$$ m_{\Delta'} = m_\Delta \cdot \dim_{\kappa(x)}(\sO_{X',x'}).$$
\end{rem}

The following result will be used to define weights on germs of PL tropicalizations.
\begin{lem}\label{lem:pl.trop.germ.weight}
	Let $h\colon X\to\R^n$ be a PL tropicalization. For $x\in X$,
	there is a compact strictly analytic neighbourhood $V(x)$ of $x$ such that for any compact strictly analytic neighbourhood $W \subset V(x)$ of $x$, the weighted germs  of $h(V(x))_d$ and $h(W)_d$ at $h(x)$ agree.
\end{lem}

\begin{proof}
	We choose compatible face structures on $S_h(X)_d$ and on $h(X)_d$ as in~\secref{rem:pl.mult.d-skeleton}. The faces of $S_h(X)_d$ not containing $x$ form a closed subset $A$ of $X$. Let $V$ be any compact strictly analytic neighbourhood of $x$ with $V \cap A = \emptyset$. We have seen in~\secref{rem:pl.d-skeleton} that the map $f$ restricts to a surjective piecewise $\z$-linear map $S_h(X)_d \to h(X)_d$ which is finite to one. We conclude that the germ of $h(V)_d$ at $h(x)$ agrees with the germ $\bigcup h(\Delta)$ at $h(x)$ where $\Delta$ ranges over all $d$-dimensional faces of $S_h(X)_d$ containing $x$. Moreover, it follows from \eqref{tropical multiplicities as push-forward} that the tropical multiplicity $m_\sigma$ in a $d$-dimensional face of $h(V)_d$ containing $h(x)$ is given by $m_\sigma = \sum_\Delta m_\Delta$, where $\Delta$ ranges over all $d$-dimensional faces of $S_h(X)_d$ with $x \in \Delta$ and $h(\Delta)=\sigma$. All this does not depend on the particular choice of $V$.  We pick for $V(x)$	any compact strictly analytic neighbourhood of $x$ with $V(x) \cap A = \emptyset$. Then the claim follows by using this independence for $W \subset V(x)$.
\end{proof}

\begin{art}[Tropical variety of a germ] \label{germ of piecewise linear tropicalization}  \label{tropical cycle of a germ}
	Let $h\colon X\to\R^n$ be a PL tropicalization, let $x\in X$, and let $V(x)$ be a compact strictly analytic neighbourhood of $x$ as in Lemma~\ref{lem:pl.trop.germ.weight}.
	We define the \defi{tropical variety $h(X,x)$ of the germ $(X,x)$} as the germ of $h(V(x))$ at $h(x)$.  It is the germ of a fan centered at $h(x)$.
	We use the weights on $h(V(x))$ to endow the
	$d$-dimensional part $h(X,x)_d$ of the
	germ $h(X,x)$ with weights on its $d$-dimensional faces.  This makes $h(X,x)_d$ into a weighted fan.
\end{art}

\begin{rem} \label{local vs global weights}
	Let $h\colon X\to\R^n$ be a PL tropicalization.  It follows from the definitions that for any $\omega \in \R^n$, we have the identity
	$$(h(X)_d,\omega)= \sum_{\substack{x \in X\\ h(x)=\omega}} h(X,x)_d$$
	of weighted fans in $\omega$, well-defined up to subdivision.
\end{rem}

\begin{rem}[Weights on germs via the $d$-skeleton]\label{rem:pl.d-skeleton.germ.weights}
	Let $x\in X$.
	If $x\notin S_h(X)_d$ then $\dim h(X,x) < d$, so $h(X,x)_d=0$ as a weighted fan.  If $x\in S_h(X)_d$, then it follows from \eqref{tropical multiplicities as push-forward}  that $h(X,x)_d$ is the push-forward of the weighted germ of  $S_h(X)_d$ at $x$.
\end{rem}

For a piecewise $\z$-linear map $L$ of piecewise $\z$-linear spaces, we have a push-forward $L_*$ of weighted $\z$-polyhedral complexes, well-defined up to subdivision (see \secref{Polyhedral geometry}). Tropical varieties satisfy the following functoriality property with respect to such $L$.

\begin{prop}\label{functoriality from the right}
	Let $h\colon X\to\R^n$ be a PL tropicalization, let $L\colon \R^n \to \R^m$ be a piecewise $\z$-linear map, and let $h'\coloneq L \circ h\colon X \to \R^m$.  Then for every $x\in X$, we have the equality
	\begin{equation} \label{Sturmfels-Tevelev from right}
	L_*(h(X,x)_d)= h'(X,x)_d
	\end{equation}
	of weighted fans (up to subdivision).  Moreover, we have $L_*(h(X)_d) = h'(X)_d$ as weighted $\z$-polytopal complexes (up to subdivision).
\end{prop}

\begin{proof}
	The first statement is a special case of the second: by Lemma~\ref{lem:pl.trop.germ.weight}, the weighted fans $h(X,x)_d$ and $h'(X,x)_d$ can be computed by restricting to a small enough neighbourhood $V(x)$ of $x$.  After subdividing the source and target, we can assume that if $\sigma'$ is a $d$-dimensional face of $h'(X)$, then the inverse image in $h(X)$ of its relative interior consists of the relative interiors of finitely many $d$-dimensional faces $\sigma$ of $h(X)$.  Moreover, we can assume that $h$ is smooth over each $\relint(\sigma)$ and that $L$ is $\z$-linear on $\sigma$.  Translating on the source, we can even assume that $L$ is $\Z$-linear on $\sigma$.

	Let $x\in h\inv(\relint(\sigma)) \subset h'^{-1}(\relint(\sigma'))$ and let $U$ be a compact strictly analytic neighbourhood of $x$ such that $h|_U = \phi_{\trop}$ for a moment map $\phi\colon U\to\bGm^{n,\an}$.  Then $h'|_U = \phi'_{\trop}$ for the moment map $\phi'\colon U\to\bGm^{m,\an}$ obtained by composing with the homomorphism $\ell\colon \bGm^{n,\an}\to\bGm^{m,\an}$ lifting $L|_\sigma$.  Choose a generic quotient $Q'\colon\Z^m\to\Z^d$ with respect to $\sigma'$ in the sense of Definition~\ref{def:generic.quotient}, and let $q'\colon\bGm^{m,\an}\to\bGm^{d,\an}$ be the associated homomorphism of tori.  Let $q = q'\circ \ell\colon \bGm^{n,\an}\to\bGm^{d,\an}$, let $Q\colon\Z^n\to\Z^d$ be the map on cocharacter lattices, and let $\psi = q'\circ\phi' = q\circ\phi$.  Then Proposition~\ref{item:trop.mult.projection} implies
	\[ m_h(X,x) = \frac 1{[\Z^d:Q(N_\sigma)]}m_\psi(X,x) \qquad
	m_{h'}(X,x) = \frac 1{[\Z^d:Q'(N'_{\sigma'})]}m_\psi(X,x),
	\]
	where $N_\sigma$ (resp.\ $N'_{\sigma'}$) is the sublattice of $\Z^n$ (resp.\ $\Z^m$) underlying the linear span of $\sigma$ (resp.\ $\sigma'$).  Taking the quotient, we have
	\[ \frac{m_{h'}(X,x)}{m_h(X,x)} = \frac{[\Z^d:Q(N_\sigma)]}{[\Z^d:Q'(N'_{\sigma'})]}
	= [Q'(N'_{\sigma'}):Q(N_\sigma)] = [N'_{\sigma'}:L(N_\sigma)].
	\]
	We deduce that $m_{h'}(X,\sigma') = [N'_{\sigma'}:L(N_\sigma)]\,m_h(X,\sigma)$. Summing over all $d$-dimensional faces $\sigma$ of $h(X)$ mapping onto $\sigma'$, the lemma follows from the definition of the push-forward.
\end{proof}

\begin{art}[Tropicalization and Pushforward of Cycles] \label{Sturmfels-Tevelev}
	Recall from~\secref{sec:fund.cycle} that we have defined the fundamental cycle $[X]$ of an analytic space $X$.  More generally, a \defi{cycle} on $X$ is a finite sum of irreducible Zariski-closed subsets of $X$.  Let $h\colon X\to\R^n$ be a PL tropicalization.  By linearity, we define a weighted $\z$-polytopal complex $h(Z)$ for all cycles $Z$ on $X$ using Definition~\ref{def:pl.trop.mult.face} on prime cycles.  It follows from Lemma~\ref{lem:mult.fund.cycle} that $h([X]) = h(X)_d$.

	Given a proper morphism $f\colon X'\to X$ of strictly analytic spaces,  a proper push-forward of cycles was defined in~\cite[2.6]{gubler-crelle}. By linearity and by restriction to components, it is enough to define $f_*([X'])$ for a proper surjective morphism of irreducible and reduced compact strictly analytic spaces. If $\dim(X)<\dim(X')$, then we set  $f_*([X'])=0$.  Otherwise, we have $\dim(X)=\dim(X')$ and $f$ is finite over the complement of a Zariski closed subset of $X$ of dimension $<d$ and hence has a a well-defined degree $\delta$. Then we set  $f_*([X'])=\delta \, [X]$.
\end{art}

In the following, we discuss generalizations of the Sturmfels--Tevelev formula in tropical geometry. First, we give an analogue for $d$-skeletons.

\begin{prop} \label{rem:sturmfels.tevelev.d-skeleton}
	Let $f\colon X' \to X$ be a surjective proper morphism of $d$-dimensional compact, irreducible, reduced, good strictly analytic spaces.  Let  $\delta$ be the degree of $f$. Let $h \colon X \to \R^n$ be a PL tropicalization and let $h' \coloneqq h \circ f$. Recall from~\secref{morphisms and d-skeleton} that the restriction of $f$ gives a piecewise $\z$-linear map $F\colon S_{h'}(X')_d \to S_h(X)_d$. Then we have the identity
	\begin{equation} \label{Sturmfels-Tevelev for skeletons}
	F_*(S_{h'}(X')_d)= \delta \cdot S_h(X)_d
	\end{equation}
	of weighted abstract $\z$-polytopal complexes up to subdivision.
\end{prop}

\begin{proof}
	We may prove the statement locally in $x \in X$. Clearly, we may asume that $x \in \relint(\Delta)$ for a $d$-dimensional face $\Delta$ of $S_h(X)_d$.
	Since $X$ is reduced, we have seen in Remark~\ref{functoriality in reduced case} that the local ring $\sO_{X,x}$ is the field $\kappa(x)$ and it follows as in~\cite[2.4.3]{CLD} that $x$ has a compact strictly analytic neighbourhood $U$ so that $f\inv(U)\to U$ is finite and flat. Replacing $f$ by this restriction, we may assume that $f$ is a finite flat morphism. It follows from Remark~\ref{preimage of the d-skeleton} that $f^{-1}(S_h(X)_d)=S_{h'}(X')_d$.
	There are finitely many preimages $y_1,\dots,y_r \in X'$ of $x$ with respect to $f$ and we conclude that all $y_j \in S_{h'}(X')_d$.

	We choose compatible face structures on $S_{h'}(X')_d$, $S_h(X)_d$ and $h(X) = h'(X')$ as in~\secref{functoriality of weights}. We may assume that $\Delta$ is such a face of $S_h(X)_d$. It follows from the above that for any $y_j$ there is a unique face $\Delta_j$ of $S_{h'}(X')_d$ such that $f(\Delta_j)=\Delta$ and
	$y_j \in \relint(\Delta_j)$.
	For the degree $\delta$ of $f$,  using~\cite[3.2.5]{ducros14:structur_des_courbes_analytiq}, we deduce
	$\delta = \sum_{j=1}^r \dim_{\kappa(x)}(\sO_{X',y_j})$.
	Now we apply Remark~\ref{functoriality in reduced case} to get
	$$\delta \cdot m_\Delta  = \sum_{j=1}^r \dim_{\kappa(x)}(\sO_{X',y_j})\,m_\Delta = \sum_{j=1}^r m_{\Delta_j}$$
	and by definition of push-forward with respect to $F$ we get the claim.
\end{proof}

The following proposition generalizes the original formula given in~\cite{StTe} for trivial valuations and in~\cite{BPR} for a general non-archimededean field, both of which treat closed subschemes of a torus.

\begin{prop}[Sturmfels--Tevelev Formula]\label{prop:sturmfels.tevelev}
	Let $f\colon X'\to X$ be a proper morphism of compact, good, strictly $K$-analytic spaces.  Assume we are given a commutative square
	\[ \xymatrix @=.25in{
		{X'} \ar[r]^{h'} \ar[d]_f & {\R^m} \ar[d]^L \\
		X \ar[r]_h & {\R^n}
	} \]
	where $h$ and $h'$ are PL tropicalizations and $L$ is piecewise $\z$-linear.  Then for any cycle $Z'$ on $X'$, we have
	\begin{equation} \label{Sturmfels-Tevelev formula}
	h(f_*(Z')_d) = L_*(h'(Z')_d)
	\end{equation}
	as weighted $\z$-polytopal complexes (up to subdivision).
\end{prop}

\begin{proof}
	We may assume that $L$ is the identity map on $\R^n$ by Proposition~\ref{functoriality from the right}.  By linearity we may assume that $Z'=X'$, so that $X'$ is reduced and irreducible.  Let $d = \dim(X')$.  We may replace $X$ by $f(X')$ to assume $f$ is surjective and $X$ is reduced.  If $\dim(X) < \dim(X')$ then $f_*[X'] = 0$ and $\dim(h'(X')) < d$, so that $h'(X')_d = 0$ and the claim follows. 
	So we may suppose then that $\dim(X) = \dim(X')$ and we denote by $\delta$ the degree of $f$.
	Using the restriction $F\colon S_{h'}(X')_d \to S_h(X)_d$, Proposition~\ref{Sturmfels-Tevelev for skeletons} and  Remark~\ref{rem:pl.d-skeleton.germ.weights} twice, we get
	$$h'(X')_d= h_*'(S_{h'}(X')_d)=h_*\circ F_*(S_{h'}(X')_d)=h_*(\delta \cdot S_h(X)_d)=\delta \cdot h_*(S_h(X)_d)= \delta \cdot h_*(X)_d.$$
	Since the right hand side agrees with $h_*(f_*[X'])$, this proves   \eqref{Sturmfels-Tevelev formula}.
\end{proof}

Finally, we present a trick for writing a PL tropicalization as a telescoping sum.  This will be used later to reduce to the case of smooth tropicalization maps.

\begin{prop} \label{tropical decomposition}
	Let $X = \bigcup_{j\in J} X_j$ be a finite $\rG$-covering by compact strictly analytic subdomains $X_j$.  For any $I \subset J$, we get a strictly analytic subdomain $X_I \coloneqq \bigcap_{i \in I} X_i$.  Let $h\colon X\to\R^n$ be a PL tropicalization.  Then we have the following identity of weighted $\z$-polytopal complexes in $\R^n$ up to subdivision:
	$$\sum_{I \subset J} (-1)^{|I|}\, h(X_I)_d = 0.$$
\end{prop}

\begin{proof}
	There is a $\z$-polytopal complex $\Pi$ with support $h(X)$ such that $\partial X_I$ is contained in a $(d-1)$-dimensional subcomplex of $\Pi$ for all $I \subset J$. Let $\omega$ be in the relative interior of a $d$-dimensional face of $\Pi$.   By construction, if $x\in X_I$ and $h(x)\in\omega$ then $x \in \Int(X_I)$, so $h(X_I,x) = h(X,x)$. By Remark~\ref{local vs global weights}, for all $I \subset J$ we have the identity
	$$(h(X_I)_d,\omega) = \sum_{\substack{x \in X_I\\h(x)=\omega}} h(X_I,x)_d = \sum_{\substack{x \in X_I\\h(x)=\omega}} h(X,x)_d$$
	of germs of weighted $\z$-polytopal complexes in $\omega$. The inclusion-exclusion principle then gives
	$$\sum_{I \subset J} (-1)^{|I|}\, (h(X_I)_d,\omega) = 0.$$
\end{proof}

\section{The balancing condition} \label{section: balancing condition}

In this section, we prove the balancing condition for harmonic tropicalizations of compact good strictly analytic spaces over $K$.
We start with some polyhedral geometry using the notions from~\secref{Polyhedral geometry}. We fix a  subgroup $\Gamma$ of $\R$ which will be later the value group of $K$.

\begin{art}[The Balancing Condition] \label{recall balancing condition}
	We first recall the \emph{balancing condition} for a weighted $\z$-polyhedral complex $\Pi$ of dimension $d$ in $\R^m$, with weights $m_\sigma \in \Z$ for each $d$-dimensional face $\sigma\in \Pi$. For any face $\sigma \in \Pi$, let $\Linear_\sigma$ be the linear subspace of $\R^m$ generated by $\{\omega-\rho\mid \omega,\rho \in \sigma\}$. The lattice $N \coloneqq \Z^m$ induces a canonical lattice $N_\sigma \coloneqq \Linear_\sigma \cap N$ in $\Linear_\sigma$. For each face $\tau$ of $\sigma$ of codimension $1$, we choose  $\omega_{\sigma,\tau} \in N_\sigma$ such that $\omega_{\sigma,\tau}+N_\tau$ generates the lattice $N_\sigma/N_\tau$ of rank $1$. Note that the generator $\omega_{\sigma,\tau}+N_\tau$ is determined up to sign; we fix the sign by requiring that $\omega_{\sigma,\tau}$ points inwards at the boundary $\tau$ of $\sigma$. Then we say that $\Pi$ is \defi{balanced} at a face $\tau \in \Pi$ of dimension $d-1$ if we have
	$$\sum_{\sigma \succ \tau} m_\sigma \omega_{\sigma,\tau} \in \Linear_\tau,$$
	where $\sigma$ ranges over all $d$-dimensional faces of $\Pi$ containing $\tau$. We say that $\Pi$ is \defi{balanced} if it is balanced at every $(d-1)$-dimensional face $\tau \in \Pi$.
	In this case, the weighted polyhedral complex is called a \emph{tropical cycle of dimension $d$.}

\end{art}

In the following, we consider  a weighted $\z$-polyhedral complex $\Pi$ in $\R^m$ of dimension $d$.
Let $\phi \colon |\Pi| \to \R$ be a piecewise $\z$-linear function with tropical Weil divisor (corner locus)  $\divisor(\phi)$ (see~\cite[Section~6]{AllermannRau}). By subdividing $\Pi$, we may assume that the restriction of $\phi$ to every face $\sigma \in \Pi$ is $\z$-linear.
Let ${\Pi'} := \Gamma_{\Pi}(\phi)$ be the graph of $\phi$ with the same weights as $\Pi$.
We denote by $\iota \colon \Pi \to {\Pi'}$ the canonical isomorphism.

\begin{prop} \label{purely tropical statement}
	The graph ${\Pi'}$ is balanced if and only if $\Pi$ is balanced and $\divisor(\phi) = 0$.
\end{prop}

\begin{proof}
	We always denote by $\tau' \in \Pi'$ the face $\iota(\tau)$ corresponding to $\tau \in \Pi$ and vice versa.
	Let $\tau'$ be a codimension-$1$ face of $\Pi'$.
	By translation, we may assume that $\phi$ restricts to a linear function on $\tau$, i.e., that both $\tau'$ and $\tau$ contain the origin and $\phi(0) = 0$.
	Then ${\Pi'}$ is balanced at $\tau'$ if and only if
	\begin{equation} \label{eq balancing I}
	\sum_{\sigma' \succ \tau' } m_{\sigma'} \omega_{\sigma', \tau'} \in \Linear_{\tau'},
	\end{equation}
	where $\sigma'$ ranges over all $d$-dimensional faces of $\Pi'$ containing $\tau'$. Noting that $\omega_{\sigma',\tau'}= \left(\omega_{\sigma, \tau}, \frac{\del \phi}{\del \omega_{\sigma, \tau}}\right)$, equation~\eqref{eq balancing I} holds if and only if
	\begin{align*}
	\left( \sum_{\sigma \succ \tau} m_\sigma \omega_{\sigma, \tau},\; \sum_{\sigma\succ \tau} m_\sigma \frac{\del \phi} {\del \omega_{\sigma, \tau}} \right) \in \Linear_{\tau'},
	\end{align*}
	where $\sigma$ ranges over all $d$-dimensional faces of $\Pi$ containing $\tau$. We have  $(v, w) \in \Linear_{\tau'}$ if and only if $v \in \Linear_{\tau}$ and $w = \frac{\del \phi}{\del v}$.
	Thus we find that \eqref{eq balancing I} holds if and only if
	\begin{equation} \label{eq balancing II}
	\sum_{\sigma \succ \tau} m_\sigma \omega_{\sigma, \tau} \in \Linear_{\tau}
	\end{equation}
	and
	\begin{equation} \label{eq balancing III}
	\sum_{\sigma \succ \tau} m_\sigma \frac{\del \phi} {\del \omega_{\sigma, \tau}} = \frac {\del \phi}{\del(\sum_{\sigma \succ \tau} m_\sigma \omega_{\sigma, \tau})}.
	\end{equation}
        Now we observe that, by the definitions, equation~\eqref{eq balancing I} holds for all $\tau'$ if and only if ${\Pi'}$ is balanced,
	whereas~\eqref{eq balancing II} holds for all $\tau$ if and only if $\Pi$ is balanced, and~\eqref{eq balancing III} holds for all $\tau$ if and only if $\divisor(\phi) = 0$.
\end{proof}

\begin{lem} \label{compatibility of corner locus}
	Let $\Pi$ be a tropical cycle in $\R^m$ and let $\phi, \psi \colon \Pi \to \R$ be piecewise $\z$-linear functions. 
	Let ${\Pi'} = \Gamma_\Pi (\phi)$ be the graph of $\phi$ and let $\pi\colon {\Pi'} \to \Pi$ be  the canonical map.
	If the tropical Weil divisor $\divisor(\psi)$ is equal to zero, then $\divisor(\pi^*(\psi))=0$ on ${\Pi'}$.
\end{lem}
\begin{proof}
	Let $\tau'$ be a face of ${\Pi'}$ of codimension $1$, corresponding to a
	$\z$-polyhedron $\tau$ of $\Pi$.
	By definition of a tropical Weil divisor, the multiplicity of $\tau'$ in $\divisor(\pi^*\psi)$ is
	\begin{align*}
	\sum_{\sigma' \succ \tau'} m_{\sigma'} \frac{\del \pi^* \psi} {\del \omega_{\sigma', \tau'}}  -  \frac {\del\pi^* \psi}{\del(\sum_{\sigma' \succ \tau'} m_{\sigma'} \omega_{\sigma', \tau'})} =
	\sum_{\sigma \succ \tau} m_{\sigma} \frac{\del \psi} {\del \omega_{\sigma, \tau}} - \frac {\del \psi}{\del(\sum_{\sigma \succ \tau} m_\sigma \omega_{\sigma, \tau})}      =0
	\end{align*}
	where $\sigma'$ (resp.\ $\sigma$) is ranging over all $d$-dimensional faces of $\Pi'$ (resp.\ $\Pi$) containing $\tau'$ (resp.\ $\tau$). This proves the claim.
\end{proof}

\begin{cor} \label{linear covering of tropical cycle}
	Let $\Pi$ be a tropical cycle in $\R^m$ of  dimension $d$ and let $L\colon \Pi \to \R^n$ be a piecewise $\z$-linear map.  Suppose that $L(x)=(\phi_1(x),\dots,\phi_n(x))$ for piecewise $\z$-linear functions $\phi_i$ on $\Pi$ with $\divisor(\phi_i)=0$. Then the graph ${\Pi'} \coloneqq \Gamma_\Pi(L)$ is a tropical cycle in $\R^m \times \R^n$ of  dimension $d$, endowed with the weights given by transport of structure via the canonical isomorphism $\iota\colon \Pi \to {\Pi'}$. Moreover, we have $L=p_2 \circ \iota$ for the second projection $p_2\colon \R^m \times \R^n \to \R^n$.
\end{cor}

\begin{proof}
	This follows from an inductive application of Proposition~\ref{purely tropical statement} and Lemma~\ref{compatibility of corner locus}.
\end{proof}

Now we apply the above considerations to tropicalizations over a non-trivially valued non-Archimedean field $K$ with additive value group $\Gamma \subset \R$. In the following, we assume that $X$ is a compact good strictly analytic space of pure dimension $d$.

For a smooth tropicalization map $F\colon X \to \R^m$,
Chambert-Loir and Ducros have shown that that the $d$-dimensional part $F(X)_d$ of the tropical variety $F(X)$ is balanced outside $F(\partial X)$. We have here the following immediate consequence of their result:

\begin{lem} \label{balancing for smooth germs}
	Let $F\colon X \to \R^m$ be a smooth tropicalization map. For $x \in X \setminus \partial X$, the germ $F(X,x)_d$ from~\secref{germ of piecewise linear tropicalization}  is a balanced fan.
\end{lem}

\begin{proof}
	By Lemma~\ref{lem:pl.trop.germ.weight}, there is a compact strictly analytic neighbourhood $V(x)$ of $x$ such that the weighted germ $F(X,x)_d$ at $x$ agrees with the weighted germ of $F(V(x))_d$ at $F(x)$. By \cite{chambert_ducros12:forms_courants}, the latter is balanced, proving the claim.
\end{proof}

\begin{prop} \label{harmonic functions and smooth trop}
	Let $F\colon X \to \R^m$ be a smooth tropicalization map, let $x \in \Int(X)$,
	and let $\phi$ be a piecewise $\z$-linear function on $F(X,x)$ such that $h \coloneqq \phi \circ F$ is harmonic in a neighbourhood of $x$.
	Then the tropical Weil divisor $\divisor(\phi,x)$ of $\phi$ on the fan $F(X,x)_d$ is equal to zero.
\end{prop}

\begin{proof}
  By Corollary~\ref{Chern current of numerically trivial}, after potentially shrinking $X$, we have $\d'\d''[h]=0$ on $\Int(X)=X \setminus \partial X$. Since $\partial X$ is closed in $X$, it follows from Lemma~\ref{lem:pl.trop.germ.weight}
	that there is a strictly affinoid neighbourhood $W$ of $X$ contained in $\Int(X)$ such that the weighted germ of $F(W)_d$ in $F(x)$ agrees with $F(V,x)_d$. Clearly, we may also assume that $\phi$ is defined on $F(W)$ and that $h$ is harmonic on $W$.  Note that  $\d'\d''[h]=0$ on  $W \setminus \partial W$.
	We may assume that $x \in S(X,x)_d$; otherwise the weighted fan $F(X,x)_d=F(W,x)_d$ is $0$ (see Remark~\ref{rem:pl.d-skeleton.germ.weights}).
	Then $x$ is an Abhyankar point of $X$ as seen in~\secref{rem:pl.d-skeleton}.
	By \cite[Proposition~3.4.4]{CLD}, there is a strictly affinoid neighbourhood $U$ of $x$ contained in $\Int(W) \subset \Int(X)$ and a smooth tropicalization map $G_0\colon U \to \R^{n_0}$ with $G_0(\partial U)\cap G_0(\{x\})= \emptyset$.  Setting $n = m+n_0$ and $G = (F,G_0)\colon U\to\R^n$, we have $G(\partial U)\cap G(\{x\})= \emptyset$ and $F|_U=L \circ G$, where $L\colon\R^m\times\R^{n_0}\to\R^m$ is the projection onto the first factor.
	The piecewise $\z$-linear function $\psi \coloneqq \phi \circ L$ on $G(U)$ satisfies $\psi \circ G = h|_U$. For any Lagerberg form $\omega \in \cA^{d,d}(G(U))$ with compact support in the open neighbourhood $\Omega \coloneqq G(U)\setminus G(\partial U)$ of $G(x)$ in $G(U)$,
	we have the identity
	$$\label{current identity} \langle \omega, \d'\d''[\psi] \rangle = \langle G^*(\omega), \d'\d''[h] \rangle =0,$$
	and hence $\d'\d''[\psi|_\Omega] = 0$.
	Using that $x \not\in \partial U$,
	Theorem~\ref{thm:trop.dim.reduction} shows that the germ $G(U,x)$ is a fan of pure dimension $d$, so
	we may shrink the open neighbourhood $\Omega$ of $G(x)$ such that $\Omega$ generates the fan $G(U,x)$.
	Since $x \not\in \partial U$,  Lemma~\ref{balancing for smooth germs} yields that the fan $G(U,x)_d$ is balanced.
	The tropical Poincar\'e--Lelong formula (see~\secref{Lagerberg forms})   shows that $\d'\d''[\psi|_\Omega]$ is the current of integration  over the tropical Weil divisor $\divisor(\psi|_\Omega)$.  Using $\d'\d''[\psi|_\Omega] = 0$, we deduce that $\divisor(\psi|_\Omega)=0$, and hence the tropical divisor $\divisor(\psi,x)$ on the fan $G(U,x)=G(U,x)_d$ is zero.
	By Proposition~\ref{functoriality from the right}, we have $L_*(G(U,x)_d)=F(X,x)_d$ and hence the projection formula \cite[Proposition~7.7]{AllermannRau} gives
	$\divisor(\phi,x)= L_*(\divisor(\psi,x))=0$,
	which proves the claim.
\end{proof}

We will generalize the balancing result of Chambert-Loir and Ducros in Lemma~\ref{balancing for smooth germs} to a harmonic tropicalization map $h\colon X \to \R^n$. We recall from~\secref{tropical cycle of a germ} that  a germ $(X,x)$ gives rise to a  canonical weighted rational fan  $h(X,x)_d$ of pure dimension $d$. In Definition~\ref{def:pl.trop.mult.face}, we have seen that $h(X)_d$ is canonically a weighted $\z$-polytopal complex.

\begin{thm} \label{thm balancing}
	Let $X$ be a good compact strictly analytic space of pure dimension $d$, and let
        $h\colon X \to \R^n$ be a harmonic tropicalization map. Then we have the following properties:
	\begin{enumerate}
		\item \label{first condition in balancing}
		For any $x \in \Int(X)$, the weighted fan $h(X,x)_d$ of the germ $(X,x)$ is balanced.
		\item  \label{second condition in balancing}
		The weighted polytopal complex $h(X)_d$ satisfies the balancing condition in every $(d-1)$-dimensional face not contained in $h({S_h(X)_d\cap} \partial X)$.
	\end{enumerate}
\end{thm}
Note here that balanced might mean that there simply are no $d$-dimensional faces. {In a previous version, we required in 
\eqref{second condition in balancing} only that the $(d-1)$-dimensional face omits $h(\partial X)$. An argument of Andreas Mihatsch allows us to replace the latter by $h({S_h(X)_d\cap} \partial X)$ where $S_h(X)_d$ is the $d$-skeleton of $h$ from \ref{rem:pl.d-skeleton}.}

\begin{proof}
	Let $x \in \Int(X)$. By~\secref{smooth tropicalization covering}, there is a compact strictly analytic subdomain $W$ of $X$ which is a neighbourhood of $x$, a smooth tropicalization map $F\colon W \to \R^m$ and a piecewise $\z$-linear map $L\colon F(W) \to \R^n$ such that $h=L \circ F$.
	By Lemma~\ref{balancing for smooth germs}, the fan $F(W,x)_d$ is  balanced.

	Suppose that $L = (\phi_1,\ldots,\phi_n)$ for piecewise $\z$-linear functions $\phi_1,\ldots,\phi_n\colon F(W)\to\R$. Since $h$ is harmonic, Proposition~\ref{harmonic functions and smooth trop} shows that the tropical Weil divisor
	of $\phi_i$ is zero in a neighbourhood of $\omega \coloneqq F(x)$. We transfer $F(X,x)$ into $\R^m \times \R^n$ by the graph embedding $\iota$ of $L$. Then Corollary~\ref{linear covering of tropical cycle}
	shows that $\iota_*(F(X,x)_d)$ is a balanced fan. It follows from \eqref{Sturmfels-Tevelev from right} that
	$$(p_2)_*(\iota_*(F(X,x)_d)) = L_*(F(W,x)_d)=h(X,x)_d.$$
	Since $p_2$ is a \emph{linear} map, we conclude from~\cite[Proposition~2.25]{GKM} that $h(X,x)_d$ is balanced, proving the first claim.

	To prove the second claim, let $\tau$ be any $(d-1)$-dimensional face of $h(X)_d$ not contained in $h({S_h(X)_d\cap} \partial X)$.  {For $\omega \in \relint(\tau)$, we have to show that the weighted germ of $h(X)_d$ in $\omega$ is balanced. 
	By Remark~\ref{local vs global weights}, the weighted germ of $h(X)_d$ in $\omega$ is the sum of the weighted germs $h(X,x)_d$ with $x$ ranging over $h^{-1}(\omega)$. 
	Note that we choose always the face structure on $h(X)_d$ as in \ref{rem:pl.mult.d-skeleton} and hence $\omega\notin h(S_h(X)_d\cap\del X)$. If $x\in h\inv(\omega)$ then either $x\in\del X\setminus S_h(X)_d$, in which case $h(X,x)_d = 0$ by the definition of the $d$-skeleton in \ref{rem:pl.d-skeleton}, or $x\in S_h(X)_d\setminus\del X$, in which case $h(X,x)_d$ is balanced by the first claim.}
\end{proof}

\begin{rem} \label{thm: maximum principle for semipositive functions}
{In a first version of this paper, we have stated here that a harmonic $\R$-PL function satisfies the local maximum principle  on $ X \setminus \partial X$. We will shift this to a subsequent paper, where we prove this more generally for  semipositive $\R$-PL functions.}
	\end{rem}

\section{Differential forms} \label{section: differential forms}

In this section, we will generalize the smooth differential forms of Chambert-Loir and Ducros \cite{CLD} on a good strictly analytic space $X$ over the non-trivially valued non-Archimedean field $K$. The idea is to replace smooth tropicalization maps by harmonic tropicalization maps, thus yielding \emph{weakly smooth} forms. Many properties of smooth differential forms are also true for weakly smooth forms, but  we will see later that weakly smooth forms are better behaved for cohomological questions.

\begin{art}[Lagerberg forms] \label{smooth forms on polytopal set}
	Let $S$ be a piecewise $\Z$-linear space in $\R^n$, as in~\secref{Polyhedral geometry}. A typical example is the tropical variety of a compact strictly analytic space: see~\secref{sec:trop.is.polyhedral} in case of a smooth tropicalization map and~\secref{choice of polytopal structure} for the generalization to PL tropicalizations. In~\secref{Lagerberg forms}, we have recalled the sheaf $\cA^{\bullet,\bullet}$ of Lagerberg forms on $S$; they form a sheaf of differential bigraded $\R$-algebras with respect to differentials $\d'$ and $\d''$. For  a piecewise $\Z$-linear space $T$ in $\R^m$ and an affine $\Z$-linear map $F \colon \R^m \to \R^n$ with $F(T) \subset S$, there is a functorial pull-back $F^*\colon \cA_S \to F_*(\cA_T)$.
\end{art}

\begin{art}[Piecewise Smooth Forms] \label{piecewise smooth forms on polytopal set}
	Also important for us is the sheaf of \defi{piecewise smooth forms on the piecewise $\z$-linear space $S$}, which we denote by $\cAps$. Here, a piecewise smooth form $\alpha$ on an open subset $\Omega$ of $S$  is  given in a neighbourhood $\Omega_p$ of any point $p \in \Omega$ by a $\z$-polytopal complex $\Pi$ with support $S$ and $\alpha_\Delta \in \cA_\Delta(\Delta \cap \Omega_p)$ for $\Delta\in\Pi$ such that $\alpha_\Delta$ restricts to $\alpha_\tau$ for every face $\tau$ of $\Delta$. We identify two piecewise smooth forms $\alpha,\alpha'$ if for every $p \in \Omega$, we have $\alpha_\Delta=\alpha_{\Delta'}'$ on $\Omega_p \cap \Omega_{p}' \cap \Delta \cap \Delta'$ for every $\Delta \in \Pi$ and every $\Delta' \in \Pi'$ for local representations of $\alpha,\alpha'$ in $p$. We have natural differentials $\d_{\rm P}'$ and $\d_{\rm P}''$ on $\cAps$ called \defi{polyhedral derivatives}. They are induced by the differentials $\d',\d''$ on the smooth pieces of a piecewise smooth form.
	We get a bigraded sheaf $\cAps^{\bullet,\bullet}$ of differential $\R$-algebras on $S$ containing $\cA$ as a subsheaf. We refer to \cite[Definition~3.10]{gubler_kunneman:tropical_arakelov} for a similar construction using $\Z$-polyhedral complexes instead of $\z$-polytopal complexes.
	For a \emph{piecewise} $\z$-linear map $F \colon T \to S$ of $\z$-polytopal sets, there is a functorial pull-back $F^*\colon \cA_{S,\rm ps} \to F_*(\cA_{T,\rm ps})$.
\end{art}

\begin{defn}[Preforms] \label{piecewise smooth preforms}
	Let $U$ be an open subset of $X$ whose closure $\bar U$ is a compact strictly analytic domain.
	A \defi{weakly smooth preform of type $(p,q)$} on $U$ is given by a tuple $(h, \alpha)$, where $h \colon \bar U \to \R^n$ is a harmonic tropicalization map and $\alpha \in \cA_P^{p,q}(P)$ for $P \coloneqq h(\bar U)$.  Two tuples $(h, \alpha)$ and $(h', \alpha')$ are considered equal if $p_1^* \alpha$ and $p_2^* \alpha'$ agree on $(h \times h')(\bar U)$, where $p_1$ and $p_2$ are the projections of $\R^n \times \R^{n'}$ to the two factors.
\end{defn}

\begin{lem}\label{sec:ps.form.on.complex}
  For a weakly smooth preform $(h,\alpha)$ on $U$, we have $(h,\alpha) = 0$ if and only if $\alpha = 0$ in $\cA(h(\bar U))$.
\end{lem}

\begin{proof}
	One direction is clear, so suppose that $(h,\alpha) = 0$.  This means that there is a harmonic tropicalization map $h'\colon \bar U\to\R^{n'}$ such that $p_1^*\alpha = 0$ on $(h\times h')(\bar U)$.  Since $p_1(h\times h')(\bar U) = h(\bar U)$, we can cover $h(\bar U)$ by the images of polytopes in $(h\times h')(\bar U)$, so $p_1^*\alpha = 0$ implies $\alpha=0$.  See the proof of~\cite[Lemme~3.2.2]{CLD} for more details.
\end{proof}

\begin{art}[Restriction of preforms] \label{restriction of piecewise smooth preforms}
	Let $U\subset U'$ be open subsets of $X$ whose closures $\bar U\subset\bar U{}'$ are compact strictly analytic domains.  The \defi{restriction} of a weakly smooth preform $(h, \alpha)$ on $U'$ is the weakly smooth preform given by $(h|_{\bar U}, \alpha|_{h(\bar U)})$.  The weakly smooth preforms thus form a presheaf of bigraded $\R$-algebras on the category of open subsets whose closure is a compact strictly analytic domain.  Note that such open sets form a base for the analytic topology of our good strictly analytic space $X$.
\end{art}

\begin{art}[Wedge product and differentials]\label{sec:wedge.product}
	Let $U\subset X$ be an open subset whose closure $\bar U$ is a compact strictly analytic domain. Let $(h,\alpha)$ and $(h',\alpha')$ be weakly smooth preforms on $U$ of type $(p,q)$ and $(p',q')$, respectively.
	Then $g\coloneq (h,h')\colon \bar U\to\R^{n+n'}$ is a harmonic tropicalization map, and one checks that $(g,p_1^*\alpha) = (h,\alpha)$ and $(g,p_2^*\alpha') = (h',\alpha')$.  We define the wedge product $(h,\alpha)\wedge(h',\alpha')$ to be the weakly smooth preform $(g,p_1^*\alpha\wedge p_2^*\alpha')$ of type $(p+p',q+q')$.  This satisfies $(h',\alpha')\wedge(h,\alpha) = (-1)^{(p+p')(q+q')}(h,\alpha)\wedge(h',\alpha')$.

	We define differentials of preforms by $\d'(h,\alpha) = (h,\d'\alpha)$ and $\d''(h,\alpha) = (h,\d''\alpha)$.  The differentials satisfy the Leibniz rule
	\[ \d'\bigl((h,\alpha)\wedge(h',\alpha')\bigr) = \d'(h,\alpha)\wedge(h',\alpha') + (-1)^{(p+q)}(h,\alpha)\wedge\d''(h',\alpha'), \]
	and likewise for $\d''$.
	We conclude that the weakly smooth preforms form a presheaf of bigraded differential $\R$-algebras with respect to the differentials $\d'$ and $\d''$.
\end{art}

\begin{defn}[Weakly smooth forms] \label{piecewise smooth forms}
	The presheaf of weakly smooth preforms is defined on the open subsets of $X$ whose closure is a compact strictly analytic domain. We denote the associated sheaf on the underlying topological space of $X$ by $\cA_X$ or simply by $\cA$.   It is a sheaf of bigraded differential $\R$-algebras $\cA= \bigoplus_{p,q}\cA^{p,q}$.  It is also a sheaf of $\cA^{0,0}$-algebras.  The elements of $\cA$ are called \defi{weakly smooth forms}.  The elements
	$\cA^{0,0}$ are well-defined real functions which we call \defi{weakly smooth functions}.
\end{defn}

\begin{art}[Pullback of Forms] \label{functoriality}
	For a morphism $f\colon X'\to X$ of good strictly analytic spaces, there is a \defi{pull-back} homomorphism $f^*\colon\cA_X \to f_*\cA_{X'}$ of sheaves of bigraded differential $\R$-algebras, defined on preforms as follows: if $U\subset X$ and $U'\subset f\inv(U)$ are open subsets whose closures are compact strictly analytic domains, and if $(h,\alpha)$ is a weakly smooth preform on $U$, then $f^*(h,\alpha) \coloneq (h\circ f,\alpha)$.
\end{art}

Let $(h,\alpha)$ be a weakly smooth preform on a compact good strictly analytic space $X$ over $K$. We denote the associated weakly smooth form by $h^*(\alpha)$.

\begin{prop} \label{harmonic lift to weakly smooth forms}
Let $X$ be a compact good strictly analytic space $X$ over $K$ and let $h\colon X\to\R^n$ be a harmonic tropicalization map.
\begin{enumerate}
	\item \label{preforms to forms}
	The natural map from weakly smooth preforms to weakly smooth forms is injective.
	\item \label{injectivity of lift}
	The pullback $h^*\colon \cA^{\bullet,\bullet}(h(X))\to \cA^{\bullet,\bullet}(X)$ is an injective homomorphism of bigraded differential $\R$-algebras.
	\item \label{compatibility of pull-backs}
	If $f \colon X' \to X$ is a morphism of good strictly analytic spaces, then $(h \circ f)^*=f^* \circ h^*$.
\end{enumerate}
\end{prop}

\begin{proof}
It is clear that \eqref{preforms to forms} follows from \eqref{injectivity of lift} and Lemma~\ref{sec:ps.form.on.complex}, and \eqref{compatibility of pull-backs} is obvious from the definitions as $h\circ f$ is a harmonic tropicalization map by Remark~\ref{remark on tropicalization maps}. It remains to show \eqref{injectivity of lift}. Obviously, the map $h^*$ is a homomorphism of bigraded differential $\R$-algebras. To show the crucial injectivity, let us pick $\alpha \in  \cA^{\bullet,\bullet}(h(X))$ with $h^*(\alpha)=0$.
Using that the presheaf of weakly smooth forms and the sheaf of weakly smooth forms have the same stalks,
every $x \in X$ has a compact strictly affinoid neighbourhood $W$ such that the preform $(h|_W,\alpha|_{h(W)})$ is zero. By Lemma~\ref{sec:ps.form.on.complex}, we have $\alpha|_{h(W)}=0$. Using that $X$ is compact, we see that finitely many such $W$ cover $X$ and hence the corresponding $h(W)$ cover $h(X)$. We conclude that $\alpha=0$, proving injectivity in \eqref{injectivity of lift}.  
\end{proof}

\begin{rem}[Smooth forms] \label{smooth forms on X}
  In the above, if we replace harmonic tropicalization maps by smooth tropicalization maps, then we get the \defi{smooth differential forms} introduced by Chambert-Loir and Ducros (see below for a proof). They give rise to a sheaf $\cAsm=\bigoplus_{p,q}\cAsm^{p,q}$ of bigraded differential $\R$-algebras with respect to differentials $\d'$ and $\d''$. The elements of $\cAsm^{0,0}$ are called \defi{smooth functions}. Again, we have a functorial pull-back with respect to morphisms of analytic spaces.  

  Now we show that $\cAsm$ agrees with the sheaf $\cA_{\rm CD}$ of smooth forms defined in \cite{chambert_ducros12:forms_courants}. Let $U\subset X$ be an open set whose closure $\bar U$ is a compact strictly analytic domain, and let $(h,\alpha)$ be a smooth preform on $U$ as defined in~\secref{piecewise smooth preforms} with respect to a smooth tropicalization map $h \colon \bar U \to \R^n$.  Choose a moment map $\phi\colon\bar U\to\bGm^{n,\an}$ such that $h = \phi_{\trop}$.  Then $(\phi|_U, h(\bar U))$ is a tropical chart in the sense of~\cite[3.1.1]{CLD}, so $\alpha\in\cA(h(\bar U))$ defines an element of $\cA_{\rm CD}(U)$, which does not depend on the choice of~$\phi$ by Lemme~3.1.10 of \textit{loc.\ cit.}  Sheafifying, we get a homomorphism of sheaves $\cA_{\rm sm}\to\cA_{\rm CD}$.  This homomorphism is surjective by~\S3.1.9 of \textit{loc.\ cit.}, and it is injective by~\secref{harmonic lift to weakly smooth forms} and Lemme~3.2.2 of \textit{loc.\ cit.}
\end{rem}

\begin{rem}[Piecewise smooth forms] \label{piecewise smooth forms on X}
  In the above, if we replace harmonic tropicalization maps by smooth tropicalization maps and the sheaves $\cA_S$ of smooth Lagerberg forms on the piecewise $\Z$-linear spaces $S$  by the sheaves $\cA_{S,{\rm ps}}$ of piecewise smooth Lagerberg forms (see~\secref{piecewise smooth forms on polytopal set}), then we get a variant of the \defi{piecewise smooth differential forms} introduced in \cite[Section~9]{gubler_kunneman:tropical_arakelov}. They give rise to a sheaf $\cAps=\bigoplus_{p,q}\cAps^{p,q}$ of bigraded differential  $\R$-algebras with respect to differentials $\d_{\rm P}'$ and $\d_{\rm P}''$. The latter are called polyhedral derivatives; they do not coincide in general with the differentials of the associated currents (this is a consequence of the tropical Poincar\'e--Lelong formula~\cite[Theorem~0.1]{GK}). The elements of $\cAps^{0,0}$ are called \defi{piecewise smooth functions}. Again, we have a  functorial pull-back with respect to morphisms of analytic spaces.
\end{rem}

\begin{rem} \label{injectivity lemma for smooth and ps}
	Note that Lemma~\ref{sec:ps.form.on.complex} and Proposition~\ref{harmonic lift to weakly smooth forms} hold also for smooth and for piecewise smooth preforms. Indeed, the smooth case is just~\cite[Lemme~3.2.2]{CLD}, and one readily checks that the argument works also in the piecewise smooth case.
\end{rem}

\begin{art}[Relations between the forms] \label{relations between forms}
	By Remark~\ref{remark on tropicalization maps}, every smooth tropicalization map is harmonic, so we get a homomorphism $\cAsm \to \cA$ of sheaves of bigraded differential $\R$-algebras.  We also have a canonical homomorphism $\cA \to \cAps$   of sheaves of bigraded differential $\R$-algebras, defined as follows.  Let $(h,\alpha)$ be a weakly smooth preform on an open subset $U$ of $X$  whose closure $\bar U$ is a compact strictly analytic domain.  We proved in~\secref{smooth tropicalization covering} that for any $x \in U$, there is a open neighbourhood $U_x$ of $x$ whose closure is a compact strictly analytic domain contained in $U$, a smooth tropicalization map $h_x$ on $\bar U_x$, and a piecewise $\z$-linear map $F\colon h'(\bar U_x) \to h(\bar U)$ such that $h=F \circ h_x$ on $\bar U_x$. Since $\alpha$ is a smooth Lagerberg form on $h(\bar U)$, we conclude that $\alpha_x \coloneqq F^*(\alpha)$ is a piecewise smooth form on $h'(\bar U_x)$. One checks that the piecewise smooth preforms $(h_x,\alpha_x)$ glue to a well-defined piecewise smooth form on $U$, giving rise to the desired canonical homomorphism  $\cA \to \cAps$.
\end{art}

\begin{prop} \label{subsheaves of forms}
  The above canonical homomorphisms are injective: that is, we have inclusions of sheaves of bigraded differential $\R$-algebras
  \[ \cAsm \inject \cA \inject \cAps. \]
\end{prop}

\begin{proof}
  The map $\cAsm \to \cAps$ is obviously injective and factors through $\cA \to \cAps$, so it is enough to show that the homomorphism $\cA \to \cAps$ is injective. Let $U$ be an open subset of $X$ and let $\omega$ be weakly smooth form on $U$ such that the associated piecewise smooth form constructed in~\secref{relations between forms} is zero. Arguing locally, we may assume that $\bar U$ is a compact strictly analytic domain and that $\omega$ is given on $U$ by a weakly smooth preform $(h,\alpha)$, where $h \colon \bar U \to \R^n$ is a harmonic tropicalization and $\alpha \in \cA(h(\bar U))$. Replacing $U$ by $U_x$ from~\secref{relations between forms}, we may assume that $h$ is covered by a smooth tropicalization map $h' \colon \bar U \to \R^m$, i.e.,~there is a piecewise $\z$-linear map $F \colon h'(\bar U) \to h(\bar U)$ with $h= F \circ h'$, and that the piecewise smooth form $(h')^*(F^*\alpha)$ on $\bar U$ is zero.  Since $\cAps(h'(\bar U))\to\cAps(\bar U)$ is injective, we have $F^*\alpha = 0.$
        We conclude that there is a finite covering of $h'(\bar U)$ by $\z$-polytopes $\Delta'$ maped $\z$-linearly by $F$ onto polytopes $\Delta$ covering $h(\bar U)$. Using $F^*\alpha=0$, we conclude as in the proof of Lemma~\ref{sec:ps.form.on.complex}  that $\alpha=0$ on $h(\bar U)$. This proves $\omega=0$, showing injectivity.
\end{proof}

\begin{rem} \label{piecewise smooth and G-topology}
  The above result shows that weakly smooth forms are examples of piecewise smooth forms. We introduced the latter as they form a natural class for defining integrals in the next section. As an example of a piecewise smooth function that is not smooth, take any piecewise smooth function $\varphi\colon \R \to \R$ that is not smooth; then $\varphi \circ \trop$ is piecewise smooth but not smooth on $\T^\an$. A more sophisticated example of a function which is harmonic but not smooth is given in Proposition \ref{non-smooth T_i}{\tiny }.
 
 The piecewise smooth forms on $X$ can be characterized  as follows.  A piecewise smooth form $\omega$ is defined by a $\rG$-covering $(V_i)_{i \in I}$ by compact strictly analytic domains $V_i$ and smooth forms $\omega_i$ on $V_i$ such that $\omega_i|_{V_i \cap V_j} = \omega_j|_{V_i \cap V_j}$ for all $i,j \in I$.  This is proved as for piecewise linear functions in Proposition~\ref{local description of piecewise linear}.
\end{rem}

\begin{art}[Supports] \label{support}
	The \emph{support} of $\omega	\in \cA^{p,q}(X)$ is the set of points of $X$  such that the germ of $\omega$ in $x$ is non-zero. This is a closed subset which we denote by $\supp(\omega)$. We denote by $\cA_c^{p,q}$ the cosheaf of weakly smooth $(p,q)$-forms with compact support on $X$.

	If $x\in X$ satisfies $\trdeg(\td\sH(x)^\bullet/\td K^\bullet)<\max(p,q)$, then by~\cite[Lemme~3.2.5]{CLD} we have $x\notin\supp(\omega)$ .  This applies in particular to any point $x$ contained in a Zariski-closed subset of dimension less than $\max(p,q)$.
\end{art}

Recall that the elements of $\cA^{0,0}$ are called weakly smooth functions. It is easy to see that they are continuous real functions.

\begin{prop}  \label{weakly smooth harmonic functions}
	Let $f\colon X \to \R$ be a  function.  The following are equivalent:
	\begin{enumerate}
		\item $f$ is weakly smooth and $\d'\d''f = 0$.
		\item $f$ is harmonic (and hence  $\R$-PL).
	\end{enumerate}
\end{prop}

\begin{proof}
  Assume first that $f$ is harmonic (hence $\R$-PL).  By Proposition~\ref{prop: equivalence of harmonicity definitions}, we may shrink $X$ to assume that $f$ is an $\R$-linear combination of $\Z$-harmonic functions.  By linearity, we may assume that $f$ is itself  $\Z$-harmonic.  Then we can use $f$ as a harmonic tropicalization. Using $f= \id_\R \circ f$, we deduce that $f \in \cA^{0,0}(X)$ and $\d'\d''f=f^*(\d'\d''\id_\R)=0$.

	Conversely, let $f$ be a weakly smooth function with $\d'\d''f=0$. We have to show that $f$ is a harmonic function. This can be checked locally at $x \in X$, so we may assume that $X$ is a strictly affinoid domain and that $f$ is the composition of a harmonic tropicalization $h = (h_1,\ldots,h_n)\colon X\to\R^n$ with a smooth function $g\colon\R^n\to\R$. We use the Taylor expansion in $\nu \coloneqq h(x)$ to write
	$$g(\omega)=g(\nu)+a_1\omega_1+ \dots + a_n \omega_n + O(\|\omega-\nu\|^2)$$
	for $\omega = (\omega_1,\ldots,\omega_n)$ in a neighbourhood of $\nu$. Using injectivity of $h^*$ from Proposition~\ref{harmonic lift to weakly smooth forms}, we deduce from $\d'\d''h=0$ that $\d'\d''(g|_{h(X)})=0$, and hence that $g$ is affine on every face of the tropical variety $h(X)$. Thus the restrictions of $g$ and of $g(\nu)+a_1\omega_1+ \dots + a_n \omega_n$ to $h(X)$ agree in a neighbourhood of $\nu$. We conclude that $f=g(\nu)+a_1h_1+ \cdots+ a_nh_n$ in a neighbourhood of $x$, and therefore $f$ is a harmonic function.
\end{proof}

\section{Integration}   \label{section: integration and currents}

In this section, $X$ is a good strictly analytic space of dimension $d$ over the non-trivially valued non-Archimedean field $K$.   First, we recall integration of a compactly supported form $\omega \in\cAps^{d,d}(X)$ over $X$ (resp.~the boundary integral of a compactly supported form $\eta$ in  $\cAps^{d-1,d}(X)$ or in $\cAps^{d,d-1}(X)$) from~\cite{CLD}.

\begin{prop} \label{integration of piecewise smooth forms}
For every good strictly analytic space $X$ of pure dimension $d$ over $K$, there is a unique linear functional $\int_X \colon \cA_{{\rm ps},c}^{d,d} \to \R$ with the following properties for $\omega \in \cA_{{\rm ps},c}^{d,d}(X)$:
\begin{enumerate}
	\item \label{restriction} If $\supp(\omega)\subset W$  for a strictly analytic  domain $W$ of $X$, then $\int_W \omega = \int_X \omega$.
	\item \label{inclusion exclusion} For  closed analytic domains $V,W$ of $X$, we have $\int_{V\cup W} \omega = \int_{V} \omega + \int_{W} \omega - \int_{V \cap W} \omega$.
	\item \label{integration of preform} If $X$ is compact with PL-tropicalization $h \colon X \to \R^n$ such that $\omega = h^*(\alpha)$ for $\alpha \in \cAps^{d,d}(h(X))$, then $\int_X \omega = \int_{h(X)_d} \alpha$.
\end{enumerate}
\end{prop}

In \eqref{integration of preform}, we see $h(X)_d$ as a weighted $\z$-polytopal complex as in Definition~\ref{def:pl.trop.mult.face} and use integration   for weighted $\z$-polytopal complexes~\cite[Sections~2--3]{Gubler}. For supports, we have  $h(\supp \omega)= \supp(\alpha)$  \cite[Corollaire~3.2.3]{CLD} and hence $\supp(\alpha)$ is compact.

\begin{proof}
We prove first uniqueness. Since $\supp(\omega)$ is compact, there is a finite covering $(X_i)_{i \in I}$ of $\supp(\omega)$ by compact strictly analytic domains $X_i$ of $X$ such that $\omega|_{X_i}= h_i^*(\alpha_i)$ for some $\alpha_i \in \cA_{{\rm ps}}^{d,d}(h(X_i))$ with respect  to a PL-tropicalization $h_i \colon X_i \to \R^{n_i}$. From \eqref{restriction} and \eqref{inclusion exclusion}, we deduce the inclusion-exclusion formula
\begin{equation} \label{general inclusion exclusion formula}
\int_X \omega = \sum_{ \emptyset \not = J \subset I} (-1)^{|J|+1} \int_{X_J} \omega
\end{equation}
for $X_J \coloneqq \bigcap_{j \in J} X_j$. Using \eqref{integration of preform} to compute $\int_{X_J} \omega$, we deduce uniqueness.

For existence, we note first that such an integration theory with properties \eqref{restriction}--\eqref{integration of preform} was set up for smooth forms and smooth tropicalization maps by Chambert-Loir and Ducros \cite[3.7--3.11]{CLD}. We deduce from it an integration theory for piecewise smooth forms and PL tropicalization maps as follows: For $\omega \in \cA_{{\rm ps},c}^{d,d}(X)$, there is a finite covering  $(X_i)_{i \in I}$ of $\supp(\omega)$ by compact strictly analytic domains $X_i$ of $X$ such that $\omega|_{X_i}$ is a smooth form for all $i \in I$ (see Remark~\ref{piecewise smooth and G-topology}). Then we use the inclusion exclusion formula \eqref{general inclusion exclusion formula} to define $\int_X \omega$  based on the smooth case. It is elementary to deduce from the smooth case that this is well-defined and that properties \eqref{restriction}--\eqref{integration of preform} hold in general.
\end{proof}

\begin{rem} \label{boundary integrals of piecewise smooth forms}
Similarly, we can define boundary integrals $\int_{\partial X} \eta$ for compactly supported $\eta$ in $\cAps^{d-1,d}(X)$ or in $\cAps^{d,d-1}(X)$. We have to replace all the integrals in Proposition~\ref{integration of piecewise smooth forms} by boundary integrals, $\omega$ by $\eta$ and $\alpha$ by a suitable form $\beta$ in $\cAps^{d-1,d}(h(X))$ or in $\cAps^{d,d-1}(h(X))$, respectively. We leave the details to the reader.
\end{rem}

For the \emph{theorem of Stokes}, we have the following 
analogue of \cite[3.12]{CLD}.  Note that for the last claim, we need the Balancing Theorem~\ref{thm balancing}.%
\footnote{Using that piecewise smooth forms are given $\rG$-locally by smooth forms (see Remark \ref{piecewise smooth and G-topology}) and using  Proposition \ref{integration of piecewise smooth forms}(\ref{inclusion exclusion}), the first part of the theorem is true more generally for piecewise smooth forms using polyhedral deivatives, but then the last claim is no longer true in the boundaryless case, see Example \ref{illustration for contributions outside the boundary}.}  

\begin{thm} \label{theorem of Stokes}
	Let $\omega \in \cA^{d-1,d}_c(X)$ and $\eta \in \cA^{d,d-1}_c(X)$.
	Then we have
	\begin{align*}
	\int_{X} \d' \omega = \int_{\partial X} \omega \quad \text{and} \quad \int_{X} \d'' \eta = \int_{\partial X} \eta.
	\end{align*}
In particular, if $X$ has no boundary, then the above integrals are zero.
\end{thm}
\begin{proof}
	By the way integration is set up,
	it is sufficient to show this when $X$ is compact with a harmonic tropicalization
	map $h \colon X \to \R^n$ and with $\omega = h^*(\alpha)$ for some $\alpha \in \cA^{d-1,d}(h(X))$.
	By the tropical theorem of Stokes in \cite[Proposition~3.5]{Gubler}, we get
	\begin{align*}
	\int_X \omega   =  \int_{h(X)_d} {\d'} \alpha = \int_{\partial h(X)_d} 
	{\alpha} = \int_{\partial X} \omega
	\end{align*}
	 proving the first formula; the second is similar. Assume that $\partial X = \emptyset$.  Then $h(X)_d$ is balanced by Theorem~\ref{thm balancing}, and hence $\int_{h(X)_d}\d'\alpha =0$ by \cite[Proposition~3.8]{Gubler}.
\end{proof}

\begin{art}[Lagerberg Involution] \label{Lagerberg involution}
	The \emph{Lagerberg involution} is the unique involution $J$ of the sheaf  $\cA$ of $\R$-algebras which leaves the smooth functions fixed and which satisfies $J\d'=\d''J$. It maps $\cA^{p,q}$ isomorphically onto $\cA^{q,p}$. The Lagerberg involution is the reason that the differentials $\d'$ and $\d''$ is play a completely symmetric role in non-Archimedean geometry. We call $\omega \in \cA^{p,p}(X)$ \defi{symmetric} if $J\omega=(-1)^{p}\omega$.
\end{art}

\begin{thm}[Green's formula] \label{theorem of Green}
	For symmetric forms $\omega \in \cA_c^{p,p}(X)$ and $\eta \in \cA_c^{q,q}(X)$ with $p+q=d-1$,  we have
	$$\int_X \bigl( \omega \wedge \d'\d'' \eta - \d'\d'' \omega \wedge \eta \bigr) = \int_{\partial X} \bigl( \omega \wedge \d''\eta - \d''\omega \wedge \eta \bigr).$$
\end{thm}

\begin{proof}
Similarly as in \cite[Lemme~1.3.8]{CLD}, this follows from the theorem of Stokes above and the Leibniz formula.
\end{proof}

\begin{art}[Forms as Currents] \label{forms as currents}
	Assume that $X$ is  Hausdorff and that $\partial X =\emptyset$. Then the space $\cDsm(X)$ of currents on $X$ is defined and studied in \cite[Section~4]{CLD}. For $\omega \in \cA_{\rm ps}(X)$, there is an \emph{associated current} $[\omega] \in \cDsm(X)$ defined by
	$$\langle \eta, [\omega] \rangle \coloneqq \int_X \omega \wedge \eta$$
	for any  $\eta \in \cA_{{\rm sm},c}(X)$. We claim that the induced linear map $\cA_{\rm ps}(X) \to \cDsm(X)$ is injective. Indeed, it is elementary to check that the corresponding map from the space of piecewise smooth Lagerberg forms to the space of Lagerberg currents on a piecewise $\z$-linear space in $\R^n$ is injective. Then the  claim follows from Remark~\ref{injectivity lemma for smooth and ps}.

	The \emph{$\d'$-residue of $\omega$} is defined as $\d'[\omega]-[\d'\omega]$ and similarly we define the \emph{$\d''$-residue of $\omega$}. The theorem of Stokes shows that the $\d'$- and $\d''$-residues of $\omega \in \cA^{d-1,d-1}(X)$ are zero. In this sense, weakly smooth forms behave like smooth differential forms.
\end{art}

We illustrate the theory of integration in a simple example explaining why it is natural to consider integration of piecewise smooth forms.  
Another reason to work in the piecewise smooth setting is that harmonic functions and harmonic  tropicalizations are piecewise smooth, but not necessarily smooth as we will see in Remark~\ref{harmonic non-smooth functions}.

\begin{eg} \label{illustration for integration}
  Let us consider the split torus $\T=\bGm^r$ with character lattice $M=\Z^r$. 
  We denote by $u_1,\dots, u_r$ the standard basis of $M$. 
  We choose a $(\Z,\Gamma)$-polytope $\Delta$ in $N_\R=\R^r$, so that $X \coloneqq \trop^{-1}(\Delta)$ is a compact strictly analytic domain in $\T^\an$.  Let $\phi\colon X\inject\T^\an$ be the inclusion.  We want to describe 
  integration of $\omega = \phi_{\trop}^*(d'u_1\wedge d''u_1\wedge \dots \wedge d'u_r\wedge d''u_r) \in \cA_{{\rm sm}}^{r,r}(X)$. It follows from \cite[Proposition 2.3.6]{bgjk2021} that $\omega$ is supported in the skeleton $S(\T)$. We identify the latter with $\R^r$ using the canonical section of $\trop$ sending $u \in \R^r$ to the weighted Gauss point with weight $u$. By \eqref{integration of preform} in Proposition \ref{integration of piecewise smooth forms}, we have
  $$ \int_X \omega = \int_\Delta \d'u_1\wedge \d''u_1\wedge \dots \wedge \d'u_r\wedge \d''u_r = \int_\Delta \d u_1 \wedge \dots \wedge \d u_r,$$
  where on the right we use the usual exterior derivative and integral from analysis. More generally, let $g \colon \R^n \to \R$ be a smooth function and let $f \coloneqq g \circ \phi_{\trop}$.  Then $f \omega \in \cA_{{\rm sm}}^{r,r}(X)$, and the same argument gives 
  \begin{equation} \label{integral formula for illustration}
    \int_X f\omega = \int_\Delta g \, \d'u_1\wedge \d''u_1\wedge \dots \wedge \d'u_r\wedge \d''u_r = \int_\Delta g \,\d u_1 \wedge \dots \wedge \d u_r.
  \end{equation}

  Most smooth functions $f\colon X \to \R$ do not arise in this way.  In general, the restriction of a smooth function $f$ to $S(\T)\cap X = \Delta$ is only piecewise smooth.  Indeed, we have seen in Remark~\ref{piecewise smooth and G-topology} that there is a $\rG$-covering of $X$ by compact strictly analytic domains $V_i$ such that $f|_{V_i}$ is the pullback of a smooth function under a moment map.  Ducros' theory of skeletons shows that every $V_i \cap S(\T)$ is a piecewise $\q$-linear subspace of $S(\T)$---see Remark \ref{rem: piecewise linear structure on tropical skeleton}---which readily yields that $f|_{\Delta}$ is a piecewise smooth function.  Even if a smooth function $f\colon X\to\R$ factorizes through  $\phi_{\trop}\colon X \to \R^r$, which means that there is a function $g\colon \Delta \to \R$ with $f=g \circ \trop$, then $g$ is not necessarily smooth; it is only piecewise smooth, as we have seen above. An example of such a function is $f \coloneqq g \circ \varphitrop$ for $g\coloneqq \min(u_1,0)$; see~\cite[Example~2.3.5]{bgjk2021}. 

  In any case, if $f$ is any piecewise smooth function on $X$, then let $g \coloneqq f|_{\Delta}$; 
  we claim that $\int_X f\omega = \int_\Delta g\,\d u_1\wedge\cdots\wedge\d u_r$, as in~\eqref{integral formula for illustration}.  
  Indeed, there are finitely many compact strictly affinoid domains $V_i$ of $X$ such that $f|_{V_i}$ is smooth (see Remark \ref{piecewise smooth and G-topology}). We may assume that there is a smooth tropicalization map $\psi_{\trop}\colon V_i \to \R^m$ and a smooth function $F\colon \R^{m} \to \R$ with $f|_{V_i} = F \circ \psi_{\trop}$.   By \cite[Example 4.7, (4.4.2)]{ducros12:squelettes_modeles}, we have that $V_i\cap S(\T)$ is a piecewise $\q$-linear subspace of $S(\T)$ and $\psi_{\trop}$ induces a piecewise $\q$-linear map  $V_i\cap S(\T) \to \R^m$. We conclude that $\Delta$ can be subdivided into $\z$-polytopes $\Delta_i$ such that $g|_{\Delta_i}=f|_{\Delta_i}$ is a smooth function on $\Delta_i$ and hence the formula~\eqref{integral formula for illustration} follows from the smooth case, using that $\omega$ is supported in $S(\T)$.

  Now assume again that $f$ is a smooth function on $X$ such that $f=g \circ \varphitrop$ for a  function $g\colon \Delta \to \R$. We want to describe the boundary integral $\int_{\partial X} f \eta$, where $\eta=\phi_{\trop}^*(\alpha)$ for any $\alpha \in \cA^{r-1,r}(N_\R)$. We showed above that $g$ is piecewise smooth on $\Delta$, i.e., that 
  there is a subdivision of $\Delta$ into $\z$-polyhedra $\Delta_i$ such that $g|_{\Delta_i}$ is smooth. 
  There is a unique form $\alpha' \in \cA^{r-1,0}(N_\R)$ such that $\alpha=(-1)^{r(r-1)/2} \alpha'\wedge d''u_1\wedge \dots \wedge d''u_r$. The boundary integral is defined such that
  $$\int_{\partial X} f\eta = \sum_i\int_{\partial \Delta_i} g \alpha = \sum_i\int_{\partial \Delta_i} g\alpha'$$
  where on the right, we have the usual boundary integral of an $(r-1)$-form from analysis. In this way, the theorem of Stokes follows from the usual Stokes' formula by
  \begin{equation*}
    \begin{split}
      \int_X \d'(f\eta)=\int_X \phi_{\trop}^*(\d'(g \alpha))=(-1)^{\frac{r(r-1)}{2}}\sum_i\int_{\Delta_i} \d'(g \alpha)\wedge \d''u_1 \wedge \dots \wedge \d''u_r\\
      =  \sum_i\int_{\Delta_i} \d'(g \alpha')
      =  \sum_i \int_{\partial \Delta_i}g \alpha'= \int_{\partial X} f\eta, 
    \end{split}
  \end{equation*}
  where we identify $\alpha'\in \cA^{r-1,0}(N_\R)$ in the second line with a usual real $(r-1)$-form and $\d'$ with the usual exterior derivative $\d$. 

  Finally, it is important to note that the boundary integral cannot be interpreted as an integral over the boundary $\del X$.  It is not hard to see that $\del X$ is contained in $\phi_{\trop}\inv(\del\Delta)$.  The boundary  $\del X$ can be seen as an $(r-1)$-dimensional analytic subvariety over a huge base extension (the completed residue field over some Gauss point, see \cite[Lemme 3.1]{ducros12:squelettes_modeles}), but the boundary integral $\int_{\partial X} f \eta$ does not make sense as an integral over this  $(r-1)$-dimensional analytic subvariety as the form $f \eta$ has bidegree $(r-1,r)$ instead of $(r-1,r-1)$. 

\end{eg}

\begin{eg} \label{illustration for contributions outside the boundary}
   Consider the split torus $\T=\bGm^2$, and let $u_1,u_2$ be the standard basis of the character lattice $M=\Z^2$. Let $g$ be any compactly supported smooth function on $N_\R=\R^2$, let $\alpha \in \cA_{\rm ps}^{1,2}(N_\R)$ be given by $\d'u_2 \wedge \d''u_1 \wedge \d''u_2$ on $\{u_1 \geq 0\}$ and by $0$ on $\{u_1 \leq 0\}$, and let $\eta \coloneqq \trop^*(g \alpha) \in \cA_{{\rm ps},c}^{1,2}(\T^\an)$. We compute as in Example \ref{illustration for integration} that
  $$\int_{\partial \T^\an} \eta = \int_{\{u_1 = 0\}} g \,\d'u_2 \wedge \d''u_1 \wedge \d''u_2= -\int_{\{u_1 = 0\}} g\, \d u_2,$$
  where we have the usual boundary integral of a $1$-form on the right. This boundary integral is usually non-zero even though $\del\bT^\an=\emptyset$.  Note that  $\int_{\partial \T^\an} \eta=0$ if $\eta$ is smooth, however, by Stokes' Theorem~\ref{theorem of Stokes}.
\end{eg}

\section{Strong Currents and the Poincar\'e--Lelong Theorem} \label{section: strong currents}

In this section, we denote by $X$ a Hausdorff strictly analytic space $X$ of pure dimension $d$ with no boundary. Note that  $X$ is a good analytic space which is separated over $K$. We have already seen that the sheaf $\cA^{p,q}$ of weakly smooth forms on $X$ has many properties analogous to its subsheaf $\cA_{\rm sm}^{p,q}$ of smooth forms.  We now turn to the notion of currents, following~\cite[Section~4]{CLD}.  Since $\cAsmc^{p,q}\subset\cA_c^{p,q}$, 
we call the dual notion of currents on $\cA_c^{p,q}$ ``strong currents'' to distinguish them from the currents on $\cAsmc^{p,q}$ in the sense of Chambert-Loir and Ducros. Note that the restriction of a strong current to $\cAsmc^{p,q}$ is a current, hence strong currents can be seen as a generalization of currents.

\begin{art}[Topology on the Space of Forms] \label{topology on forms}
  As in~\cite[4.1.1]{CLD}, replacing smooth tropicalization maps by harmonic tropicalization maps, we define a (locally convex) topology on the space $\cA_c^{p,q}(X)$ of compactly supported weakly smooth $(p,q)$-forms on $X$. Roughly speaking, convergence means that the supports are covered by finitely many compact strictly analytic subdomains $V_i$ and the strictly analytic subdomains tropicalize the forms such that all the derivatives of the corresponding Lagerberg forms converge uniformly.
	Then the topology on the subspace of compactly supported smooth forms on $X$ as defined in \cite[4.1.1]{CLD} is just the subspace topology induced by the inlusion $\cAsmc^{p,q}(X)\subset\cA_c^{p,q}(X)$.
\end{art}

\begin{art}[Strong Currents] \label{currents}
	A \emph{strong current $S$ on $X$ of bidegree $(p,q)$} is defined as a continuous  linear functional  $S\colon \cA_c^{p,q}(X)\to \R$.  As usual we write $\angles{\omega,S} = S(\omega)$ for $\omega\in\cA_c^{p,q}(X)$. We denote by $\cD_{p,q}(X)$ the space of strong currents on $X$. Note that the restriction of $S$ to $\cA_{{\rm sm},c}^{p,q}(X)$ is a current in the sense of \cite[4.2]{CLD}, i.e.~we have a canonical linear map $\cD_{p,q}(X) \to \cDsm_{p,q}(X)$.  

	We set $\cD(X) \coloneqq \bigoplus_{p,q} \cD_{p,q}(X)$.
	As in \cite[4.2]{CLD}, one shows that $\cD(X)$ is an $\cA(X)$-module and that there is a direct image of currents with respect to proper morphisms.
	The presheaf $U\mapsto \cD(U)$ of strong currents on open subsets $U$ of $X$ is a sheaf on $X$.
\end{art}

\begin{eg}[Current of Integration] \label{current of integration}
	The \emph{current of integration}  $\delta_X \in \cD_{d,d}(X)$ is defined by
	$$\langle \eta , \delta_X \rangle  \coloneqq \int_X \eta.$$
	It is shown in \cite[4.5]{CLD} that the irreducible components $X_i$ of $X$ come with natural multiplicities $m_i$ defined similarly as in algebraic geometry, and that
	$\delta_X = \sum_i m_i \delta_{V_i}$.

	More generally, a cycle $Z$ on $X$ of dimension $r$ is a locally finite formal sum $\sum_i m_i Z_i$ of irreducible Zariski closed subsets $Z_i$ of dimension $r$ and we define
	$$\langle \eta , \delta_Z \rangle  \coloneqq \sum_i m_i \int_{Z_i} \eta,$$
	leading to a strong current $\delta_Z \in \cD_{r,r}(X)$.
\end{eg}

\begin{art}[Forms as Currents] \label{current associated to a form}
	For any $\omega \in \cAps^{p,q}(X)$, the \emph{associated current} $[\omega]\in \cD_{d-p,d-q}(X)$ is defined by
	$$\langle \eta, [\omega] \rangle  \coloneqq \langle \eta,\omega \wedge \delta_X \rangle \coloneqq \int_X \omega \wedge \eta.$$
	By~\secref{forms as currents}, the linear map $\cAps^{p,q}(X) \to \cD_{d-p,d-q}(X)$, $\omega \to [\omega]$ is injective.
\end{art}

\begin{art}[Differentials on Currents] \label{differential on currents}
  As in \cite[4.4]{CLD}, we define differentials $\d'\colon \cD_{p,q}(X)\to \cD_{p-1,q}(X)$ and $\d''\colon \cD_{p,q}(X)\to \cD_{p,q-1}(X)$ by
	$$\langle \omega, \d'T \rangle = (-1)^{p+q+1} \langle \d'\omega, T \rangle \quad \text{and} \quad
	\langle \eta, \d''T \rangle = (-1)^{p+q+1} \langle \d''\eta, T \rangle $$
	for $T \in \cD_{p,q}(X),\omega \in \cA_c^{p-1,q}(X),$ and $\eta \in \cA_c^{p,q-1}(X)$.
	The sign is chosen such that the homomorphism $\cA \to \cD$ defined in~\secref{current associated to a form} respects the differentials $\d'$ and $\d''$; this uses the Theorem of Stokes~\ref{theorem of Stokes}.
\end{art}

\begin{art}[Current of a Meromorphic Function] \label{current of meromorphic function}
  Let $f$ be any regular meromorphic function on $X$ (recall that \defi{regular} means here that $f$ is a nonzero quotient of analytic functions that restrict to non-zero divisors on the stalks). We denote by $\divisor(f)$  the associated Weil divisor on $X$. Note that the support $|\div(f)|$ is a Zariski closed subset of $X$ of codimension at least $1$, as it is also the support of a principal Cartier divisor. By~\secref{support}, we have $\supp(\omega)\cap|\div(f)| = \emptyset$ for any $\omega \in A^{p,q}(X)$ if $p+q \geq 2d-1$. It follows that
	$$\langle \omega, [\log|f|] \rangle \coloneqq \int_X \log|f| \,\omega$$
	is well-defined for any $\omega \in A_c^{d,d}(X)$, since $\log|f|$ is smooth on $\supp(\omega)$.  One shows as in \cite[Lemma~4.6.1]{CLD} that $ [\log|f|] \in \cD_{d,d}(X)$.
\end{art}

We have the following strengthening of the \emph{Poincar\'e--Lelong formula}, as proved by Chambert-Loir and Ducros for the space of currents $\cDsm(X)$.

\begin{thm}[Poincar\'e--Lelong Formula] \label{Poincare-Lelong}
Let $f$ be a regular meromorphic function on $X$. Then we have the identity $\d'\d''[\log|f|]= \delta_{\divisor(f)}$ of strong currents in $\cD_{d-1,d-1}(X)$.
\end{thm}

\begin{proof}
We note that the Poincar\'e--Lelong formula in \cite[Th\'eor\`eme~4.6.5]{CLD} shows that this identity holds if we restrict to $[\log|f|]$ to $\cAsm^{d,d}(X)\subset\cA^{d,d}(X)$, so we have to show that it still holds for weakly smooth forms. The arguments are the same as in the proof of \cite[Th\'eor\`eme~4.6.5]{CLD}; we just indicate the small adjustments.

The result is local in $X$ and linear in $f$, so we may assume that $f \in \Ocal(X)$. We denote the corresponding morphism $X \to \A^{1,\an}$ also by $f$. As in the proof of \textit{ibid}, we may assume that $X$ satisfies the property  $S_1$ of Serre and that the morphism $f\colon X \to \A^{1,\an}$ is purely of relative dimension $d-1$.
Let $F\coloneqq -\log|f|$ be the corresponding smooth tropicalization map $X \setminus \{f=0\} \to \R$.

Crucial for the proof is the following result given in \cite[Proposition~4.6.6]{CLD}, which we adapt to our situation as follows. Let $W$ be a strictly affinoid domain in $X$ with a harmonic tropicalization $g \colon W \to \R^n$. Then $h \coloneqq (F,g)\colon (W\setminus\{f=0\}) \to \R \times \R^n$ is also a harmonic tropicalization.
For $t \in \A^{1,\an}$, we define $W_t \coloneqq f^{-1}(t)\cap W$ of $W$. For a closed interval $I \subset (0,\infty)$, we consider the analytic subdomain $W_I \coloneqq |f|^{-1}(I)$ of $W$.  By \cite[Th\'eor\`eme~3.2]{Ducros}, we get $(\Z,\R)$-polytopal sets $g(W_t)$ and $h(W_I)$ of dimension at most $d-1$ and $d$ respectively.

Let $B_r$ be the closed ball of radius $r$ in $\A^{1,\an}$ centered at the origin. We claim that there is $r >0$
with the following equalities of weighted polytopal complexes:
\begin{enumerate}
\item \label{claim A}
For all $t \in B_r$, we have $g(W_t)_{d-1} = g(\divisor(f))_{d-1}$.
\item \label{claim B}
For all closed intervals $I \subset (0,r]$ and  using the Cartesian product of weighted polytopal complexes, we have $(-\log(I),1) \times g(\divisor(f|_W))_{d-1}=h(W_I)_d$.
\end{enumerate}
This is shown in \cite[Proposition~4.6.6]{CLD} in case of a smooth tropicalization map $g$.
Since a harmonic tropicalization is piecewise linear, there is
a finite $\rG$-covering $(W_j)_{j \in J}$ of $W$ by compact strictly analytic subdomains such that $g|_{W_j}$ is a smooth tropicalization map. Using the telescoping sum from Proposition~\ref{tropical decomposition}, we see that the general case follows from the smooth case.

The remaining part of the proof of the Poincar\'e--Lelong formula works as in the proof of \cite[Theorem~4.6.5]{CLD} replacing smooth tropicalization maps by harmonic tropicalization maps.
\end{proof}

\section{Dolbeault Cohomology} \label{section: Dolbeault cohomology}

Let $X$ be a good strictly analytic space of dimension $d$ over a non-trivially valued non-Archimedean field $K$. To introduce Dolbeault cohomology groups, the following result is crucial. Its proof  will be completely analogous to the $\d'$-Poincar\'e Lemma  for smooth forms on Berkovich spaces given in \cite[Theorem~4.5]{Jell}; both follow easily from the second author's tropical Poincar\'e lemma. Note that by applying the Lagerberg involution $J$, every result for $\d''$ has a natural counterpart for $\d'$ after switching the role of $p$ and $q$.

\begin{thm} \label{Poincare Lemma in the stalks}
  For all $p\geq 0$, the differential operator $\d''$ on the sheaf $\AS$ induces a complex
	\begin{align*}
	0 \to \AS^{p,0} \to \dots \to \AS^{p,d} \to 0,
	\end{align*}
	of  sheaves on $X$ which is exact at $\AS^{p,q}$ for all  $q > 0$.
\end{thm}
\begin{proof}
	By definition of $\d''$, this is a complex. To prove exactness, we pick a $\d''$-closed element in the stalk of $\AS^{p,q}$ at $x \in X$ represented by a  weakly smooth form $\omega \in\cA^{p,q}(U)$ for an open neighbourhood $U$ of $x$. Shrinking $U$ and using the definition of weakly smooth forms, we may assume that $W \coloneqq \overline{U}$ is a compact strictly analytic domain of $X$,
	that $\omega \in\cA^{p,q}(W)$ with $\d''\omega=0$,
	and that there is a harmonic tropicalization $h \colon W \to \R^n$ and a Lagerberg form $\alpha \in\cA^{p,q}(h(W))$ with $h^*(\alpha)= \omega$. We have seen in Proposition~\ref{harmonic lift to weakly smooth forms} that $h^*$ is injective, so $\d''\alpha=0$. Since $h(W)$ is the support of a $\z$-polytopal complex $\Pi$, the union $\Omega$ of the relative interiors of all $\sigma \in \Pi$ with $h(x) \in \sigma$ is a star shaped open neighbourhood of $h(x)$ in $h(W)$. By \cite[Lemmas~3.5 and~3.7]{Jell}, $\Omega$ is a polyhedrally star shaped open subset of   $h(W)$. The tropical Poincar\'e lemma  \cite[Theorem~2.18]{Jell} shows that there is $\beta \in\cA^{p,q-1}(\Omega)$ with $\alpha=\d''\beta$ on $\Omega$. We conclude that $\omega= h^*(\alpha)=\d''h^*(\beta)$ on $h^{-1}(\Omega) \cap U$, proving exactness at  $\AS_x^{p,q}$.
\end{proof}

\begin{defn} \label{definition Dolbeault cohomology}
	Applying the global section functor to the complex in Theorem~\ref{Poincare Lemma in the stalks}, we get the \emph{Dolbeault complex}
	\begin{equation} \label{Dolbeault complex}
	0 \to \AS^{p,0}(X) \to \dots \to \AS^{p,d}(X) \to 0.
	\end{equation}
	The resulting cohomology is called the \emph{Dolbeault cohomology} and the cohomology groups are denoted by $H^{p,q}(X)$ for any $p,q \in \N$.
\end{defn}

The following result and its proof are now analogous to \cite[Corollary~4.6]{Jell}

\begin{cor} \label{identification with singular cohomology}
	Assume that $X$ is a good paracompact  strictly analytic space over $K$. For $p \in N$, let $\Lcal^p$ be the kernel of $\d''\colon\AS^{p,0}\to \AS^{p,1}$. Then $H^{p,q}(X)$ is naturally isomorphic to the sheaf cohomology $H^q(X,\Lcal^p)$. For $p=0$, this means that $H^{0,q}(X)$ is naturally isomorphic to the singular cohomology $H_{\rm sing}^q(X,\R)$.
\end{cor}

\begin{proof}
	Since $X$ admits smooth partitions of unity \cite[Propostion~3.3.6]{CLD}, we conclude that the sheaves $\AS^{p,q}$ in the complex in Theorem~\ref{Poincare Lemma in the stalks} are fine and hence acyclic. This proves the first claim. For $p=0$, we have $\Lcal^p = \underline{\R}$ and hence we get the last claim.
\end{proof}

\begin{art} \label{comparison of Dolbeault cohomologies}
	Replacing weakly smooth forms by smooth forms in \eqref{Dolbeault complex}, we obtain a Dolbeault complex for smooth differential forms on $X$; we denote the corresponding cohomology by $H_{\rm sm}^{p,q}$. The  inclusion $\cA_{\rm sm}^{p,q} \to \cA^{p,q}$ induces a homomorphism between the corresponding Dolbeault complexes, so we get a linear map $H_{\rm sm}^{p,q}(X) \to H^{p,q}(X)$ for every $p,q \in \N$.
\end{art}

\begin{rem} \label{Liu's cycle map}
	Using~\secref{comparison of Dolbeault cohomologies}, it is clear that we also have \emph{Liu's cycle class map} for the Dolbeault cohomology based on weakly smooth forms. The details are as follows.

	Let $Y$ be a smooth separated scheme of finite type over $K$. We denote by $Y^\an$ the Berkovich space of $Y$ obtained by analytification \cite[3.4]{BerkovichSpectral} and let $\CH^p(Y)$ be the Chow group of $Y$ graded by codimension. In \cite[Section~3]{Liu}, Liu constructed a tropical cycle class map ${\rm cl}_{\trop}\colon\CH^p(Y)\to H_{\rm sm}^{p,p}(Y^\an).$ Composing with the linear map between the Dolbeault cohomologies from~\secref{comparison of Dolbeault cohomologies}, we get the desired \emph{cycle class map}
	$${\rm cl}\colon\CH^p(Y) \stackrel{{\rm cl}_{\trop}}{\longrightarrow} H_{\rm sm}^{p,p}(Y^\an) \longrightarrow  H^{p,p}(Y^\an).$$
	Note that ${\rm cl}_{\trop}\colon \CH^\bullet(Y) \to H_{\rm sm}^{\bullet,\bullet}(Y^\an)$ is a homomorphism of graded rings and hence the same is true for the cycle class map
	${\rm cl}\colon\CH^\bullet(Y) {\longrightarrow}H^{\bullet,\bullet}(Y^\an)$.
\end{rem}

We have the following generalization of  \cite[Theorem~3.7]{Liu} to weakly smooth forms.

\begin{thm} \label{Liu's cycle class theorm}
	Let $Y$ be a separated smooth scheme of pure dimension $d$. Let $Z$ be a cycle of codimension $p$ in $Y$ and let $\omega \in A_c^{n-p,n-p}(Y^\an)$ with $\d''\omega=0$. Then we have
	\begin{equation} \label{Liu integration}
	\int_{Z^\an} \omega = \int_{Y^\an} {\rm cl}(Z) \wedge \omega.
	\end{equation}
\end{thm}

\begin{proof}
	Liu proved this for smooth forms $\omega$ in  \cite[Theorem~3.7]{Liu}. 	We show here how his proof has to be adjusted to get the identity for weakly smooth $\omega$. We replace the sheaf of currents $\cDsm$ by the sheaf $\cD$ of strong currents on $Y^\an$. We have seen that all of the tools used in in the proof of~\cite[Theorem~3.7]{Liu} are available also for weakly smooth forms: most importantly, Stokes's Theorem~\ref{theorem of Stokes} and the Poincar\'e--Lelong Formula~\ref{Poincare-Lelong}.  Liu's proof in \textit{ibid.}\ thus works \textit{mutatis mutandis}.
\end{proof}

\begin{art}[Failure of the K\"unneth Formula] \label{Liu and degree}
	If $Y$ is a smooth proper scheme over $K$ of pure dimension $d$ and $Z$ is a $0$-cycle in $Y$, then applying \eqref{Liu integration} to the constant function $\omega$, we get
	\begin{equation} \label{Liu degree}
	\deg(Z)= \int_{Y^\an}{\rm cl}(Z)
	\end{equation}
	as in \cite[Corollary~3.10]{Liu}.
	Let ${\rm NS}^{p}(Y)$ be the quotient of $\CH^p(Y)$ modulo numerical equivalence. Then we have
	\begin{equation} \label{Neron-Severi group lower bound}
	\dim_\R(H^{p,p}(Y^\an)) \geq \dim_\Q({\rm NS}^{p}(Y)_\Q).
	\end{equation}
	Liu has shown this for the tropical Dolbeault cohomology \cite[Corollory~3.11]{Liu}, and the same arguments using \eqref{Liu degree} give \eqref{Neron-Severi group lower bound}. As in~\cite[Example~3.12]{Liu}, this gives a counterexample for the K\"unneth formula in Dolbeault cohomology by considering $Y=C \times C$ for an irreducible smooth projective  curve $C$ of genus $g \geq 1$, using Theorem~\ref{intro: Hodge diamond}.
\end{art}

\section{Galois action on forms} \label{section: Galois action}

We consider a good strictly analytic space $X$ over a non-trivially valued field $K$.
Let $K'/K$ be a finite Galois extension with Galois group $G$, and let $X'$ be the base change of $X$ to $K'$.

\begin{art}[Galois Action on Forms] \label{action on forms}
	We first explain how $G$ operates on weakly smooth forms $\omega' \in \cA^{p,q}(X')$. Recall that $\omega'$ is given by an open covering $\{U'\}_{U' \in I}$ of $X'$  and by a family $(h_{U'},\alpha_{U'})_{U' \in I}$ where each $\overline{U'}$ is a compact  strictly analytic domain with a harmonic tropicalization  $h_{U'}\colon \overline{U'} \to \R^{n_{U'}}$ and where $\alpha_{U'} \in \cA^{p,q}(\R^{n_{U'}})$. We require compatibility for the tuples $(h_{U'},\alpha_{U'})$ on intersections.

	For $\sigma \in G$, we define $\sigma^*(\omega') \in \cA^{p,q}(X')$ by the open covering $\{\sigma^{-1}(U')\}_{U' \in I}$ of $X'$ and by the family $(h_{U'} \circ \sigma ,\alpha_{U'})_{U' \in I}$. Obviously, the compatibility conditions are again satisfied and one checks that this gives a well-defined weakly smooth form.
\end{art}

\begin{prop} \label{Galois invariant forms}
  The above $G$-action and pullback give $\cA^{p,q}(X)\cong \cA^{p,q}(X')^G $.
\end{prop}

\begin{proof}
	Let $\pi\colon X' \to X$ the structure map. Base change of forms gives a natural homomorphism $\pi^* \colon \AS_X \to (\pi_*(\AS_{X'}))^G$ which is clearly a monomorphism. We claim that
	\begin{equation} \label{sheaf invariance}
	\AS_{X} \cong (\pi_*(\AS_{X'}))^G.
	\end{equation}
	We have to show that the homomorphism between the stalks in $x$ is bijective.  This is local in $x$ and so we may assume that $X$ is strictly affinoid.
	By \cite[Proposition~1.3.5]{BerkovichSpectral}, we have $X \cong X'/G$.
	Injectivity on stalks in $x$ follows from Proposition~\ref{harmonic lift to weakly smooth forms}.

	It remains to prove surjectivity.
	We pick a $(p,q)$-form $\omega'$ in the stalk of $(\pi_*(\AS_{X'}))^G$ at $x$. We have to show that $\omega'$ is  given by a weakly smooth form in a neighbourhood of $x$.
	Choose $x' \in \pi^{-1}(x)$, let $H$ be the stabilizer of $x'$, and let $R$ be a system of representatives of $G/H$ in $G$.
        There is a family of open neighbourhoods $U_\rho$ of $\rho(x')$ for $\rho\in R$ whose closures $\bar U_\rho$ are pairwise disjoint strictly affinoid domains, such that $\omega'|_{U_\rho}$ is given by a preform $(h_\rho,\alpha_\rho')$ for harmonic tropicalizations $h_\rho\colon\bar U_\rho\to\R^{n_\rho}$ and forms $\alpha_\rho'\in\cA^{p,q}(\R^{n_\rho})$.  Passing to smaller $U_\rho$'s, we may assume that $\sigma(U_\rho)= U_{\sigma \rho}$ for all $\sigma \in G$ and all $\rho \in R$. Moreover, $G$-invariance of $\omega'$ means that we may assume
	\begin{equation} \label{consequence of G-invariance}
	h_{\sigma \rho}^*(\alpha_{\sigma \rho}')=(h_\rho \circ \sigma^{-1})^*(\alpha_\rho')
	\end{equation}
	for all $\sigma \in G$ and all $\rho \in R$. We now define a harmonic tropicalization
	$$h'\colon \bigcup_{\rho \in R} {\bar U_\rho} \longrightarrow \prod_{\mu \in R} \R^{n_\mu}$$
	given on ${\bar U_\rho}$ by $\prod_{\mu \in R} h_\mu \circ \mu \circ \rho^{-1}$.
	Since $\bigcup_{\rho \in R} {U_\rho}$ is $G$-invariant, there is an open subset $U$ of $X$ with $\pi^{-1}(U) =\bigcup_{\rho \in R} {U_\rho}$.
	Passing to a smaller affinoid neighbourhood of $x$, we may assume that the closure $\bar U$ of $U$ is strictly affinoid.
	By construction, the harmonic tropicalization $h'$ is $G$-invariant, so Proposition~\ref{infinite Galois extensions}  shows that there is a $\Q$-harmonic map 
	$h\colon \bar U \to \prod_{\mu \in R} \R^{n_\mu}$
	with $h'=h \circ \pi$. 
	Replacing all $(h_\rho,\alpha_\rho')$ by $(mh_\rho,[\frac{1}{m}]^*\alpha_\mu')$ for a suitable non-zero integer $m$, we may assume that $h$ is $\Z$-harmonic and is hence a harmonic tropicalization map. 
	Let $p_\mu\colon \prod_{\mu \in R} \R^{n_\mu} \to \R^{n_\mu}$ be the projection to the $\mu$-th factor. It follows from \eqref{consequence of G-invariance} that the restriction of $p_\mu^*(\alpha_\mu')$ to $h'(\bigcup_{\rho \in R} {\bar U_\rho})$ is independent of the choice of $\mu$. For any $\mu \in R$, this shows that
	the weakly smooth form $\omega \coloneqq h^*(p_\mu^*(\alpha_\mu'))$ on $U$ satisfies $\pi^*(\omega)=\omega'$, proving surjectivity and hence the Proposition.
      \end{proof}

      The Galois action on forms respects the differential graded algebra structure on $\cA(X)$, so it induces an action on the Dolbeault cohomology groups $H^{p,q}(X)$.

\begin{cor} \label{Galois action on cohomology}
	We have $H^{p,q}(X) \cong H^{p,q}(X')^G$.
\end{cor}

\begin{proof}
	Since higher cohomology groups of a finite group are torsion~\cite[Proposition~VII.6]{serre68:corps_locaux}, it is clear that $W \to W^G$ is an exact functor on the category of $\R$-vector spaces with $G$-action. We conclude that the fixed point sets of the Dolbeault complex \eqref{Dolbeault complex} for $X'$ compute  $H^{p,q}(X')^G$, so the claim follows from Proposition~\ref{Galois invariant forms}.
\end{proof}

\begin{art}[Galois Action on Piecewise Smooth Forms]  \label{action on ps forms}
	In the same way as~\secref{action on forms}, we can define the $G$-action on the sheaf $\cAps^{\bullet,\bullet}$ of piecewise smooth forms on
	$X$. It is clear that Proposition~\ref{Galois invariant forms} holds also for piecewise smooth forms.
	By construction,
	this $G$-action is compatible with the $G$-action on the subsheaf $\cA^{\bullet,\bullet}$ of weakly smooth forms.
\end{art}

\begin{prop} \label{integration and conjugation}
	Let $X$ be a good $d$-dimensional strictly analytic space $X$ over $K$, let $K'/K$ be a finite Galois extension with Galois group $G$, and let $X'$ be the base change of $X$ to $K'$.
	\begin{enumerate}
		\item \label{first claim conjugation}
		For $\omega \in \cA_{{\rm ps},c}^{d,d}(X)$, we have $\int_{X'} \pi^*(\omega)= \int_X \omega$.
		\item \label{second claim conjugation}
		For $\omega' \in \cA_{{\rm ps},c}^{d,d}(X')$ and $\sigma \in G$, we have $\int_{X'} \sigma^*(\omega')= \int_{X'} \omega'$.
	\end{enumerate}

\end{prop}

\begin{proof}
	Suppose first that $X$ is compact and that $\omega$ is a piecewise smooth preform given by a PL tropicalization $h \colon X \to \R^n$ and a Lagerberg form $\alpha \in \cA^{p,q}(h(X))$ with $h^*(\alpha)=\omega$. Then $\pi^*(\omega)$ is a piecewise smooth preform on $X'$ given by the PL tropicalization $h' \coloneqq h \circ \pi \colon X' \to \R^n$.  Using that $h'(X')_d=h(X)_d$ as weighted polytopal complexes by Proposition \ref{base change and tropical multiplicities} and that $(h')^*(\alpha)=\pi^*(h^*(\alpha))= \pi^*(\omega)$, it is clear that both integrals in \eqref{first claim conjugation} coincide with $\int_{h(X)_d} \alpha$. Proposition~\ref{integration of piecewise smooth forms} then yields~\eqref{first claim conjugation} in general.

	To prove \eqref{second claim conjugation}, we choose a finite covering $(X_i') _{i \in I}$ of $\supp(\omega')$ by compact strictly analytic domains $X_i'$ such that  $\omega'|_{X_i'}= h_i^*(\alpha_i')$ for some $\alpha_i' \in \cA_{{\rm ps}}^{d,d}(h(X_i'))$ and a piecewise linear tropicalization $h_i \colon X_i' \to \R^{n_i}$. Then the compact strictly analytic domains $\sigma^{-1}(X_i')$ cover the support of $\sigma^*(\omega')$, and the restriction of $\sigma^*(\omega')$ to $\sigma^{-1}(X_i')$ is given by $(\sigma \circ h_i)^*(\alpha_i')$. By Proposition~\ref{integration of piecewise smooth forms}(\ref{integration of preform}), we have
	$$\int_{\sigma^{-1}(X_i')} \sigma^*(\omega)=\int_{h_i(X_i')} \alpha_i' = \int_{X_i'} \omega.$$
	The integrals in \eqref{second claim conjugation} can be computed by the inclusion-exclusion formula \eqref{general inclusion exclusion formula} as in the proof of Proposition~\ref{integration of piecewise smooth forms}, so~\eqref{second claim conjugation} follows.
\end{proof}

\section{Harmonic functions on curves} \label{section: results for curves}

Let $K$ be a non-Archimedean field with nontrivial value group $\Gamma\subset\R$, and let $\sqrt{\Gamma} \coloneqq \{t \in \R \mid \text{$\exists k \in \N_{>0}$ such that $kt \in \Gamma$}\}$ be the divisible hull.
We will show that our notion of harmonic functions introduced in Section~\ref{section: pl and harmonic functions} agrees in the case of a rig-smooth curve with the harmonic functions in Thuillier's thesis \cite{Thuillier}.  This will be important in the companion paper~\cite{GJR2}.

First, we recall some notions from the theory of graphs and their relationship to the structure of non-Archimedean analytic curves. For more details, we refer to~\cite{GJR2}.

\begin{art}[Metric Graphs with Boundary] \label{recall metric graphs}
	We  consider a weighted metric graph $(\Sigma,\partial \Sigma)$ with boundary.  This is a finite multigraph with no loop edges endowed with positive edge lengths assumed to be in $\Gamma$.
	The boundary $\partial \Sigma$ is any subset of the vertices of $\Sigma$. Every edge is endowed with the additional datum of a \defi{weight} $w(e)\in \Z_{>0}$.
	If we ignore the boundary, then we get a compact piecewise $\z$-linear space of  dimension at most $1$ as introduced in~\secref{Polyhedral geometry}.  (Note that isolated vertices are allowed.)

	A function $h\colon \Sigma \to \R$ is called \emph{harmonic} if the restriction of $h$ to any edge is affine $\R$-linear, and if for any vertex $x$ of $\Sigma \setminus \partial \Sigma$, we have $\sum_{e^-=x} w(e) d_e(h)=0$, where $e$ ranges over all edges with tail vertex $x$ and  $d_e(h)$ is the outgoing slope of $h$ at $x$ in the direction of $e$.

	For a subgroup $\Lambda$ of $\R$ containing $\Gamma$, we let $\Sigma(\Lambda)$ denote the subset of $\Sigma$ consisting of all points $x$ such that  $f(x)\in \Lambda$ for all $\z$-linear functions $f$ on an edge $e$ containing $x$. Let $U\subset\Sigma$ be a subgraph with vertices contained in $\Sigma(\sqrt\Gamma)$.
        A function $h\colon U \to \R$ is called \emph{piecewise $\z$-linear} if there is a subdivision of the subgraph $U$ such that  the vertices of the subdivision are in $\Sigma(\sqrt{\Gamma})$ and such that the restriction of $h$ to any edge $e$ of the subdivision extends to a $\z$-linear function on the edge of $\Sigma$ containing $e$.
\end{art}

\begin{art}[Skeletons of Curves] \label{strictly semistable model and skeleton}
	Metric graphs with boundary arise in non-Archimedean geometry as follows. Let $X$ be a compact  rig-smooth curve over $K$, i.e., a compact rig-smooth strictly analytic space of pure dimension $1$. We assume that $X$ has a strictly semistable model $\XF$ over the valuation ring $\kcirc$. This is a formal $\kcirc$-model of $X$ (see~\secref{Formal geometry and models}) whose special fiber $\XF_s$ is geometrically reduced with smooth irreducible components and at most double points as singularities.  This is equivalent to \cite[Definition~2.2.8]{Thuillier}, where such models are called \textit{simplement semi-stable}. To such a model is associated a weighted metric graph with boundary $\Sigma\coloneq S(\fX)$, called the \defi{skeleton} of $\fX$.  There is a canonical inclusion $\Sigma\inject X$.  The vertices of $\Sigma$ correspond to the irreducible components of $\fX_s$ and the edges correspond to the singular points of $\fX_s$, so that $\Sigma$ is the incidence graph of $\fX_s$.  The \defi{boundary} $\del\Sigma$ consists of the vertices of $\Sigma$ corresponding to non-proper irreducible components of $\fX_s$.  If $x\in\fX_s$ is a singular point with corresponding edge $e$, then the \defi{weight} of $e$ is defined to be $w(e) = [\td K(x):\td K]$.  There is an \'etale neighbourhood of $x$ of the form $\Spf(K\{X,Y\}/(XY-\varpi))$ for some $\varpi\in K^{\circ\circ}\setminus\{0\}$; the \defi{length} of $e$ is then defined to be $v(\varpi)$. There is a canonical deformation retraction map $\tau\colon X \to S(\XF)$. 
\end{art}

\begin{art}[Subdivisions of the Skeleton] \label{subdivisions of skeleton}
	Consider a subdivision $\Sigma'$ of the skeleton $\Sigma = S(\XF)$ such that all vertices of $\Sigma'$ are in $\Sigma(\sqrt{\Gamma})$. Vilsmeier \cite[Construction~2.6]{Vilsmeier} has defined a canonical formal scheme $\XF'$ with generic fiber $X$ and reduced special fiber, and a canonical morphism $\iota\colon \XF' \to \XF$ extending $\id_X$, such that for any edge $e'$ of $\Sigma'$, the analytic subdomain $(\tau)^{-1}(e')$ of $X$ is the generic fiber of an open subset of $\XF'$. Note that $\XF'$ is not necessarily an admissible formal scheme, i.e.~not necessarily of topologically finite type. By \cite[Propositions~6.7 and~6.9]{Gubler2}, the formal scheme $\XF'$ is admissible if the value group $\Gamma$ is discrete or if the vertices of $\Sigma'$ are contained in $\Sigma(\Gamma)$.
\end{art}

\begin{art}[Piecewise Affine Functions on the Skeleton] \label{piecewise affine functions}
	Let $\Sigma = S(\XF)$ as in~\secref{strictly semistable model and skeleton}, let $U$ be a subgraph of $\Sigma$ with vertices contained in $\Sigma(\sqrt\Gamma)$, and let $h\colon U \to \R$ be a piecewise $\z$-linear function in the sense of~\secref{recall metric graphs}.  Passing to a subdivision of $U$, we may assume that  $h$ is $\z$-linear on each edge of $U$. 
        The vertices of $U$ induce a subdivision $\Sigma'$ of $\Sigma$; we denote the associated formal $\kcirc$-model of $X$ by $\XF'$ as in~\secref{subdivisions of skeleton}, with canonical morphism $\iota \colon  \XF' \to \XF$. When $\Gamma$ dense in $\R$, we assume that  all vertices of $U$ are in $\Sigma(\Gamma)$; then $\fX'$ is an admissible formal scheme. 
        The subset $Y \coloneqq \tau^{-1}(U)$ is the generic fiber of a unique open formal subscheme $\YF'$ in $\XF'$. By the arguments used in the proof of \cite[Proposition~2.11]{Vilsmeier}, there is a canonical vertical Cartier divisor $D$ on $\YF'$ with associated formal metric on $\Ocal_Y$ determined by
	$-\log \|1\|_{\Ocal(D)} = h \circ \tau$.  See~\secref{model function}.
\end{art}

\begin{prop} \label{PL and strictly semistable}
  Let $X$ be a compact rig-smooth curve over $K$ and let $h\colon X \to \R$.
  \begin{enumerate}
  \item \label{factorization through the skeleton}
    If  $\fX$ is a strictly semistable model  of $X$ over $\kcirc$, then $h=h_\fL$ for a line bundle $\fL \in M(\fX)$ if and only if there is a function $f\colon S(\fX) \to \R$ which is $\z$-linear on each edge such that $h=f \circ \tau$.
  \item \label{boundary and harmonic}
    If $x \in \partial X$, then $h$ is harmonic at $x$ if and only if $h$ is $\R$-PL at $x$. 
  \end{enumerate}	
\end{prop}

\begin{proof}
The ``if'' part in (\ref{factorization through the skeleton}) follows from \secref{piecewise affine functions}. Conversely, assume that $h=h_\fL$ for  $\fL \in M(\fX)$. Using the identification $\fL|_X=\sO_X$, the meromorphic section $1$ of $\fL$ induces a vertical Cartier divisor $D$ on $\fX$. There is a covering of $\fX$ by  open subsets $\fU$ such that $S(\fU)$ is either an edge or a vertex of $S(\fX)$ and 	such that $D|_\fU = \div(\gamma)$ for some $\gamma \in \cO(U)^\times$, where $U=\fU_\eta$. It follows from \cite[Theorem 5.3 and Step 13 in its proof]{berkovic99:locally_contractible_I} that $h|_U=-\log|\gamma|$ factorizes through a $\z$-linear function on $S(\fU)$. This proves (\ref{factorization through the skeleton}).

Since $x \in \partial X$, the curve $\partial X$ is affine and hence the nef condition  in Proposition \ref{prop:def.semipositive} is vacuous.  This proves (\ref{boundary and harmonic}).
\end{proof}

\begin{art}[Thuillier's Harmonic Functions] \label{def Thuillier harmonic}
	We consider now an arbitrary rig-smooth curve $X$ over $K$.
	We call a function $h\colon X \to \R$ \emph{harmonic in the sense of Thuillier}  if every $x \in X$ has a strictly affinoid neighbourhood $V$,   a finite Galois extension $K'/K$, and a strictly semistable model $\VF'$ of $V' \coloneq V\tensor_K K'$,
	such that the pullback $h'$ of $h|_V$  factorizes through the skeleton $S(\VF')$ and such that $h'|_{S(\VF')}$ is a harmonic function on the skeleton $S(\VF')$ considered as a weighted metric graph with boundary as in~\secref{strictly semistable model and skeleton}.
	
        For $x \in \partial X$,  Proposition~\ref{PL and strictly semistable} and Corollary~\ref{cor:pl.basechange} show that the above harmonicity condition at $x$ is equivalent to piecewise $\R$-linearity at $x$, using that  the boundary of $S(\fX)$ is $\partial X$.  Note that Thuillier~\cite[Section~2.3]{Thuillier} used this definition only for strictly affinoid or boundaryfree $X$, so the above definition is slightly more general than his.
\end{art}

\begin{prop}\label{Thuillier's harmonic}  \label{comparison harmonic in general}
	Let $X$ be a rig-smooth curve. Then a function $h\colon X \to \R$ is harmonic in the sense of Definition~\ref{definition harmonic} if and only if $h$ is harmonic in the sense of Thuillier.
\end{prop}

\begin{proof}
	The claim is local on $X$, so we may assume $X$ strictly affinoid. We first show that the harmonicity notions agree for a piecewise linear function $h\in\PL(X)$. By~\secref{model theorem}, we have that $h=-\log \|1\|_{\LS}$ for a model ${\LS}$ of $\OS_X$ on a formal $\kcirc$-model $\fX$ of $X$. Using
	base change, Proposition~\ref{prop:harmonic.properties} and \cite[Remarque~2.3.17]{Thuillier}, we are free to pass to a finite Galois extension of $K$.
	Arguing locally, the semistable reduction theorem (\cite[Th\'eor\`eme~2.3.8]{Thuillier} under these hypotheses; see also~\cite[Section~7]{BL} and~\cite[Theorem~5.1.14(iv), 6.1.3]{ducros14:structur_des_courbes_analytiq}) shows that we may assume that $\fX$ is a strictly semistable model.
	By Proposition \ref{PL and strictly semistable}, there is a function $f$ on the skeleton $S(\fX)$ which is $\z$-linear on each edge such that $h=f \circ \tau$.
	For $x \in  S(\fX)\setminus \partial X$ with corresponding proper irreducible component $C \coloneqq \overline{\red_{\fX}(x)}$ of $\fX_s$, it follows as in \cite[Theorem~2.6]{KRZ1} that the sum of outgoing slopes of $f$ at $x$ is equal to $\deg(\LS|_C)$.
	We conclude that $h$ is harmonic in the sense of Thuillier if and only  if ${\LS}$ is numerically equivalent to $0$. The latter is equivalent to harmonicity of $h$ in the sense of Definition~\ref{definition harmonic}, as we have shown in Lemma~\ref{lem:harmonic.global}.

	Now we prove the proposition for any function $h\colon X \to \R$.
	Assume that $h$ is a harmonic function on $X$ as in Definition~\ref{definition harmonic}. By Proposition~\ref{prop: equivalence of harmonicity definitions}, the function $h$ is locally an $\R$-linear combination of  $\Z$-harmonic functions, so we are reduced to the case handled above since the set of Thuillier-harmonic functions forms a real vector space.

	Conversely, let $h$ be a harmonic function in the sense of Thuillier.
	We have to check for any $x \in X$ that $h$ is harmonic at $x$.
        As above, we are free to pass to a finite Galois extension of $X$, so we may extend scalars and shrink $X$ to assume that $X$ has a strictly semistable model  $\fX$ and that $h$ factorizes through a harmonic function $f\colon S(\fX)\to\R$.	  It is enough to show that $f$ is an $\R$-linear combination of piecewise $\z$-linear harmonic functions $f_i$ on $S(\fX)$.  Indeed, by Proposition~\ref{PL and strictly semistable}, we have that $f_i\circ\tau$ is piecewise linear, so $f_i\circ\tau$ is harmonic by the the first paragraph of this proof, and hence $h$ is harmonic.

        Suppose first that $x\in\del X = \del S(\fX)$.  After shrinking $\fX$, we may assume that $S(\fX)$ is a star around the vertex $x$, so that $f$ is affine $\R$-linear on each edge.   Translating by a constant function, we may assume $f(x) = 0$.  Then $f$ is an $\R$-linear combination of piecewise $\Z$-linear functions $f_e$ on $S(\fX)$: namely, the functions $f_e$ that have slope $1$ on a given edge $e$ and slope $0$ on the others.  Since  $x\in\del S(\fX)$, Proposition \ref{PL and strictly semistable} yields that each $f_e$ is harmonic and hence $h$ is harmonic.

	If $x \in X \setminus \partial X$, then as above, we may assume that $S(\fX)$ is either an edge containing $\tau(x)$ in its interior, or that $S(\fX)$ is a star around the vertex $\tau(x)$.  In the first case, $f$ is an $\R$-multiple of a $\z$-linear function on the edge, so $h$ is harmonic. 	In the second case, we may assume that $f(\tau(x))=0$.  Then $f$ is harmonic  if and only if $f$ is linear on each edge and the weighted sum of the outgoing slopes at $\tau(x)$ is zero. This is a  linear equation for the slopes which is defined over $\Q$, so there is a basis of solutions defined over $\Q$ and even over $\Z$.	The elements $f_i$ of such a basis induce piecewise $\z$-linear functions $f_i$ which are harmonic on $S(\fX)$.  Since $f$ is an $\R$-linear combination of the $f_i$, we have that $h$ is harmonic.
\end{proof}

Note that $\Gamma$ is either dense in $\R$ or is a discrete subgroup. In the following, we set $\Lambda \coloneqq \Gamma$ if $\Gamma$ is dense and $\Lambda \coloneqq \sqrt{\Gamma}$ (the divisible hull in $\R$) if $\Gamma$ is discrete.  When $\Gamma$ is discrete, we may rescale the valuation to assume $\Gamma=\Z$.

\begin{art} \label{setup for simple neighbourhood proposition}
  Let $X$ be a compact rig-smooth curve over $K$ with strictly semistable model $\XF$. As in~\secref{strictly semistable model and skeleton}, the skeleton   $S(\XF)$ is  viewed as  a weighted metric graph with boundary.  
  Let $U$ be a subgraph of  $\Sigma \coloneqq S(\fX)$ with vertices contained in  $\Sigma(\Lambda)$, then $\tau\inv(U)$ is a compact strictly analytic domain in $X$.
\end{art}

\begin{prop} \label{simple neighbourhoods and harmonic functions on skeletons}
  With the notation above,  let $f\colon U \to \R$ be a piecewise $\z$-linear harmonic function as in~\secref{recall metric graphs}.  Then $h \coloneq f\circ\tau\colon \tau\inv(U)\to\R$ is $\Z$-harmonic.
\end{prop}

\begin{proof}
  Replacing $U$ be a subdivision, we may assume that $f$ is affine $\z$-linear on each edge. The vertices of $U$ induce a subdivison of $S(\fX)$ and hence \secref{subdivisions of skeleton} yields a canonical formal $\kcirc$-model $\fX'$ of $X$ over $\fX$. 
  By~\secref{piecewise affine functions}, we get a canonical open formal subscheme $\YF'$ of $\XF'$ with generic fiber $Y=\tau^{-1}(U)$.  It follows from~\secref{model theorem} and~\secref{piecewise affine functions} that $h$ is piecewise linear, so by definition, it is $\Z$-harmonic if and only if it is harmonic.  After a finite, separable extension of the ground field, the graph $U$ becomes a skeleton of $Y$ with respect to some strictly semistable model (see Remark~\ref*{II-rem:non-rational-subgraph} in~\cite{GJR2}).  Thus $h$ is harmonic in the sense of Thuillier, so by Proposition~\ref{Thuillier's harmonic}, it is harmonic in the sense of Definition~\ref{definition harmonic}.
\end{proof}

\begin{rem} \label{harmonic non-smooth functions}
  Let $X$ be a \emph{smooth} curve (that is, a rig-smooth curve \emph{without boundary}), let $\sH$ be the sheaf of harmonic functions on $X$, and let $\sF$ be the subsheaf of $\R$-vector spaces generated by the germs $\log|f|$ for local invertible analytic functions $f$. Then
  Thuillier~\cite[2.3.4]{Thuillier} has shown that $\sF$ might be a strict subsheaf of $\sH$, with equality when $X$ is potentially locally isomorphic to $\P^{1,\an}$ or when $\ktilde$ is algebraic over a finite field.
	It was shown in \cite[5.2, 5.3]{wanner16:harmonic_functions} that $\sF$ is the kernel of $\d'\d'' \colon \cA_{\rm sm}^{0,0} \to \cA_{\rm sm}^{1,1}$ and that $\sF$ is the sheaf of smooth harmonic functions on $X$.  The latter follows also by Proposition~\ref{weakly smooth harmonic functions}. We conclude that there is a non-smooth harmonic function on an open subset of $X$ if and only if the inclusion $\sF \subset \sH$ is strict.
	If $\sF = \sH$, then every harmonic tropicalization map is a smooth tropicalization map, and $\cAsm=\cA$.
\end{rem}

Wanner's regularization theorem \cite[Corollary~5.4]{Wanner19}, obtained for Mumford curves or in case of an algebraic residue field over a finite field, can now be generalized as follows (compare Remark~\ref{harmonic non-smooth functions}).

\begin{prop} \label{regularization theorem}
  Let $X$ be a smooth curve (a rig-smooth curve without boundary) over a non-trivially valued non-Archimedean field $K$, and let $f$ be a subharmonic function on $X$ as defined in \cite[3.1.2]{Thuillier}. Then for any relatively compact open subset $U$ of $X$, there is a decreasing net of weakly smooth psh functions $f_i$ on $U$ converging pointwise to $f|_U$.
\end{prop}

\begin{proof}
  By Thuillier's regularization theorem \cite[Th\'eor\`eme~3.4.2]{Thuillier}, there is a decreasing net of  piecewise $\R$-linear subharmonic functions on $U$ converging pointwise to $f$, so we may assume that $f$ is piecewise $\R$-linear. (Note that piecewise $\R$-linear functions are called \defi{lisse} in  \cite[Definition~3.2.1]{Thuillier}.)
	As in the proof of \cite[Theorem~3.4.2]{Thuillier}, we may assume that $U= Y \setminus \partial Y$ for a compact strictly analytic domain $Y$ of $X$. By \cite[Proposition~3.2.4]{Thuillier}, there is a finite Galois extension $K'/K$ such that $Y$ has a strictly semistable model $\YF$ over the valuation ring of $K'$ and such that $f=F \circ \tau$ for a piecewise $\R$-linear function $F$ on the skeleton $S(\YF)$, where $\tau\colon Y \to S(\YF)$ is the retraction. Since the subharmonic functions over $K$ are the Galois-invariant subharmonic functions over $K'$ and the same holds for weakly smooth functions, it is enough to prove the claim over $K'$, so we may assume $K=K'$.

	Now we can follow  Wanner's arguments~\cite[Lemma~5.2]{Wanner19}. 
	Note that the curvature $\d\d^c(f|_U)$ is a discrete measure on $U$; we denote by $S$ its support.  In what follows, we allow subdivisions of $\Sigma\coloneq S(\fY)$ at arbitrary points in $\Sigma$ (not just type-II points); we do not claim that the subdivision is the skeleton of a formal model of $U$.  By subdividing the skeleton, we may assume that $\Sigma$ contains $S$, that $F$ is affine on the edges of $\Sigma$, and that the points of $S$ are vertices of $\Sigma$. For a point $x \in S$, let $\Sigma_x$ be the union of all edges of $\Sigma$ containing $x$. By further subdividing $\Sigma$, we may assume that $\Sigma_x \cap S = \{x\}$.
	Since $x$ is in the support of $\d\d^c(f|_U)$, the weighted sum of outgoing slopes of $F$ in $x$ is positive. 
	Then there is a harmonic function $G_x \leq F$ on $\Sigma_x$, affine on each edge of $\Sigma_x$, and with $F(y)=G_x(y)$ if and only if $y=x$.  For $\varepsilon >0$ sufficiently small, the regularized maximum $m_\varepsilon$ as defined in \cite[4.6]{Wanner19} leads to a function
	$$f_\varepsilon \coloneqq m_{\varepsilon/2}(G_x \circ \tau + \varepsilon,f) = m_{\varepsilon/2}(G_x+\varepsilon,F)\circ \tau   $$
	on $\tau^{-1}(\Sigma_x)$ which agrees with $f$ over a neighbourhood  of the boundary of $\Sigma_x$ and agrees with $G_x \circ \tau + \varepsilon$ over a neighbourhood of $x$ in $\Sigma_x$. We conclude that the functions $f_\varepsilon$, for varying $x \in S$, can be glued to get a function $f_\varepsilon$ defined on $Y$ which agrees with $f$ on  $V \coloneqq Y \setminus \bigcup_{x \in S} \tau^{-1}(P_x)$ for suitable compact  neighbourhoods $P_x$ of $x$ in $\relint(\Sigma_x)$. Since $f$ is harmonic on $U \cap V$ by omitting $S$, we see that $f_\varepsilon$ is harmonic on $U \cap V$ and on a neighbourhood of $S$. Since the regularized maximum of two affine functions on a segment is a convex smooth function, we see that $f_\varepsilon$ is a smooth psh function over the interior of each edge of any $\Sigma_x$. Overall, it follows that $f_\varepsilon$ is a weakly smooth psh function on $U$. In the proof of \cite[Lemma~5.2]{Wanner19}, Wanner extracted from the functions $f_\varepsilon$ a  decreasing sequence which converges pointwise to $f$, which proves the proposition.
\end{proof}

\section{Canonical tropicalizations of abelian varieties} \label{section: canonical tropicalization of abelian varieties}

Let $A$ be an abelian variety over the non-Archimedean field $K$ with non-trivial valuation $v$. The uniformization theory of Bosch, L\"utkebohmert and Raynaud yields that after replacing $K$ by a finite separable extension, there is a canonical exact sequence
\[
0 \longrightarrow \torus \longrightarrow E
\stackrel{q}{\longrightarrow} B \longrightarrow 0
\]
of algebraic groups such that $A^{\rm an}= E^{\rm an}/\Lambda$ for a  discrete subgroup
$\Lambda$ of $E(K)$.
Here $B$ is the generic fiber of an abelian scheme $\sB$ over $\kcirc$
and $\torus$ is a split torus over $K$
(see \cite{bosch-luetkebohmert1991} and \cite[6.5]{BerkovichSpectral}).
Note that the quotient homomorphism  $p\colon E^{\rm an} \to A^{\rm an}$
is only an analytic morphism, but the \emph{Raynaud extension} $E$ is algebraic
by the GAGA principle, and the quotient homomorphism $q\colon E \to B$
is locally trivial.

\begin{art}[Canonical Tropicalizations] \label{canonical tropicalization}
	We denote the character lattice of the split torus $\T$ by $M$ and we
	set $N \coloneqq \Hom(M,\Z)$.
	For $u \in M$, the pushout of the Raynaud extension with respect to
	the character  $\chi_u\colon \torus \to \bGm$ gives rise to a
	translation invariant extension of $B$ by $\bGm$, and hence to
	a rigidified translation invariant line bundle $E_u$ on $B$ (see
	\cite{bosch-luetkebohmert1991} and \cite[3.2]{foster-rabinoff-shokrieh-soto}).
	Note that the line bundle $q^*(E_u)$ is trivial over $E$, and the above
	pushout construction gives a canonical frame $e_u\colon E \to q^*(E_u)$.
	Following~\cite[3.2]{foster-rabinoff-shokrieh-soto}), we define
	the \emph{canonical tropicalization map} as the unique map
	\[
	\trop\colon E^{\rm an} \longrightarrow N_\R
	\]
	satisfying
	\[
	\langle \trop(x), u \rangle = - \log q^*\|e_u(x)\|_{E_u},
	\]
	where $\|\phantom{a}\|_{E_u}$ is the canonical metric on the rigidified
	line bundle $E_u$ which is the formal metric associated to the unique (up to isomorphism) formal $\kcirc$-model of $E_u$ on the formal completion $\fB$  of $\sB$ \cite[3.1]{foster-rabinoff-shokrieh-soto}.
	For $x \in \T^{\rm an}$, we have
	$\langle \trop(x),u \rangle = v \circ \chi^u(x)$ and hence
	$\trop$ agrees with the usual tropicalization map on the split
	torus $\T^{\rm an}$.
	Moreover, $\trop$ maps $\Lambda$ onto a complete lattice in $N_\R$
	(see \cite[Theorem~1.2]{bosch-luetkebohmert1991}).
\end{art}

\begin{prop} \label{canonical tropicalization is harmonic}
	The canonical tropicalization map of an abelian variety is a harmonic tropicalization.
\end{prop}

\begin{proof}
	We have seen in~\secref{canonical tropicalization} that the canonical tropicalization map is induced by functions $h_u \coloneqq - \log q^*\|e_u(x)\|_{E_u}$ for $u \in M$.
	Since the rigidified line bundle $E_u$ on $B$ is translation invariant, the formal model $\ES_u$ of $E_u$ on $\fB$ inducing the canonical metric $\metr_{E_u}$ of $E_u$ is numerically trivial on $\fB$. Using the theory of formal models on the paracompact space $E^\an$ as in \cite[2.2]{gubler-martin}, we conclude that $q^*\|\phantom{a}\|_{E_u}$ is the formal metric on a formal model $\fX$ of $E^\an$ with a morphism $\fX \to \fB$ extending $q$. Denoting this extension also by $q$, the metric $q^*\|\phantom{a}\|_{E_u}$ is induced by the formal model $q^*(\ES_u )$ of
	$q^*(E_u)$. Note that $q^*(E_u)$ is trivial with canonical frame $e_u$, so we may identify it with $\OS_E$. Since $q^*(\ES_u )$ is numerically trivial on $\fX$, we conclude from Lemma~\ref{lem:harmonic.global} that $h_u$ is $\Z$-harmonic, proving the proposition.
\end{proof}

\begin{rem} \label{Berkovich open splittings}
	For questions related to smoothness in the sense of Chambert-Loir--Ducros, we have to consider smooth tropicalization maps (Remark~\ref{smooth forms on X}). As canonical tropicalizations of abelian varieties are defined on $\rG$-open subsets which are not open in the analytic topology, the canonical tropicalization is not necessarily smooth. Indeed, in Section~\ref{abelian threefold} we will give 
	an example of an abelian threefold over some non-Archimedean field for which the canonical tropicalization is \emph{not} smooth.
\end{rem}

\begin{rem} \label{canonical metrics and delta-metrics}
	It was claimed in \cite[Example~8.15]{GK}	that the canonical metric on any line bundle of an abelian variety is locally the tensor product of a smooth and a formal metric. This was used in \cite[Example~9.17]{GK} to prove that such canonical metrics are $\delta$-metrics. The latter means that locally there is a smooth tropicalization map such that the first Chern current is induced by a $(1,1)$-form with piecewise smooth coefficients on the tropical chart (see \cite[9.13]{GK}). However, this was done under the incorrect assumption that canonical tropicalization maps of abelian varieties are smooth.  Therefore, it remains an open question whether canonical metrics are $\delta$-metrics. As this paper clearly shows, it is better to consider a more general notion of $\delta$-metrics replacing smoothness by weak smoothness.
\end{rem}

\section{A non-smooth canonical tropicalization map}    \label{abelian threefold}

We fix a complete algebraically closed non-Archimedean field with nontrivial valuation $v\colon K\to\R\cup\{\infty\}$.  Let $\td K$ be the residue field of $K$. It follows from \cite[Lemma~3.4.1/4]{bosch_guntzer_remmert84:non_archimed_analysis} that $\td K$ is algebraically closed.   We assume that $\td K$ is not an algebraic closure of a finite field. As usual, we denote the value group by $\Gamma$.

\begin{art}[A Theta Graph] \label{theta graph}
We consider a metric $\Theta$-graph
with edge lengths $\ell(e)$ contained in the value group $\Gamma$, as illustrated on the left side of Figure~\ref{fig:a-j}. We denote this metric graph suggestively by $\Theta$. We label the vertices by $P,Q$, and we orient the edges $a,b,c$ to start at $P$ and end at $Q$. This makes $\Theta$ into an oriented metric graph. Recall that the \defi{edge length pairing} $\langle \cdot, \cdot \rangle$ on the free abelian group generated by the edges is defined such that the edges form an orthogonal basis with $\langle e, e \rangle = \ell(e)$ for each edge $e$. It induces a positive definite pairing $\angles{\cdot,\cdot}\colon H_1(\Theta,\Z)\times H_1(\Theta,\Z)\to\R$
and hence  an embedding $\iota\colon H_1(\Theta,\Z)\to H^1(\Theta,\R)$, as in~\cite[2.7]{BRab}.
\end{art}

\begin{art}[The Jacobian of the Theta Graph] \label{Jacobian of Theta}
As in~\cite[Section~8]{MikZharII}, we consider the  Jacobian of the $\Theta$-graph   defined by
$$\Jac(\Theta) = H^1(\Theta,\R)/\iota H_1(\Theta,\Z).$$
Then $\Jac(\Theta)$ is a two-dimensional real torus which has the following explicit description. Let $\gamma_1$ be the closed path starting at $P$ and traversing $c$ and $b$, and let $\gamma_2$ be the closed path starting at $P$ and traversing $c$ and $a$.  Then $\{\gamma_1,\gamma_2\}$ forms a basis of $H_1(\Theta,\Z)$.  Choosing the dual basis gives an isomorphism $H^1(\Theta,\R)\cong\R^2$, with respect to which we have $\iota(\gamma_1) = (\ell(a)+\ell(c),\,\ell(c))$ and $\iota(\gamma_2) = (\ell(c),\,\ell(b)+\ell(c))$.  The Jacobian is the quotient of $\R^2$ by the lattice generated by $\iota(\gamma_1)$ and $\iota(\gamma_2)$.  See Figure~\ref{fig:a-j}.

Let $\alpha_P\colon\Theta\to\Jac(\Theta)$ be the Abel--Jacobi map based at $P$.  Identifying $c$ with the segment $[0,\ell(c)]$, with $0$ identified with $P$, we have
$\alpha_P(x) = (x, x)$ for $x\in c$, modulo translation by $H_1(\Theta,\Z)$.
Identifying $a$ with $[0,\ell(a)]$, with $0$ identified with $P$, we have $\alpha_P(x) = (-x+\ell(a)+\ell(c),\ell(c))$ for $x\in a$; doing likewise for the edge $b$, we have $\alpha_P(x) = (\ell(c),-x+\ell(b)+\ell(c))$.
\end{art}

\begin{figure}[ht]\label{fig:a-j}
	\centering
	\begin{tikzpicture}
	\node[point, "$P$" below] (P) at (0, 0) {};
	\node[point, "$Q$" above] (Q) at (0, 2) {};
	\draw[thick] (P) edge["$c$" right] (Q);
	\draw[thick] (P) edge["$a$" right, bend left=95, looseness=2] (Q);
	\draw[thick] (P) edge["$b$" left, bend right=115, looseness=3] (Q);
	\draw[very thick] (-1.3,.2)
	edge[<-, "$\gamma_1$", bend left=40, looseness=1] (-1.3,1.8);
	\draw[very thick] (1.8,.2)
	edge[<-, "$\gamma_2$" right, bend right=30, looseness=1] (1.8,1.8);
	\node at (0, -2.7) {$\Theta$};
	\node at (7, -2.7) {$\Jac(\Theta)$};

	\draw[very thick] (3, 1) edge[->, bend left, "$\alpha_P$"] (5, 1);

	\begin{scope}[xshift=7cm, scale=.5]
	\begin{scope}[cm={3.5,1,1,3,(0,0)}]
	\draw[help lines] (-1, -1) grid (2,2);
	\end{scope}
	\node[point] at (0, 0) {};
	\node[point] at (1, 1) {};
	\draw[thick] (0,0) -- (1,1) (1,1) -- (3.5,1) (1,1) -- (1,3);
	\begin{scope}[xshift=1cm, yshift=3cm]
	\node[point] at (0, 0) {};
	\node[point] at (1, 1) {};
	\node[point] at (1, 3) {};
	\draw[thick] (0,0) -- (1,1) (1,1) -- (3.5,1) (1,1) -- (1,3);
	\end{scope}
	\begin{scope}[xshift=-1cm, yshift=-3cm]
	\node[point] at (0, 0) {};
	\node[point] at (1, 1) {};
	\draw[thick] (0,0) -- (1,1) (1,1) -- (3.5,1) (1,1) -- (1,3);
	\end{scope}
	\begin{scope}[xshift=3.5cm, yshift=1cm]
	\node[point] at (0, 0) {};
	\node[point] at (1, 1) {};
	\node[point] at (3.5, 1) {};
	\draw[thick] (0,0) -- (1,1) (1,1) -- (3.5,1) (1,1) -- (1,3);
	\end{scope}
	\begin{scope}[xshift=4.5cm, yshift=4cm]
	\node[point] at (0, 0) {};
	\node[point] at (1, 1) {};
	\node[point] at (3.5, 1) {};
	\node[point] at (1, 3) {};
	\draw[thick] (0,0) -- (1,1) (1,1) -- (3.5,1) (1,1) -- (1,3);
	\end{scope}
	\begin{scope}[xshift=2.5cm, yshift=-2cm]
	\node[point] at (0, 0) {};
	\node[point] at (1, 1) {};
	\node[point] at (3.5, 1) {};
	\draw[thick] (0,0) -- (1,1) (1,1) -- (3.5,1) (1,1) -- (1,3);
	\end{scope}
	\begin{scope}[xshift=-3.5cm, yshift=-1cm]
	\node[point] at (0, 0) {};
	\node[point] at (1, 1) {};
	\draw[thick] (0,0) -- (1,1) (1,1) -- (3.5,1) (1,1) -- (1,3);
	\end{scope}
	\begin{scope}[xshift=-2.5cm, yshift=2cm]
	\node[point] at (0, 0) {};
	\node[point] at (1, 1) {};
	\node[point] at (1, 3) {};
	\draw[thick] (0,0) -- (1,1) (1,1) -- (3.5,1) (1,1) -- (1,3);
	\end{scope}
	\begin{scope}[xshift=-4.5cm, yshift=-4cm]
	\node[point] at (0, 0) {};
	\node[point] at (1, 1) {};
	\draw[thick] (0,0) -- (1,1) (1,1) -- (3.5,1) (1,1) -- (1,3);
	\end{scope}
	\end{scope}

	\end{tikzpicture}

	\caption{The graph $\Theta$ and its Jacobian.}
\end{figure}

\begin{art}[A Metrized Complex of Curves]  \label{metrized complex and curve}
  We enrich $\Theta$ with the structure of a metrized complex of curves in the sense of~\cite[Definition~2.17]{ABBR}, as follows.  Choose an elliptic curve $E$ over $\td K$.  Let $z$ be the identity in the group law on $E(\td K)$, and pick distinct non-torsion points $x,y\in E(\td K)$.  We attach $E$ to $P$, identifying the adjacent edges $a,b,c$ with the points $x,y,z$, respectively.  We attach $\bP^1_{\td K}$ to $Q$, identifying the edges with $0,1,\infty$.  This forms a metrized complex of curves.  By~\cite[Theorem~3.24]{ABBR}, there exists a smooth, proper, connected $K$-curve $C$ with minimal skeleton identified with $\Theta$ (as a metrized complex of curves). Using \cite[4.18.1]{BPR2}, we deduce that $C$ has genus $g(C)=g(\Theta)+g(E)+g(\bP^1)=3$. By the stable reduction theorem,  $C$ has a unique stable  model $\fC$ over $K^\circ$ whose associated metrized complex is $\Theta$, and hence the irreducible components of $\fC_s$ are $E$ and $\bP_{\td K}^1$, meeting in the three double points corresponding to $x=0$, $y=1$ and $z=\infty$.  See \cite[4.16]{BPR2} for details.
\end{art}

Let $J$ be the Jacobian of $C$.  By~\cite[Theorem~2.9]{BRab}, the canonical skeleton $\Sigma$ of $J$ is naturally isomorphic to $\Jac(\Theta)$.  Moreover, choosing a point $p\in C(K)$ retracting to $P\in\Theta$ and letting $\alpha_P\colon C\to J$ be the Abel--Jacobi map based at $p$, we have a commutative square
\begin{equation*} \label{eq:intro.aj.commutes}\xymatrix @=.25in{
	{C^\an} \ar[r]^{\alpha_p} \ar[d]^\tau &
	{J^\an} \ar[d]^\tau \\
	{\Theta} \ar[r]_(.4){\alpha_P} & {\Sigma}
}\end{equation*}
under the identification $\Sigma\cong \Jac(\Theta)$, where the vertical maps are retraction onto the skeleton.  Let $T\colon C^\an\to\Sigma$ be either composition: this is the canonical tropicalization of $J$ restricted to $\alpha_p(C^\an)$.

\begin{prop} \label{non-smooth T_i}
	Choose a lift $\td T\colon C^\an\to H^1(\Theta,\R)$ of $T$ in a neighbourhood of $P\in\Theta = S(\fC)$.  Then the coordinate functions $\td T_i(\cdot) = \angles{\gamma_i, \td T(\cdot)}$ are not smooth at $P$, for $i=1,2$.
\end{prop}

\begin{proof}
  Let $Q_a,Q_b,Q_c$ be the points in the middle of the edges $a,b,c$. We subdivide the graph $\Theta$ at $Q_a,Q_b,Q_c$ and we denote the resulting subdivision by $\Theta'$. The lift
  $\td T_1$ is defined on an open subset of $C^{\an}$ containing the edges of $\Theta'$ joining $P$ with $Q_a,Q_b,Q_c$.  According to the  description in~\secref{Jacobian of Theta}, the function $\td T_1$ has slope $1$ on the edge from $P$ to $Q_a'$, slope $-1$ along the edge from $P$ to $Q_b$, and slope $0$ along the edge from $P$ to $Q_c$. We extend this to a piecewise $\z$-linear function $h$ on $\Theta'$ which is affine on all the edges. By~\secref{subdivisions of skeleton}, the subdivision $\Theta'$ induces a canonical model $\fC'$ of $C$ over the minimal model $\fC$.  Using \cite[Proposition~B.17]{gubler-hertel}, we see that $\fC'$ is strictly semistable.
  By~\secref{piecewise affine functions}, there is a unique vertical Cartier divisor $D$ on $\fC'$ such that $h \circ \tau = -\log \|1\|_{\sO(D)}$.
  Since $h$ is piecewise $\z$-linear and harmonic at $P$, we conclude using Proposition~\ref{simple neighbourhoods and harmonic functions on skeletons} that $\td T_1$ is a $\Z$-harmonic function in a neighbourhood of $P$ in $C^\an$.
  Note that the reduction of $P$ is the generic point of an irreducible component $E'$ of $\fC_s'$. Since $E'$ is a proper smooth curve and the canonical morphism $E' \to E$ is birational, it is an isomorphism, so we can identify $E'$ with $E$.
	Using the above description of the slopes of $h$, the slope formula (\cite[Theorem~2.6]{KRZ1}, \cite[Theorem~5.15]{BPR2})
	shows that
	$\sO(D)|_E$ is induced by the divisor $[z]-[x]$. Since $x$ is a non-torsion point of $E$ and $z$ is the zero element of $E$, we conclude that the class of  $[z]-[x]$
	in $\Pic^0(E)\tensor\R$ is non-zero.

	Let $\sH$ be the sheaf of harmonic functions on $C^\an$ and let $\sF$ be the subsheaf of smooth harmonic functions, as in Remark~\ref{harmonic non-smooth functions}. Thuillier~\cite[Lemme~2.3.22]{Thuillier} has shown
	that  $\sH_P/\sF_P\cong \Pic^0(E)\tensor\R$. This isomorphism maps $\td T_1$ to the class of $[z]-[x]$ by the above construction, proving that $\td T_1$ is not smooth in a neighbourhood of $P$.
	The argument for $T_2$ proceeds similarly.
\end{proof}

\begin{cor} \label{non-smooth can trop}
	Let $J$ be the Jacobian of the curve $C$ constructed in Remark~\ref{metrized complex and curve}. Then the canonical tropicalization map of $J$ is not a smooth tropicalization map.
\end{cor}

\begin{proof}
  Let $E$ be the Raynaud uniformization of $J$ with $J^\an\cong E^\an /\Lambda$ for a discrete subgroup $M$ of $E(K)$. It follows from \cite[6.5]{BerkovichSpectral} that the canonical tropicalization map $\trop$ of $J$ is the lift $E^\an \to H^1(\Theta,\R)$  of the retraction $J \to \Sigma$ to the universal covering spaces.
	We lift the Abel--Jacobi map $\alpha_p\colon C^\an \to J^\an$ in an analytic neighbourhood $U$ of $P$ in $C^\an$ to a morphism $\td \alpha_p\colon U\to E^\an$.  Then $\trop \circ \td \alpha_p $ is given
	locally at $P$ by  $\td T$. Since the pull-back of a smooth function is smooth, we deduce from Proposition~\ref{non-smooth T_i} that the canonical tropicalization map of $J$ is not a smooth tropicalization map.
\end{proof}

\bibliographystyle{egabibstyle}
\bibliography{papers}

\end{document}